\pdfoutput=1
\documentclass[a4paper,reqno]{amsart}
\usepackage[margin=3cm]{geometry}
\usepackage[T3,T1]{fontenc}

\newcommand{\ifarticle}[2]{
    \csname@ifclassloaded\endcsname{beamer}{#2}{#1}
}

\newcommand{\ifbook}[2]{
    \csname@ifclassloaded\endcsname{amsbook}{#1}{#2}
}

    \usepackage[T1]{fontenc} 
    \usepackage{xparse} 
    \usepackage{xspace} 
    \usepackage{etoolbox} 
    \usepackage{xpatch} 

    \usepackage{amsmath,amssymb} 
    \usepackage{keytheorems} 
    \allowdisplaybreaks[1]

    \usepackage{mathtools} 
    \usepackage{mathrsfs} 
    \usepackage{stmaryrd} 
    \usepackage{bbm} 
    \usepackage{quiver} 
    \usepackage{xcolor} 
    \usepackage{scalerel} 
    \usepackage{booktabs} 
    \usepackage{nameref} 
    \usepackage[british]{babel} 
    \usepackage{csquotes} 
    \usepackage[UKenglish]{isodate} 
    \usepackage{microtype} 
    \usepackage[splitrule]{footmisc} 

    \ifarticle{
        \PassOptionsToPackage{hyphens}{url}
        \usepackage[pdfusetitle,linktoc=all,bookmarksdepth=2,colorlinks,citecolor=cyan,linkcolor=purple,urlcolor=blue]{hyperref} 
        \urlstyle{rm}

        \usepackage{zref-clever} 
        \zcsetup{nameinlink,noabbrev,cap=true}
        \newcommand{\cref}{\zcref}
        \newcommand{\Cref}{\zcref[cap=true]}
        \usepackage{footnotebackref} 

        \usepackage[shortlabels]{enumitem} 
        \setlist{topsep=2pt,itemsep=2pt,partopsep=2pt,parsep=2pt} 

        \setcounter{tocdepth}{1}
    }{}
    \usepackage[backend=biber,style=alphabetic,sorting=nyt,abbreviate=false,backref=true,backrefstyle=three,maxnames=9,minalphanames=3,maxalphanames=4]{biblatex} 
    \DeclareUnicodeCharacter{0301}{\TODO[Invalid symbol.]} 

    \DeclareFieldFormat*{citetitle}{\emph{#1}} 
    \DeclareCiteCommand{\citetitle}
        {\boolfalse{citetracker}%
        \boolfalse{pagetracker}%
        \usebibmacro{prenote}}
        {\ifciteindex
            {\indexfield{indextitle}}
            {}%
        \printtext[bibhyperref]{\printfield[citetitle]{labeltitle}}}
        {\multicitedelim}
        {\usebibmacro{postnote}}
    \DeclareCiteCommand*{\citeauthor}
        {\defcounter{maxnames}{99}%
        \defcounter{minnames}{99}%
        \defcounter{uniquename}{2}%
        \boolfalse{citetracker}%
        \boolfalse{pagetracker}%
        \usebibmacro{prenote}}
        {\ifciteindex{\indexnames{labelname}}{}%
        \printnames{labelname}}
        {\multicitedelim}
        {\usebibmacro{postnote}}

    \usepackage{calc} 
    \usepackage{tabularray} 
    \usepackage{ebproof} 
    \usepackage{forest} 
    \usepackage{adjustbox} 
    \usepackage{float} 
    \usepackage{stfloats}

    \makeatletter
    \ifarticle{
        \@ifclassloaded{amsart}{
        \makeatletter
        \def\subsection{\@startsection{subsection}{2}%
          \z@{.5\linespacing\@plus.7\linespacing}{.5\baselineskip}%
          {\normalfont\centering\bfseries}}
        \makeatother
        }{}

        \makeatletter\def\@makefntext{\indent\@makefnmark}\makeatletter

        \xpretocmd{\@adminfootnotes}{\let\@makefntext\BHFN@OldMakefntext}{}{}
        \xpatchcmd{\@maketitle}{\let\@makefnmark\relax}{\let\@makefnmark\no@makefnmark}{}{}
        \def\no@makefnmark{}
        \renewcommand\@makefntext[1]{%
          \ifx\@makefnmark\no@makefnmark
            \BHFN@OldMakefntext{#1}%
          \else
            \renewcommand\@makefnmark{%
            \mbox{%
                \textsuperscript{%
                \normalfont
                \hyperref[\BackrefFootnoteTag]{\@thefnmark}%
                }%
            }\,%
            }%
            \BHFN@OldMakefntext{#1}%
          \fi
        }

        \LetLtxMacro{\BHFN@Old@footnotemark}{\@footnotemark}
        \renewcommand*{\@footnotemark}{%
            \refstepcounter{BackrefHyperFootnoteCounter}%
            \xdef\BackrefFootnoteTag{bhfn:\theBackrefHyperFootnoteCounter}%
            \label{\BackrefFootnoteTag}%
            \BHFN@Old@footnotemark
        }

        \makeatletter
        \def\paragraph{\@startsection{paragraph}{4}%
          \z@\z@{-\fontdimen2\font}%
          {\normalfont\bfseries}}
        \makeatother

        \makeatletter
        \def\subsection{\@startsection{subsection}{2}%
          \z@{.5\linespacing\@plus.7\linespacing}{.5\baselineskip}%
          {\normalfont\centering\bfseries}}
        \makeatother
    }{}
    \makeatother

    \ifarticle{
        \theoremstyle{plain}
        \ifbook{
            \newkeytheorem{theorem}[name=Theorem, parent=chapter, refname={theorem, theorems}, Refname={Theorem, Theorems}]
        }{
            \newkeytheorem{theorem}[name=Theorem, parent=section, refname={theorem, theorems}, Refname={Theorem, Theorems}]
        }

        \newkeytheorem{theorem*}[name=Theorem, numbered=false]

        \newkeytheorem{corollary}[name=Corollary, sibling=theorem, refname={corollary, corollaries}, Refname={Corollary, Corollaries}]
        \newkeytheorem{lemma}[name=Lemma, sibling=theorem, refname={lemma, lemmas}, Refname={Lemma, Lemmas}]
        \newkeytheorem{proposition}[name=Proposition, sibling=theorem, refname={proposition, propositions}, Refname={Proposition, Propositions}]

        \theoremstyle{definition}

        \newkeytheorem{axiom}[name=Axiom, sibling=theorem, qed=$\lrcorner$, refname={axiom, axioms}, Refname={Axiom, Axioms}]
        \newkeytheorem{conjecture}[name=Conjecture, sibling=theorem, qed=$\lrcorner$, refname={conjecture, conjectures}, Refname={Conjecture, Conjectures}]
        \newkeytheorem{construction}[name=Construction, sibling=theorem, qed=$\lrcorner$, refname={construction, constructions}, Refname={Construction, Constructions}]
        \newkeytheorem{definition}[name=Definition, sibling=theorem, qed=$\lrcorner$, refname={definition, definitions}, Refname={Definition, Definitions}]
        \newkeytheorem{example}[name=Example, sibling=theorem, qed=$\lrcorner$, refname={example, examples}, Refname={Example, Examples}]
        \newkeytheorem{notation}[name=Notation, sibling=theorem, qed=$\lrcorner$]
        \newkeytheorem{note}[name=Note, sibling=theorem, qed=$\lrcorner$, refname={note, notes}, Refname={Note, Notes}]
        \newkeytheorem{remark}[name=Remark, sibling=theorem, qed=$\lrcorner$, refname={remark, remarks}, Refname={Remark, Remarks}]

        \newenvironment{sketch}{\proof}{\endproof}
    }{}

    \ExplSyntaxOn
    \clist_map_inline:nn { axiom, conjecture, construction, corollary, definition, example, lemma, notation, note, proposition, remark, theorem } {
        \AtBeginEnvironment{#1} {
            \zcsetup{countertype={enumi=#1, enumii=#1}}
            \setlist[enumerate,1]{ref={\csname the#1\endcsname.(\arabic*)}}
            \setlist[enumerate,2]{ref={\csname the#1\endcsname.(\arabic*).(\alph*)}}
        }
    }
    \ExplSyntaxOff

    \newcommand{\qedshift}{\vspace*{-\baselineskip}}

    \ExplSyntaxOn
    \NewDocumentCommand{\mathcommand}{mO{0}m}
     {
      \exp_args:Nc \NewCommandCopy {khue_kept_\cs_to_str:N #1} { #1 }
      \exp_args:Nc \newcommand {khue_new_\cs_to_str:N #1}[#2]{#3}
      \DeclareDocumentCommand {#1} {}
       {
        \mode_if_math:TF
         {
          \use:c {khue_new_\cs_to_str:N #1}
         }
         {
          \use:c {khue_kept_\cs_to_str:N #1}
         }
       }
     }
    \ExplSyntaxOff

    \newenvironment{iffseq}{%
        \global\let\externaldblbackslash\\
        \[\begin{array}{cl}
        \ifundef{\internaldblbackslash}{%
            \global\let\internaldblbackslash\\%
            \gdef\\{\internaldblbackslash\cmidrule{1-1}\morecmidrules\cmidrule{1-1}}%
        }{}
    }{%
        \end{array}\]
        \global\undef\internaldblbackslash
        \global\let\\\externaldblbackslash
    }

    \newsavebox\tikzcdbox

    \newcommand{\nbh}{\nobreakdash-\hspace{0pt}}

    \mathcommand{\h}{\textup{-}}

    \newcommand{\tx}{\mathrm}
    \mathcommand{\b}{\mathbf}
    \newcommand{\s}{\mathsf}
    \newcommand{\cl}{\mathcal}
    \mathcommand{\bb}{\mathbb}
    \DeclareMathAlphabet{\bbn}{U}{bbold}{m}{n}
    \newcommand{\dc}[1]{\TextOrMath{double category\xspace#1}{\b{\bb#1}}}
    
    \mathcommand{\sf}{\mathsf}
    \mathcommand{\u}{\underline}
    \mathcommand{\o}{\overline}

    \newcommand{\TODO}[1][TODO]{\textcolor{orange}{\textup{#1}}\xspace}

    \newcommand{\flip}[1]{\text{\rotatebox[origin=c]{-180}{$#1$}}}
    
    \newcommand{\datetoday}{\date{\cleanlookdateon\today}}

    \newcommand{\defeq}{\mathrel{:=}}
    \mathcommand{\d}{\mathbin{;}}
    \mathcommand{\c}{\circ}
    \newcommand{\ph}[1][]{{({-}_{#1})}}
    
    \newcommand{\iso}{\cong}
    \renewcommand{\equiv}{\simeq}
    \newcommand{\biequiv}{\sim}

    \newcommand{\from}{\leftarrow}
    \newcommand{\xto}{\xrightarrow}
    
    \newcommand{\tto}{\Rightarrow}

    \newcommand{\ffto}{\hookrightarrow}
    
    \newcommand{\epito}{\twoheadrightarrow}
    
    \newcommand*\cocolon{%
            \nobreak
            \mskip6mu plus1mu
            \mathpunct{}%
            \nonscript
            \mkern-\thinmuskip
            {:}%
            \mskip2mu
            \relax
    }

    \makeatletter
    \def\slashedarrowfill@#1#2#3#4#5{%
    $\m@th\thickmuskip0mu\medmuskip\thickmuskip\thinmuskip\thickmuskip
    \relax#5#1\mkern-7mu%
    \cleaders\hbox{$#5\mkern-2mu#2\mkern-2mu$}\hfill
    \mathclap{#3}\mathclap{#2}%
    \cleaders\hbox{$#5\mkern-2mu#2\mkern-2mu$}\hfill
    \mkern-7mu#4$%
    }
    \def\rightslashedarrowfill@{%
    \slashedarrowfill@\relbar\relbar\mapstochar\rightarrow}
    \newcommand\xslashedrightarrow[2][]{%
    \ext@arrow 0055{\rightslashedarrowfill@}{#1}{#2}}
    \def\leftslashedarrowfill@{%
    \slashedarrowfill@\leftarrow\relbar\mapsfromchar\relbar}
    \newcommand\xslashedleftarrow[2][]{%
    \ext@arrow 0055{\leftslashedarrowfill@}{#1}{#2}}
    \def\rightdoubleslashedarrowfill@{%
    \slashedarrowfill@\relbar\relbar{\mapstochar\mkern-2mu\mapstochar}\rightarrow}
    \newcommand\xdoubleslashedrightarrow[2][]{%
    \ext@arrow 0055{\rightdoubleslashedarrowfill@}{#1}{#2}}
    \makeatother

    \newcommand{\xlto}{\xslashedrightarrow}
    \newcommand{\lto}{\xlto{}}
    \newcommand{\xlfrom}{\xslashedleftarrow}

    \newcommand{\op}{{}^\tx{op}}
    \newcommand{\co}{{}^\tx{co}}

    \newcommand{\tp}[1]{\langle#1\rangle}
    
    \DeclareMathOperator*{\colim}{colim}

    \newcommand{\rf}{\mathbin{\blacktriangleleft}}
    
    \newcommand{\rx}{\mathbin{\blacktriangleright}}
    \newcommand{\plx}{\mathbin{{\rhd}\mathclap{\mspace{-17.5mu}\cdot}}}
    \newcommand{\prx}{\mathbin{{\blacktriangleright}\mathclap{\mspace{-17.5mu}\textcolor{white}{\cdot}}}}

    \newcommand{\adj}{\dashv}
    
    \newcommand{\radj}[1]{\mathrel{\adj_{#1}}}

    \newcommand{\ob}[1]{|#1|}

    \DeclareFontFamily{U}{min}{}
    \DeclareFontShape{U}{min}{m}{n}{<-> udmj30}{}
    \newcommand{\yo}{\!\text{\usefont{U}{min}{m}{n}\symbol{'210}}\!}

    \mathcommand{\comma}{\downarrow}

    \newcommand{\copi}{\flip\pi}
    
    \newcommand{\ostar}{\mathbin{\scalerel*{\kern.075em\tikz{\node[draw, circle, line width=.36pt, inner sep=0]{\ensuremath{*}};}\kern.075em}{\odot}}}
    \newcommand{\bulletstar}{\mathbin{\scalerel*{\kern.075em\tikz{\node[fill, text=white, circle, line width=.36pt, inner sep=0]{\ensuremath{*}};}\kern.075em}{\odot}}}

    \newcommand{\copow}{\cdot}

    \makeatletter
    \def\widebreve{\mathpalette\wide@breve}
    \def\wide@breve#1#2{\sbox\z@{$#1#2$}%
         \mathop{\vbox{\m@th\ialign{##\crcr
    \kern0.08em\brevefill#1{0.8\wd\z@}\crcr\noalign{\nointerlineskip}%
                        $\hss#1#2\hss$\crcr}}}\limits}
    \def\brevefill#1#2{$\m@th\sbox\tw@{$#1($}%
      \hss\resizebox{#2}{\wd\tw@}{\rotatebox[origin=c]{90}{\upshape(}}\hss$}
    \makeatother

    \makeatletter
    \def\widearch{\mathpalette\wide@arch}
    \def\wide@arch#1#2{\sbox\z@{$#1#2$}%
         \mathop{\vbox{\m@th\ialign{##\crcr
    \kern0.08em\archfill#1{0.8\wd\z@}\crcr\noalign{\nointerlineskip}%
                        $\hss#1#2\hss$\crcr}}}\limits}
    \def\archfill#1#2{$\m@th\sbox\tw@{$#1($}%
      \hss\resizebox{#2}{\wd\tw@}{\rotatebox[origin=c]{-90}{\upshape(}}\hss$}
    \makeatother

    \NewDocumentCommand{\jrule}{om}{%
        \IfNoValueTF{#1}
            {\textsc{#2}}
            {$#1$-\textsc{#2}}%
    }

    \newcommand{\Set}{{\b{Set}}}
    
    \newcommand{\Cat}{\b{Cat}}
    \newcommand{\CAT}{\b{CAT}}
    \newcommand{\VCat}{{\V\h\Cat}}
    
    \newcommand{\Mnd}{\b{Mnd}}
    \newcommand{\RMnd}{\b{RMnd}}

    \newcommand{\ff}{fully faithful}
    \newcommand{\Ff}{Fully faithful}
    \newcommand{\ffness}{full faithfulness}

    \newcommand{\ioo}{identity-on-objects}

    \newcommand{\EM}{Eilenberg--Moore}

    \newcommand{\eg}{e.g.\@\xspace}
    \newcommand{\ie}{i.e.\@\xspace}
    \newcommand{\viz}{viz.\@\xspace}
    \newcommand{\cf}{cf.\@\xspace}
    
    \newcommand{\aka}{a.k.a.\@\xspace}
    \NewDocumentCommand{\etc}{t.}{etc.\@\xspace}
    \NewDocumentCommand{\ibid}{t.}{ibid.\@\xspace}
    \NewDocumentCommand{\loccit}{t.}{loc.\ cit.\@\xspace}

    \newcommand{\dcs}{double categories\xspace}
    \newcommand{\vd}{virtual double}
    \newcommand{\vdc}{\vd{} category}
    \newcommand{\vdcs}{\vd{} categories}
    
    \newcommand{\ve}{virtual equipment}
    \newcommand{\vb}{virtual bicategory}
    \newcommand{\vbs}{virtual bicategories}
    \newcommand{\fct}{formal category theory}

    \newcommand{\lfp}{locally finitely presentable}

\makeatletter

\define@key{beamerframe}{c}[true]{
    \beamer@frametopskip=0pt plus 1fill\relax%
    \beamer@framebottomskip=0pt plus 1fill\relax%
}

\patchcmd{\beamer@sectionintoc}{\vfill}{\vskip\itemsep}{}{}

\makeatother

\ifarticle{}{

  \colorlet{colour-bg}{black!85} 
  \definecolor{colour-primary}{HTML}{cc80ff} 
  \colorlet{colour-text}{black!10} 
  \colorlet{colour-subtle}{black!40} 
  \colorlet{colour-block-bg}{black!80} 
  \definecolor{colour-warning-bg}{HTML}{ffea80} 
  \definecolor{colour-warning-primary}{HTML}{e08152} 

  \hypersetup{linktoc=all,colorlinks, citecolor={colour-primary}, linkcolor={colour-primary}, urlcolor={colour-primary}} 
  \urlstyle{sf}

  \usepackage{changepage} 

  \usepackage{ragged2e} 

  \setbeamertemplate{section in toc}[sections numbered]

  \apptocmd{\frame}{}{\justifying}{}

  \usefonttheme[onlymath]{serif}

  \beamertemplatenavigationsymbolsempty
  \setbeamertemplate{footline}{
    \hfill%
    \usebeamercolor[fg]{page number in head/foot}%
    \usebeamerfont{page number in head/foot}%
    \setbeamertemplate{page number in head/foot}[pagenumber]%
    \usebeamertemplate*{page number in head/foot}\kern.2cm\vskip.2cm%
  }

  \setbeamertemplate{frametitle}[default][center]
  \addtobeamertemplate{frametitle}{\vspace*{1ex}}{\vspace*{1ex}}

  \setbeamertemplate{itemize item}{$\bullet$}

  \setbeamertemplate{theorems}[numbered]
  \newtheorem{proposition}[theorem]{\translate{Proposition}}

  \addtobeamertemplate{block begin}
  {}
  {\vspace{0ex} 
  \begin{adjustwidth}{1em}{1em} 
  \tikzcdset{background color=colour-block-bg}
  }
  \addtobeamertemplate{block end}{\end{adjustwidth}%
  \vspace{1ex}} 
  {}
  \renewenvironment<>{block}[1]{%
      \begin{actionenv}#2%
        \def\insertblocktitle{\begin{adjustwidth}{.8em}{.8em}#1\end{adjustwidth}}%
        \par%
        \usebeamertemplate{block begin}}
      {\par%
        \usebeamertemplate{block end}%
      \end{actionenv}}

  \addtobeamertemplate{block example begin}
  {}
  {\vspace{0ex} 
  \begin{adjustwidth}{1em}{1em} 
  \tikzcdset{background color=colour-block-bg}
  }
  \addtobeamertemplate{block example end}{\end{adjustwidth}%
  \vspace{1ex}} 
  {}
  \renewenvironment<>{exampleblock}[1]{%
      \begin{actionenv}#2%
          \def\insertblocktitle{\begin{adjustwidth}{.8em}{.8em}#1\end{adjustwidth}}%
          \par%
          \only<presentation>{
            \setbeamercolor{local structure}{parent=example text}}%
          \usebeamertemplate{block example begin}}
        {\par%
          \usebeamertemplate{block example end}%
        \end{actionenv}}

    \setbeamercolor{background canvas}{bg=colour-bg}
    \tikzcdset{background color=colour-bg}

    \setbeamercolor{block title}{fg=colour-primary,bg=colour-block-bg}
    \setbeamercolor{block title example}{fg=colour-primary,bg=colour-block-bg}
    \setbeamercolor{block body}{bg=colour-block-bg}
    \setbeamercolor{block body example}{bg=colour-block-bg}

    \setbeamercolor{title}{fg=colour-primary}
    \setbeamercolor{frametitle}{fg=colour-primary}
    \setbeamercolor{alerted text}{fg=colour-primary}

    \setbeamercolor{normal text}{fg=colour-text}
    \setbeamercolor{item}{fg=colour-text}
    \setbeamercolor{section in toc}{fg=colour-text}

    \setbeamercolor{page number in head/foot}{fg=colour-subtle}

    \setbeamercolor*{bibliography entry author}{fg=colour-text}
    \setbeamercolor*{bibliography entry note}{fg=colour-text}

}

\usetikzlibrary{nfold}
\IfFileExists{tangle.sty}{\usepackage{tangle}}{}
\usepackage{bbm}
\usepackage{wasysym}
\usepackage{suffix}

\newcommand{\X}{\bb X}
\newcommand{\tX}{\u\X}
\renewcommand{\Cat}{\dc{Cat}}
\newcommand{\D}{\bb D}

\newcommand{\Y}{\bb Y}

\newcommand{\E}{\cl E}
\newcommand{\K}{\cl K}
\newcommand{\I}{\s I}

\newcommand{\lh}[1]{\llbracket #1 \rrbracket}
\mathcommand{\L}{\textup{\sagittarius}}

\newcommand{\jAE}{j \colon A \to E}

\newcommand{\jadj}{\radj j}
\newcommand{\ljr}{\ell \jadj r}

\newcommand{\cp}[1]{\mathbin{\smallsmile_{#1}}}
\newcommand{\pc}[1]{\mathbin{\smallfrown_{#1}}}

\newcommand{\VcoCat}{\V\co\h\Cat}

\newcommand{\Kl}{\b{Kl}}

\newcommand{\wc}{\ostar}
\newcommand{\wl}{\bulletstar}

\renewcommand{\defeq}{\mathrel{:=}}
\newcommand{\opcart}{\tx{opcart}}
\newcommand{\odotl}{\mathbin{{\odot}_L}}
\newcommand{\odotr}{\mathbin{{\odot}_R}}
\newcommand{\cart}{\tx{cart}}

\renewcommand{\EM}{\relax\ifmmode\b{EM}\else{}Eilenberg\nobreakdash--Moore\fi}

\newcommand{\ex}{\underline}

\newcommand{\rfP}[3]{\mathcal{P}{#1}(#2, #3)}

\mathcommand{\P}{\cl P}

\makeatletter
\@ifpackageloaded{tangle}{
\tgColours{F5A3A3,F5CCA3,F5F5A3,CCF5A3,A3F5A3,A3F5CC,A3F5F5,A3CCF5,A3A3F5,CCA3F5,F5A3F5,F5A3CC}

}{}
\makeatother

\newcounter{eqstep}
\newcounter{eqsubstep}[eqstep]
\newcommand{\nexteqstep}{\stepcounter{eqstep}}

\makeatletter
\newcommand{\raisedtarget}[1]{\Hy@raisedlink{\hypertarget{#1}{}}}
\makeatother

\newcommand{\tangleeq}{\refstepcounter{eqsubstep}\raisedtarget{eqstep:\theeqstep.\theeqsubstep}{\mathclap{\overset{(\theeqstep.\theeqsubstep)}{=}}}}

\WithSuffix\newcommand\tangleeq*{\mathclap{=}}

\newcommand{\tangleeql}{\refstepcounter{eqsubstep}\raisedtarget{eqstep:\theeqstep.\theeqsubstep}{\mathllap{\overset{(\theeqstep.\theeqsubstep)}{=}}}}

\WithSuffix\newcommand\tangleeql*{\mathllap{=}}

\newenvironment{tangleeqs*}{%
    \mathcommand{\=}{\tangleeq*}
    \global\let\externaldblbackslash\\
    \csname gather*\endcsname
    \ifundef{\internaldblbackslash}{%
        \global\let\internaldblbackslash\\%
        \gdef\\{\internaldblbackslash\tangleeql*}%
    }{}
}{%
    \csname endgather*\endcsname
    \global\undef\internaldblbackslash
    \global\let\\\externaldblbackslash
}

\bibliography{references}

\DeclareSymbolFont{tipa}{T3}{cmr}{m}{n}
\DeclareMathAccent{\arch}{\mathalpha}{tipa}{16}
\DeclareFontFamily{U}{mathx}{}
\DeclareFontShape{U}{mathx}{m}{n}{<-> mathx10}{}
\DeclareSymbolFont{mathx}{U}{mathx}{m}{n}
\DeclareMathAccent{\widehat}{0}{mathx}{"70}
\DeclareMathAccent{\widecheck}{0}{mathx}{"71}

\newcommand{\opcartl}{\opcart_l}
\newcommand{\lcto}{\lto \cdots \lto}

\newcommand{\A}{{\b A}}
\newcommand{\B}{{\b B}}
\newcommand{\C}{{\b C}}
\renewcommand{\D}{{\b D}}
\renewcommand{\E}{{\b E}}
\newcommand{\M}{{\b M}}
\newcommand{\Q}{{\b Q}}
\newcommand{\W}{{\b W}}
\newcommand{\Sind}{\b{Sind}}
\newcommand{\Ind}{\b{Ind}}
\newcommand{\Coc}{\h\b{Coc}}
\newcommand{\Span}{{\b{Span}}}

\newcommand{\PP}{\dc P}
\newcommand{\QQ}{\dc Q}
\newcommand{\J}{{\cl J}}

\newcommand{\V}{{\bb V}}

\newcommand{\colimequiv}{\mathrel{\hat\approx}}
\newcommand{\colimleq}{\mathrel{\hat\lesssim}}
\newcommand{\limequiv}{\mathrel{\check\approx}}
\newcommand{\limleq}{\mathrel{\check\lesssim}}
\newcommand{\repl}{^{{\iso}}}
\newcommand{\satw}{^{{\ostar}}}
\newcommand{\limclosure}[2]{#1^{{\wl}\h\tx{cl}}_{#2}}
\newcommand{\colimclosure}[2]{#1^{{\wc}\h\tx{cl}}_{#2}}

\newcommand{\psh}[2]{{\widehat{#1}}[{#2}]}
\newcommand{\pshe}[2]{{\hat{#1}}_{#2}}

\newcommand{\coc}[2]{{\widehat{#1}}(#2)}
\newcommand{\coce}[2]{{\hat{#1}}_{#2}}
\newcommand{\pshm}{\arch}
\newcommand{\widepshm}{\widearch}

\newcommand{\cpsh}[2]{{\widecheck{#1}}[{#2}]}
\newcommand{\cpshe}[2]{{\check{#1}}_{#2}}

\newcommand{\com}[2]{{\check{#1}}(#2)}
\newcommand{\come}[2]{{\check{#1}}_{#2}}
\newcommand{\cpshm}{\breve}

\theoremstyle{plain}

\newtheorem{manualtheoreminner}{Theorem}
\newenvironment{manualtheorem}[1]{%
  \renewcommand\themanualtheoreminner{#1}%
  \manualtheoreminner
}{\endmanualtheoreminner}

\newtheorem{manualpropositioninner}{Proposition}
\newenvironment{manualproposition}[1]{%
  \renewcommand\themanualpropositioninner{#1}%
  \manualpropositioninner
}{\endmanualpropositioninner}

\theoremstyle{definition}

\newtheorem{manualdefinitioninner}{Definition}
\newenvironment{manualdefinition}[1]{%
  \renewcommand\themanualdefinitioninner{#1}%
  \manualdefinitioninner
}{\endmanualdefinitioninner}

\foreach \env in {manualdefinition}{
\AtBeginEnvironment{\env}{%
    \pushQED{\qed}%
}
\AtEndEnvironment{\env}{\popQED\endexample}
}

\title{Presheaves and cocompletions in formal category theory}
\author{Nathanael Arkor}
\address{Department of Software Science, Tallinn University of Technology, Estonia}
\author{Dylan McDermott}
\address{Department of Computer Science, University of Oxford, United Kingdom}
\datetoday

\subjclass[2020]{18D70,18D65,18D60,18A30,18A35,18F20,18D20,18N10}

\hypersetup{pdfauthor={Nathanael Arkor, Dylan McDermott}}

\begin{document}

\begin{abstract}
  We study the relationship between presheaf constructions and free cocompletions in the context of formal category theory, elucidating the coincidence between the two concepts in familiar settings. We show that, in a \ve{} satisfying mild assumptions, free cocompletions under classes of weights are exhibited by presheaf constructions. We furthermore extend the theory of weighted colimits from enriched category theory to this setting, developing the concepts of atomicity and rank, and providing recognition theorems for presheaf objects, free cocompletions, and cocomplete objects. As an application of our methods, we construct free cocompletions, under arbitrary classes of colimit-small weights, of (possibly large) categories enriched in (not necessarily symmetric) monoidal categories and bicategories; this resolves a longstanding omission in the literature on enriched category theory. As a consequence, we obtain a categorified Eilenberg--Watts theorem, which extends the classical theorem for categories of modules for rings (seen as one-object $\b{Ab}$-enriched categories) to categories enriched in a bicategory.
\end{abstract}

\maketitle

\tableofcontents

\section{Introduction}
\label{introduction}

The construction of the category of presheaves $\widehat{\A} \defeq \Set^{\A\op}$ on a locally small category $\A$ is perhaps the most fundamental construction in category theory. When $\A$ is small, the presheaf construction is characterised by two important universal properties.
\begin{enumerate}
  \item It classifies distributors (\aka{} profunctors or bimodules), in the sense that functors $\A \to \Set^{\B\op}$ are in natural bijection with distributors $\A \lto \B$.
  \item It is the free cocompletion of $\A$, in that (a) $\widehat{\A}$ admits all small colimits, and (b) functors $\A \to \b X$, for $\b X$ admitting small colimits, are in pseudo natural equivalence with functors $\widehat{\A} \to \b X$ that preserve small colimits.
\end{enumerate}

When $\A$ is large, the situation is more subtle. While $\widehat{\A}$ still classifies distributors, it no longer exhibits the free cocompletion of $\A$. However, we may obtain the free cocompletion by restricting $\widehat{\A}$ to the full subcategory spanned by the \emph{small presheaves}, \ie{} the small colimits of representable presheaves. More generally, for any class $\Phi$ of \emph{colimit-small weights} (here meaning a class of distributors satisfying an appropriate size condition), corresponding to a class of colimit shapes, the free cocompletion of $\A$ under $\Phi$-weighted colimits may be obtained as a full subcategory of $\widehat{\A}$ by restricting to the closure of the representable presheaves under $\Phi$-weighted colimits.

For enriched categories, there is a similar story, though additional care is required. For one, given an arbitrary monoidal category $\b V$ and a $\b V$-enriched category $\A$, it may not be possible to form the $\b V$-category of presheaves $\widehat{\A}$, even assuming $\A$ is small. For another, even supposing we may form $\widehat{\A}$, it is not necessarily constructed as a $\b V$-category of contravariant $\b V$-functors from $\A$ to $\b V$ (indeed, for $\b V$ to itself form a $\b V$-category requires $\b V$ to be closed, and to construct opposites of $\b V$-categories requires some form of symmetry on $\b V$). However, as long as $\b V$ is sufficiently nice\footnotemark{}, it is nevertheless again true that we may form free cocompletions (and, more generally, $\Phi$-cocompletions) of $\b V$-categories by restricting to subcategories of $\widehat{\A}$~\cite{kelly1982basic}.
\footnotetext{For instance, \textcite{kelly1982basic} assumes that $\b V$ is a complete and cocomplete symmetric closed monoidal category, and we will show in \cref{presheaves-and-cocompletions-in-V-Cat} that these assumptions can be substantially weakened.}%

This intimate relationship between presheaf constructions and free cocompletions is highly useful in practice, as it gives a simple concrete construction of an a priori entirely abstract object. Furthermore, the knowledge that free cocompletions are exhibited by categories of presheaves allows one to deduce important properties of free cocompletions that are not evidently derivable merely from their universal property. It is therefore desirable to understand how presheaf constructions relate to free cocompletions in general, permitting these techniques to be extended to settings beyond that of categories enriched in a nice monoidal category.

To explain what we mean by \emph{presheaf construction} and \emph{free cocompletion} in general, we recall the philosophy of \emph{\fct}~\cite{gray1974formal,street1978yoneda,wood1982abstract}. Category theory provides an expressive language with which to describe totalities of structures and their relationships. Formal category theory arises when we apply the principles of category theory to category theory itself, studying the totalities of categories -- as well as variants such as enriched categories, internal categories, and fibred categories -- using (two-dimensional) category theoretic techniques. While several frameworks for \fct{} have been proposed, in recent years the formalism of \emph{\ve s} has emerged as a particularly apposite setting in which to carry out \fct~\cite{cruttwell2010unified,arkor2024formal,arkor2024relative,arkor2025nerve,koudenburg2024formal}.

A \ve{} is a two-dimensional structure that axiomatises the totality of (large, enriched) categories, functors, distributors, and natural transformations. A \ve{} has two kinds of morphism: \emph{tight-cells} (denoted $\to$), which axiomatise functors; and \emph{loose-cells} (denoted $\lto$), which axiomatise distributors. They thus extend the concept of a 2-category by regarding distributors also as a primitive notion. It is possible to define notions of \emph{presheaf object} and \emph{weighted colimit} in a \ve{}~\cite{arkor2025nerve,arkor2024formal,koudenburg2024formal}, which recover the usual notions of enriched category theory when instantiated in appropriate \ve s of enriched categories. We may therefore ask to what extent the relationship between presheaf constructions and free cocompletions, known classically to hold for categories enriched in nice monoidal categories, holds purely formally.

In this paper, we introduce the notions of \emph{$\PP$-presheaf object} and \emph{$\Phi$-cocompletion} in a \ve{}, for $\PP$ a class of loose-cells and $\Phi$ a class of weights, and study their relationship. In doing so, we extend much of the classical theory of presheaves and cocompletions -- such as atomicity, rank, and recognition theorems -- from enriched categories to formal category theory. Our central result, \cref{presheaf-object-is-cocompletion}, gives a construction of $\Phi$-cocompletions via presheaf objects, establishing that the classical relationship between presheaves and cocompletions is not specific to enriched categories, but holds for entirely formal reasons.

To demonstrate the utility of our development, we apply our results in the setting of categories enriched in a \emph{normal \vb} (a structure simultaneously generalising a unital multicategory and a bicategory). While the theory of categories enriched in a monoidal category is well developed~\cite{kelly1982basic}, the theory of categories enriched in multicategories~\cite{lambek1969deductive} or in bicategories~\cite{walters1981sheaves,street1983enriched} is much less mature, despite the abundance of motivating examples~\cite{walters1982sheaves,betti1983variation,garner2014topological,garner2018enriched}. For a normal \vb{} $\V$ satisfying some mild assumptions, we show that it is possible to construct $\Phi$-cocompletions of $\V$-enriched categories via enriched presheaf categories. In particular, when $\V$ is a bicategory, we have the following.

\begin{theorem*}[{\ref{cocompletion-in-nice-situation}}]
  Let $\V$ be a locally complete and locally cocomplete bicategory with right lifts and let $\Phi$ be a class of colimit-small weights. For every $\V$-category $\A$ (not assumed to be small), the free cocompletion of $\A$ under $\Phi$-weighted colimits exists, and is given by the smallest sub-$\V$-category of the $\V$-category of small presheaves on $\A$ that is closed under $\Phi$-weighted colimits and representable presheaves.
\end{theorem*}

Despite the historical maturity of bicategory-enrichment, which dates back to the 1980s (originating in \citeauthor{walters1981sheaves}' observation that sheaves may be seen as categories enriched in a bicategory~\cite{walters1981sheaves,walters1982sheaves}), the construction of arbitrary $\Phi$-cocompletions appears to be entirely new; we give a survey of free cocompletions of enriched categories in \cref{historical-context}. In fact, many of the results throughout this paper, which extend the classical theory of colimits and cocompletions for categories enriched in symmetric monoidal categories~\cite{kelly1982basic} to the context of \fct, appear to be new even when instantiated in the setting of enrichment in a non-symmetric monoidal category. Our paper may thus be seen as a contribution to the theory of non-symmetric monoidal and bicategorical enrichment, of independent interest to our formal development.

In the future, we intend to apply this framework to situations beyond enriched categories. There are numerous examples in the literature of category-like structures that admit notions of presheaf construction and free cocompletion, which satisfy the same kinds of theorems as in enriched category theory, and yet have heretofore been studied on a case-by-case basis (necessarily, as no general theory existed). For instance, in each of the following examples, a notion of cocompletion is constructed for a category-like structure by considering an appropriate notion of presheaf construction.
\begin{itemize}
  \item For internal categories, there are at least two notions of presheaf construction, depending on one's attitude to size. For instance, \textcite{street1980cosmoi,weber2007yoneda} consider exponentials of internal full subcategories (playing the role of an internal `category of sets'), whereas \textcite[\S B2.5]{johnstone2002sketches} considers a presheaf \emph{indexed category} on an internal category.
  \item In \cite{street1981conspectus}, \citeauthor{street1981conspectus} studies the theory of \emph{variable categories}, \viz{} pseudo functors $\K\op \to \b{Cat}$ from a bicategory $\K$. In \S7 -- \S9 \ibid, \citeauthor{street1981conspectus} examines what it means for a variable category to be cocomplete, and constructs the free cocompletion of a variable category under a class of colimits via a presheaf construction~\cite[Theorem~9.3]{street1981conspectus}.
  \item \textcite{im1986universal} show that, for a small monoidal category $\b A$, the convolution monoidal structure on the category of presheaves $\Set^{\A\op}$ exhibits it as the free monoidal cocompletion of $\b A$; this is generalised to pseudo algebras for certain pseudomonads by \textcite{walker2019distributive} (\cf~\cite[\S8]{koudenburg2015double}, in which \citeauthor{koudenburg2015double} develops a formal category theoretic approach to convolution algebraic structures).
  \item In \cite{garner2012lex}, \citeauthor{garner2012lex} develop the theory of \emph{lex colimits}, which is the theory of cocompletions for finitely complete enriched categories, and show that many notions of exactness in category theory are thereby captured.
  \item \textcite{bunge2013tightly,shulman2013enriched} independently introduced the notion of \emph{indexed enriched} (\aka{} \emph{enriched indexed}) categories, providing a simultaneous generalisation of indexed categories and enriched categories. In particular, \citeauthor{shulman2013enriched} exhibits free cocompletions of indexed enriched categories via a presheaf construction.
  \item In \cite{garner2020cocompletion}, \citeauthor{garner2020cocompletion} develop the theory of presheaves and cocompletions for restriction categories, which are axiomatisations of categories of partial functions~\cite{cockett2002restriction}.
\end{itemize}
Our expectation is that these arise, for appropriate choices of \ve, as instances of the formal theory of presheaves and cocompletions we develop herein, permitting the general theory to be applied immediately in these settings and facilitating comparison between them.

\subsection{Outline of the paper}

In \cref{fct}, we recall the definition of \ve{} and develop some of the basic \fct{} that will be used throughout the paper, focusing on the calculus of right lifts. This may be viewed as an extension of the material in \cite[\S3]{arkor2024formal} (though familiarity with that work is not assumed here). In \cref{presheaves-in-a-virtual-equipment}, we introduce the notion of \emph{$\PP$-presheaf object}, for a class $\PP$ of loose-cells in a \ve{}, and study its properties, showing it to satisfy the usual properties satisfied by full subcategories of enriched presheaf categories. In \cref{atomicity}, we introduce the notion of \emph{atomicity}, which generalises the concept of a finitely presentable object in a category, and show how it can be used to characterise presheaf objects. In \cref{rank-etc}, we introduce the notion of \emph{rank}, which generalises the concept of a finitary functor. Rank is fundamental to \cref{cocompleteness-in-a-virtual-equipment}, in which we introduce the concept of a \emph{$\Phi$-cocompletion} in a \ve{} and study its properties, showing it to satisfy the usual properties satisfied by enriched cocompletions. Our formal development culminates in \cref{presheaves-and-cocompletions}, in which we establish sufficient conditions under which presheaf objects in \ve s exhibit free cocompletions. (Naturally, while we have spoken so far only about cocompletions under classes of colimits, the theory of completions under classes of limits is entirely dual: for completeness, we discuss this in \cref{copresheaf-objects-and-completions}.)

To demonstrate the utility of our framework, we use our central result to construct in \cref{presheaves-and-cocompletions-in-V-Cat} free (co)completions, under arbitrary classes of (co)limit-small weights, of categories enriched in normal \vbs. The general theory of sections \ref{fct} -- \ref{copresheaf-objects-and-completions} is then immediately applicable. We conclude in \cref{examples} by giving various examples of cocompletions of categories enriched in bicategories. \Cref{local-(co)completeness} provides some basic material on local completeness and local cocompleteness for \vdcs, which is used in \cref{presheaves-and-cocompletions-in-V-Cat}.

\subsection{Referencing convention}

Throughout, we frequently annotate our results with references to comparable results in standard accounts (\eg{} in \textcite{kelly1982basic}). `Comparable' here is intended loosely: it is typically the case that, even when specialised to the setting of enriched categories, our results are more general than those we reference, and the proofs we give simpler by virtue of the formal setting in which we work. The references are intended rather to make it convenient for readers of those texts to find a formal analogue (or generalisation) of a given result in the setting of ordinary or enriched categories.

\subsection{Related work}

The relationship between distributors, presheaf constructions, and free cocompletions is of historical significance, as it provides the motivation for the three primary approaches to \fct{} in the literature. \citeauthor{street1978yoneda}' formalism of \emph{Yoneda structures}~\cite{street1978yoneda} is an axiomatisation of the presheaf construction. \citeauthor{weber2007yoneda} showed that, under strong assumptions (holding, for instance, for internal categories, but not for enriched categories), if the presheaf objects on small objects in a Yoneda structure are cocomplete, then they form free cocompletions~\cites[Theorem~3.20]{weber2007yoneda}[Remark~7.7]{koudenburg2024formal}. Crucially, cocompleteness is assumed rather than derived (though specific examples are given where this assumption does hold~\cite[\S7]{weber2007yoneda}).

\citeauthor{wood1982abstract}'s formalism of \emph{proarrow equipments}~\cite{wood1982abstract}, as well as the later \emph{\ve s} of \textcite{cruttwell2010unified} and the \emph{augmented \ve s} of \textcite{koudenburg2024formal}, are axiomatisations of the structure of distributors. \textcite{koudenburg2024formal} introduced presheaf constructions in the setting of augmented \ve s, established that this setting generalises that of Yoneda structures~\cite[\S4.34]{koudenburg2024formal}, and proved that, under mild assumptions, every presheaf object exhibits a free cocompletion with respect to a class of weighted diagrams. This result is closely related to our central result. However, the universal property that we consider -- which is double categorical in nature, rather than merely 2-categorical -- is significantly stronger than that of \cite{koudenburg2024formal}, and is necessary to generalise important classical theorems about enriched cocompletions to the formal setting. We give a more detailed comparison to \citeauthor{koudenburg2024formal}'s work in \cref{relationship-to-koudenburg}. \citeauthor{koudenburg2024formal} furthermore studied conditions under which the presheaf construction is exhibited as an internal hom, as in the classical setting~\cite[Theorem~8.21]{koudenburg2024formal}.

\citeauthor{bunge1999bicomma} demonstrated that \emph{lax-idempotent pseudomonads} (\aka{} \emph{KZ-doctrines}), conceived by \textcite{kock1967limit,zoberlein1976doctrines} to axiomatise free cocompletions, also provide a setting for \fct. In particular, they give conditions under which the existence of a pseudo algebra structure for a lax-idempotent pseudomonad corresponds to the existence of certain pointwise left extensions~\cite[Theorem~1.13]{bunge1999bicomma}; here, \emph{pointwiseness} is defined in terms of bicomma objects, which is suitable for internal categories but not for enriched categories. It was shown by \textcite{walker2018yoneda} that every locally \ff{} lax-idempotent pseudomonad induces a Yoneda structure\footnote{Modulo the right-ideal condition, which plays no role in the theory of Yoneda structures.}; in particular, the formal category theory of lax-idempotent pseudomonads is subsumed by that of \ve s. Conversely, we show in \cref{induced-relative-pseudomonad} that every class of weights in a \ve{} induces a lax-idempotent \emph{relative} pseudomonad~\cite{fiore2018relative}.

In her PhD thesis, \textcite{despalungue2023operads} introduced a formalism for formal category theory based on an internalisation of the notion of enrichment. Let $\K$ be a cartesian closed bicategory with a duality involution ${\ph^\circ \colon \K\co \xto\iso \K}$ (therein called an \emph{oppositisation}). An \emph{entangled enrichment} on $\K$ is a coherent, functorial assignment equipping each object $A$ with a hom-1-cell $A^\circ \times A \to S$ into a distinguished object $S$ (playing the role of the category of sets)~\cite[\S I.4]{despalungue2023operads}. This suffices to capture categories enriched in a symmetric closed monoidal category~\cite[\S I.4.6]{despalungue2023operads}, although not enrichment in non-braided categories or bicategories. In the context of an entangled enrichment, the canonical 1-cell $A \to S^{A^\circ}$, obtained by transposing the hom-1-cell, is shown to possess several of the expected properties of a presheaf embedding~\cite[\S I.4.5]{despalungue2023operads}, and exhibits the free cocompletion of $A$ in a suitable sense~\cite[Proposition~I.4.5.8]{despalungue2023operads}. In contrast to previous approaches, $A$ is not required for this result to be small with respect to $S$ in any sense.\footnote{Note that \citeauthor{despalungue2023operads} uses the term \emph{small category} differently from the convention in set-theoretic foundations~\cite[\S A.7]{despalungue2023operads}.} Consequently, \citeauthor{despalungue2023operads}'s statement does not directly recover the classical universal property of the presheaf construction on a small category; however, the statement and proof may be adapted with minimal changes to recover the classical statement. Free cocompletions of large categories under small colimits are not considered therein.

\subsection{Size conventions}

We shall have need throughout the paper to speak of \emph{small} and \emph{large} categories, as well as to assemble them into \emph{huge} categories. Formally, this is achieved by positing the existence of a Grothendieck universe $\mathfrak U$: throughout the paper, \emph{small}, \emph{large}, and \emph{huge} will refer to what typically would be called \emph{$\mathfrak U$-small}, \emph{small} (\viz{} set-sized), and \emph{large} (\viz{} class-sized).

\subsection{Acknowledgements}

We thank Roald Koudenburg for explaining the relationship to his work, as outlined in \cref{relationship-to-koudenburg}, thank Giacomo Tendas for a helpful conversation about petty presheaves, and thank Omer Cantor, Bryce Clarke, James Deikun, Sophie d'Espalungue, Roald Koudenburg, Luc Saccoccio--Le Guennec, and Ross Street for providing helpful comments on the paper. During the writing of this paper, the first-named author was supported by a departmental postdoctoral grant from the Department of Software Science at Tallinn University of Technology, and by the Estonian Research Council grant PSG1242; the second-named author was supported by Icelandic Research Fund grant \textnumero~228684-053.

\section{Formal category theory in a \ve}
\label{fct}

We begin by recalling the basic concepts of formal category theory in a \ve{}, including weighted colimits, (pointwise) extensions, and relative adjunctions. This section is based on \cite[\S3]{arkor2024formal}, though we generalise several of the definitions and results therein -- for instance, by considering colimits weighted by chains of loose-cells, rather than individual loose-cells. Familiarity with \cite{arkor2024formal} is not assumed.

\subsection{Virtual \dcs{} and \ve s}
\label{vdcs-and-ves}

A \emph{\ve} is a two-dimensional structure that axiomatises the totality of enriched categories, functors, distributors, and natural transformations.

\begin{definition}[{\cite[61]{burroni1971tcategories}}]
    A \emph{\vdc{}} comprises the following data.
    \begin{enumerate}
        \item A (huge) category of \emph{objects} and \emph{tight-cells}.
        We denote a tight-cell $f$ from an object $A$ to an object $B$ by an arrow $f \colon A \to B$; denote the composition of tight-cells $f \colon A \to B$ and $g \colon B \to C$ both by $(f \d g) \colon A \to C$ and by $g f \colon A \to C$; and denote the identity of an object $A$ by $1_A \colon A \to A$, or simply by $=$ in pasting diagrams.
        \item For each pair of objects $A$ and $B$, a class of \emph{loose-cells} from $A$ to $B$. We denote a loose-cell $p$ from $A$ to $B$ by an arrow with a vertical stroke $p \colon A \lto B$.
        \item For each chain of loose-cells $p_1, \ldots, p_n$ ($n \geq 0$) and compatible tight-cells $f, f'$ and loose-cell $q$ (together forming a \emph{frame}), a class of \emph{2-cells} with the given frame.
        \[\begin{tikzcd}
            {A_0} & {A_1} & \cdots & {A_{n - 1}} & {A_n} \\
            B &&&& {B'}
            \arrow[""{name=0, anchor=center, inner sep=0}, "f"', from=1-1, to=2-1]
            \arrow["{p_1}"'{inner sep=.8ex}, "\shortmid"{marking}, from=1-2, to=1-1]
            \arrow["{p_2}"'{inner sep=.8ex}, "\shortmid"{marking}, from=1-3, to=1-2]
            \arrow["{p_{n - 1}}"'{inner sep=.8ex}, "\shortmid"{marking}, from=1-4, to=1-3]
            \arrow["{p_n}"'{inner sep=.8ex}, "\shortmid"{marking}, from=1-5, to=1-4]
            \arrow[""{name=1, anchor=center, inner sep=0}, "{f'}", from=1-5, to=2-5]
            \arrow["q"{inner sep=.8ex}, "\shortmid"{marking}, from=2-5, to=2-1]
            \arrow["\phi"{description}, draw=none, from=1, to=0]
        \end{tikzcd}\]
        Observe that our pasting diagrams are written from right-to-left (matching nondiagrammatic composition order of loose-cells). Such a 2-cell is \emph{globular} if $f$ and $f'$ are identities.
        We denote a globular 2-cell by $\phi \colon p_1, \dots, p_n \tto q$.

        In pasting diagrammatic notation, we denote a nullary 2-cell by a square of the following form\footnote{Some authors, \eg{} \cite{cruttwell2010unified}, prefer to render nullary 2-cells as triangles. However, we find this to be both less uniform (for instance, such authors do not render binary 2-cells as trapeziums) and less visually pleasing than rendering all 2-cells as rectangles.}.
        \[\begin{tikzcd}
          A & A \\
          B & {B'}
          \arrow[""{name=0, anchor=center, inner sep=0}, "f"', from=1-1, to=2-1]
          \arrow[equals, nfold, from=1-2, to=1-1]
          \arrow[""{name=1, anchor=center, inner sep=0}, "{f'}", from=1-2, to=2-2]
          \arrow["p"{inner sep=.8ex}, "\shortmid"{marking}, from=2-2, to=2-1]
          \arrow["\phi"{description}, draw=none, from=1, to=0]
        \end{tikzcd}\]
        \item For every configuration of 2-cells of the following shape,
        \[\begin{tikzcd}
            \cdot & \cdots & \cdot & \cdots & \cdot & \cdots & \cdot \\
            \cdot && \cdot & \cdots & \cdot && \cdot \\
            \cdot &&&&&& \cdot
            \arrow["\shortmid"{marking}, from=2-7, to=2-5]
            \arrow[""{name=0, anchor=center, inner sep=0}, "{f'}", from=1-7, to=2-7]
            \arrow[""{name=1, anchor=center, inner sep=0}, from=1-5, to=2-5]
            \arrow["{p_{m_n}^n}"', "\shortmid"{marking}, from=1-7, to=1-6]
            \arrow["{p_{1}^n}"', "\shortmid"{marking}, from=1-6, to=1-5]
            \arrow["{p_{m_1}^1}"', "\shortmid"{marking}, from=1-3, to=1-2]
            \arrow["{p_1^1}"', "\shortmid"{marking}, from=1-2, to=1-1]
            \arrow["\shortmid"{marking}, from=2-3, to=2-1]
            \arrow[""{name=2, anchor=center, inner sep=0}, from=1-3, to=2-3]
            \arrow[""{name=3, anchor=center, inner sep=0}, "f"', from=1-1, to=2-1]
            \arrow[""{name=4, anchor=center, inner sep=0}, "{g'}", from=2-7, to=3-7]
            \arrow["q", "\shortmid"{marking}, from=3-7, to=3-1]
            \arrow[""{name=5, anchor=center, inner sep=0}, "g"', from=2-1, to=3-1]
            \arrow["{\phi_n}"{description}, draw=none, from=0, to=1]
            \arrow["{\phi_1}"{description}, draw=none, from=2, to=3]
            \arrow["\psi"{description}, draw=none, from=4, to=5]
        \end{tikzcd}\]
        a 2-cell,
        \[\begin{tikzcd}
            \cdot & \cdots & \cdot \\
            \cdot && \cdot
            \arrow["{p_{m_n}^n}"', "\shortmid"{marking}, from=1-3, to=1-2]
            \arrow["q", "\shortmid"{marking}, from=2-3, to=2-1]
            \arrow[""{name=0, anchor=center, inner sep=0}, "{f' \d g'}", from=1-3, to=2-3]
            \arrow[""{name=1, anchor=center, inner sep=0}, "{f \d g}"', from=1-1, to=2-1]
            \arrow["{p_1^1}"', "\shortmid"{marking}, from=1-2, to=1-1]
            \arrow["{(\phi_1, \ldots, \phi_n) \d \psi}"{description}, draw=none, from=0, to=1]
        \end{tikzcd}\]
        the \emph{composite}.
        \item For each loose-cell $p \colon A' \lto A$, a 2-cell $1_p \colon p \tto p$, the \emph{identity} of $p$, denoted simply as $=$ in pasting diagrams.
        \[\begin{tikzcd}
            A & {A'} \\
            A & {A'}
            \arrow[""{name=0, anchor=center, inner sep=0}, Rightarrow, no head, nfold, from=1-1, to=2-1]
            \arrow["p"', "\shortmid"{marking}, from=1-2, to=1-1]
            \arrow[""{name=1, anchor=center, inner sep=0}, Rightarrow, no head, nfold, from=1-2, to=2-2]
            \arrow["p", "\shortmid"{marking}, from=2-2, to=2-1]
            \arrow["{=}"{description}, draw=none, from=1, to=0]
        \end{tikzcd}\]
    \end{enumerate}
    We shall sometimes denote by
    \[\begin{tikzcd}
        \cdot & \cdots & \cdot \\
        \cdot & \cdots & \cdot
        \arrow["\shortmid"{marking}, from=1-3, to=1-2]
        \arrow["\shortmid"{marking}, from=1-2, to=1-1]
        \arrow["\shortmid"{marking}, from=2-3, to=2-2]
        \arrow["\shortmid"{marking}, from=2-2, to=2-1]
        \arrow[""{name=0, anchor=center, inner sep=0}, Rightarrow, no head, nfold, from=1-3, to=2-3]
        \arrow[""{name=1, anchor=center, inner sep=0}, Rightarrow, no head, nfold, from=1-1, to=2-1]
        \arrow["{=}"{description}, draw=none, from=0, to=1]
    \end{tikzcd}\]
    the juxtaposition of identity 2-cells:
    \[\begin{tikzcd}
        \cdot & \cdot & \cdot & \cdot \\
        \cdot & \cdot & \cdot & \cdot
        \arrow["\shortmid"{marking}, from=1-4, to=1-3]
        \arrow["\shortmid"{marking}, from=2-2, to=2-1]
        \arrow[""{name=0, anchor=center, inner sep=0}, Rightarrow, no head, nfold, from=1-4, to=2-4]
        \arrow[""{name=1, anchor=center, inner sep=0}, Rightarrow, no head, nfold, from=1-1, to=2-1]
        \arrow["\shortmid"{marking}, from=1-2, to=1-1]
        \arrow["\shortmid"{marking}, from=2-4, to=2-3]
        \arrow["\cdots"{description}, draw=none, from=1-3, to=1-2]
        \arrow["\cdots"{description}, draw=none, from=2-3, to=2-2]
        \arrow[""{name=2, anchor=center, inner sep=0}, Rightarrow, no head, nfold, from=1-3, to=2-3]
        \arrow[""{name=3, anchor=center, inner sep=0}, Rightarrow, no head, nfold, from=1-2, to=2-2]
        \arrow["{=}"{description}, draw=none, from=0, to=2]
        \arrow["{=}"{description}, draw=none, from=3, to=1]
    \end{tikzcd}\]
    Composition of 2-cells is required to be associative and unital in the evident manner~\cite[Definition~2.1]{cruttwell2010unified}: associativity is implicit in pasting diagrams.
\end{definition}

\begin{example}[label=Cat]
  The \vdc{} $\Cat$ is defined as follows.
  \begin{enumerate}
    \item The underlying category is the category of locally small categories and functors.
    \item A loose-cell from $A$ to $B$ is a distributor from $\A$ to $\B$, \ie{} a functor $\B\op \times \A \to \Set$.
    \item A 2-cell with the following frame
      \[\begin{tikzcd}
        {\b A_0} & \cdots & {\b A_n} \\
        {\B} && {\B'}
        \arrow[""{name=0, anchor=center, inner sep=0}, "f"', from=1-1, to=2-1]
        \arrow["{p_1}"'{inner sep=.8ex}, "\shortmid"{marking}, from=1-2, to=1-1]
        \arrow["{p_n}"'{inner sep=.8ex}, "\shortmid"{marking}, from=1-3, to=1-2]
        \arrow[""{name=1, anchor=center, inner sep=0}, "{f'}", from=1-3, to=2-3]
        \arrow["q"{inner sep=.8ex}, "\shortmid"{marking}, from=2-3, to=2-1]
        \arrow["\phi"{description}, draw=none, from=1, to=0]
      \end{tikzcd}\]
      comprises a family of functions
      \[\phi_{a_0, \ldots, a_n} \colon p_1(a_0, a_1) \times \cdots \times p_n(a_{n - 1}, a_n) \to q(f (a_0), f' (a_n))\]
      for each $a_0 \in \ob{\b A_0}, \dots, a_n \in \ob{\b A_n}$, satisfying naturality conditions~\cite[Definition~8.1]{arkor2024formal}.
  \end{enumerate}
  The identity 2-cell on a distributor is given by the family of identity functions, and the composition of 2-cells is induced by functoriality of the cartesian product of sets.
\end{example}

In contrast to 2-categories, which admit three notions of dual, there is a single notion of dual for \vdcs{}, which for the \vdc{} $\Cat$ incorporates both the 2-category $\b{Cat}\co$ and the bicategory $\b{Dist}\op$.

\begin{example}[{\cite[Definition~2.12]{arkor2024formal}}]
  \label{dual}
  The \emph{dual} $\X\co$ of a \vdc{} $\X$ is the \vdc{} with the same objects and tight-cells as $\X$, for which a loose-cell $A \lto B$ in $\X\co$ is a loose-cell $B \lto A$ in $\X$, and for which a 2-cell with the frame on the left below in $\X\co$ is a 2-cell with the frame on the right below in $\X$.
  \[
  \begin{tikzcd}
    {A_0} & \cdots & {A_n} \\
    {B_0} && {B_n}
    \arrow["{f_0}"', from=1-1, to=2-1]
    \arrow["{p_1}"'{inner sep=.8ex}, "\shortmid"{marking}, from=1-2, to=1-1]
    \arrow["{p_n}"'{inner sep=.8ex}, "\shortmid"{marking}, from=1-3, to=1-2]
    \arrow["{f_n}", from=1-3, to=2-3]
    \arrow["q"{inner sep=.8ex}, "\shortmid"{marking}, from=2-3, to=2-1]
  \end{tikzcd}
  \hspace{6em}
  \begin{tikzcd}
    {A_n} & \cdots & {A_0} \\
    {B_n} && {B_0}
    \arrow["{f_n}"', from=1-1, to=2-1]
    \arrow["{p_n}"'{inner sep=.8ex}, "\shortmid"{marking}, from=1-2, to=1-1]
    \arrow["{p_1}"'{inner sep=.8ex}, "\shortmid"{marking}, from=1-3, to=1-2]
    \arrow["{f_0}", from=1-3, to=2-3]
    \arrow["q"{inner sep=.8ex}, "\shortmid"{marking}, from=2-3, to=2-1]
  \end{tikzcd}
  \qedshift
  \]
\end{example}

The \vdcs{} in which we shall be interested satisfy several important properties. The first is the ability to restrict loose-cells along tight-cells: this axiomatises the restriction of distributors along functors.

\begin{definition}[{\cite[Definition~7.1]{cruttwell2010unified}}]
	\label{cartesian-cell}
  A 2-cell
	\[\begin{tikzcd}
		\cdot & \cdot \\
		\cdot & \cdot
		\arrow["p"', "\shortmid"{marking}, from=1-2, to=1-1]
		\arrow["q", "\shortmid"{marking}, from=2-2, to=2-1]
		\arrow[""{name=0, anchor=center, inner sep=0}, "g", from=1-2, to=2-2]
		\arrow[""{name=1, anchor=center, inner sep=0}, "f"', from=1-1, to=2-1]
		\arrow["\cart"{description}, draw=none, from=0, to=1]
	\end{tikzcd}\]
  in a \vdc{} is \emph{cartesian} if any 2-cell, as on the left below, factors uniquely therethrough:
	\[\begin{tikzcd}
		\cdot & \cdots & \cdot \\
		\cdot && \cdot
		\arrow["q", "\shortmid"{marking}, from=2-3, to=2-1]
		\arrow[""{name=0, anchor=center, inner sep=0}, "{g' \d g}", from=1-3, to=2-3]
		\arrow[""{name=1, anchor=center, inner sep=0}, "{f' \d f}"', from=1-1, to=2-1]
		\arrow["{r_n}"', "\shortmid"{marking}, from=1-3, to=1-2]
		\arrow["{r_1}"', "\shortmid"{marking}, from=1-2, to=1-1]
		\arrow["\phi"{description}, draw=none, from=0, to=1]
	\end{tikzcd}
  \quad = \quad
	\begin{tikzcd}
		\cdot & \cdots & \cdot \\
		\cdot && \cdot \\
		\cdot && \cdot
		\arrow["q", "\shortmid"{marking}, from=3-3, to=3-1]
		\arrow["{r_n}"', "\shortmid"{marking}, from=1-3, to=1-2]
		\arrow[""{name=0, anchor=center, inner sep=0}, "g", from=2-3, to=3-3]
		\arrow[""{name=1, anchor=center, inner sep=0}, "f"', from=2-1, to=3-1]
		\arrow["p"{description}, from=2-3, to=2-1]
		\arrow[""{name=2, anchor=center, inner sep=0}, "{g'}", from=1-3, to=2-3]
		\arrow[""{name=3, anchor=center, inner sep=0}, "{f'}"', from=1-1, to=2-1]
		\arrow["{r_1}"', "\shortmid"{marking}, from=1-2, to=1-1]
		\arrow["\cart"{description}, draw=none, from=0, to=1]
		\arrow["\hat\phi"{description}, draw=none, from=2, to=3]
	\end{tikzcd}\]
  In this case, we call $p$ the \emph{restriction} $q(f, g)$.
  If $q$ is a loose-identity $A(1, 1)$ (see \cref{opcartesian}), we denote $p = A(1, 1)(f, g)$ simply by $A(f, g)$. For a tight-cell $f \colon A \to B$, the restriction ${B(1, f) \colon A \lto B}$, when it exists, is the \emph{companion} of $f$; and the restriction ${B(f, 1) \colon B \lto A}$, when it exists, is the \emph{conjoint} of $f$. A loose-cell is \emph{representable} if it is the companion of some tight-cell; and \emph{corepresentable} if it is the conjoint of some tight-cell.
\end{definition}

\begin{example}[label=restriction-in-Cat]
  The restriction of a distributor $q \colon \B' \lto \A'$ along functors $f \colon \A \to \A'$ and $g \colon \B \to \B'$ in $\Cat$ is the distributor $q(f, g) \defeq (a, b) \mapsto q(f(a), g(b))$.
\end{example}

Note that a \vdc{} $\X$ admits restrictions if and only if $\X\co$ does~\cite[Lemma~2.14]{arkor2024formal}.

The second property we shall be interested in is the existence of identity loose-cells, corresponding to the hom-set distributors $\A({-}, {-}) \colon \A\op \times \A \to \Set$. While \vdcs{} do not in general admit composites of loose-cells, we may identify the existence of composites via a universal property.

\begin{definition}[{\cites[Definition~2.7]{dawson2006paths}[Definition~5.1]{cruttwell2010unified}}]
	\label{opcartesian}
  A 2-cell
	\[\begin{tikzcd}
		\cdot & \cdots & \cdot \\
		\cdot && \cdot
		\arrow["{q_m}"', "\shortmid"{marking}, from=1-3, to=1-2]
		\arrow["{q_1}"', "\shortmid"{marking}, from=1-2, to=1-1]
		\arrow["q", "\shortmid"{marking}, from=2-3, to=2-1]
		\arrow[""{name=0, anchor=center, inner sep=0}, Rightarrow, no head, nfold, from=1-3, to=2-3]
		\arrow[""{name=1, anchor=center, inner sep=0}, Rightarrow, no head, nfold, from=1-1, to=2-1]
		\arrow["\opcart"{description}, draw=none, from=0, to=1]
	\end{tikzcd}\]
  in a \vdc{} is \emph{opcartesian} if any 2-cell
	\[\begin{tikzcd}
		\cdot & \cdots & \cdot & \cdots & \cdot & \cdots & \cdot \\
		\cdot &&&&&& \cdot
		\arrow["{r_n}"', "\shortmid"{marking}, from=1-7, to=1-6]
		\arrow["{r_1}"', "\shortmid"{marking}, from=1-6, to=1-5]
		\arrow["{q_m}"', "\shortmid"{marking}, from=1-5, to=1-4]
		\arrow["{q_1}"', "\shortmid"{marking}, from=1-4, to=1-3]
		\arrow["{p_l}"', "\shortmid"{marking}, from=1-3, to=1-2]
		\arrow["{p_1}"', "\shortmid"{marking}, from=1-2, to=1-1]
		\arrow[""{name=0, anchor=center, inner sep=0}, "g", from=1-7, to=2-7]
		\arrow[""{name=1, anchor=center, inner sep=0}, "f"', from=1-1, to=2-1]
		\arrow["s", "\shortmid"{marking}, from=2-7, to=2-1]
		\arrow["\phi"{description}, draw=none, from=0, to=1]
	\end{tikzcd}\]
  factors uniquely therethrough:
	\[\begin{tikzcd}
		\cdot & \cdots & \cdot & \cdots & \cdot & \cdots & \cdot \\
		\cdot & \cdots & \cdot && \cdot & \cdots & \cdot \\
		\cdot &&&&&& \cdot
		\arrow["{r_n}"', "\shortmid"{marking}, from=1-7, to=1-6]
		\arrow["{r_1}"', "\shortmid"{marking}, from=1-6, to=1-5]
		\arrow["{q_m}"', "\shortmid"{marking}, from=1-5, to=1-4]
		\arrow["{q_1}"', "\shortmid"{marking}, from=1-4, to=1-3]
		\arrow["{p_l}"', "\shortmid"{marking}, from=1-3, to=1-2]
		\arrow["{p_1}"', "\shortmid"{marking}, from=1-2, to=1-1]
		\arrow["s", "\shortmid"{marking}, from=3-7, to=3-1]
		\arrow["q"{description}, from=2-5, to=2-3]
		\arrow[""{name=0, anchor=center, inner sep=0}, Rightarrow, no head, nfold, from=1-5, to=2-5]
		\arrow[""{name=1, anchor=center, inner sep=0}, "g", from=2-7, to=3-7]
		\arrow[""{name=2, anchor=center, inner sep=0}, "f"', from=2-1, to=3-1]
		\arrow[""{name=3, anchor=center, inner sep=0}, Rightarrow, no head, nfold, from=1-7, to=2-7]
		\arrow[""{name=4, anchor=center, inner sep=0}, Rightarrow, no head, nfold, from=1-1, to=2-1]
		\arrow["{r_n}"{description}, from=2-7, to=2-6]
		\arrow["{r_1}"{description}, from=2-6, to=2-5]
		\arrow["{p_l}"{description}, from=2-3, to=2-2]
		\arrow["{p_1}"{description}, from=2-2, to=2-1]
		\arrow[""{name=5, anchor=center, inner sep=0}, Rightarrow, no head, nfold, from=1-3, to=2-3]
		\arrow["{=}"{description}, draw=none, from=3, to=0]
		\arrow["{=}"{description}, draw=none, from=5, to=4]
		\arrow["\opcart"{description}, draw=none, from=0, to=5]
		\arrow["\check\phi"{description}, draw=none, from=1, to=2]
	\end{tikzcd}\]
  In this case, we call $q$ the \emph{(loose) composite} $q_1 \odot \cdots \odot q_m$ of $q_1, \ldots, q_m$.
	When $m = 0$, we call $q \colon A \lto A$ the \emph{loose-identity} and denote it by $A(1, 1)$, or simply by $=\!\!\!|\!\!\!=$ in pasting diagrams. A \vdc{} is \emph{normal} if every object admits a loose-identity.

	We denote a nullary 2-cell with loose-identity codomain by $\phi \colon f \tto g$.
	\[\begin{tikzcd}
		\cdot & \cdot \\
		\cdot & \cdot
		\arrow["\shortmid"{marking}, Rightarrow, no head, nfold, from=2-2, to=2-1]
		\arrow[""{name=0, anchor=center, inner sep=0}, "g", from=1-2, to=2-2]
		\arrow[""{name=1, anchor=center, inner sep=0}, "f"', from=1-1, to=2-1]
		\arrow[Rightarrow, no head, nfold, from=1-2, to=1-1]
		\arrow["\phi"{description}, draw=none, from=0, to=1]
	\end{tikzcd}\qedshift\]
\end{definition}

\begin{example}
  As indicated above, every object $\A$ in $\Cat$ admits a loose-identity, given by the hom-set distributor $(a, a') \mapsto \A(a, a')$. Furthermore, given distributors $p \colon \B \lto \A$ and $q \colon \C \lto \B$, if the following coend exists (for instance, if $\B$ is small), then it exhibits the composite $p \odot q \colon \C \lto \A$.
  \[(a, c) \mapsto \int^{b \in \B} p(a, b) \times q(b, c) \qedhere\]
\end{example}

In fact, for the results with which we shall be concerned, it will typically suffice to consider a weaker universal property than that of an opcartesian 2-cell.

\begin{definition}[label=left-opcartesian, note={\cite[Definition~2.5]{arkor2024formal}}]
	A 2-cell in a \vdc{} is \emph{left-opcartesian} if it satisfies the universal property of an opcartesian 2-cell (\cref{opcartesian}) in the special case where $l = 0$.\footnote{The definition of left-opcartesianness in \cite[Definition~2.5]{arkor2024formal} is slightly weaker than that defined here, but the two definitions coincide in the presence of restrictions.} In this case, we call $q$ the \emph{left-composite} $q_1 \odotl \cdots \odotl q_m$ of $q_1, \ldots, q_m$.
  Dually, a 2-cell is \emph{right-opcartesian} if it satisfies the universal property in the special case where $n = 0$. In this case, we call $q$ the \emph{right-composite} $q_1 \odotr \cdots \odotr q_m$.
\end{definition}

Left-composites are unique up to isomorphism and are left-associative in the sense that there is a canonical isomorphism,
\begin{equation}
  \label{left-associativity-of-left-composites}
  q_1 \odotl q_2 \odotl \cdots \odotl q_m \iso (\cdots (q_1 \odotl q_2) \odotl \cdots) \odotl q_m
\end{equation}
together with a canonical reassociating 2-cell (which is invertible if the left-composites are actual composites).
\begin{equation}
  q_1 \odotl q_2 \odotl \cdots \odotl q_m \tto q_1 \odotl (\cdots \odotl (q_{m - 1} \odotl q_m) \cdots)
\end{equation}
Dual statements hold for right-composites, a left-composite in $\X\co$ being precisely a right-composite in $\X$. We reserve no notation for nullary left- or right-composites, since we shall not make use of them here (all loose-identities we consider have the full universal property of \cref{opcartesian}). Observe that loose-composites are, in particular, left- and right-composites (but the converse is not true in general). We record the following useful interaction between left-composites and loose-composites, permitting reassociation.

\begin{lemma}
  \label{left-composite-reassociation}
  Suppose that the left-composite of a chain $q_1, \ldots, q_m$ exists and that the composite $q_i \odot \cdots \odot q_j$ exists ($1 \leq i \leq j \leq m$). Then the left-composite on the left below exhibits the left-composite on the right below.
  \[q_1 \odotl \cdots \odotl q_i \odotl \cdots \odotl q_j \odotl \cdots \odotl q_m \iso q_1 \odotl \cdots \odotl (q_i \odot \cdots \odot q_j) \odotl \cdots \odotl q_m\]
\end{lemma}

\begin{proof}
  First, factor the left-opcartesian 2-cell associated with $q_1 \odotl \cdots \odotl q_m$ through the opcartesian 2-cell associated with $q_i \odot \cdots \odot q_j$.
  \[\begin{tikzcd}[column sep=huge]
    \cdot & \cdots & \cdot & \cdots & \cdot & \cdots & \cdot \\
    \cdot & \cdots & \cdot && \cdot & \cdots & \cdot \\
    \cdot &&&&&& \cdot
    \arrow[""{name=0, anchor=center, inner sep=0}, equals, nfold, from=1-1, to=2-1]
    \arrow["{q_1}"'{inner sep=.8ex}, "\shortmid"{marking}, from=1-2, to=1-1]
    \arrow["{q_{i - 1}}"'{inner sep=.8ex}, "\shortmid"{marking}, from=1-3, to=1-2]
    \arrow[""{name=1, anchor=center, inner sep=0}, equals, nfold, from=1-3, to=2-3]
    \arrow["{q_i}"'{inner sep=.8ex}, "\shortmid"{marking}, from=1-4, to=1-3]
    \arrow["{q_j}"'{inner sep=.8ex}, "\shortmid"{marking}, from=1-5, to=1-4]
    \arrow[""{name=2, anchor=center, inner sep=0}, equals, nfold, from=1-5, to=2-5]
    \arrow["{q_{j + 1}}"'{inner sep=.8ex}, "\shortmid"{marking}, from=1-6, to=1-5]
    \arrow["{q_m}"'{inner sep=.8ex}, "\shortmid"{marking}, from=1-7, to=1-6]
    \arrow[""{name=3, anchor=center, inner sep=0}, equals, nfold, from=1-7, to=2-7]
    \arrow[""{name=4, anchor=center, inner sep=0}, equals, nfold, from=2-1, to=3-1]
    \arrow["{q_1}"{description}, from=2-2, to=2-1]
    \arrow["{q_{i - 1}}"{description}, from=2-3, to=2-2]
    \arrow["{q_i \odot \cdots \odot q_j}"{description}, from=2-5, to=2-3]
    \arrow["{q_{j + 1}}"{description}, from=2-6, to=2-5]
    \arrow["{q_m}"{description}, from=2-7, to=2-6]
    \arrow[""{name=5, anchor=center, inner sep=0}, equals, nfold, from=2-7, to=3-7]
    \arrow["{q_1 \odotl \cdots \odotl q_i \odot \cdots \odot q_j \odotl \cdots \odotl q_m}"{inner sep=.8ex}, "\shortmid"{marking}, from=3-7, to=3-1]
    \arrow["{=}"{description}, draw=none, from=0, to=1]
    \arrow["\opcart"{description}, draw=none, from=1, to=2]
    \arrow["{=}"{description}, draw=none, from=2, to=3]
    \arrow["{\check{\opcartl}}"{description}, draw=none, from=4, to=5]
  \end{tikzcd}\]
  We will show that $\check{\opcartl}$ is left-opcartesian. Given a 2-cell as follows,
  \[\begin{tikzcd}[column sep=huge]
    \cdot & \cdots & \cdots & \cdot & \cdots & \cdot \\
    \cdot &&&&& \cdot
    \arrow[""{name=0, anchor=center, inner sep=0}, "f"', from=1-1, to=2-1]
    \arrow["{q_1}"'{inner sep=.8ex}, "\shortmid"{marking}, from=1-2, to=1-1]
    \arrow["{q_i \odot \cdots \odot q_j}"'{inner sep=.8ex}, "\shortmid"{marking}, from=1-3, to=1-2]
    \arrow["{q_m}"'{inner sep=.8ex}, "\shortmid"{marking}, from=1-4, to=1-3]
    \arrow["{r_1}"'{inner sep=.8ex}, "\shortmid"{marking}, from=1-5, to=1-4]
    \arrow["{r_n}"'{inner sep=.8ex}, "\shortmid"{marking}, from=1-6, to=1-5]
    \arrow[""{name=1, anchor=center, inner sep=0}, "g", from=1-6, to=2-6]
    \arrow["s"{inner sep=.8ex}, "\shortmid"{marking}, from=2-6, to=2-1]
    \arrow["\phi"{description}, draw=none, from=1, to=0]
  \end{tikzcd}\]
  first precompose the opcartesian 2-cell associated with $q_i \odot \cdots \odot q_j$, then factorise the resulting 2-cell through the left-opcartesian 2-cell associated with $q_1 \odotl \cdots \odotl q_m$. This produces an appropriate factorisation of $\phi$, whose uniqueness follows from the universal properties of the opcartesian and left-opcartesian 2-cells.
\end{proof}

The ambient setting for formal category theory within which we shall work in this paper is that of a \emph{strict \ve}.

\begin{definition}[{\cite[Definition~7.6]{cruttwell2010unified}}]
  A \emph{\ve} is a \vdc{} that admits all loose-identities and restrictions.
\end{definition}

\begin{definition}[{\cite[Definition~5.1]{arkor2025nerve}}]
    \label{strict-ve}
    A \ve{} is \emph{strict} if it is equipped with a strictly functorial choice of restrictions, in the sense that the following 2-cells exhibit the chosen cartesian 2-cells (so that $p(1, 1) = p$ and $p(f, g)(f', g') = p(ff', gg')$), for all objects $A, A', A'', B, B', B''$, tight-cells $f, f', g, g'$, and loose-cell $p$.
    \[
    \begin{tikzcd}[column sep=large]
      A & B \\
      A & B
      \arrow[""{name=0, anchor=center, inner sep=0}, equals, nfold, from=1-1, to=2-1]
      \arrow["{p(1, 1)}"'{inner sep=.8ex}, "\shortmid"{marking}, from=1-2, to=1-1]
      \arrow[""{name=1, anchor=center, inner sep=0}, equals, nfold, from=1-2, to=2-2]
      \arrow["p"{inner sep=.8ex}, "\shortmid"{marking}, from=2-2, to=2-1]
      \arrow["{=}"{description}, draw=none, from=0, to=1]
    \end{tikzcd}
    \hspace{4em}
    \begin{tikzcd}
      {A''} && {B''} \\
      {A'} && {B'} \\
      A && B
      \arrow[""{name=0, anchor=center, inner sep=0}, "{f'}"', from=1-1, to=2-1]
      \arrow["{p(ff', gg')}"'{inner sep=.8ex}, "\shortmid"{marking}, from=1-3, to=1-1]
      \arrow[""{name=1, anchor=center, inner sep=0}, "{g'}", from=1-3, to=2-3]
      \arrow[""{name=2, anchor=center, inner sep=0}, "f"', from=2-1, to=3-1]
      \arrow["{p(f, g)}"{description}, from=2-3, to=2-1]
      \arrow[""{name=3, anchor=center, inner sep=0}, "g", from=2-3, to=3-3]
      \arrow["p"{inner sep=.8ex}, "\shortmid"{marking}, from=3-3, to=3-1]
      \arrow["\cart"{description}, draw=none, from=1, to=0]
      \arrow["\cart"{description}, draw=none, from=2, to=3]
    \end{tikzcd}
    \qedshift
    \]
\end{definition}

Note that strictness is a mild assumption: every \ve{} is equivalent to a strict \ve{}~\cite[Theorem~A.1]{arkor2025nerve}, and examples of \ve s that arise in practice (for instance, $\Cat$) tend to be strict. Furthermore, strictness is a useful assumption, as it generally allows us to reason about universal properties involving restrictions up to isomorphism rather than up to equivalence (this permits our definition of presheaf object in \cref{Phi-presheaf-object} to have a strict universal property, for instance, which simplifies calculations considerably).

\begin{notation}[label=cp-notation]
  Let $f \colon A \to B$ be a tight-cell in a \ve. Denote by $\cp f \colon B(1, f), B(f, 1) \tto B(1, 1)$ the following 2-cell.
  \[
  \cp f \defeq
  \begin{tikzcd}
    B & A & B \\
    B & B & B \\
    B && B
    \arrow[""{name=0, anchor=center, inner sep=0}, equals, from=1-1, to=2-1]
    \arrow["{B(1, f)}"'{inner sep=.8ex}, "\shortmid"{marking}, from=1-2, to=1-1]
    \arrow[""{name=1, anchor=center, inner sep=0}, "f"{description}, from=1-2, to=2-2]
    \arrow["{B(f, 1)}"'{inner sep=.8ex}, "\shortmid"{marking}, from=1-3, to=1-2]
    \arrow[""{name=2, anchor=center, inner sep=0}, equals, from=1-3, to=2-3]
    \arrow[""{name=3, anchor=center, inner sep=0}, equals, from=2-1, to=3-1]
    \arrow["\shortmid"{marking}, equals, from=2-2, to=2-1]
    \arrow["\shortmid"{marking}, equals, from=2-3, to=2-2]
    \arrow[""{name=4, anchor=center, inner sep=0}, equals, from=2-3, to=3-3]
    \arrow["\shortmid"{marking}, equals, from=3-3, to=3-1]
    \arrow["\cart"{description}, draw=none, from=0, to=1]
    \arrow["\cart"{description}, draw=none, from=1, to=2]
    \arrow["\opcart"{description}, draw=none, from=3, to=4]
  \end{tikzcd}\qedshift
  \]
\end{notation}

\subsection{Lifts and extensions}
\label{weights}

Next, we shall explain how limits and colimits in categories are expressed in a \ve. We shall be concerned in this paper with free cocompletions under various colimit shapes. In \fct, as in enriched category theory, the appropriate data by which to parameterise a colimit shape is a \emph{weight}. In enriched category theory, it is customary to define a weight for a colimit simply to be a presheaf~\cite{kelly1982basic} (\cf~\cref{presheaves-and-cocompletions-in-V-Cat}). However, it has long been observed that it is more flexible to define a weight to be an arbitrary distributor~\cite{street1978yoneda}; in particular, this permits a single notion of weight for both limits and colimits. In maximum generality, for settings in which composites of distributors may not exist (for instance, if we are concerned with categories that are not necessarily small), it is most appropriate to consider \emph{chains} of distributors~\cite[\S1.11]{koudenburg2024formal}.

\begin{definition}
    \label{weight}
    A \emph{weight} in a \vdc{} is a chain of loose-cells.
    \[\underbrace{X_0 \xlfrom{p_1} X_1 \xlfrom{p_2} \cdots \xlfrom{p_m} X_m}_{\vec p} \tag{$m \ge 0$}\]
    We write $(p_1, \ldots, p_m) \colon X_m \lcto X_0$ or simply $\vec p \colon X_m \lcto X_0$ for this situation.
    A weight is \emph{unary} if $m = 1$, \ie{} it comprises a single loose-cell.
\end{definition}

Note that, in contrast to classical definitions of weight in enriched category theory~\cite{kelly1982basic,albert1988closure,kelly2005notes}, we do not impose any smallness assumptions in the definition of a weight. In practice, it is often useful to talk about certain large colimits (\eg~adjoints, or large cointersections of epimorphisms~\cite{kelly1981large}), and most of the properties of colimits we will consider are independent of size concerns.

The universal properties of weighted limits and weighted colimits in a \ve{} are expressed in terms of right \emph{extensions} and \emph{lifts} respectively~\cite{wood1982abstract}. In the following, we generalise the notions of \emph{right lift} and \emph{right extension} from \cites[Definitions~3.1 \& 3.2]{arkor2024formal}[Definition~9.1.2]{riehl2022elements} to permit arbitrary (rather than unary) weights (\cf~\cite[\S1]{koudenburg2024formal}). We note that, for lifts and extensions in an arbitrary \vdc{}, a stronger universal property involving tight-cells would be necessary, but in a \ve{} this simpler definition suffices (\cf~\cite[\S7]{pare2024retro}).

\begin{definition}
	\label{right-lift}
  Let $\vec p = (p_1, \ldots, p_m) \colon Y \lcto Z$ be a weight and let $q \colon X \lto Z$ be a loose-cell.
  A loose-cell $q \rf \vec p \colon X \lto Y$, equipped with a 2-cell $\varpi \colon p_1, \ldots, p_m, q \rf \vec p \tto q$, is the \emph{right lift} of $q$ through $p_1, \ldots, p_m$ when every 2-cell as on the left below factors uniquely as a diagram of the form on the right below (where $n \geq 0$). We call $\varpi$ the \emph{counit} of the right lift.
	\[
  \begin{tikzcd}[row sep=5em]
    Z & \cdots & Y & \cdots & X \\
    Z &&&& X
    \arrow[""{name=0, anchor=center, inner sep=0}, equals, nfold, from=1-1, to=2-1]
    \arrow["{p_1}"', "\shortmid"{marking}, from=1-2, to=1-1]
    \arrow["{p_m}"', "\shortmid"{marking}, from=1-3, to=1-2]
    \arrow["{r_1}"', "\shortmid"{marking}, from=1-4, to=1-3]
    \arrow["{r_n}"', "\shortmid"{marking}, from=1-5, to=1-4]
    \arrow[""{name=1, anchor=center, inner sep=0}, equals, nfold, from=1-5, to=2-5]
    \arrow["q", "\shortmid"{marking}, from=2-5, to=2-1]
    \arrow["\phi"{description}, draw=none, from=1, to=0]
  \end{tikzcd}
	\hspace{1em} = \hspace{1em}
  \begin{tikzcd}
    Z & \cdots & Y & \cdots & X \\
    Z & \cdots & Y && X \\
    Z &&&& X
    \arrow[""{name=0, anchor=center, inner sep=0}, equals, nfold, from=1-1, to=2-1]
    \arrow["{p_1}"'{inner sep=.8ex}, "\shortmid"{marking}, from=1-2, to=1-1]
    \arrow["{p_m}"'{inner sep=.8ex}, "\shortmid"{marking}, from=1-3, to=1-2]
    \arrow[""{name=1, anchor=center, inner sep=0}, equals, nfold, from=1-3, to=2-3]
    \arrow["{r_1}"'{inner sep=.8ex}, "\shortmid"{marking}, from=1-4, to=1-3]
    \arrow["{r_n}"'{inner sep=.8ex}, "\shortmid"{marking}, from=1-5, to=1-4]
    \arrow[""{name=2, anchor=center, inner sep=0}, equals, nfold, from=1-5, to=2-5]
    \arrow[""{name=3, anchor=center, inner sep=0}, equals, nfold, from=2-1, to=3-1]
    \arrow["{p_1}"{description}, from=2-2, to=2-1]
    \arrow["{p_m}"{description}, from=2-3, to=2-2]
    \arrow["{q \rf (p_1, \ldots, p_m)}"{description}, from=2-5, to=2-3]
    \arrow[""{name=4, anchor=center, inner sep=0}, equals, nfold, from=2-5, to=3-5]
    \arrow["q"{inner sep=.8ex}, "\shortmid"{marking}, from=3-5, to=3-1]
    \arrow["{=}"{description}, draw=none, from=1, to=0]
    \arrow["{\lambda\phi}"{description}, draw=none, from=2, to=1]
    \arrow["\varpi"{description}, draw=none, from=3, to=4]
  \end{tikzcd}
  \qedshift
	\]
\end{definition}

\begin{definition}[label=lifting]
  A weight $\vec p \colon Y \lcto Z$ is \emph{lifting} if all right lifts through $\vec p$ exist, \ie{} if, for every loose-cell $q \colon X \lto Z$, a right lift $q \rf \vec p \colon X \lto Y$ exists.
\end{definition}

A right extension in a \ve{} $\X$ is, by definition, a right lift in $\X\co$. We spell out the definition for convenience.

\begin{definition}
  \label{right-extension}
  Let $\vec p = (p_1, \ldots, p_m) \colon X \lcto Y$ be a weight and let $q \colon X \lto Z$ be a loose-cell. A loose-cell $\vec p \rx q \colon Y \lto Z$ equipped with a 2-cell $\varpi \colon \vec p \rx q, p_1, \ldots, p_m \tto q$ is the \emph{right extension} of $q$ along $p_1, \ldots, p_m$ when every 2-cell as on the left below factors uniquely as a diagram of the form on the right below (where $n \geq 0$). We call $\varpi$ the \emph{counit} of the right extension.
	\[
  \begin{tikzcd}[row sep=5em]
    Z & \cdots & Y & \cdots & X \\
    Z &&&& X
    \arrow[""{name=0, anchor=center, inner sep=0}, equals, nfold, from=1-1, to=2-1]
    \arrow["{r_1}"', "\shortmid"{marking}, from=1-2, to=1-1]
    \arrow["{r_n}"', "\shortmid"{marking}, from=1-3, to=1-2]
    \arrow["{p_1}"', "\shortmid"{marking}, from=1-4, to=1-3]
    \arrow["{p_m}"', "\shortmid"{marking}, from=1-5, to=1-4]
    \arrow[""{name=1, anchor=center, inner sep=0}, equals, nfold, from=1-5, to=2-5]
    \arrow["q", "\shortmid"{marking}, from=2-5, to=2-1]
    \arrow["\phi"{description}, draw=none, from=1, to=0]
  \end{tikzcd}
	\hspace{1em} = \hspace{1em}
  \begin{tikzcd}
    Z & \cdots & Y & \cdots & X \\
    Z && Y & \cdots & X \\
    Z &&&& X
    \arrow[""{name=0, anchor=center, inner sep=0}, equals, nfold, from=1-1, to=2-1]
    \arrow["{r_1}"'{inner sep=.8ex}, "\shortmid"{marking}, from=1-2, to=1-1]
    \arrow["{r_n}"'{inner sep=.8ex}, "\shortmid"{marking}, from=1-3, to=1-2]
    \arrow[""{name=1, anchor=center, inner sep=0}, equals, nfold, from=1-3, to=2-3]
    \arrow["{p_1}"'{inner sep=.8ex}, "\shortmid"{marking}, from=1-4, to=1-3]
    \arrow["{p_m}"'{inner sep=.8ex}, "\shortmid"{marking}, from=1-5, to=1-4]
    \arrow[""{name=2, anchor=center, inner sep=0}, equals, nfold, from=1-5, to=2-5]
    \arrow[""{name=3, anchor=center, inner sep=0}, equals, nfold, from=2-1, to=3-1]
    \arrow["{(p_1, \ldots, p_m) \rx q}"{description}, from=2-3, to=2-1]
    \arrow["{p_1}"{description}, from=2-4, to=2-3]
    \arrow["{p_m}"{description}, from=2-5, to=2-4]
    \arrow[""{name=4, anchor=center, inner sep=0}, equals, nfold, from=2-5, to=3-5]
    \arrow["q"{inner sep=.8ex}, "\shortmid"{marking}, from=3-5, to=3-1]
    \arrow["{\lambda\phi}"{description}, draw=none, from=1, to=0]
    \arrow["{=}"{description}, draw=none, from=2, to=1]
    \arrow["\varpi"{description}, draw=none, from=3, to=4]
  \end{tikzcd}
  \qedshift
  \]
\end{definition}

\begin{definition}[label=extending]
  A weight $\vec p \colon X \lcto Y$ is \emph{extending} if all right extensions along $\vec p$ exist, \ie{} if, for every loose-cell $q \colon X \lto Z$, a right extension $\vec p \rx q \colon Y \lto Z$ exists.
\end{definition}

In the presence of left- and right-composites, lifts through, and extensions along, arbitrary weights are no more general than lifts through, and extensions along, unary weights.

\begin{lemma}
  \label{right-lift-through-left-composite}
  Let $q \colon X \lto Z$ be a loose-cell.
  \begin{enumerate}
    \item Let $(p_1, \ldots, p_m) \colon Y \lcto Z$ be a weight. If the left-composite $p_1 \odotl \cdots \odotl p_i$ exists ($1 \leq i \leq m$), then, if either side of the following exists, so does the other, in which case they are isomorphic.
    \[q \rf (p_1, \ldots, p_m) \iso q \rf (p_1 \odotl \cdots \odotl p_i, p_{i + 1}, \ldots, p_m)\]
    \item Let $(p_1, \ldots, p_m) \colon X \lcto Y$ be a weight. If the right-composite $p_i \odotr \cdots \odotr p_m$ exists ($1 \leq i \leq m$), then, if either side of the following exists, so does the other, in which case they are isomorphic.
    \[(p_1, \ldots, p_m) \rx q \iso (p_1, \ldots, p_{i - 1}, p_i \odotr \cdots \odotr p_m) \rx q\]
  \end{enumerate}
\end{lemma}

\begin{proof}
  We focus on (1), since (2) follows by duality. Suppose that $q \rf (p_1, \ldots, p_m)$ exists. By the universal property of $p_1 \odotl \cdots \odotl p_i$, the counit $\varpi$ factors uniquely through a 2-cell $\check\varpi$.
  \[\begin{tikzcd}[column sep=large]
    Z & \cdots & \cdot & \cdots & Y && X \\
    Z && \cdot & \cdots & Y && X \\
    Z &&&&&& X
    \arrow[""{name=0, anchor=center, inner sep=0}, equals, nfold, from=1-1, to=2-1]
    \arrow["{p_1}"'{inner sep=.8ex}, "\shortmid"{marking}, from=1-2, to=1-1]
    \arrow["{p_i}"'{inner sep=.8ex}, "\shortmid"{marking}, from=1-3, to=1-2]
    \arrow[""{name=1, anchor=center, inner sep=0}, equals, nfold, from=1-3, to=2-3]
    \arrow["{p_{i + 1}}"'{inner sep=.8ex}, "\shortmid"{marking}, from=1-4, to=1-3]
    \arrow["{p_m}"'{inner sep=.8ex}, "\shortmid"{marking}, from=1-5, to=1-4]
    \arrow[""{name=2, anchor=center, inner sep=0}, equals, nfold, from=1-5, to=2-5]
    \arrow["{q \rf (p_1, \ldots, p_m)}"'{inner sep=.8ex}, "\shortmid"{marking}, from=1-7, to=1-5]
    \arrow[""{name=3, anchor=center, inner sep=0}, equals, nfold, from=1-7, to=2-7]
    \arrow[""{name=4, anchor=center, inner sep=0}, equals, nfold, from=2-1, to=3-1]
    \arrow["{p_1 \odotl \cdots \odotl p_i}"{description}, from=2-3, to=2-1]
    \arrow["{p_{i + 1}}"{description}, from=2-4, to=2-3]
    \arrow["{p_m}"{description}, from=2-5, to=2-4]
    \arrow["{q \rf (p_1, \ldots, p_m)}"{description}, from=2-7, to=2-5]
    \arrow[""{name=5, anchor=center, inner sep=0}, equals, nfold, from=2-7, to=3-7]
    \arrow["q"{inner sep=.8ex}, "\shortmid"{marking}, from=3-7, to=3-1]
    \arrow["\opcartl"{description}, draw=none, from=0, to=1]
    \arrow["{=}"{description}, draw=none, from=1, to=2]
    \arrow["{=}"{description}, draw=none, from=2, to=3]
    \arrow["{\check\varpi}"{description}, draw=none, from=4, to=5]
  \end{tikzcd}\]
  Now, suppose we are given a 2-cell $\phi$ as follows, which we may precompose by the left-opcartesian 2-cell defining $p_1 \odotl \cdots \odotl p_i$.
  \[\begin{tikzcd}[column sep=large]
    Z & \cdots & \cdot & \cdots & Y & \cdots & X \\
    Z && \cdot & \cdots & Y & \cdots & X \\
    Z &&&&&& X
    \arrow[""{name=0, anchor=center, inner sep=0}, equals, nfold, from=1-1, to=2-1]
    \arrow["{p_1}"'{inner sep=.8ex}, "\shortmid"{marking}, from=1-2, to=1-1]
    \arrow["{p_i}"'{inner sep=.8ex}, "\shortmid"{marking}, from=1-3, to=1-2]
    \arrow[""{name=1, anchor=center, inner sep=0}, equals, nfold, from=1-3, to=2-3]
    \arrow["{p_{i + 1}}"'{inner sep=.8ex}, "\shortmid"{marking}, from=1-4, to=1-3]
    \arrow["{p_m}"'{inner sep=.8ex}, "\shortmid"{marking}, from=1-5, to=1-4]
    \arrow[""{name=2, anchor=center, inner sep=0}, equals, nfold, from=1-5, to=2-5]
    \arrow["{r_1}"'{inner sep=.8ex}, "\shortmid"{marking}, from=1-6, to=1-5]
    \arrow["{r_n}"'{inner sep=.8ex}, "\shortmid"{marking}, from=1-7, to=1-6]
    \arrow[""{name=3, anchor=center, inner sep=0}, equals, nfold, from=1-7, to=2-7]
    \arrow[""{name=4, anchor=center, inner sep=0}, equals, nfold, from=2-1, to=3-1]
    \arrow["{p_1 \odotl \cdots \odotl p_i}"{description}, from=2-3, to=2-1]
    \arrow["{p_{i + 1}}"{description}, from=2-4, to=2-3]
    \arrow["{p_m}"{description}, from=2-5, to=2-4]
    \arrow["{r_1}"{description}, from=2-6, to=2-5]
    \arrow["{r_n}"{description}, from=2-7, to=2-6]
    \arrow[""{name=5, anchor=center, inner sep=0}, equals, nfold, from=2-7, to=3-7]
    \arrow["q"{inner sep=.8ex}, "\shortmid"{marking}, from=3-7, to=3-1]
    \arrow["\opcartl"{description}, draw=none, from=0, to=1]
    \arrow["{=}"{description}, draw=none, from=1, to=2]
    \arrow["{=}"{description}, draw=none, from=2, to=3]
    \arrow["\phi"{description}, draw=none, from=5, to=4]
  \end{tikzcd}\]
  By the universal property of the right lift $q \rf (p_1, \ldots, p_m)$, this factors uniquely through $\varpi$ as $\lambda\phi$. Finally, pasting the left-opcartesian 2-cell on top of the following 2-cell exhibits the composite as equal to the 2-cell above, so that the following 2-cell is equal to $\phi$ by the universal property of the left-opcartesian 2-cell.
  \[\begin{tikzcd}[column sep=6em]
    Z & \cdot & \cdots & Y & \cdots & X \\
    Z & \cdot & \cdots & Y && X \\
    Z &&&&& X
    \arrow[""{name=0, anchor=center, inner sep=0}, equals, nfold, from=1-1, to=2-1]
    \arrow["{p_1 \odotl \cdots \odotl p_i}"'{inner sep=.8ex}, "\shortmid"{marking}, from=1-2, to=1-1]
    \arrow["{p_{i + 1}}"'{inner sep=.8ex}, "\shortmid"{marking}, from=1-3, to=1-2]
    \arrow["{p_m}"'{inner sep=.8ex}, "\shortmid"{marking}, from=1-4, to=1-3]
    \arrow[""{name=1, anchor=center, inner sep=0}, equals, nfold, from=1-4, to=2-4]
    \arrow["{r_1}"'{inner sep=.8ex}, "\shortmid"{marking}, from=1-5, to=1-4]
    \arrow["{r_n}"'{inner sep=.8ex}, "\shortmid"{marking}, from=1-6, to=1-5]
    \arrow[""{name=2, anchor=center, inner sep=0}, equals, nfold, from=1-6, to=2-6]
    \arrow[""{name=3, anchor=center, inner sep=0}, equals, nfold, from=2-1, to=3-1]
    \arrow["{p_1 \odotl \cdots \odotl p_i}"{description}, from=2-2, to=2-1]
    \arrow["{p_{i + 1}}"{description}, from=2-3, to=2-2]
    \arrow["{p_m}"{description}, from=2-4, to=2-3]
    \arrow["{q \rf (p_1, \ldots, p_m)}"{description}, from=2-6, to=2-4]
    \arrow[""{name=4, anchor=center, inner sep=0}, equals, nfold, from=2-6, to=3-6]
    \arrow["q"{inner sep=.8ex}, "\shortmid"{marking}, from=3-6, to=3-1]
    \arrow["{=}"{description}, draw=none, from=1, to=0]
    \arrow["{\lambda\phi}"{description}, draw=none, from=2, to=1]
    \arrow["{\check\varpi}"{description}, draw=none, from=3, to=4]
  \end{tikzcd}\]
  Thus $\lambda\phi$ exhibits $q \rf (p_1, \ldots, p_m)$ as having the universal property of the right lift $q \rf (p_1 \odotl \cdots \odotl p_i, p_{i + 1}, \ldots, p_m)$. The converse follows by entirely analogous reasoning.
\end{proof}

In the presence of actual loose-composites, we have the following stronger statement.

\begin{lemma}
  \label{right-lift-through-composite}
  Let $q \colon X \lto Z$ be a loose-cell.
  \begin{enumerate}
    \item Let $(p_1, \ldots, p_m) \colon Y \lcto Z$ be a weight. If the loose-composite ${p_i \odot \cdots \odot p_j}$ exists ($1 \leq i \leq j \leq m$), then, if either side of the following exists, so does the other, in which case they are isomorphic.
    \[q \rf (p_1, \ldots, p_m) \iso q \rf (p_1, \ldots, p_i \odot \cdots \odot p_j, \ldots, p_m)\]
    \item Let $(p_1, \ldots, p_m) \colon X \lcto Y$ be a weight. If the loose-composite ${p_i \odot \cdots \odot p_j}$ exists ($1 \leq i \leq j \leq m$), then, if either side of the following exists, so does the other, in which case they are isomorphic.
    \[(p_1, \ldots, p_m) \rx q \iso (p_1, \ldots, p_i \odot \cdots \odot p_j, \ldots, p_m) \rx q\]
  \end{enumerate}
\end{lemma}

\begin{proof}
  The reasoning is analogous to that of \cref{right-lift-through-left-composite}, making use of the stronger universal property of a loose-composite.
\end{proof}

Iterated lifts may be reduced to individual lifts, and similarly for extensions.

\begin{lemma}
  \label{iterated-lift}
  Let $q \colon X \lto Z$ be a loose-cell.
  \begin{enumerate}
    \item Let $(p_1, \ldots, p_m) \colon Y \lcto Z$ be a weight with $m > 0$. If the right lift $q \rf (p_1, \ldots, p_i)$ exists ($0 \leq i \leq m$), then, if either side of the following exists, so does the other, in which case they are isomorphic.
    \[q \rf (p_1, \ldots, p_i) \rf (p_{i + 1}, \ldots, p_m) \iso q \rf (p_1, \ldots, p_m)\]
    \item Let $(p_1, \ldots, p_m) \colon X \lcto Y$ be a weight. If the right extension $(p_{i + 1}, \ldots, p_m) \rx q$ exists ($0 \leq i \leq m$), then, if either side of the following exists, so does the other, in which case they are isomorphic.
    \[(p_1, \ldots, p_i) \rx (p_{i + 1}, \ldots, p_m) \rx q \iso (p_1, \ldots, p_m) \rx q\]
  \end{enumerate}
\end{lemma}

\begin{proof}
  The reasoning in (1) follows as in \cite[Lemma~3.5]{arkor2024formal}, the universal properties of the loose-cells in question evidently being equivalent. (2) is dual.
\end{proof}

If all right lifts exist, then every left-composite is furthermore a composite.

\begin{lemma}[label=right-lifts-imply-left-composites-are-composites]
  Suppose that all restrictions exist and that every loose-cell is lifting. Then every left-opcartesian 2-cell is moreover opcartesian.
\end{lemma}

\begin{proof}
  Suppose we are given a left-opcartesian 2-cell as follows.
	\[\begin{tikzcd}
		\cdot & \cdots & \cdot \\
		\cdot && \cdot
		\arrow["{q_m}"', "\shortmid"{marking}, from=1-3, to=1-2]
		\arrow["{q_1}"', "\shortmid"{marking}, from=1-2, to=1-1]
		\arrow["q", "\shortmid"{marking}, from=2-3, to=2-1]
		\arrow[""{name=0, anchor=center, inner sep=0}, Rightarrow, no head, nfold, from=1-3, to=2-3]
		\arrow[""{name=1, anchor=center, inner sep=0}, Rightarrow, no head, nfold, from=1-1, to=2-1]
		\arrow["\opcart"{description}, draw=none, from=0, to=1]
	\end{tikzcd}\]
  To verify that this 2-cell is opcartesian, suppose we are given a 2-cell as follows; since restrictions exist, it suffices to consider a globular 2-cell.
  \[\begin{tikzcd}
    \cdot & \cdots & \cdot & \cdots & \cdot & \cdots & \cdot \\
    \cdot &&&&&& \cdot
    \arrow[""{name=0, anchor=center, inner sep=0}, equals, from=1-1, to=2-1]
    \arrow["{p_1}"'{inner sep=.8ex}, "\shortmid"{marking}, from=1-2, to=1-1]
    \arrow["{p_l}"'{inner sep=.8ex}, "\shortmid"{marking}, from=1-3, to=1-2]
    \arrow["{q_1}"'{inner sep=.8ex}, "\shortmid"{marking}, from=1-4, to=1-3]
    \arrow["{q_m}"'{inner sep=.8ex}, "\shortmid"{marking}, from=1-5, to=1-4]
    \arrow["{r_1}"'{inner sep=.8ex}, "\shortmid"{marking}, from=1-6, to=1-5]
    \arrow["{r_n}"'{inner sep=.8ex}, "\shortmid"{marking}, from=1-7, to=1-6]
    \arrow[""{name=1, anchor=center, inner sep=0}, equals, from=1-7, to=2-7]
    \arrow["s"{inner sep=.8ex}, "\shortmid"{marking}, from=2-7, to=2-1]
    \arrow["\phi"{description}, draw=none, from=1, to=0]
  \end{tikzcd}\]
  Since every loose-cell is lifting, every weight is lifting by \cref{iterated-lift}. Thus, using the universal property of the right lift, followed by the left-opcartesian 2-cell, the 2-cell above induces a 2-cell $\widecheck{\lambda\phi}$ such that the following equation holds.
  \[
  \phi \quad = \quad
  \begin{tikzcd}
    \cdot & \cdots & \cdot & \cdots & \cdot & \cdots & \cdot \\
    \cdot & \cdots & \cdot && \cdot & \cdots & \cdot \\
    \cdot & \cdots & \cdot &&&& \cdot \\
    \cdot &&&&&& \cdot
    \arrow[""{name=0, anchor=center, inner sep=0}, equals, from=1-1, to=2-1]
    \arrow["{p_1}"'{inner sep=.8ex}, "\shortmid"{marking}, from=1-2, to=1-1]
    \arrow["{p_l}"'{inner sep=.8ex}, "\shortmid"{marking}, from=1-3, to=1-2]
    \arrow[""{name=1, anchor=center, inner sep=0}, equals, from=1-3, to=2-3]
    \arrow["{q_1}"'{inner sep=.8ex}, "\shortmid"{marking}, from=1-4, to=1-3]
    \arrow["{q_m}"'{inner sep=.8ex}, "\shortmid"{marking}, from=1-5, to=1-4]
    \arrow[""{name=2, anchor=center, inner sep=0}, equals, from=1-5, to=2-5]
    \arrow["{r_1}"'{inner sep=.8ex}, "\shortmid"{marking}, from=1-6, to=1-5]
    \arrow["{r_n}"'{inner sep=.8ex}, "\shortmid"{marking}, from=1-7, to=1-6]
    \arrow[""{name=3, anchor=center, inner sep=0}, equals, from=1-7, to=2-7]
    \arrow[""{name=4, anchor=center, inner sep=0}, equals, from=2-1, to=3-1]
    \arrow["{p_1}"{description}, from=2-2, to=2-1]
    \arrow["{p_l}"{description}, from=2-3, to=2-2]
    \arrow[""{name=5, anchor=center, inner sep=0}, equals, from=2-3, to=3-3]
    \arrow["{q_1 \odotl \cdots \odotl q_m}"{description}, from=2-5, to=2-3]
    \arrow["{r_1}"{description}, from=2-6, to=2-5]
    \arrow["{r_n}"{description}, from=2-7, to=2-6]
    \arrow[""{name=6, anchor=center, inner sep=0}, equals, from=2-7, to=3-7]
    \arrow[""{name=7, anchor=center, inner sep=0}, equals, from=3-1, to=4-1]
    \arrow["{p_1}"{description}, from=3-2, to=3-1]
    \arrow["{p_l}"{description}, from=3-3, to=3-2]
    \arrow["{s \rf (p_1, \ldots, p_l)}"{description}, from=3-7, to=3-3]
    \arrow[""{name=8, anchor=center, inner sep=0}, equals, from=3-7, to=4-7]
    \arrow["s"{inner sep=.8ex}, "\shortmid"{marking}, from=4-7, to=4-1]
    \arrow["{=}"{description}, draw=none, from=0, to=1]
    \arrow["\opcartl"{description}, draw=none, from=1, to=2]
    \arrow["{=}"{description}, draw=none, from=2, to=3]
    \arrow["{=}"{description}, draw=none, from=4, to=5]
    \arrow["{\widecheck{\lambda\phi}}"{description}, draw=none, from=6, to=5]
    \arrow["\varpi"{description}, draw=none, from=7, to=8]
  \end{tikzcd}
  \]
  That this provides a unique factorisation of $\phi$ follows from the universal properties of the left-composite and right lift.
\end{proof}

While, in practice, the existence of \emph{all} right lifts or extensions tends to be too strong an assumption, the following lemma provides a useful special case of this fact.

\begin{lemma}[label=left-composite-reassociation-lifting]
  Suppose we are given chains $\vec p$ and $\vec q = (q_1, \dots, q_l)$ and $\vec r$ of loose-cells, for which the left-composites
  $
  q_1 \odotl \cdots \odotl q_l
  $
  and
  $
  \vec p \odotl (q_1 \odotl \cdots \odotl q_l) \odotl \vec r
  $
  exist.
  If $\vec p$ is lifting, then the following 2-cell is left-opcartesian.
\[\begin{tikzcd}[ampersand replacement=\&]
	\cdot \& \cdots \& \cdot \& \cdots \& \cdot \& \cdots \& \cdot \\
	\cdot \& \cdots \& \cdot \&\& \cdot \& \cdots \& \cdot \\
	\cdot \&\&\&\&\&\& \cdot
	\arrow[""{name=0, anchor=center, inner sep=0}, equals, nfold, from=1-1, to=2-1]
	\arrow["{p_1}"'{inner sep=.8ex}, "\shortmid"{marking}, from=1-2, to=1-1]
	\arrow["{p_k}"'{inner sep=.8ex}, "\shortmid"{marking}, from=1-3, to=1-2]
	\arrow[""{name=1, anchor=center, inner sep=0}, equals, nfold, from=1-3, to=2-3]
	\arrow["{q_1}"'{inner sep=.8ex}, "\shortmid"{marking}, from=1-4, to=1-3]
	\arrow["{q_l}"'{inner sep=.8ex}, "\shortmid"{marking}, from=1-5, to=1-4]
	\arrow[""{name=2, anchor=center, inner sep=0}, equals, nfold, from=1-5, to=2-5]
	\arrow["{r_1}"'{inner sep=.8ex}, "\shortmid"{marking}, from=1-6, to=1-5]
	\arrow["{r_m}"'{inner sep=.8ex}, "\shortmid"{marking}, from=1-7, to=1-6]
	\arrow[equals, nfold, from=1-7, to=2-7]
	\arrow[""{name=3, anchor=center, inner sep=0}, equals, nfold, from=2-1, to=3-1]
	\arrow["{p_1}"{description}, from=2-2, to=2-1]
	\arrow["{p_k}"{description}, from=2-3, to=2-2]
	\arrow["{q_1 \odotl \cdots \odotl q_l}"{description}, from=2-5, to=2-3]
	\arrow["{r_1}"{description}, from=2-6, to=2-5]
	\arrow["{r_m}"{description}, from=2-7, to=2-6]
	\arrow[""{name=4, anchor=center, inner sep=0}, equals, nfold, from=2-7, to=3-7]
	\arrow["{\vec p \odotl (q_1 \odotl \cdots \odotl q_l) \odotl \vec r}"{inner sep=.8ex}, "\shortmid"{marking}, from=3-7, to=3-1]
	\arrow["{=}"{description}, draw=none, from=0, to=1]
	\arrow["\opcartl"{description}, draw=none, from=1, to=2]
	\arrow["\opcartl"{description}, draw=none, from=3, to=4]
\end{tikzcd}\]
\end{lemma}

\begin{proof}
  Suppose we are given a 2-cell as follows.
\[\begin{tikzcd}[ampersand replacement=\&]
	\cdot \& \cdots \& \cdot \& \cdots \& \cdot \& \cdots \& \cdot \& \cdots \& \cdot \\
	\cdot \&\&\&\&\&\&\&\& \cdot
	\arrow[""{name=0, anchor=center, inner sep=0}, equals, nfold, from=1-1, to=2-1]
	\arrow["{p_1}"'{inner sep=.8ex}, "\shortmid"{marking}, from=1-2, to=1-1]
	\arrow["{p_k}"'{inner sep=.8ex}, "\shortmid"{marking}, from=1-3, to=1-2]
	\arrow["{q_1}"'{inner sep=.8ex}, "\shortmid"{marking}, from=1-4, to=1-3]
	\arrow["{q_l}"'{inner sep=.8ex}, "\shortmid"{marking}, from=1-5, to=1-4]
	\arrow["{r_1}"'{inner sep=.8ex}, "\shortmid"{marking}, from=1-6, to=1-5]
	\arrow["{r_m}"'{inner sep=.8ex}, "\shortmid"{marking}, from=1-7, to=1-6]
	\arrow["{s_1}"'{inner sep=.8ex}, "\shortmid"{marking}, from=1-8, to=1-7]
	\arrow["{s_n}"'{inner sep=.8ex}, "\shortmid"{marking}, from=1-9, to=1-8]
	\arrow[""{name=1, anchor=center, inner sep=0}, "g", from=1-9, to=2-9]
	\arrow["t"{inner sep=.8ex}, "\shortmid"{marking}, from=2-9, to=2-1]
	\arrow["\phi"{description}, draw=none, from=0, to=1]
\end{tikzcd}\]
  Invoking in turn the universal properties of the right lift $t \rf (p_1, \ldots, p_k)$, the left-composite $q_1 \odotl \cdots \odotl q_l$, and the left-composite $\vec p \odotl (q_1 \odotl \cdots \odotl q_l) \odotl \vec r$, the 2-cell $\phi$ is equal to the following composite, in which $\psi$ is the factorisation of $((1_{p_1}, \ldots, 1_{p_k}, \widecheck{\lambda \phi}, 1_{r_1}, \dots, 1_{r_m}) \d \varpi)$ through the latter left-opcartesian 2-cell, establishing the existence aspect of the required universal property.
\[\begin{tikzcd}[ampersand replacement=\&]
	\cdot \& \cdots \& \cdot \& \cdots \& \cdot \& \cdots \& \cdot \& \cdots \& \cdot \\
	\cdot \& \cdots \& \cdot \&\& \cdot \& \cdots \& \cdot \& \cdots \& \cdot \\
	\cdot \&\&\&\&\&\& \cdot \& \cdots \& \cdot \\
	\cdot \&\&\&\&\&\&\&\& \cdot
	\arrow[""{name=0, anchor=center, inner sep=0}, equals, nfold, from=1-1, to=2-1]
	\arrow["{p_1}"'{inner sep=.8ex}, "\shortmid"{marking}, from=1-2, to=1-1]
	\arrow["{p_k}"'{inner sep=.8ex}, "\shortmid"{marking}, from=1-3, to=1-2]
	\arrow[""{name=1, anchor=center, inner sep=0}, equals, nfold, from=1-3, to=2-3]
	\arrow["{q_1}"'{inner sep=.8ex}, "\shortmid"{marking}, from=1-4, to=1-3]
	\arrow["{q_l}"'{inner sep=.8ex}, "\shortmid"{marking}, from=1-5, to=1-4]
	\arrow[""{name=2, anchor=center, inner sep=0}, equals, nfold, from=1-5, to=2-5]
	\arrow["{r_1}"'{inner sep=.8ex}, "\shortmid"{marking}, from=1-6, to=1-5]
	\arrow["{r_m}"'{inner sep=.8ex}, "\shortmid"{marking}, from=1-7, to=1-6]
	\arrow[""{name=3, anchor=center, inner sep=0}, equals, nfold, from=1-7, to=2-7]
	\arrow["{s_1}"'{inner sep=.8ex}, "\shortmid"{marking}, from=1-8, to=1-7]
	\arrow["{s_n}"'{inner sep=.8ex}, "\shortmid"{marking}, from=1-9, to=1-8]
	\arrow[""{name=4, anchor=center, inner sep=0}, equals, nfold, from=1-9, to=2-9]
	\arrow[""{name=5, anchor=center, inner sep=0}, equals, nfold, from=2-1, to=3-1]
	\arrow["{p_1}"{description}, from=2-2, to=2-1]
	\arrow["{p_k}"{description}, from=2-3, to=2-2]
	\arrow["{q_1 \odotl \cdots \odotl q_l}"{description}, from=2-5, to=2-3]
	\arrow["{r_1}"{description}, from=2-6, to=2-5]
	\arrow["{r_m}"{description}, from=2-7, to=2-6]
	\arrow[""{name=6, anchor=center, inner sep=0}, equals, nfold, from=2-7, to=3-7]
	\arrow["{s_1}"{description}, from=2-8, to=2-7]
	\arrow["{s_n}"{description}, from=2-9, to=2-8]
	\arrow[""{name=7, anchor=center, inner sep=0}, equals, nfold, from=2-9, to=3-9]
	\arrow[""{name=8, anchor=center, inner sep=0}, equals, nfold, from=3-1, to=4-1]
	\arrow["{\vec p \odotl (q_1 \odotl \cdots \odotl q_l) \odotl \vec r}"{description}, from=3-7, to=3-1]
	\arrow["{s_1}"{description}, from=3-8, to=3-7]
	\arrow["{s_n}"{description}, from=3-9, to=3-8]
	\arrow[""{name=9, anchor=center, inner sep=0}, "g", from=3-9, to=4-9]
	\arrow["t"{inner sep=.8ex}, "\shortmid"{marking}, from=4-9, to=4-1]
	\arrow["{=}"{description}, draw=none, from=0, to=1]
	\arrow["\opcartl"{description}, draw=none, from=1, to=2]
	\arrow["{=}"{description}, draw=none, from=3, to=4]
	\arrow["\opcartl"{description}, draw=none, from=5, to=6]
	\arrow["{=}"{description}, draw=none, from=7, to=6]
	\arrow["\psi"{description}, draw=none, from=8, to=9]
\end{tikzcd}\]
  Uniqueness follows from the conjunction of the uniqueness aspects of the universal properties of the respective left-composite, left-composite, and right lift.
\end{proof}

The proofs of the following lemmas, which establish the interaction between right lifts, right extensions, and restrictions, are exactly as their correspondents in \cite[\S3]{arkor2024formal} and are omitted.

\begin{lemma}[{\cf~\cite[Lemma~3.4]{arkor2024formal}}]
  Let $\vec p = p_1, \ldots, p_m \colon X \lto \cdots \lto Y$ and $\vec r = (r_1, \ldots, r_n) \colon Y \lto \cdots \lto Z$ be weights and $q \colon X \lto Z$ be a loose-cell.
  Suppose that the right extension $\vec p \rx q$ and right lift $q \rf \vec r$ exist. Then, if either side of the following exists, so does the other, in which case they are isomorphic.
  \[(\vec p \rx q) \rf \vec r \iso \vec p \rx (q \rf \vec r) \tag*{\qed}\]
\end{lemma}

\begin{lemma}[{\cf~\cite[Lemma~3.6]{arkor2024formal}}]
  \label{restriction-of-lift}
  Let $q \colon X \lto Z$ be a loose-cell and $y \colon Y' \to Y$ be a tight-cell.
  \begin{enumerate}
    \item Let $\vec p = (p_1, \ldots, p_m) \colon Y \lto \cdots \lto Z$ be a weight and $x \colon X' \to X$ be a tight-cell, where $m > 0$. If the right lift $q \rf \vec p$ exists, then $(q \rf \vec p)(y, x) \iso q(1, x) \rf (p_1, \ldots, p_m(1, y))$.
    \item Let $\vec p = (p_1, \ldots, p_m) \colon X \lto \cdots \lto Y$ be a weight and $z \colon Z' \to Z$ be a tight-cell, where $m > 0$. If the right extension $\vec p \rx q$ exists, then $(\vec p \rx q)(z, y) \iso (p_1(y, 1), \ldots, p_m) \rx q(z, 1)$.
    \qed
  \end{enumerate}
\end{lemma}

\subsection{Limits and colimits}

We are now ready to introduce weighted limits and weighted colimits in \ve s. The following definitions generalise the notions in \cite[Definitions~3.8 \& 3.9]{arkor2024formal} to permit arbitrary (rather than unary) weights. Consequently, they coincide with the \emph{left Kan extensions} and \emph{right Kan extensions} of \cite[Definition~1.9 \& Remark~1.6]{koudenburg2024formal} (though \citeauthor{koudenburg2024formal} defines these in a \vdc{} more generally, without assuming the existence of companions or conjoints).

\begin{definition}
	\label{weighted-colimit}
  Let $\vec p = (p_1, \ldots, p_m) \colon Y \lcto Z$ be a weight and $f \colon Z \to X$ be a tight-cell. A \emph{$\vec p$-weighted cocone} (or simply \emph{$\vec p$-cocone}) for $f$ is a pair $(c, \gamma)$ of a tight-cell $c \colon Y \to X$ and a 2-cell $\gamma \colon p_1, \ldots, p_m \tto X(f, c)$.
  A cocone $(\vec p \wc f, \lambda)$ is \emph{colimiting} (alternatively is the \emph{$\vec p$-weighted colimit}, or simply \emph{$\vec p$-colimit}) of $f$ when the 2-cell
  \[\begin{tikzcd}[column sep=8em]
    Z & \cdots & Y & X \\
    Z && Y & X \\
    Z &&& X
    \arrow[""{name=0, anchor=center, inner sep=0}, equals, nfold, from=1-1, to=2-1]
    \arrow["{p_1}"'{inner sep=.8ex}, "\shortmid"{marking}, from=1-2, to=1-1]
    \arrow["{p_m}"'{inner sep=.8ex}, "\shortmid"{marking}, from=1-3, to=1-2]
    \arrow[""{name=1, anchor=center, inner sep=0}, equals, nfold, from=1-3, to=2-3]
    \arrow["{X(\vec p \wc f, 1)}"'{inner sep=.8ex}, "\shortmid"{marking}, from=1-4, to=1-3]
    \arrow[""{name=2, anchor=center, inner sep=0}, equals, nfold, from=1-4, to=2-4]
    \arrow[""{name=3, anchor=center, inner sep=0}, equals, nfold, from=2-1, to=3-1]
    \arrow["{X(f, \vec p \wc f)}"{description}, from=2-3, to=2-1]
    \arrow["{X(\vec p \wc f, 1)}"{description}, from=2-4, to=2-3]
    \arrow[""{name=4, anchor=center, inner sep=0}, equals, nfold, from=2-4, to=3-4]
    \arrow["{X(f, 1)}"{inner sep=.8ex}, "\shortmid"{marking}, from=3-4, to=3-1]
    \arrow["\lambda"{description}, draw=none, from=1, to=0]
    \arrow["{=}"{description}, draw=none, from=2, to=1]
    \arrow["{\cp{\vec p \wc f}(f, 1)}"{description}, draw=none, from=4, to=3]
  \end{tikzcd}\]
	exhibits $X(\vec p \wc f, 1)$ as the right lift $X(f, 1) \rf \vec p$.
	A tight-cell $g \colon X \to X'$ \emph{preserves} the colimit $\vec p \wc f$ when the cocone $(((\vec p \wc f) \d g), (\lambda \d g))$ is the $\vec p$-colimit of $(f \d g) \colon Z \to X'$.
\end{definition}

A weighted limit in $\X$ is a weighted colimit in $\X\co$. We spell out the definition for convenience.

\begin{definition}
	\label{weighted-limit}
	Let $\vec p = (p_1, \ldots, p_m) \colon Z \lcto Y$ be a weight and $f \colon Z \to X$ be a tight-cell. A \emph{$\vec p$-weighted cone} (or simply \emph{$\vec p$-cone}) for $f$ is a pair $(c, \gamma)$ of a tight-cell $c \colon Y \to X$ and a 2-cell $\gamma \colon \vec p \tto X(c, f)$. A cone $(\vec p \wl f, \mu)$ is \emph{limiting} (alternatively is the \emph{$\vec p$-weighted limit}, or simply \emph{$\vec p$-limit}) of $f$ when the 2-cell
  \[\begin{tikzcd}[column sep=8em]
    X & Y & \cdots & Z \\
    X & Y && Z \\
    X &&& Z
    \arrow[""{name=0, anchor=center, inner sep=0}, equals, nfold, from=1-1, to=2-1]
    \arrow["{X(1, \vec p \wl f)}"'{inner sep=.8ex}, "\shortmid"{marking}, from=1-2, to=1-1]
    \arrow[""{name=1, anchor=center, inner sep=0}, equals, nfold, from=1-2, to=2-2]
    \arrow["{p_1}"'{inner sep=.8ex}, "\shortmid"{marking}, from=1-3, to=1-2]
    \arrow["{p_m}"'{inner sep=.8ex}, "\shortmid"{marking}, from=1-4, to=1-3]
    \arrow[""{name=2, anchor=center, inner sep=0}, equals, nfold, from=1-4, to=2-4]
    \arrow[""{name=3, anchor=center, inner sep=0}, equals, nfold, from=2-1, to=3-1]
    \arrow["{X(1, \vec p \wl f)}"{description}, from=2-2, to=2-1]
    \arrow["{X(\vec p \wl f, f)}"{description}, from=2-4, to=2-2]
    \arrow[""{name=4, anchor=center, inner sep=0}, equals, nfold, from=2-4, to=3-4]
    \arrow["{X(1, f)}"{inner sep=.8ex}, "\shortmid"{marking}, from=3-4, to=3-1]
    \arrow["{=}"{description}, draw=none, from=1, to=0]
    \arrow["\mu"{description}, draw=none, from=2, to=1]
    \arrow["{\cp{\vec p \wl f}(1, f)}"{description}, draw=none, from=4, to=3]
  \end{tikzcd}\]
	exhibits $X(1, \vec p \wl f)$ as the right extension $\vec p \rx X(1, f)$.
  A tight-cell $g \colon X \to X'$ \emph{preserves} the limit $\vec p \wl f$ when the cone $((\vec p \wl f \d g), (\mu \d g))$ is the $\vec p$-limit of $(f \d g) \colon Z \to X'$.
\end{definition}

Just as we can speak of \emph{tight}-cells that preserve weighted colimits, there is an analogous notion of compatibility between \emph{loose}-cells and weighted colimits.

\begin{definition}
  \label{respects-colimit}
  A loose-cell $p \colon W \lto X$ \emph{respects} a colimit $\vec q \wc g$ in $X$ when the canonical 2-cell $\vec q, p(\vec q \wc g, 1) \tto p(g, 1)$ exhibits $p(\vec q \wc g, 1)$ as the right lift of $p(g, 1)$ through $\vec q$.
  \[
    p(\vec q \wc g, 1) \iso p(g, 1) \rf \vec q \qedhere
  \]
\end{definition}

In particular, a tight-cell $f \colon X \to W$ preserves a colimit if and only if its conjoint $W(f, 1) \colon W \lto X$ respects the colimit.

\begin{example}
  \label{distributor-respects-colimit}
  In $\Cat$, a distributor $p \colon \b X \lto \b Y$ respects the $(q \colon \A \lto \B)$-weighted colimit of a functor $g \colon \B \to \b Y$ when there is a canonical isomorphism as follows, natural in $a \in \A$ and $x \in \b X$.
  \[p((q \wc g)(a), x) \iso \widehat{\B}(q({-}, a), p(g{-}, x)) \qedhere\]
\end{example}

We obtain the dual notion for limits by working in $\X\co$. Once again, we spell out the definition in this case, but henceforth shall generally leave routine applications of duality to the reader.

\begin{definition}
  \label{respects-limit}
  A loose-cell $p \colon X \lto W$ \emph{respects} a limit $\vec q \wl g$ in $X$ when the canonical 2-cell $p(1, \vec q \wl g), \vec q \tto p(1, g)$ exhibits $p(1, \vec q \wl g)$ as the right extension of $p(1, g)$ along $\vec q$.
  \[
    p(1, \vec q \wl g) \iso \vec q \rx p(1, g) \qedhere
  \]
\end{definition}

An \emph{absolute colimit} is a colimit that is preserved by any functor or, equivalently, a colimit that is preserved by the presheaf embedding~\cite{pare1969absolute}. We may generalise the latter description by parameterising it with respect to a distributor.

\begin{definition}[label=j-absolute]
  Let $\vec p = (p_1, \ldots, p_m) \colon Y \lcto Z$ be a weight, $j \colon E \lto A$ be a loose-cell, and $f \colon Z \to E$ be a tight-cell.
  A colimit $(\vec p \wc f, \lambda)$ is \emph{$j$-absolute} if
  the following 2-cell is left-opcartesian (\cref{left-opcartesian}).
  \[\begin{tikzcd}[column sep=huge]
    A & Z & \cdots & Y \\
    A & Z && Y \\
    A &&& Y
    \arrow["{j(1, f)}"'{inner sep=.8ex}, "\shortmid"{marking}, from=1-2, to=1-1]
    \arrow["{p_1}"'{inner sep=.8ex}, "\shortmid"{marking}, from=1-3, to=1-2]
    \arrow["{p_m}"'{inner sep=.8ex}, "\shortmid"{marking}, from=1-4, to=1-3]
    \arrow[""{name=0, anchor=center, inner sep=0}, equals, nfold, from=2-1, to=1-1]
    \arrow[""{name=1, anchor=center, inner sep=0}, equals, nfold, from=2-2, to=1-2]
    \arrow["{j(1, f)}"{description}, from=2-2, to=2-1]
    \arrow[""{name=2, anchor=center, inner sep=0}, equals, nfold, from=2-4, to=1-4]
    \arrow["{E(f, \vec p \wc f)}"{description}, from=2-4, to=2-2]
    \arrow[""{name=3, anchor=center, inner sep=0}, equals, nfold, from=3-1, to=2-1]
    \arrow[""{name=4, anchor=center, inner sep=0}, equals, nfold, from=3-4, to=2-4]
    \arrow["{j(1, \vec p \wc f)}"{inner sep=.8ex}, "\shortmid"{marking}, from=3-4, to=3-1]
    \arrow["{=}"{description}, draw=none, from=1, to=0]
    \arrow["\lambda"{description}, draw=none, from=2, to=1]
    \arrow["{j \odot \cp f(1, \vec p \wc f)}"{description}, draw=none, from=4, to=3]
  \end{tikzcd}\]
  For a tight-cell $j \colon A \to E$, a colimit in $E$ is \emph{$j$-absolute} if it is \emph{$E(j, 1)$}-absolute~\cite[Definition~3.21]{arkor2024formal}. A colimit in $E$ is \emph{absolute} if it is $1_E$-absolute.
\end{definition}

When $j = 1$, \cref{j-absolute} is closely related to the \emph{left Beck--Chevalley condition} of \cite[Definition~5.10]{koudenburg2024formal}, in which context \cite[Theorem~5.16]{koudenburg2024formal} corresponds to our \cref{absoluteness-and-exactness}.

Note that, when we talk of \emph{preservation of absolute colimits}, we simply mean that absolute colimits are sent to (not necessarily absolute) colimits. The following lemma provides a useful method to recognise absolute colimits, and generalises \cite[Lemma~3.23]{arkor2024formal} from tight-cells to loose-cells (upon taking $j$ to be corepresentable), for which we require the notion of \emph{density} for a loose-cell.

\begin{definition}[{\cite[Definition~6.8]{arkor2025nerve}}]
  \label{dense-loose-cell}
  A loose-cell $p \colon B \lto A$ is \emph{dense} when the identity 2-cell $1_p \colon p \tto p$ exhibits $B(1, 1) \colon B \lto B$ as the right lift $p \rf p$. A tight-cell is \emph{dense} if its conjoint is.
\end{definition}

\begin{lemma}
  \label{absolute-implies-colimit}
  Let $\vec p \colon Y \lcto Z$ be a weight, let $j \colon E \lto A$ be a dense loose-cell, and let $f \colon Z \to E$ be a tight-cell.
  A tight-cell $c \colon Y \to E$ forms the $j$-absolute $\vec p$-weighted colimit of $f$ if and only if there is an isomorphism:
  \[
     j(1, f) \odotl \vec p \iso j(1, c)
  \]
\end{lemma}

\begin{proof}
  The only if direction is trivial.
  For the other direction, assume there is such an isomorphism.
  Then $c$ forms the colimit $\vec p \wc f$ by the following chain of isomorphisms.
  \begin{align*}
    E(c, 1)
    &\iso j \rf j(1, c)
    \\&\iso j \rf (j(1, f) \odotl \vec p)
    \\&\iso (j \rf j(1, f)) \rf \vec p
    \\&\iso (j \rf j)(f, 1) \rf \vec p
    \\&\iso E(f, 1) \rf \vec p
  \end{align*}
  The universal 2-cell $\lambda \colon \vec p \tto E(f, c)$ is the unique 2-cell such that the left-opcartesian 2-cell ${j(1, f), \vec p \tto j(1, c)}$ witnessing the isomorphism above is equal to that in the definition of $j$-absoluteness (\cref{j-absolute}).
  Hence this colimit is $j$-absolute.
\end{proof}

The following lemma provides a common source of loose-cells that respect $j$-absolute colimits.

\begin{lemma}[label=right-lifts-respect-absolute-colimits]
  Every right lift through a loose-cell $j \colon E \lto A$ respects $j$-absolute colimits.
\end{lemma}

\begin{proof}
  We have the following for each $j$-absolute colimit $q \wc g$.
  \begin{align*}
    (p \rf j)(q \wc g, 1)
    ~\iso~&
    p \rf j(1, q \wc g)
    \\~\iso~&
    p \rf (j(1, g) \odotl q)
    \\~\iso~&
    (p \rf j(1, g)) \rf q
    \\~\iso~&
    (p \rf j)(g, 1) \rf q
    \qedhere
  \end{align*}
\end{proof}

Moreover, in the presence of sufficient right lifts, this property characterises $j$-absoluteness.

\begin{lemma}[label=absoluteness-and-exactness]
  Let $j \colon E \lto A$ be a lifting loose-cell (\cref{lifting}). A colimit in $E$ is $j$-absolute if and only if every right lift through $j$ respects the colimit.
\end{lemma}

\begin{proof}
  That right lifts through $j$ respect $j$-absolute colimits is \cref{right-lifts-respect-absolute-colimits}.
  For the converse, let $\vec q \wc g$ be a colimit in $E$ and assume that it is respected by every right lift through $j$, so that we have the following isomorphism.
  \[
    p \rf j(1, \vec q \wc g) \iso (p \rf j(1, g)) \rf \vec q
  \]
  We have to show that the canonical 2-cell $j(1, g), \vec q \tto j(1, \vec q \wc g)$ is left-opcartesian, meaning that composition with this 2-cell exhibits a bijection between 2-cells
  $
    j(1, \vec q \wc g), r_1, \dots, r_n \tto p
  $
  and 2-cells
  $
    j(1, g), \vec q, r_1, \dots, r_n \tto p
  $.
  This is the case because composition with the canonical 2-cell is the composition of the following bijections, where the second bijection is the fact that $\vec q \wc g$ is respected, and the others are the universal property of the right lifts.
  \begin{iffseq}
    j(1, \vec q \wc g), r_1, \dots, r_n \tto p
    \\
    r_1, \dots, r_n \tto p \rf j(1, \vec q \wc g)
    \\
    r_1, \dots, r_n \tto (p \rf j(1, g)) \rf \vec q
    \\
    \vec q, r_1, \dots, r_n \tto p \rf j(1, g)
    \\
    j(1, g), \vec q, r_1, \dots, r_n \tto p
    \qedshift
  \end{iffseq}
\end{proof}

In addition to considering the preservation of colimits, we may consider their reflection.

\begin{definition}
  Let $\vec p \colon Y \lcto Z$ be a weight, and let $f \colon Z \to W$ and $g \colon W \to X$ be tight-cells. We say that $g$ \emph{reflects $\vec p$-colimits of $f$} if a $\vec p$-cocone $(c, \gamma)$ for $f$ is colimiting if the $\vec p$-cocone $((c \d g), (\gamma \d g))$ for $(f \d g)$ is colimiting.
\end{definition}

In particular, all limits and colimits are reflected by \ff{} tight-cells.

\begin{definition}[{\cite[Corollary~3.28]{arkor2024formal}}]
  \label{ff}
  A tight-cell $f \colon A \to B$ is \emph{\ff} if $A(1, 1) \iso B(f, f)$.
\end{definition}

\begin{lemma}
  \label{ff-reflects-limits}
  \Ff{} tight-cells reflect weighted limits and colimits.
\end{lemma}

\begin{proof}
  Let $\vec p \colon Y \lcto Z$ be a weight, let $f \colon Z \to W$ be a tight-cell, and let $g \colon W \ffto X$ be a \ff{} tight-cell. Suppose that there exists a cocone $(c \colon Y \to W, \gamma)$ such that $((c \d g), (\gamma \d g))$ exhibits the colimit $\vec p \wc (f \d g)$. Then $W(c, 1) \iso X(gc, g) \iso X(gf, g) \rf \vec p \iso W(f, 1) \rf \vec p$ using \ffness{} of $g$. The statement for limits follows by duality.
\end{proof}

Absoluteness allows us to characterise when weighted colimits are \ff{} (we shall not use this result here, but record it as a useful lemma).

\begin{lemma}[label=ff-weighted-colimit]
  Let $p \colon Y \lto Z$ be a loose-cell and $f \colon Z \to X$ be a \ff{} tight-cell, such that the colimit $p \wc f$ exists and is $f$-absolute.
  Then $p \wc f$ is \ff{} if and only if $p$ is dense.
\end{lemma}

\begin{proof}
  $p$ is dense when $Y(1, 1) \iso p \rf p$, and $p \wc f$ is \ff{} when $Y(1, 1) \iso X(p \wc f, p \wc f)$.
  These are equivalent by the following.
  \begin{align*}
    p \rf p
      & \iso (X(f, f) \odotl p) \rf p \tag{$f$ is \ff} \\
      & \iso X(f, p \wc f) \rf p \tag{$p \wc f$ is $f$-absolute}\\
      & \iso (X(f, 1) \rf p)(1, p \wc f) \tag{\cref{restriction-of-lift}} \\
      & \iso X(p \wc f, p \wc f)
      \tag*{\qedhere}
  \end{align*}
\end{proof}

\subsection{Adjunctions and reflections}

Some of the most important properties concerning limits and colimits regard their interaction with adjunctions and, more generally, relative adjunctions. Relative adjunctions were studied in detail in \cite[\S5]{arkor2024formal}; we recall some of the basic properties that will be relevant to us here.

\begin{definition}[{\cite[Definition~5.1]{arkor2024formal}}]
    \label{relative-adjunction}
    A \emph{relative adjunction} in a \ve{} comprises the following data.
    \[\begin{tikzcd}
    	& C \\
    	A && E
    	\arrow[""{name=0, anchor=center, inner sep=0}, "\ell"{pos=0.4}, from=2-1, to=1-2]
    	\arrow[""{name=1, anchor=center, inner sep=0}, "r"{pos=0.6}, from=1-2, to=2-3]
    	\arrow["j"', from=2-1, to=2-3]
    	\arrow["\dashv"{anchor=center}, shift right=2, draw=none, from=0, to=1]
    \end{tikzcd}\]
    \begin{enumerate}
        \item A tight-cell $\jAE$, the \emph{root}.
        \item A tight-cell $\ell \colon A \to C$, the \emph{left (relative) adjoint}.
        \item A tight-cell $r \colon C \to E$, the \emph{right (relative) adjoint}.
        \item An isomorphism $C(\ell, 1) \iso E(j, r)$.
    \end{enumerate}
    In this situation, we write $\ljr$. An \emph{adjunction} is a relative adjunction whose root is the identity, in which case we write $\ell \adj r$.
\end{definition}

The proof of the following proposition is exactly as its correspondents in \cite{arkor2024formal} (where only unary weights were considered) and is omitted.

\begin{proposition}
	\label{adjoints-preserve-(co)limits}
    Let $\jAE$ be a tight-cell.
    \begin{enumerate}
      \item If $\ell \colon A \to C$ is a left $j$-adjoint, then $\ell$ preserves every colimit that $j$ preserves~\cite[Proposition~5.11]{arkor2024formal}.
      \item If $j$ is dense and $r \colon C \to E$ is a right $j$-adjoint, then $r$ preserves limits~\cite[Proposition~5.12]{arkor2024formal}.
      \qed
    \end{enumerate}
\end{proposition}

\begin{corollary}[label=dense-and-ff-implies-continuous]
  Every dense and \ff{} tight-cell preserves and reflects weighted limits.
\end{corollary}

\begin{proof}
  Let $f \colon A \to B$ be a dense and \ff{} tight-cell. \Ff{} tight-cells reflect weighted limits by \cref{ff-reflects-limits}. Furthermore, $1 \radj f f$ since $f$ is \ff{}. Hence, it preserves weighted limits since it is dense, by \cref{adjoints-preserve-(co)limits}.
\end{proof}

A particular class of adjunctions in which we shall be interested are the \emph{reflections}, which are those with invertible counit; equivalently, they are those adjunctions for which the right adjoint is \ff{} in the sense of \cref{ff}. Our main interest in reflections herein is the following lemma, which demonstrates that reflective subobjects inherit colimits.

\begin{lemma}
  \label{reflected-colimits}
  Let $\ell \adj r$ and suppose that the right adjoint is \ff.
  \[\begin{tikzcd}
    A & B
    \arrow[""{name=0, anchor=center, inner sep=0}, "r"', shift right=2, hook, from=1-1, to=1-2]
    \arrow[""{name=1, anchor=center, inner sep=0}, "\ell"', shift right=2, from=1-2, to=1-1]
    \arrow["\dashv"{anchor=center, rotate=-90}, draw=none, from=1, to=0]
  \end{tikzcd}\]
  Let $\vec p$ be a weight and let $f \colon X \to A$ be a tight-cell. If $B$ admits the $\vec p$-weighted colimit of $(f \d r)$, then $A$ admits the $\vec p$-weighted colimit of $f$.

  Suppose furthermore that $\ell' \adj r'$ with $r'$ \ff{}, and that there is a pseudocommutative square as below left.
  Let $\phi$ be the mate of the invertible 2-cell.
  If $b \colon B \to B'$ preserves the $\vec p$-weighted colimit of $(f \d r)$, then $a \colon A \to A'$ preserves the $\vec p$-weighted colimit of $f$ exactly when the whiskering of $\phi$ with $\vec p \wc (f \d r)$ is invertible.
  \[
  \begin{tikzcd}
    B & {B'} \\
    A & {A'}
    \arrow["b", from=1-1, to=1-2]
    \arrow[""{name=0, anchor=center, inner sep=0}, "r", hook, from=2-1, to=1-1]
    \arrow["a"', from=2-1, to=2-2]
    \arrow[""{name=1, anchor=center, inner sep=0}, "{r'}"', hook, from=2-2, to=1-2]
    \arrow["\iso"{description}, draw=none, from=0, to=1]
  \end{tikzcd}
  \hspace{6em}
  \begin{tikzcd}[ampersand replacement=\&]
  	B \& {B'} \\
  	A \& {A'}
  	\arrow["b", from=1-1, to=1-2]
  	\arrow["\ell"', from=1-1, to=2-1]
  	\arrow["\phi", between={0.3}{0.7}, Rightarrow, nfold, from=1-2, to=2-1]
  	\arrow["{\ell'}", from=1-2, to=2-2]
  	\arrow["a"', from=2-1, to=2-2]
  \end{tikzcd}
  \]
\end{lemma}

\begin{proof}
  For the first statement, since $\ell$ is a left adjoint, it preserves colimits (\cref{adjoints-preserve-(co)limits}), and, since $r$ is \ff{}, the counit of the adjunction is invertible. It follows that $(\vec p \wc (f \d r)) \d \ell \iso \vec p \wc (f \d r \d \ell) \iso \vec p \wc f$.

  For the second statement, we have
  \begin{align*}
    \vec p \wc (f \d a) & \iso (\vec p \wc (f \d a \d r')) \d \ell' \\
      & \iso (\vec p \wc (f \d r \d b)) \d \ell' \\
      & \iso (\vec p \wc (f \d r)) \d b \d \ell' \\
      & \tto (\vec p \wc (f \d r)) \d \ell \d a \\
      & \iso (\vec p \wc f) \d a
  \end{align*}
  where the not-necessarily-invertible 2-cell is the whiskering of $\phi$.
\end{proof}

The following lemma, which establishes the good behaviour of right lifts through \ff{} tight-cells, generalises \cite[Lemma~3.29]{arkor2024formal} from tight-cells to loose-cells (upon taking $p$ to be corepresentable).

\begin{lemma}
  \label{lift-through-ff-tight-cell}
  Let $j \colon A \to E$ be a tight-cell. If $j$ is \ff, then, for every loose-cell $p \colon X \lto A$ for which the right lift $(p \rf E(j, 1), \varpi)$ exists, the 2-cell $\varpi \colon (p \rf E(j, 1))(j, 1) \tto p$ is invertible.
\end{lemma}

\begin{proof}
  The 2-cell $\varpi$ is equal to the composite of the following isomorphisms,
  \[(p \rf E(j, 1))(j, 1) \iso p \rf E(j, j) \iso p \rf A(1, 1) \iso p\]
  using \cref{restriction-of-lift} and \ffness{} of $j$.
\end{proof}

\subsection{Pointwise extensions}

A particularly important class of limits and colimits are extensions of one tight-cell along another: these correspond to the \emph{pointwise} extensions in enriched category theory~\cite[\S3.3]{arkor2024formal}.

\begin{definition}
  \label{left-extension}
  Let $j \colon Z \to Y$ and $f \colon Z \to X$ be tight-cells. A \emph{(pointwise) left extension} $j \plx f \colon Y \to X$ of $f$ along $j$ is a $Y(j, 1)$-weighted colimit of $f$.
\end{definition}

A (pointwise) right extension in $\X$ is a left extension in $\X\co$.

\begin{definition}
  \label{pointwise-right-extension}
  Let $j \colon Z \to Y$ and $f \colon Z \to X$ be tight-cells. A \emph{(pointwise) right extension} $j \prx f \colon Y \to X$ of $f$ along $j$ is a $Y(1, j)$-weighted limit of $f$.
\end{definition}

As in \cite[\S3]{arkor2024formal}, we shall speak simply of \emph{left} and \emph{right extensions}, the notion of right extensions of loose-cells (\cref{right-extension}) being distinguished from right extensions of tight-cells (\cref{pointwise-right-extension}) by the context. In any case, the two concepts are closely related: given tight-cells $j \colon Z \to Y$ and $f \colon Z \to X$, the defining property of the right extension of tight-cells is $X(1, j \prx f) \iso Y(1, j) \rx X(1, f)$.

We shall require a few properties of left extensions. The following lemma is handy for composing left extensions.

\begin{lemma}
  \label{composing-left-extensions}
  Let $j$, $f$ and $g$ be tight-cells. Suppose that $j$ is \ff, that the left extensions $f \plx j$ and $j \plx g$ exist, and that the latter preserves the former. Then $(f \plx j) \d (j \plx g) \iso f \plx g$.
\end{lemma}

\begin{proof}
  Using the preservation assumption, together with \cite[Lemma~3.29]{arkor2024formal}, we have the following.
  \[(f \plx j) \d (j \plx g) \iso f \plx (j \d (j \plx g)) \iso f \plx g \qedhere\]
\end{proof}

The following lemma characterises when left extensions witness equivalences.

\begin{lemma}
  \label{nerve--realisation-equivalence}
  Let $j \colon A \to E$ and $f \colon A \to B$ be tight-cells for which the left extensions $f \plx j \colon B \to E$ and $j \plx f \colon E \to B$ exist and are adjoint: $f \plx j \adj j \plx f$.
  \begin{enumerate}
    \item The adjunction is coreflective if $j$ is \ff{}, $f$ is dense, and $f \plx j$ is $j$-absolute.
    \item The adjunction is reflective if $f$ is \ff{} and $j$ is dense.
  \end{enumerate}
  In particular, the adjunction is an equivalence if both $f$ and $j$ are dense and \ff, and $f \plx j$ is $j$-absolute.
\end{lemma}

\begin{proof}
  (1) Since $f \plx j$ is $j$-absolute by assumption, it is preserved by $j \plx f$ by \cite[Lemma~3.22]{arkor2024formal}, hence \cref{composing-left-extensions} implies $(f \plx j) \d (j \plx f) \iso f \plx f \iso 1$, exhibiting coreflectivity. (2) Since $f \plx j$ is a left adjoint, it preserves all left extensions, so \cref{composing-left-extensions} implies $(j \plx f) \d (f \plx j) \iso j \plx j \iso 1$.
\end{proof}

\subsection{Adjoints as absolute lifts}

We now have almost everything we need to begin our study of presheaf constructions and cocompletions. Before doing so, we explain the relation between adjunctions and extensions, which will be useful in studying adjointness properties for cocomplete objects.

\begin{definition}[{\cite[Definition~2.18]{arkor2024formal}}]
	\label{loose-adjunction}
	A \emph{loose adjunction} in a \vdc{} comprises the following data.
	\begin{enumerate}
    \item Loose-cells $\ell \colon A \lto C$ and $r \colon C \lto A$.
    \item A loose-identity $C(1, 1) \colon C \lto C$ and a loose-composite $r \odot \ell \colon A \lto A$.
		\item 2-cells $\eta \colon \tto r \odot \ell$ and $\varepsilon \colon \ell, r \tto C(1, 1)$
	\end{enumerate}
	These data must satisfy the evident triangle identities. We denote this situation by $\ell \adj r$.
\end{definition}

\begin{example}[{\cite[Lemma~2.19]{arkor2024formal}}]
  Let $f \colon A \to C$ be a tight-cell in a \ve{}. Then $C(1, f) \adj C(f, 1)$. The unit is denoted $\pc f \colon {} \tto C(f, f)$ and the counit is $\cp f \colon C(1, f), C(f, 1) \tto C(1, 1)$ (\cref{cp-notation}).
\end{example}

The following observation is classical for 2-categories~\cite[Proposition~2]{street1978yoneda}, but requires a little more care in the situation in which composites may not exist.

\begin{proposition}
  \label{left-loose-adjoint-is-absolute-right-lift}
  Let $\ell \colon B \lto A$ and $r \colon A \lto B$ be loose-cells admitting a loose-composite ${r \odot \ell \colon B \lto B}$ and let $\varepsilon \colon \ell, r \tto A(1, 1)$ be a 2-cell. The following are equivalent.
  \begin{enumerate}
    \item The 2-cell $\varepsilon$ exhibits the counit of a loose adjunction $\ell \adj r$.
    \item For any loose-cell $p \colon X \lto A$ for which the loose-composite $r \odot p$ exists, the 2-cell below exhibits the counit of a right lift $r \odot p \iso p \rf \ell$.
    \[\begin{tikzcd}
      A & B & A & X \\
      A && A & X \\
      A &&& X
      \arrow[""{name=0, anchor=center, inner sep=0}, equals, nfold, from=1-1, to=2-1]
      \arrow["\ell"'{inner sep=.8ex}, "\shortmid"{marking}, from=1-2, to=1-1]
      \arrow["r"'{inner sep=.8ex}, "\shortmid"{marking}, from=1-3, to=1-2]
      \arrow[""{name=1, anchor=center, inner sep=0}, equals, nfold, from=1-3, to=2-3]
      \arrow["p"'{inner sep=.8ex}, "\shortmid"{marking}, from=1-4, to=1-3]
      \arrow[""{name=2, anchor=center, inner sep=0}, equals, nfold, from=1-4, to=2-4]
      \arrow[""{name=3, anchor=center, inner sep=0}, equals, nfold, from=2-1, to=3-1]
      \arrow["\shortmid"{marking}, equals, nfold, from=2-3, to=2-1]
      \arrow["p"{description}, from=2-4, to=2-3]
      \arrow[""{name=4, anchor=center, inner sep=0}, equals, nfold, from=2-4, to=3-4]
      \arrow["p"{inner sep=.8ex}, "\shortmid"{marking}, from=3-4, to=3-1]
      \arrow["\varepsilon"{description}, draw=none, from=0, to=1]
      \arrow["{=}"{description}, draw=none, from=1, to=2]
      \arrow["\opcart"{description}, draw=none, from=3, to=4]
    \end{tikzcd}\]
    \item The 2-cell $\varepsilon$ exhibits the counit of a right lift $r \iso A(1, 1) \rf \ell$ and the 2-cell $\varepsilon \odot \ell$ exhibits the counit $\ell, r \odot \ell \tto \ell$ of a right lift $r \odot \ell \iso \ell \rf \ell$.
  \end{enumerate}
\end{proposition}

\begin{proof}
  (1 $\implies$ 2) Given a loose adjunction $\ell \adj r$, the universal property of the right lift $p \rf \ell$ follows by transposition, \ie pasting with the unit $\eta \colon {} \tto r \odot \ell$, uniqueness of the factorisation following from the triangle identities.

  (2 $\implies$ 3) The composite $r \odot A(1, 1)$ exists trivially, while the composite $r \odot \ell$ exists by assumption.

  (3 $\implies$ 1) Given the specified right lifts, we define a unit ${} \tto \ell \rf \ell$ by applying the universal property of $\ell \rf \ell$ to the identity $1_\ell \colon \ell \tto \ell$; one triangle identity then follows immediately from the universal property.
  \[\begin{tikzcd}
    A & B && B \\
    A & B && B \\
    A &&& B
    \arrow[""{name=0, anchor=center, inner sep=0}, Rightarrow, no head, nfold, from=1-1, to=2-1]
    \arrow["\ell"', "\shortmid"{marking}, from=1-2, to=1-1]
    \arrow[""{name=1, anchor=center, inner sep=0}, Rightarrow, no head, nfold, from=1-2, to=2-2]
    \arrow[Rightarrow, no head, nfold, from=1-4, to=1-2]
    \arrow[""{name=2, anchor=center, inner sep=0}, Rightarrow, no head, nfold, from=1-4, to=2-4]
    \arrow[""{name=3, anchor=center, inner sep=0}, Rightarrow, no head, nfold, from=2-1, to=3-1]
    \arrow["\ell"{description}, from=2-2, to=2-1]
    \arrow["{\ell \rf \ell}"{description}, from=2-4, to=2-2]
    \arrow[""{name=4, anchor=center, inner sep=0}, Rightarrow, no head, nfold, from=2-4, to=3-4]
    \arrow["\ell", "\shortmid"{marking}, from=3-4, to=3-1]
    \arrow["{=}"{description}, draw=none, from=1, to=0]
    \arrow["{\lambda(1_\ell)}"{description}, draw=none, from=2, to=1]
    \arrow["{\varepsilon \odot \ell}"{description}, draw=none, from=3, to=4]
  \end{tikzcd}\]
  For the other triangle identity, observe that the universal property of $A(1, 1) \rf \ell$ implies that, to show the following 2-cell is the identity, it suffices to show that pasting $\ell$ on the left and postcomposing $\varepsilon$ produces $\varepsilon$. This follows straightforwardly from the first triangle identity.
  \[\begin{tikzcd}[column sep=huge]
    B & B & A \\
    B & B & A \\
    B && A
    \arrow[""{name=0, anchor=center, inner sep=0}, Rightarrow, no head, nfold, from=1-1, to=2-1]
    \arrow[Rightarrow, no head, nfold, from=1-2, to=1-1]
    \arrow[""{name=1, anchor=center, inner sep=0}, Rightarrow, no head, nfold, from=1-2, to=2-2]
    \arrow["{A(1, 1) \rf \ell}"', "\shortmid"{marking}, from=1-3, to=1-2]
    \arrow[""{name=2, anchor=center, inner sep=0}, Rightarrow, no head, nfold, from=1-3, to=2-3]
    \arrow[""{name=3, anchor=center, inner sep=0}, Rightarrow, no head, nfold, from=2-1, to=3-1]
    \arrow["{\ell \rf \ell}"{description}, from=2-2, to=2-1]
    \arrow["{A(1, 1) \rf \ell}"{description}, from=2-3, to=2-2]
    \arrow[""{name=4, anchor=center, inner sep=0}, Rightarrow, no head, nfold, from=2-3, to=3-3]
    \arrow["{A(1, 1) \rf \ell}", "\shortmid"{marking}, from=3-3, to=3-1]
    \arrow["{\lambda(1_\ell)}"{description}, draw=none, from=0, to=1]
    \arrow["{=}"{description}, draw=none, from=1, to=2]
    \arrow["{(A(1, 1) \rf \ell) \odot \varepsilon}"{description}, draw=none, from=3, to=4]
  \end{tikzcd}\qedshift\]
\end{proof}

As a consequence, we obtain the following `adjoint tight-cell theorem'.

\begin{corollary}[{\cf~\cites[Proposition~21]{street1978yoneda}[Corollary~10]{wood1982abstract}}]
  \label{adjointness-via-absolute-left-extension}
  Let $f \colon A \to B$ be a tight-cell. The following conditions are equivalent.
  \begin{enumerate}
    \item $f$ admits a right adjoint.
    \item The left extension $f \plx 1_A \colon B \to A$ exists and is absolute.
    \item The left extension $f \plx 1_A \colon B \to A$ exists and is preserved by $f$.
  \end{enumerate}
  In this case, the unit of the left extension exhibits the unit of the adjunction $f \adj f \plx 1_A$.
\end{corollary}

\begin{proof}
  (1 $\iff$ 2) follows from \cite[Proposition~5.10]{arkor2024formal}. For (1 $\iff$ 3), observe that,
  for each tight-cell $u \colon B \to A$, the loose-composite $A(u, 1) \odot B(f, 1) \colon B \lto B$ is exhibited by $B(f u, 1)$. By \cref{left-loose-adjoint-is-absolute-right-lift}, a 2-cell $B(f, 1) \odot A(u, 1) \iso A(u f, 1) \tto A(1, 1)$ therefore exhibits the counit of a loose-adjunction $B(f, 1) \adj A(u, 1)$ if and only if it exhibits the counit of a right lift $A(u, 1) \iso A(1, 1) \rf B(f, 1)$ and its composite exhibits the counit of a right lift $B(f u, 1) \iso B(f, 1) \rf B(f, 1)$: equivalently if the 2-cell exhibits a left extension $u \iso f \plx 1_A$, and this left extension is preserved by $f$. In particular, $B(f, 1) \adj A(u, 1)$ if and only if $f \adj u$. The unit of the left extension is seen to coincide with the unit of the adjunction by tracing it through the correspondence.
\end{proof}

\section{Presheaves in a \ve}
\label{presheaves-in-a-virtual-equipment}

A \emph{presheaf object} in a \ve{} axiomatises the structure of a category of presheaves~\cite[\S7.1]{arkor2025nerve}. Its universal property expresses the bijection discussed in \cref{introduction}, between loose-cells $A \lto B$ and tight-cells $A \to \cl P B$. In particular, the identity $1_{\cl P A} \colon \cl P A \to \cl P A$ induces a loose-cell $\pi_A \colon \cl P A \lto A$, which we require to be dense in the sense of \cref{dense-loose-cell}. (This distinguishes presheaf objects from \emph{power objects}~\cite[\S11]{lambert2022double}, in which $\pi_A$ is not required to be dense; see \cite[Remark~7.4]{arkor2025nerve} for a discussion of related concepts.) We may refine the notion of presheaf object by parameterising it by a class $\PP$ of loose-cells, requiring that the projection $\pi_A$ witnesses a bijection between loose-cells $A \lto B$ in $\PP$ and tight-cells from $A$ into the $\PP$-presheaf object on $B$.

\begin{remark}[A comment on notation]
  Before we present the definition, it is necessary to clarify a distinction between classical enriched category theory and formal category theory, which was mentioned in passing in \cref{weights}. For the theory of categories enriched in a monoidal category, it is typical to consider two notions of weight, depending on whether one is considering limits or colimits~\cite{kelly1982basic}: colimits are weighted by presheaves (equivalently distributors from the trivial $\b V$-category), while limits are weighted by copresheaves (equivalently distributors to the trivial $\b V$-category). This means that it is often reasonable to use the same notation for both colimit notions and limit notions (\eg{} free completions under a class of weights and free cocompletions under a class of weights), since these notions will be disambiguated by whether one is considering a class of weights for colimits, or a class of weights for limits. (Strictly speaking, when enriching in a braided monoidal category, presheaves on $\A$ are in bijection with copresheaves on $\A\op$, which means that the two notions of weight do coincide. However, the two notions are distinguished notationally in practice via the presence or absence of ${}\op$.) In the setting of formal category theory, the appropriate notion of weight (\cref{weight}) is entirely symmetric, and parameterises both limits and colimits. We must therefore take care in our notation to be clear about whether we are talking about limit notions or colimit notions. We decorate operations with $\hat{\ }$ and $\check{\ }$ to denote colimit notions and limit notions respectively, which aligns with the typical convention in enriched category theory to denote the category of presheaves by $\widehat{\A}$ and the category of copresheaves by~$\widecheck{\A}$.
\end{remark}

Henceforth, we shall work in the setting of an ambient strict \ve{} $\X$ (see \cref{strict-ve} and the following discussion).

\begin{definition}
  \label{Phi-presheaf-object}
  Let $\PP$ be a class of loose-cells in $\X$ and let $A$ be an object. A \emph{$\PP$-presheaf object} for $A$ comprises an object $\psh\PP A$ and a dense loose-cell $\pi^\PP_A \colon \psh\PP A \lto A$, such that
  \begin{enumerate}
      \item \label{Phi-presheaf-object-UP} for every loose-cell $p \colon X \lto A$ in $\PP$, there is a unique tight-cell $\pshm p \colon X \to \psh\PP A$ satisfying $p = \pi^\PP_A(1, \pshm p)$;
      \item \label{Phi-presheaf-object-saturation} for every tight-cell $f \colon X \to \psh\PP A$, the restriction $\pi^\PP_A(1, f) \colon X \lto A$ is in $\PP$.
  \end{enumerate}
  Typically, when there is no risk of confusion, we will denote $\pi^\PP_A$ simply by $\pi_A$.
\end{definition}

Note in particular that, for each $\PP$-presheaf object, $\pi^\PP_A = \pi^\PP_A(1, 1)$ is in $\PP$, and $\widepshm{\pi^\PP_A} = 1_{\psh\PP A}$.

\begin{example}
  \label{presheaf-object}
  Denote by $\P$ the class of all loose-cells in $\X$.\footnote{We could alternatively simply write $\X$ for this class, but it is useful to have notation that is independent of the ambient \ve.} Then an object $A$ admits a $\P$-presheaf object $\psh\P A$ if and only if it admits a \emph{presheaf object} in the sense of \cite[Definition~7.2]{arkor2025nerve}.
\end{example}

\begin{remark}[label=P-presheaf-object-is-stronger-than-presheaf-object-in-sub-equipment]
  Given a class $\PP$ of loose-cells in the \ve{} $\X$, closed under restrictions and loose-identities, we can consider the full sub-\ve{} spanned by loose-cells in $\PP$. A $\PP$-presheaf object in $\X$ is, in particular, a presheaf object in the sub-\ve{} (in the sense of \cref{presheaf-object}). However the converse is not necessarily true: density of $\pi_A \colon \psh\PP A \lto A$ in $\X$ is typically a stronger condition than density in the sub-\ve, as it refers to arbitrary loose-cells in $\X$. It therefore does not seem sufficient to study $\PP$-presheaf constructions generally merely by considering presheaf constructions in sub-\ve s{} (or sub-augmented \ve s, \cf~\cref{relationship-to-koudenburg}).
\end{remark}

Strictly speaking, the class of loose-cells in \cref{Phi-presheaf-object} is underconstrained, as the universal property of a presheaf object on an object $A$ only involves loose-cells with codomain $A$. Therefore, we may always restrict a class of loose-cells to the loose-cells with a fixed codomain, without changing the notion of presheaf object.

\begin{lemma}
  \label{PhiA-presheaf-object}
  Let $\PP$ be a class of loose-cells and let $A$ be an object. Denote by $\PP|_{{}\lto A}$ the subclass of $\PP$ spanned by loose-cells with codomain $A$. Then $A$ admits a $\PP$-presheaf object if and only if it admits a $\PP|_{{}\lto A}$-presheaf object.
\end{lemma}

\begin{proof}
  The universal property of \cref{Phi-presheaf-object} only involves loose-cells with codomain $A$.
\end{proof}

However, in practice, we typically wish to consider larger classes of weights than the one minimally determined by \cref{PhiA-presheaf-object}, as we are often interested in comparing $\PP$-presheaf objects on different objects.

\begin{remark}
  \label{canonical-class-of-loose-cells}
  Note that, if a $\PP$-presheaf object exists, then necessarily \[\PP|_{{}\lto A} = \{\ \pi^\PP_A(1, f) \mid f \colon \cdot \to \psh\PP A\ \}\] (where $\cdot$ denotes an arbitrary object). However, typically we wish to start with a class of loose-cells, and then form a presheaf object with respect to that class, rather than conversely. We will give conditions in \cref{well-behaved-is-presheaf-object} under which a given loose-cell $\pi \colon E \lto A$ exhibits a presheaf object for the canonically associated class of loose-cells $\{\ \pi(1, f) \mid f \colon \cdot \to E\ \}$.
\end{remark}

It is worth highlighting that $\PP$-presheaf objects are unique up to isomorphism, rather than up to equivalence. While this tends to be convenient in practice, it does mean that the relationship between the class $\PP$ and the notion of $\PP$-presheaf object is rather strict. To illustrate this point, let us consider a natural operation on a class of loose-cells: closing under isomorphism.

\begin{definition}
  The \emph{repletion} $\PP\repl$ of a class of loose-cells $\PP$ in a \vdc{} $\X$ is the closure of $\PP$ under (globularly) isomorphic loose-cells. A class of loose-cells is \emph{replete} if it is equal to its repletion, \ie{} $\PP = \PP\repl$.
\end{definition}

While the classes $\PP$ and $\PP\repl$ may seem `essentially the same', a $\PP$-presheaf object does not necessarily exhibit a $\PP\repl$-presheaf object, and vice versa.

\begin{lemma}
  \label{coincident-presheaf-objects}
  Let $\PP$ and $\PP'$ be classes of loose-cells and let $A$ be an object. Suppose that a loose-cell $\pi_A \colon P \lto A$ exhibits both a $\PP$- and a $\PP'$-presheaf object for $A$. Then $\PP|_{{}\lto A} = \PP'|_{{}\lto A}$.
\end{lemma}

\begin{proof}
  Let $p \colon X \lto A$ be a loose-cell in $\PP'$. Then $p = \pi_A(1, \pshm p)$ by \eqref{Phi-presheaf-object-UP} for a $\PP'$-presheaf object, and hence $p \in \PP$ by \eqref{Phi-presheaf-object-saturation} for a $\PP$-presheaf object. The converse is symmetric.
\end{proof}

We now establish some basic properties of presheaf objects. One of the most important properties is that the homs of presheaf objects exhibit right lifts.

\begin{lemma}[{\cf~\cite[Lemma~7.5]{arkor2025nerve}}]
  \label{Phi-presheaf-hom-is-right-lift}
  Let $\PP$ be a class of loose-cells and let $A$ be an object admitting a $\PP$-presheaf object. For every pair of loose-cells $p \colon X \lto A$ and $q \colon Y \lto A$ in $\PP$, there is a canonical 2-cell with the following frame,
  \[\begin{tikzcd}[column sep=huge, row sep=small]
    A & X & Y \\
    A && Y
    \arrow[equals, nfold, from=1-1, to=2-1]
    \arrow["p"'{inner sep=.8ex}, "\shortmid"{marking}, from=1-2, to=1-1]
    \arrow["{\psh\PP A(\pshm p, \pshm q)}"'{inner sep=.8ex}, "\shortmid"{marking}, from=1-3, to=1-2]
    \arrow[equals, nfold, from=1-3, to=2-3]
    \arrow["q"{inner sep=.8ex}, "\shortmid"{marking}, from=2-3, to=2-1]
  \end{tikzcd}\]
  which exhibits $\psh\PP A(\pshm p, \pshm q)$ as the right lift $q \rf p$ of $q$ through $p$. In particular, $\psh\PP A(1, \pshm p) \iso p \rf \pi_A$ and $\psh\PP A(\pshm p, 1) \iso \pi_A \rf p$.
\end{lemma}

\begin{proof}
  We have $\psh\PP A(\pshm p, \pshm q) \iso (\pi_A \rf \pi_A)(\pshm p, \pshm q) \iso \pi_A(1, \pshm q) \rf \pi_A(1, \pshm p) = q \rf p$ using density of $\pi_A$, \cref{restriction-of-lift}, and the universal property of the $\PP$-presheaf object.
\end{proof}

For instance, using this fact, we can generalise the observation that a functor is dense if and only if its nerve is \ff{} (which is recovered by taking $p$ to be corepresentable in the following).

\begin{corollary}
  \label{p-dense-iff-breve-p-ff}
  Let $\PP$ be a class of loose-cells and let $A$ be an object admitting a $\PP$-presheaf object. For each loose-cell $p \colon X \lto A$ in $\PP$, the tight-cell $\pshm p \colon X \to \psh\PP A$ is \ff{} if and only if $p$ is dense.
\end{corollary}

\begin{proof}
  $\psh\PP A(\pshm p, \pshm p) \iso p \rf p$ by \cref{Phi-presheaf-hom-is-right-lift}, and it is clear this isomorphism commutes with the canonical 2-cells from $X(1, 1)$.
\end{proof}

As a consequence, while we may not generally have $\psh\PP A \iso \psh{\PP\repl} A$, the former is nevertheless contained in the latter.

\begin{corollary}
  \label{subpresheaf-object}
  Let $\PP \subseteq \PP'$ be classes of loose-cells and let $A$ be an object admitting $\PP$- and $\PP'$-presheaf objects. Then there is a \ff{} tight-cell $\psh\PP A \ffto \psh{\PP'} A$.
\end{corollary}

\begin{proof}
  Observe that since $\pi^\PP_A \colon \psh\PP A \lto A$ is in $\PP$, it is also in $\PP'$, so the universal property of $\psh{\PP'} A$ induces a tight-cell $\widepshm{\pi^\PP_A} \colon \psh\PP A \to \psh{\PP'} A$, which by \cref{p-dense-iff-breve-p-ff} is \ff{} since $\pi^\PP_A$ is dense.
\end{proof}

While the universal property of \cref{Phi-presheaf-object} is ostensibly one-dimensional, the density of the projection $\pi_A \colon \psh\PP A \lto A$ implies that presheaf objects satisfy a two-dimensional universal property.

\begin{proposition}[{\cf~\cite[Proposition~7.6]{arkor2025nerve}}]
    \label{Phi-presheaf-two-dimensional-UP}
    Let $\PP$ be a class of loose-cells and let $A$ be an object admitting a $\PP$-presheaf object. Every 2-cell $\phi$ of the following form, in which $p$ and $q$ are loose-cells in $\PP$,
    \[\begin{tikzcd}
    	A & X & \cdots & {X'} \\
    	A &&& {X'}
    	\arrow["q", "\shortmid"{marking}, from=2-4, to=2-1]
    	\arrow[""{name=0, anchor=center, inner sep=0}, Rightarrow, no head, nfold, from=1-1, to=2-1]
    	\arrow["p"', "\shortmid"{marking}, from=1-2, to=1-1]
    	\arrow["{p_1}"', "\shortmid"{marking}, from=1-3, to=1-2]
    	\arrow["{p_n}"', "\shortmid"{marking}, from=1-4, to=1-3]
    	\arrow[""{name=1, anchor=center, inner sep=0}, Rightarrow, no head, nfold, from=1-4, to=2-4]
    	\arrow["\phi"{description}, draw=none, from=0, to=1]
    \end{tikzcd}\]
    factors uniquely through a 2-cell $\pshm\phi$, as below.
    \[\begin{tikzcd}[column sep=large]
    	A & {\psh\PP A} & X & \cdots & {X'} \\
    	A & {\psh\PP A} &&& {X'}
    	\arrow[""{name=0, anchor=center, inner sep=0}, Rightarrow, no head, nfold, from=1-2, to=2-2]
    	\arrow["{\psh\PP A(1, \pshm p)}"', "\shortmid"{marking}, from=1-3, to=1-2]
    	\arrow["{\psh\PP A(1, \pshm q)}", "\shortmid"{marking}, from=2-5, to=2-2]
    	\arrow["{p_1}"', "\shortmid"{marking}, from=1-4, to=1-3]
    	\arrow["{p_n}"', "\shortmid"{marking}, from=1-5, to=1-4]
    	\arrow[""{name=1, anchor=center, inner sep=0}, Rightarrow, no head, nfold, from=1-5, to=2-5]
    	\arrow["{\pi_A}"', "\shortmid"{marking}, from=1-2, to=1-1]
    	\arrow["{\pi_A}", "\shortmid"{marking}, from=2-2, to=2-1]
    	\arrow[""{name=2, anchor=center, inner sep=0}, Rightarrow, no head, nfold, from=1-1, to=2-1]
    	\arrow["\pshm\phi"{description}, draw=none, from=0, to=1]
    	\arrow["{=}"{description}, draw=none, from=2, to=0]
    \end{tikzcd}\]
\end{proposition}

\begin{proof}
    Since $q \rf p \iso \psh\PP A(\pshm p, \pshm q)$ by \cref{Phi-presheaf-hom-is-right-lift}, the factorisation follows directly from the universal property of the right lift.
\end{proof}

It follows that the universal property of a presheaf object may be expressed as a two-dimensional adjointness property.

\begin{notation}
  For a class of loose-cells $\PP \subseteq \X$ in a \vdc{} $\X$, and for each pair of objects $X, A \in \X$, denote by $\PP\lh{X, A}$ the category of loose-cells $X \lto A$ in $\PP$ and the globular 2-cells between them.
\end{notation}

\begin{proposition}[{\cf~\cite[Lemma~7.7]{arkor2025nerve}}]
  \label{Phi-presheaf-via-adjunction}
  Let $\PP$ be a class of loose-cells and let $A$ be an object.
  A dense loose-cell $\pi_A \colon \psh\PP A \lto A$ exhibits a $\PP$-presheaf object for $A$ if and only if, for each object $X \in \X$, the functor
  \[\pi_A(1, {-}) \colon \X[X, \psh\PP A] \to \PP\lh{X, A}\]
  admits an inverse $\widepshm{\ph} \colon \PP\lh{X, A} \rightleftarrows \X[X, \psh\PP A]$, exhibiting an isomorphism of categories.
\end{proposition}

\begin{proof}
  Suppose that $\pi_A \colon \psh\PP A \lto A$ exhibits a $\PP$-presheaf object. That the assignment forms a bijection on objects is immediate from the universal property; that it forms a bijection on morphisms is immediate from \cref{Phi-presheaf-two-dimensional-UP}.

  Conversely, suppose $\pi_A \colon \psh\PP A \lto A$ exhibits a natural isomorphism of categories. The universal property of \cref{Phi-presheaf-object} follows trivially.
\end{proof}

The following characterises those colimits that exist in presheaf objects, which will be central to relating presheaf objects to cocompletions in \cref{presheaves-and-cocompletions}.

\begin{proposition}[label=presheaf-colimits-absolute, note={\cf~\cite[Corollary~5.7]{koudenburg2024formal}}]
  Let $\PP$ be a class of loose-cells, let $\vec p \colon Y \lcto X$ be a weight and let $q \colon X \lto A$ be a loose-cell for which $A$ admits a $\PP$-presheaf object and for which $q \in \PP$. Suppose that the left-composite $q \odotl \vec p$ exists. Then $\psh\PP A$ admits a $\pi_A$-absolute colimit $\vec p \wc \pshm q$ if and only if $(q \odotl \vec p) \in \PP\repl$.
\end{proposition}

\begin{proof}
  By \cref{absolute-implies-colimit}, $\psh\PP A$ admits the required $\pi_A$-absolute colimit exactly when there exists a tight-cell
  $(\vec p \wc \pshm q) \colon Y \to \psh{\PP}{A}$
  and isomorphism
  \[
    \pi_A(1, \vec p \wc \pshm q) \iso \pi_A(1, \pshm q) \odotl \vec p
  \]
  These data exist exactly when the right-hand side is in $\PP\repl$, but the right-hand side is precisely $q \odotl \vec p$.
\end{proof}

Generally speaking, a $\PP$-presheaf object may not be `closed under representables', in that we may not have a tight-cell $A \to \psh\PP A$. However, we shall primarily be interested in situations where we have such an embedding, for which we introduce the following definition.

\begin{definition}
  \label{Phi-embedding}
  Let $\PP$ be a class of loose-cells. An object $A$ is \emph{$\PP$-admissible} if it admits a $\PP$-presheaf object, and $A(1, 1) \in \PP\repl$.
  In this case, we denote by
  \[\pshe\PP A \colon A \to \psh\PP A\]
  the \emph{$\PP$-presheaf embedding of $A$}, which is the essentially unique tight-cell satisfying $A(1, 1) \iso \pi_A(1, \pshe\PP A)$.
\end{definition}

$\PP$-embeddings are the formal analogues of Yoneda embeddings, and are well behaved in the ways one would expect.

\begin{lemma}
  Let $\PP$ be a class of loose-cells and let $A$ be a $\PP$-admissible object.
  \begin{enumerate}
    \item \label{Yoneda} There is an isomorphism $\psh\PP A(\pshe\PP A, 1) \iso \pi_A$ (\cf~\cite[Lemma~7.9]{arkor2025nerve}).
    \item \label{Phi-embedding-is-dense-and-ff} $\pshe\PP A$ is dense and \ff{}.
    \item $\pshe\PP A$ preserves and reflects weighted limits (\cf~\cite[Theorem~10.28]{shulman2013enriched}).
    \item \label{tight-cells-into-psh-are-absolute-colimits} For each weight $p \in \PP$, the tight-cell $\pshm p \colon X \to \psh\PP A$ exhibits the $\pshe\PP A$-absolute colimit $p \wc \pshe\PP A$ (\cf~\cite[Proposition~3.15]{kelly2005notes}).
  \end{enumerate}
\end{lemma}

\begin{proof}
  \begin{enumerate}
    \item We have $\psh\PP A(\pshe\PP A, 1) \iso \pi_A \rf \pi_A(1, \pshe\PP A) \iso \pi_A \rf A(1, 1) \iso \pi_A$ using \cref{Phi-presheaf-hom-is-right-lift}.
    \item Using (1), density follows immediately from density of $\pi_A$, while \ffness{} follows from \cref{p-dense-iff-breve-p-ff}.
    \item Immediate from \cref{dense-and-ff-implies-continuous} using (2).
    \item Immediate from \cref{presheaf-colimits-absolute}.
    \qedhere
  \end{enumerate}
\end{proof}

\begin{corollary}
  Let $\PP \subseteq \PP'$ be classes of loose-cells and let $A$ be a $\PP$-admissible and $\PP'$-admissible object. The induced tight-cell $\psh\PP A \ffto \psh{\PP'} A$ (\cref{subpresheaf-object}) commutes with the presheaf embeddings up to isomorphism.
  \[\begin{tikzcd}
    {\psh\PP A} && {\psh{\PP'}A} \\
    & A
    \arrow["{\widepshm{\pi^\PP_A}}", hook, from=1-1, to=1-3]
    \arrow[""{name=0, anchor=center, inner sep=0}, "{\pshe\PP A}", from=2-2, to=1-1]
    \arrow[""{name=1, anchor=center, inner sep=0}, "{\pshe{\PP'}A}"', from=2-2, to=1-3]
    \arrow["\iso"{description}, shift left=2, draw=none, from=0, to=1]
  \end{tikzcd}\]
\end{corollary}

\begin{proof}
  We have $\pi^{\PP'}(1, \widepshm{\pi^\PP_A}) = \pi^\PP_A$ by definition, so $\pi^{\PP'}_A(1, \widepshm{\pi^\PP_A} \pshe\PP A) = \pi^\PP_A(1, \pshe\PP A) \iso A(1, 1) \iso \pi^{\PP'}_A(1, \pshe{\PP'}A)$, from which the statement follows using density of $\pi^{\PP'}_A$.
\end{proof}

In \cref{presheaf-colimits-absolute}, we gave necessary and sufficient conditions for a presheaf object to admit $\pi_A$-absolute weighted colimits. We may also give necessary and sufficient conditions for a presheaf object to admit weighted limits.

\begin{proposition}[label=limit-in-presheaf-object, note={\cf~\cite[Theorem~10.29]{shulman2013enriched}}]
  Let $\PP$ be a class of loose-cells, let $\vec p \colon Y \lcto X$ be a weight and let $q \colon Y \lto A$ be a loose-cell in $\PP$. Suppose that $A$ admits a $\PP$-presheaf object.
  If the right extension $\vec p \rx q$ exists and is in $\PP$, then $\widepshm{\vec p \rx q}$ exhibits the limit $\vec p \wl \pshm q \colon X \to \psh\PP A$.
  The converse holds if $A$ is $\PP$-admissible.
\end{proposition}

\begin{proof}
  Suppose that $\vec p \rx q$ exists and is in $\PP$. We have:
  \begin{align*}
    \psh\PP A(1, \widepshm{\vec p \rx q}) & \iso (\vec p \rx q) \rf \pi_A \\
      & \iso \vec p \rx (q \rf \pi_A) \\
      & \iso \vec p \rx (\pi_A(1, \pshm q) \rf \pi_A) \\
      & \iso \vec p \rx \psh\PP A(1, \pshm q) \tag{$\pi_A$ is dense}
  \end{align*}
  Conversely, suppose that $A$ is $\PP$-admissible and that the limit $\vec p \wl \pshm q$ exists. We have:
  \begin{align*}
    \pi_A(1, \vec p \wl \pshm q) & \iso \psh\PP A(\pshe\PP A, \vec p \wl \pshm q) \\
      & \iso \psh\PP A(1, \vec p \wl \pshm q)(\pshe\PP A, 1) \\
      & \iso (\vec p \rx \psh\PP A(1, \pshm q))(\pshe\PP A, 1) \\
      & \iso \vec p \rx \psh\PP A(\pshe\PP A, \pshm q) \\
      & \iso \vec p \rx \pi_A(1, \pshm q) \\
      & \iso \vec p \rx q
      \qedhere
  \end{align*}
\end{proof}

\begin{corollary}
  \label{breve-p-is-continuous}
  Let $\PP$ be a class of loose-cells and let $q \colon Y \lto A$ be a loose-cell in $\PP$. Suppose that $A$ admits a $\PP$-presheaf object. If $q$ respects a limit $\vec p \wl f$, then $\pshm q \colon Y \to \psh\PP A$ preserves it.
\end{corollary}

\begin{proof}
  First, observe that $q(1, \vec p \wl f) \iso \vec p \rx q(1, f)$, so that the latter is in $\PP\repl$ since $q$ is. By \cref{limit-in-presheaf-object}, $\widepshm{q(1, \vec p \wl f)} = \pshm q(\vec p \wl f)$ consequently exhibits the limit $\vec p \wl \widepshm{q(1, f)} = \vec p \wl (\pshm q f)$.
\end{proof}

The properties of presheaf objects established above will be sufficient for our study of free cocompletions, which is our concern here. In a sequel, we will study presheaf objects further in relation to \emph{admissibility} and \emph{realisability}, which is the study of abstract nerve constructions~\cite{street1978yoneda,bunge1999bicomma}.

\section{Atomicity}
\label{atomicity}

In the study of $\Phi$-cocompletions, an important concept is the notion of \emph{$\Phi$-atomicity}, which extends the concept of a $\Phi$-atom in a category $\A$, \viz{} an object $a \in \ob\A$ for which $\A(a, {-}) \colon \A \to \Set$ preserves $\Phi$-colimits~\cite[\S4]{kelly2005notes}. We generalise this notion from objects in a category to loose-cells in a \ve.

\begin{definition}[label=atom]
  Let $\vec p \colon Y \lcto Z$ be a weight. A loose-cell $j \colon E \lto A$ is a \emph{$\vec p$-atom} when every $\vec p$-weighted colimit in $E$ is $j$-absolute. Given a class of weights $\Phi$, we say that $j$ is a \emph{$\Phi$-atom} if it is a $\vec p$-atom for each weight $\vec p \in \Phi$. A tight-cell is a \emph{$\vec p$-atom}, respectively \emph{$\Phi$-atom}, if its conjoint is.
\end{definition}

\begin{example}
  \label{examples-of-atoms}
  A functor $f \colon \A \to \B$ is a $\Phi$-atom if and only if the nerve functor $\B(f({-}_2), {-}_1) \colon \B \to \Set^{\A\op}$ preserves every $\Phi$-colimit that exists in $\B$~\cite[Lemma~8.10]{arkor2024formal}. Since colimits in $\Set^{\A\op}$ are computed pointwise, this equivalently says that, for each $a \in \ob\A$, the functor $\B(f(a), {-}) \colon \B \to \Set$ is $\Phi$-cocontinuous, \ie{} each object $f(a)$ is a $\Phi$-atom in the classical sense. In particular, an object $b \in \ob\B$ is a $\Phi$-atom in the classical sense if and only if the corresponding functor $b \colon \b1 \to \B$ is a $\Phi$-atom in the sense of \cref{atom}.

  Consider $\Delta$ a class of categories, which induces a class $\Phi$ of weights by $\Phi \defeq \{ * \colon \b1 \lto \b J \mid \b J \in \Delta \}$, where $*$ denotes the terminal presheaf.
    \begin{enumerate}
      \item When $\Delta = \{ \text{small discrete categories} \}$, the $\Phi$-atoms are called \emph{abstractly exclusively unary}~\cite{bunge1966categories}, \emph{Z-objects}~\cite{hoffmann1973categorical}, \emph{coprime}~\cite{borger1989multicoreflective}, or \emph{connected}~\cite{ehresmann1978multiple}.
      \item When $\Delta = \{ \text{small sifted categories} \}$, the $\Phi$-atoms are called \emph{strongly finitely presentable}~\cite{adamek2001sifted} or \emph{perfectly presentable}~\cite{adamek2010algebraic}.
      \item When $\Delta = \{ \text{small filtered categories} \}$, the $\Phi$-atoms are called \emph{finitely presentable}~\cite{gabriel1971lokal} or \emph{compact}~\cite{lurie2003infinity}.
      \item When $\Delta = \{ \text{small $\kappa$-filtered categories} \}$, for a regular cardinal $\kappa$, the $\Phi$-atoms are called \emph{$\kappa$-presentable}~\cite{gabriel1971lokal} or \emph{$\kappa$-compact}~\cite{lurie2003infinity}.
      \item When $\Delta = \{ \text{small categories} \}$, the $\Phi$-atoms are called \emph{small projectives}~\cite{kelly1982basic}, \emph{absolutely presentable}~\cite{centazzo2004characterization} or \emph{atoms}~\cite{kelly2005notes}.
      \qedhere
    \end{enumerate}
\end{example}

We shall use atomicity in \cref{well-behaved-is-presheaf-object} to give a characterisation of presheaf objects, and later in \cref{cocompletion-recognition} to give a recognition theorem for free cocompletions. Before doing so, we establish a useful relationship between relative adjointness and atomicity. Recall from \cref{adjoints-preserve-(co)limits} that left relative adjoints preserve those colimits preserved by the root, whilst relative right adjoints preserve limits when the root is dense; the following situation is reminiscent.

\begin{lemma}
  \label{right-adjoints-preserve-absolute-colimits}
  Let $\jAE$ be a dense tight-cell and let $\ljr$. The right adjoint $r$ preserves $\ell$-absolute colimits.
\end{lemma}

We note that, under the additional assumption that $j$ is \ff, this follows from the combination of the facts that (1) under the given assumptions, $r$ is the left extension $\ell \plx j$; and that (2) left extensions along $\ell$ preserve $\ell$-absolute colimits~\cite[Proposition~5.10 \& Lemma~3.22]{arkor2024formal}. However, this property holds even without the assumption that $j$ is \ff.

\begin{proof}
  We have the following chain of isomorphisms.
  \begin{align*}
    E(j, r (\vec p \wc f)) & \iso C(\ell, \vec p \wc f) \tag{$\ljr$} \\
      & \iso C(\ell, f) \odotl \vec p \tag{$\vec p \wc f$ is $\ell$-absolute} \\
      & \iso E(j, r f) \odotl \vec p \tag{$\ljr$}
  \end{align*}
  Hence $r (\vec p \wc f) \iso \vec p \wc (r f)$ by \cref{absolute-implies-colimit}.
\end{proof}

\begin{theorem}[label=atomic-left-adjoint-iff-cocontinuous-right-adjoint]
  Let $\jAE$ be a tight-cell and suppose $\ljr$.
  \begin{enumerate}
    \item If $j$ is a $\Phi$-atom and $r$ is $\Phi$-cocontinuous, then $\ell$ is a $\Phi$-atom.
    \item If $j$ is dense and $\ell$ is a $\Phi$-atom, then $r$ preserves all $\Phi$-colimits that exist.
  \end{enumerate}
\end{theorem}

\begin{proof}
  Let $\vec p \colon Z \lcto X$ be a weight in $\Phi$ and let $g \colon X \to C$ be a tight-cell.
  \begin{enumerate}
    \item \begin{align*}
      C(\ell, \vec p \wc g) & \iso E(j, r(\vec p \wc g)) \tag{$\ljr$} \\
        & \iso E(j, \vec p \wc (r g)) \tag{$r$ is $\Phi$-cocontinuous} \\
        & \iso E(j, r g) \odotl \vec p \tag{$j$ is a $\Phi$-atom} \\
        & \iso C(\ell, g) \odotl \vec p \tag{$\ljr$}
    \end{align*}
    \item Since $\ell$ is a $\Phi$-atom, $\Phi$-colimits are $\ell$-absolute, from which the result follows by \cref{right-adjoints-preserve-absolute-colimits}.
    \qedhere
  \end{enumerate}
\end{proof}

\begin{example}
  In particular, \cref{atomic-left-adjoint-iff-cocontinuous-right-adjoint} recovers the characterisation of $\Phi$-atomic functors mentioned in \cref{examples-of-atoms}: given a functor $f \colon \A \to \B$, the nerve $\B(f({-}_2), {-}_1) \colon \B \to \Set^{\A\op}$ is right-adjoint to $f$ relative to the presheaf embedding $\yo_\A \colon \A \to \Set^{\A\op}$~\cite[Lemma~7.9]{arkor2025nerve}, which is trivially a dense $\Phi$-atom (since its nerve is the identity on $\Set^{\A\op}$). Thus, \cref{atomic-left-adjoint-iff-cocontinuous-right-adjoint} states that $f$ is a $\Phi$-atom if and only if $\B(f({-}_2), {-}_1) \colon \B \to \Set^{\A\op}$ preserves every $\Phi$-colimit that exists in $\B$.
\end{example}

\begin{example}
  Let $\ell \colon \Ind(\A) \rightleftarrows \Ind(\B) \cocolon r$ be an adjunction between \lfp{} categories, for small finitely cocomplete categories $\A$ and $\B$. By precomposing the dense inclusion $j \colon \A \to \Ind(\A)$, which is a $\Phi$-atom with respect to $\Phi$ the class of filtered colimits (\cref{examples-of-atoms}), we obtain a relative adjunction~\cite[Proposition~5.29]{arkor2024formal}.
  \[\begin{tikzcd}
    & {\Ind(\B)} \\
    {\A} && {\Ind(\A)}
    \arrow[""{name=0, anchor=center, inner sep=0}, "r", from=1-2, to=2-3]
    \arrow[""{name=1, anchor=center, inner sep=0}, "{j \d \ell}", from=2-1, to=1-2]
    \arrow["j"', hook, from=2-1, to=2-3]
    \arrow["\dashv"{anchor=center, rotate=1}, shift right=2, draw=none, from=1, to=0]
  \end{tikzcd}\]
  Since $\A$ comprises the finitely presentable objects of $\Ind(\A)$, the functor $(j \d \ell)$ is $\Phi$-atomic if and only if $\ell$ preserves finitely presentable objects.
  \Cref{atomic-left-adjoint-iff-cocontinuous-right-adjoint} therefore states that $\ell$ preserves finitely presentable objects if and only if $r$ preserves filtered colimits, which is a key observation in establishing Gabriel--Ulmer duality~\cite{gabriel1971lokal}.\footnote{Building on this observation, it is possible to establish an entirely formal duality theorem. We defer this to the study of admissibility and realisability mentioned at the end of the previous section.}
\end{example}

\subsection{Characterisation of presheaf objects}

As remarked in \cref{canonical-class-of-loose-cells}, given a $\PP$-presheaf object, the class of loose-cells $\PP|_{{}\lto A}$ is canonically determined by the loose-cell $\pi_A \colon \psh\PP A \lto A$. It is therefore natural to ask whether it is possible to characterise when an arbitrary loose-cell exhibits a presheaf object for the canonical class of loose-cells induced in this manner. The following gives a sharp characterisation, subject to a mild condition on the class of loose-cells.

\begin{definition}[label=monic]
  A loose-cell $j \colon E \lto A$ is \emph{monic} when, given any pair of tight-cells $f, f' \colon X \to E$, if $j(1, f) = j(1, f')$ then $f = f'$.
\end{definition}

\begin{proposition}
  \label{well-behaved-is-presheaf-object}
  The following properties of a loose-cell $j \colon E \lto A$ are equivalent.
  \begin{enumerate}
    \item There exists a class $\PP$ of loose-cells, for which $(j(1, f) \odotl j) \in \PP\repl$ for each $f \colon A \to E$, and such that $j \colon E \lto A$ exhibits a $\PP$-presheaf object (\cref{Phi-presheaf-object}).
    \item $j$ is a dense and monic $j$-atom, and $E$ admits $j$-weighted colimits (\cref{dense-loose-cell,monic,atom}).
  \end{enumerate}
  In this case, $A$ is $\PP$-admissible (\cref{Phi-embedding}) if and only if $j$ is corepresented by a \ff{} tight-cell $A \to E$.
\end{proposition}

\begin{proof}
  (1) $\implies$ (2). By the definition of a presheaf object, $j$ must be dense and monic. That, under the given assumptions, every $j$-weighted colimit exists and is $j$-absolute is immediate from \cref{presheaf-colimits-absolute}.

  (2) $\implies$ (1). Define a class of weights $\PP$ as follows.
  \[\PP \defeq \{ j(1, f) \mid f \colon X \to E \}\]
  By construction, $\PP$ is closed under precomposition by tight-cells into $E$.
  For each $f \colon A \to E$ we have $j(1, j \wc f) \iso j(1, f) \odotl j$ because $j \wc f$ is $j$-absolute, so the repletion of $\PP$ is closed under the necessary left-composites.

  Define $\pi_A \defeq j \colon E \lto A$. This is dense by assumption. Given $j(1, f) \colon X \lto A$, we define $\widepshm{j(1, f)} \defeq f$. We thus have $\pi_A(1, \widepshm{j(1, f)}) = j(1, \widepshm{j(1, f)}) = j(1, f)$. That this choice is unique follows from monicity of $j$.

  Finally, if $A$ is $\PP$-admissible, then there exists a \ff{} tight-cell $A \to E$ by \cref{Phi-embedding-is-dense-and-ff}. Conversely, if $j$ is the conjoint of a \ff{} tight-cell $j'$, we have $A(1, 1) \iso E(j', j') \iso j(1, j') \in \PP\repl$.
\end{proof}

\begin{remark}
  Density of a loose-cell $j \colon E \lto A$ implies that it is `essentially monic', \ie{} given $f, f' \colon X \to E$, if $j(1, f) \iso j(1, f')$ then $f \iso f'$. The monicity condition in \cref{well-behaved-is-presheaf-object} is necessary because our presheaf objects are unique up to isomorphism, rather than equivalence.
\end{remark}

The previous proposition motivates the following definition, which essentially characterises presheaf objects with respect to classes of weights that are particularly well behaved; the terminology follows \cite[\S4]{altenkirch2015monads}.

\begin{definition}
  \label{well-behaved}
  A tight-cell $j \colon A \to E$ is \emph{well-behaved} if the following conditions hold.
  \begin{enumerate}
    \item $j$ is \ff{}.
    \item $j$ is dense.
    \item Left extensions of tight-cells $A \to E$ along $j$ exist.
    \item $j$ is an $E(j, 1)$-atom, \ie the left extensions above are $j$-absolute.
    \qedhere
  \end{enumerate}
\end{definition}

We shall further study this property in \cref{cocompleteness-in-a-virtual-equipment}.

\section{Rank}
\label{rank-etc}

Along with atomicity, a further important concept in the study of $\Phi$-cocompletions is the notion of \emph{rank} with respect to a tight-cell, which extends the concept of the \emph{accessibility rank} of a functor. Classically, a functor is said to have \emph{rank} $\kappa$, for a regular cardinal $\kappa$, if it preserves $\kappa$-filtered colimits (\ie{} if it is \emph{$\kappa$-accessible}, \aka{} \emph{finitary} when $\kappa = \aleph_0$). The rank of a functor intuitively captures the `size' of a functor: a functor with rank may typically be recovered by specifying a functor out of a smaller subcategory, and left-extending along the subcategory inclusion. The following definitions are intended to capture this intuition.

\begin{definition}
  \label{rank}
  Let $\jAE$ be a tight-cell.
  \begin{enumerate}
    \item A tight-cell $f \colon E \to I$ has \emph{rank} $j$ if the left extension $j \plx (j \d f)$ exists, and the canonical 2-cell $j \plx (j \d f) \tto f$ is invertible.
    \item A loose-cell $p \colon I \lto E$ has \emph{rank} $j$ if the right lift $p(j, 1) \rf E(j, 1)$ exists, and the canonical 2-cell $p \tto p(j, 1) \rf E(j, 1)$ is invertible. \qedhere
  \end{enumerate}
\end{definition}

In particular, a tight-cell $f \colon E \to I$ has rank $j$ precisely when its conjoint $I(f, 1) \colon I \lto E$ does.

\begin{example}
  The most common situation in which to consider rank is when the functor $j \colon \A \to \E$ exhibits the cocompletion of a small category under a class of colimits. In this situation, the property of having rank $j$ admits several equivalent formulations. Since we establish these equivalent characterisations in full generality in \cref{Phi-preservation-is-phi-absolute-preservation}, we shall only give a few illustrative examples here.
  \begin{enumerate}
    \item Let $j \colon \A \to \Ind(\A)$ denote the cocompletion of a small category $\A$ under filtered colimits. A functor from $\Ind(\A)$ is \emph{finitary} precisely when it has rank $j$, which holds if and only if it preserves filtered colimits.
    \item Let $j \colon \A \to \Sind(\A)$ denote the cocompletion of a small category $\A$ under sifted colimits. A functor from $\Sind(\A)$ is \emph{strongly finitary} precisely when it has rank $j$, which holds if and only if it preserves sifted colimits.
    \qedhere
  \end{enumerate}
\end{example}

\begin{example}
  \label{density-via-rank}
  Rank may be viewed as a generalisation of the concept of density: a tight-cell $\jAE$ is dense precisely when the identity $1_E$ on its codomain has rank $j$.
\end{example}

The following proposition establishes a close connection between rank, the property of being a left extension, and preservation of absolute colimits.

\begin{proposition}[{\cf~\cites[Th\'eor\`eme~2.5]{diers1974jadjonction}[Theorem~5.29]{kelly1982basic}[Proposition~7.9]{lucyshyn2016enriched}[Lemma~3.22]{arkor2024formal}}]
  \label{rank-and-cocontinuity}
  Let $\jAE$ be a tight-cell.
  Consider the following conditions on a tight-cell $f \colon E \to I$ (respectively, a loose-cell $p \colon I \lto E$).
  \begin{enumerate}
    \item $f$ has rank $j$ (respectively $p$ has rank $j$).
    \item $f$ is a left extension along $j$ (respectively $p$ is a right lift through $E(j, 1)$).
    \item $f$ preserves $j$-absolute colimits (respectively $p$ respects $j$-absolute colimits).
  \end{enumerate}
  We have (1) $\implies$ (2) $\implies$ (3).
  If $j$ is \ff{}, then (2) $\implies$ (1).
  If $j$ is \ff{} and dense, then all three conditions are equivalent.
\end{proposition}

\begin{proof}
  Each condition for a tight-cell $f$ is equivalent to that for the corresponding loose-cell $I(f, 1)$, so it is enough to give the proofs only for loose-cells.
  (1) $\implies$ (2). Trivial.
  (2) $\implies$ (3). Follows from \cref{right-lifts-respect-absolute-colimits}.

  (2) $\implies$ (1). Writing $p$ as a right lift $p' \rf E(j, 1)$, we have
  \begin{align*}
    p & \iso p' \rf E(j, 1) \\
    & \iso (p' \rf A(1, 1)) \rf E(j, 1) \\
    & \iso (p' \rf E(j, j)) \rf E(j, 1) \tag{$j$ is \ff{}} \\
    & \iso (p' \rf E(j, 1))(j, 1) \rf E(j, 1) \tag{\cref{restriction-of-lift}} \\
    & \iso p(j, 1) \rf E(j, 1)
  \end{align*}

  Finally, when $j$ is \ff{} and dense, (3) $\implies$ (1) because $1_E$ is the $j$-absolute left extension $j \plx j$:
  \[
    E(j, j \plx j) ~\iso~ E(j, 1) ~\iso~ E(j, j) \odotl E(j, 1)
  \]
  Hence $p$ respects $j \plx j$, so $p \iso p(j \plx j, 1) \iso p(j, 1) \rf E(j, 1)$ as required.
\end{proof}

\begin{remark}
  From \cref{right-adjoints-preserve-absolute-colimits,rank-and-cocontinuity}, we see that tight-cells $r$ that preserve \mbox{$\ell$-absolute} colimits provide a joint generalisation of right relative adjoints to $\ell$ with dense roots, and tight-cells with rank $\ell$.
\end{remark}

The following provides useful composability and cancellation lemmas for rank.

\begin{lemma}
  Let $\jAE$ and $a \colon A' \to A$ be tight-cells.
  Consider the following conditions for a tight-cell $f \colon E \to I$ (respectively, a loose-cell $p \colon I \lto E$).
  \begin{enumerate}
    \item $(j \d f)$ has rank $a$ (respectively $p(j, 1)$ has rank $a$).
    \item $f$ has rank $j$ (respectively $p$ has rank $j$).
    \item $f$ has rank $(a \d j)$ (respectively $p$ has rank $(a \d j)$).
  \end{enumerate}
  \label{rank-composition} If (1) holds, then (2) $\iff$ (3).
  \label{rank-from-rank} If $j$ is \ff{}, then (3) $\implies$ (1 \& 2).
\end{lemma}

\begin{proof}
  The results for a tight-cell $f$ follow from those for a loose-cell $p$, so we just give the proofs for the latter case.
  The first claim is immediate from the following isomorphisms, where the first follows from $p(j, 1)$ having rank $a$, and the second follows from \cref{iterated-lift,right-lift-through-composite}.
  \[
    p(j, 1) \rf E(j, 1)
    \iso
    (p(ja, 1) \rf A(a, 1)) \rf E(j, 1)
    \iso
    p(ja, 1) \rf E(ja, 1)
  \]

  Now suppose that $j$ is \ff{} and that (3) holds.
  The following isomorphisms prove that (1) holds, so that (2) also holds.
  \[
    p(j, 1) \iso (p(ja, 1) \rf E(ja, 1))(j, 1) \iso p(ja, 1) \rf E(ja, j) \iso p(ja, 1) \rf A(a, 1) \qedhere
  \]
\end{proof}

The following characterises tight- and loose-cells whose rank is a composite.

\begin{lemma}[label=rank-and-ff-extension]
  Let $\jAE$ be a \ff{} tight-cell.
  For every tight-cell $a \colon A' \to A$, we have the following.
  \begin{enumerate}
    \item A tight-cell has rank $(a \d j)$ if and only if it is a left extension $j \plx g$ for some $g$ of rank $a$.
    \item A loose-cell has rank $(a \d j)$ if and only if it is a right lift $q \rf E(j, 1)$ for some $q$ of rank $a$.
  \end{enumerate}
\end{lemma}

\begin{proof}
  We give the proof of (2); the proof of (1) is analogous.
  If a loose-cell $p$ has rank $(a \d j)$, then by \cref{rank-from-rank}, $p(j, 1)$ has rank $a$ and $p$ has rank $j$.
  The latter means that $p \iso p(j, 1) \rf E(j, 1)$, exhibiting $p$ as a right lift of a loose-cell through $E(j, 1)$ with rank $a$.
  In the other direction, consider a right lift $p = q \rf E(j, 1)$.
  Then $p$ has rank $j$ by \cref{rank-and-cocontinuity}, so if $p(j, 1) \iso q \rf E(j, j) \iso q$ has rank $a$, then $p$ has rank $(a \d j)$ by \cref{rank-composition}.
\end{proof}

Since, in some cases, having rank $j$ is equivalent to preserving certain colimits by \cref{rank-and-cocontinuity}, one might expect tight-cells having rank $j$ to enjoy similar properties to tight-cells that preserve certain colimits.
The following lemma supports this intuition by providing an analogue to the fact that left $k$-adjoints preserve all colimits that are preserved by $k$ (\cref{adjoints-preserve-(co)limits}).

\begin{lemma}
  Let $\jAE$ be a tight-cell.
  \begin{enumerate}
    \item If $k \colon E \to K$ has rank $j$, then every left $k$-adjoint has rank $j$.
    \item \label{rank-cancellation} In particular, if $i \colon I \to I'$ is \ff{}, and $(f \d i)$ has rank $j$, then $f$ has rank $j$.
  \end{enumerate}
\end{lemma}

\begin{proof}
  For (1), note that if a loose-cell $p$ has rank $j$, then so does $p(1, r)$ for each tight-cell $r$.
  If $k$ has rank $j$, then $K(k, 1)$ has rank $j$, and hence so does every loose-cell of the form $K(k, r)$.
  In particular, if $\ell$ is the left $k$-adjoint of some $r$, then $C(\ell, 1) \iso K(k, r)$ has rank $j$, so $\ell$ has rank $j$.

  (2) is a special case of (1) because, taking $k \defeq (f \d i)$, the isomorphism $I(f, 1) \iso I'(if, i)$ witnesses $f$ as the left $k$-adjoint of $i$.
\end{proof}

Given that the concept of rank generalises the concept of density (\cref{density-via-rank}), we may specialise the previous results to obtain composition and cancellation results for dense tight-cells.

\begin{corollary}
  Let $\jAE$ and $a \colon A' \to A$ be tight-cells.
  \begin{enumerate}
    \item \label{density-composition} If $j$ has rank $a$, then density of $j$ is equivalent to density of $(a \d j)$ (\cf{}~\cite[Proposition~5.7]{kelly1982basic}).
    \item \label{density-from-rank} If $j$ is \ff{} and has rank $a$, then $a$ is dense.
    \item \label{rank-from-density} If $j$ is \ff{} and $(a \d j)$ is dense, then $j$ has rank $a$, and both $a$ and $j$ are dense (\cf{}~\cites[Lemme~1.0]{diers1976type}[Theorem~5.13]{kelly1982basic}).
  \end{enumerate}
\end{corollary}

\begin{proof}
  (1) is precisely \cref{rank-composition} with $f = 1_E$, while (2) is \cref{rank-cancellation} with $f \defeq 1_A$ and $i \defeq j$ and rank $a$.
  For (3), $j$ having rank $a$ is \cref{rank-from-rank} with $f = 1_{E}$.
  Density of $j$ and $a$ is then immediate from (1) and (2).
\end{proof}

The following consequence complements \cref{ff-weighted-colimit}, which gave conditions for a left extension to be \ff.

\begin{corollary}[{\cf~\cite[Proposition~5.10]{kelly1982basic}}]
  Let $\jAE$ and $f \colon A \to X$ be tight-cells admitting a left extension $j \plx f \colon E \to X$. If $j$ is \ff, then $j \plx f$ is dense if and only if $f$ is dense.
\end{corollary}

\begin{proof}
  Since $j$ is \ff{}, $j \plx f$ has rank $j$ by \cref{rank-and-cocontinuity}. Consequently, by \cref{density-composition}, density of $j \plx f$ is equivalent to density of $j \d (j \plx f) \iso f$ using \cref{lift-through-ff-tight-cell}.
\end{proof}

\section{Cocompleteness and cocompletions in a \ve{}}
\label{cocompleteness-in-a-virtual-equipment}

A \emph{cocompletion} in a \ve{} axiomatises the universal property of a free cocompletion of a category under a class of colimits. When considering free cocompletions in ordinary category theory, one typically asks that every functor $f \colon \A \to \b X$ into a category $\b X$ admitting a certain class $\Phi$ of colimits may be extended to a functor $\tilde f \colon \coc\Phi{\A} \to \b X$ from the free cocompletion, in an essentially unique way. While this universal property certainly plays a role in our treatment, it is important for us to observe that free cocompletions also admit a second universal property, which concerns distributors $p \colon \b X \lto \A$ satisfying an exactness property with respect to $\Phi$. The notion of cocompletion we consider, therefore, is stronger than the 2-categorical universal property that has been considered heretofore in the literature. As we will explain later, this stronger universal property allows us to establish desirable properties of free cocompletions that do not follow from the mere 2-categorical universal property: for instance, the embedding of any object into a cocompletion in our sense is \ff{} (\cref{cocompletion-is-ff}), whereas there are pathological examples of cocompletions in the 2-categorical sense for which this is not the case (\cref{ffness-of-cocompletion-embedding}).

We shall start by introducing the appropriate notion of cocompleteness in our setting (\cref{cocompleteness}), before introducing and studying cocompletions (\cref{cocompletions}).

\subsection{Cocompleteness}
\label{cocompleteness}

The notion of cocompleteness of an object and the cocontinuity of a tight-cell with respect to a class of weights is straightforward, being formulated exactly as in enriched category theory. However, we shall also need a notion of \emph{exactness} of a loose-cell with respect to a class of weights, which plays an analogous role to cocontinuity for tight-cells.

\begin{definition}
  \label{cocomplete}
  Let $\Phi$ be a class of weights.
  \begin{enumerate}
    \item A \emph{$\Phi$-colimit} is a $\vec p$-weighted colimit for some $\vec p \in \Phi$.
    \item An object $A$ is \emph{$\Phi$-cocomplete} if, for every tight-cell $f \colon D \to A$ and weight $\vec p \colon W \lcto D$ in $\Phi$, there exists a colimit $\vec p \wc f \colon W \to A$.
    \item A tight-cell $g \colon A \to B$ with $\Phi$-cocomplete domain is \emph{$\Phi$-cocontinuous} if it preserves every $\Phi$-colimit.
    \item A loose-cell $p \colon X \lto A$ with $\Phi$-cocomplete codomain is \emph{$\Phi$-exact} if it respects every $\Phi$-colimit (in the sense of \cref{respects-colimit}).
    \qedhere
  \end{enumerate}
\end{definition}

In particular, a tight-cell $g \colon A \to B$ with $\Phi$-cocomplete domain is $\Phi$-cocontinuous if and only if its conjoint $B(g, 1) \colon B \lto A$ is $\Phi$-exact.

\subsection{Saturation}

In general, there may be many classes of weights whose notions of cocompleteness coincide: for instance, in the setting of ordinary categories, the existence of all small colimits may be reduced to many different bases, such as:
\begin{itemize}
  \item small coproducts and (reflexive) coequalisers;
  \item an initial object and wide pushouts;
  \item finite coproducts and sifted colimits;
  \item finite colimits and filtered colimits,
\end{itemize}
and so on.
We introduce the following equivalence relation on classes of weights to capture this relationship.

\begin{definition}[label=colimit-equivalence]
  Two classes of weights $\Phi$ and $\Phi'$ are \emph{colimit equivalent}, written $\Phi \colimequiv \Phi'$, if and only if $\Phi \colimleq \Phi'$ and $\Phi' \colimleq \Phi$, where $\Phi \colimleq \Phi'$ holds if and only if
  \begin{enumerate}
    \item every $\Phi'$-cocomplete object is $\Phi$-cocomplete;
    \item every $\Phi'$-cocontinuous tight-cell with $\Phi'$-cocomplete domain is $\Phi$-cocontinuous;
    \item every $\Phi'$-exact loose-cell with $\Phi'$-cocomplete codomain is $\Phi$-exact.
    \qedhere
  \end{enumerate}
\end{definition}

\begin{remark}
  Since a tight-cell $f \colon X \to Y$ is $\Phi$-cocontinuous if and only if the loose-cell $Y(f, 1)$ is $\Phi$-exact, condition (2) above is redundant (and merely included for clarity).
\end{remark}

The relation $\colimleq$ defined in \cref{colimit-equivalence} forms a preorder on the collection of classes of weights.
We shall be interested in maximal elements, in the following sense.

\begin{definition}[label=saturation]
  Let $\Phi$ be a class of weights. The \emph{colimit saturation} of $\Phi$ is the largest class of weights $\Phi\satw$ for which $\Phi\satw \colimequiv \Phi$. A class $\Phi$ is \emph{colimit saturated} if $\Phi = \Phi\satw$.
\end{definition}

\begin{remark}
  For a class $\Phi$ of weights, denote by $\Phi\Coc$ the 2-category of $\Phi$-cocomplete objects.
  Colimit saturation, as considered traditionally~\cite{albert1988closure}, is intended to address the problem of finding the largest class of weights $\Phi\satw \supseteq \Phi$ for which there is a concrete biequivalence of 2-categories $\Phi\Coc \biequiv \Phi\satw\h\b{Coc}$. However, in line with our approach to formal category theory~\cite{arkor2024formal}, we are concerned not with 2-categories, but with virtual equipments. Therefore, the appropriate notion of colimit saturation is one that also induces a correspondence between the appropriate notion of loose-cell between $\Phi$-cocomplete objects, which is precisely the notion of $\Phi$-exact loose-cell. Consequently, our notion of colimit saturation is, a priori, more constrained than the traditional definition, which omits condition (3) in \cref{colimit-equivalence}. Note also that, again in contrast to the traditional definition, we impose no size restrictions on our weights; and do not require in (2) that the \emph{codomain} of a $\Phi'$-cocontinuous tight-cell is $\Phi'$-cocomplete.
\end{remark}

\begin{remark}
  \label{limit-saturation}
  There is a dual notion of \emph{limit saturation}, denoted $\Phi^{{\bulletstar}}$ and induced by an equivalence relation $\limequiv$ arising from a preorder $\limleq$. In the setting of enrichment in a complete and cocomplete braided closed monoidal category, the two notions of saturation coincide (up to taking opposites), and so in this context the term \emph{saturation} is used for both terms unambiguously. However, in general, we must distinguish between the two kinds of saturation.
\end{remark}

We record some basic properties of the saturation of a class of weights, for use in \cref{presheaves-and-cocompletions}.

\begin{lemma}
  \label{saturation-properties}
  $\ph\satw$ defines a closure operator on the power class of weights whose fixed points are the colimit saturated classes of weights. Furthermore, for each class of weights $\Phi$, we have the following.
  \begin{enumerate}
    \item \label{sat-is-replete} $\Phi\satw$ is replete.
    \item \label{sat-contains-representables} $B(1, f) \in \Phi\satw$ for each tight-cell $f \colon A \to B$.
    \item \label{sat-contains-concatenations} Given a composable chain $\vec p, \vec q$, if $\vec p \in \Phi\satw$ and $\vec q \in \Phi\satw$, then the composed chain $(p_1, \ldots, p_m, q_1, \ldots, q_n) \in \Phi\satw$.
    \item \label{sat-contains-left-composites} If $(p_1, \ldots, p_m) \in \Phi\satw$ and a left-composite $p_1 \odotl \cdots \odotl p_i$ exists ($1 \leq i \leq m$), then $(p_1 \odotl \cdots \odotl p_i, p_{i + 1}, \ldots, p_m) \in \Phi\satw$.
    \item If $(p_1, \ldots, p_m) \in \Phi\satw$ and a composite $p_i \odot \cdots \odot p_j$ exists ($1 \leq i \leq j \leq m$), then $(p_1, \ldots, p_{i - 1}, p_i \odot \cdots \odot p_j, p_{j + 1}, \ldots, p_m) \in \Phi\satw$.
    \item If $p \in \Phi\satw$, then $p(1, f) \in \Phi\satw$.
  \end{enumerate}
  Furthermore, $\Phi \colimequiv \Phi'$ if and only if $\Phi\satw = (\Phi')\satw$ if and only if $\Phi \colimleq \Phi' \colimleq \Phi\satw$.
\end{lemma}

\begin{proof}
  That $\ph\satw$ defines such a closure operator is trivial.
  \begin{enumerate}
    \item Trivial, since colimits are essentially unique.
    \item For every tight-cell $g \colon B \to C$, $(f \d g)$ is the $B(1, f)$-weighted colimit of $g$~\cite[Example~3.11]{arkor2024formal}.
    \item Follows from \cref{iterated-lift}.
    \item Follows from \cref{right-lift-through-left-composite}.
    \item Follows from \cref{right-lift-through-composite}.
    \item Follows from (2 \& 3 \& 5) because $p(1, f) \iso p \odot B(1, f)$ (alternatively, by \cite[Lemma~3.13]{arkor2024formal}).
  \end{enumerate}
  The characterisations of colimit equivalence in terms of saturation are trivial.
\end{proof}

\subsection{Cocompletions}
\label{cocompletions}

We are ready to introduce the second concept central to the paper (following the $\PP$-presheaf objects of \cref{Phi-presheaf-object}).

\begin{definition}
  \label{cocompletion}
  Let $\Phi$ be a class of weights and let $A$ be an object. A \emph{$\Phi$-cocompletion} of $A$ comprises an object $\coc\Phi A$ and a tight-cell $\coce\Phi A \colon A \to \coc\Phi A$ such that the following conditions hold.
  \begin{enumerate}
    \item $\coc\Phi A$ is $\Phi$-cocomplete.
    \item \label{cocompletion-tight-cell-UP} For each tight-cell $f \colon A \to X$ with $\Phi$-cocomplete codomain, the left extension ${\coce\Phi A \plx f \colon \coc\Phi A \to X}$ exists, the unit $f \tto \coce\Phi A \d (\coce\Phi A \plx f)$ is invertible, and $\coce\Phi A \plx f$ is the essentially unique $\Phi$-cocontinuous tight-cell $\tilde f \colon \coc\Phi A \to X$ such that $f \iso (\coce\Phi A \d \tilde f)$.
    \item \label{cocompletion-loose-cell-UP} For each loose-cell $p \colon X \lto A$, the right lift ${p \rf \coc{\Phi}{A}(\coce\Phi A, 1) \colon X \lto \coc\Phi A}$ exists, the canonical 2-cell $(p \rf \coc{\Phi}{A}(\coce\Phi A, 1))(\coce\Phi A, 1) \tto p$ is invertible, and $p \rf \coc{\Phi}{A}(\coce\Phi A, 1)$ is the essentially unique $\Phi$-exact loose-cell $\tilde p \colon X \lto \coc\Phi A$ such that $p \iso \tilde p(\coce\Phi A, 1)$.
  \end{enumerate}
  An object is \emph{$\Phi$-cocompletable} if it admits a $\Phi$-cocompletion.
\end{definition}

In particular, for a $\Phi$-cocompletion $\coce\Phi A \colon A \to \coc\Phi A$, the conjoint $\coc{\Phi}{A}(\coce\Phi A, 1) \colon \coc\Phi A \lto A$ is lifting (\cref{lifting}).

\begin{remark}
  \label{cocompletion-alternative-formulation}
  Since a tight-cell between $\Phi$-cocomplete objects is $\Phi$-cocontinuous if and only if its conjoint is $\Phi$-exact, we may replace (2) in \cref{cocompletion} with the following condition.
  \begin{itemize}
    \item If the loose-cell $p \colon X \lto A$ in (3) is corepresentable and $X$ is $\Phi$-cocomplete, then the right lift $p \rf \coc{\Phi}{A}(\coce\Phi A, 1)$ is also corepresentable.
    \qedhere
  \end{itemize}
\end{remark}

As indicated earlier, the notion of cocompletion is invariant under colimit equivalence.

\begin{lemma}
  \label{colimit-equivalent-cocompletions}
  Let $\Phi$ and $\Phi'$ be classes of weights such that $\Phi \colimequiv \Phi'$.
  Then a tight-cell $j \colon A \to E$ exhibits the $\Phi$-cocompletion of $A$ if and only if it exhibits the $\Phi'$-cocompletion of $A$.
\end{lemma}

\begin{proof}
  Trivial.
\end{proof}

Note that, as anticipated at the start of the section, \cref{cocompletion} exhibits a stronger universal property than has traditionally been expressed for free cocompletions, as it admits a universal property with respect to loose-cells. We shall show in \cref{presheaves-and-cocompletions-in-V-Cat} that free cocompletions of enriched categories do indeed satisfy this stronger universal property. This is a typical phenomenon in double category theory: for instance, algebra objects for monads in \ve s satisfy a stronger universal property than algebra objects for monads in 2-categories (see the discussion after \cite[Definition~6.33]{arkor2024formal}). We contend that this stronger universal property is desirable. One indication is the following lemma, for which the stronger universal property is necessary (see \cref{ffness-of-cocompletion-embedding}).

\begin{lemma}
  \label{cocompletion-is-ff}
  Every $\Phi$-cocompletion $\coce\Phi A \colon A \to \coc{\Phi}{A}$ is \ff.
\end{lemma}

\begin{proof}
  By \cref{cocompletion-loose-cell-UP}, there are isomorphisms
  \[
    p \rf \coc{\Phi}{A}(\coce\Phi A, \coce\Phi A)
    \iso
    (p \rf \coc{\Phi}{A}(\coce\Phi A, 1))(\coce\Phi A, 1)
    \iso
    p
    \iso
    p \rf A(1, 1)
  \]
  natural in $p$, hence there is a natural bijection between 2-cells $A(1, 1) \tto p$ and 2-cells $\coc{\Phi}{A}(\coce\Phi A, \coce\Phi A) \tto p$, which implies that $A(1, 1) \iso \coc{\Phi}{A}(\coce\Phi A, \coce\Phi A)$, exhibiting $\coce\Phi A$ as \ff.
\end{proof}

\begin{example}
  \label{ffness-of-cocompletion-embedding}
  One might expect it to be possible to prove \ffness{} of $\coce\Phi A$ using (2) instead of (3) in the proof of \cref{cocompletion-is-ff}.
  However, this is not the case: indeed, it is possible for a tight-cell $\coce\Phi A \colon A \to \coc{\Phi} A$ to satisfy (1) and (2), while not being \ff{}.

  For instance, recall that a locally small category admitting large coproducts is necessarily posetal (\cf~\cite[78]{freyd1964abelian}). Therefore, taking $\Phi$ to be the class of weights for large coproducts, a locally small category is $\Phi$-cocomplete if and only if it is equivalent to a preordered class with large joins. For a small category $\A$, the functor
  \[\A \xto{\yo_{\A}} \Set^{\A\op} \xto{(\exists x \in \ph)^{\A\op}} \b2^{\A\op}\]
  obtained as the preorder reflection of the presheaf embedding therefore satisfies properties (1) and (2) of \cref{cocompletion}~\cite[580]{walker2019distributive}. However, it is not \ff{} in general, and therefore is not a $\Phi$-cocompletion in the sense of \cref{cocompletion}.
\end{example}

\begin{corollary}
  Let $\Phi$ be a class of weights and let $A$ be an object admitting a $\Phi$-cocompletion $\coce\Phi A \colon A \to \coc\Phi A$. Then $\widetilde{A(1, 1)} \iso \coc\Phi A(1, \coce\Phi A)$.
\end{corollary}

\begin{proof}
  We have $\coc\Phi A(1, \coce\Phi A)(\coce\Phi A, 1) \iso \coc\Phi A(\coce\Phi A, \coce\Phi A) \iso A(1, 1)$ since $\coce\Phi A$ is \ff{} by \cref{cocompletion-is-ff}, from which the statement follows by essential uniqueness.
\end{proof}

Cocompletions of categories under classes of colimits are known to satisfy many useful properties~\cite{kelly1982basic,kelly2005notes}. The following results establish that these properties generally continue to hold in the setting of \fct.

The following augments \cref{saturation-properties}, by exhibiting an additional property satisfied by colimit saturation.
\begin{lemma}
  \label{cocompletion-is-in-saturation}
  Let $\Phi$ be a class of weights and suppose that a $\Phi$-cocompletion $\coce\Phi A \colon A \to \coc\Phi A$ exists.
  Then $\coc{\Phi}{A}(\coce\Phi A, 1)$ is in the colimit saturation of $\Phi$.
\end{lemma}

\begin{proof}
  By definition, a $\coc{\Phi}{A}(\coce\Phi A, 1)$-weighted colimit is a left extension along $\coce\Phi A$.
  $\Phi$-cocomplete objects admit these left extensions by definition.
  It remains to show that each $\Phi$-exact loose-cell $p$ respects each left extension $f' = \coce{\Phi}{A} \plx f$ (it then trivially follows that $\Phi$-cocontinuous tight-cells preserve such left extensions).
  To do so, note that $f'$ is $\Phi$-cocontinuous, and hence the restriction $p(f', 1)$ is also $\Phi$-exact.
  The uniqueness clause of \cref{cocompletion-loose-cell-UP} implies that the canonical 2-cell
  $p(f', 1) \tto p(f, 1) \rf \coc{\Phi}{A}(\coce\Phi A, 1)$ is invertible; this is exactly the fact that $p$ respects the left extension.
\end{proof}

\begin{lemma}
  \label{cocompletion-absolute}
  Let $\coce\Phi A \colon A \to \coc{\Phi}{A}$ be a $\Phi$-cocompletion.
  Left extensions along $\coce\Phi A$ and $\Phi$-colimits in $\coc{\Phi}{A}$ are $\coce\Phi A$-absolute.
\end{lemma}

\begin{proof}
  By \cref{absoluteness-and-exactness}, it is enough to show that every right lift
  $p' \defeq p \rf \coc{\Phi}{A}(\coce\Phi A, 1)$ is $\Phi$-exact and respects left extensions along $\coce\Phi A$.
  $\Phi$-exactness is immediate from the definition of $\Phi$-cocompletion (\cref{cocompletion-loose-cell-UP}), and thus implies that $p'$ respects left extensions along $\coce\Phi A$ by \cref{cocompletion-is-in-saturation}.
\end{proof}

\begin{proposition}[{\cf{}~\cite[Proposition~7.9]{lucyshyn2016enriched}}]
  \label{Phi-preservation-is-phi-absolute-preservation}
  Let $\coce\Phi A \colon A \to \coc{\Phi}{A}$ be a $\Phi$-cocompletion.
  The following are equivalent for a tight-cell $f \colon \coc{\Phi}{A} \to X$ into an arbitrary object $X$ (not assumed $\Phi$-cocomplete).
  \begin{enumerate}
    \item $f$ is $\Phi$-cocontinuous.
    \item $f$ has rank $\coce\Phi A$.
    \item $f$ is a left extension along $\coce\Phi A$.
    \item $f$ preserves $\coce\Phi A$-absolute colimits.
  \end{enumerate}
  The following are equivalent for a loose-cell $p \colon X \lto \coc\Phi A$.
  \begin{enumerate}
      \item $p$ is $\Phi$-exact.
      \item $p$ has rank $\coce\Phi A$.
      \item $p$ is a right lift through $\coc\Phi A(\coce\Phi A, 1)$.
      \item $p$ respects $\coce\Phi A$-absolute colimits.
  \end{enumerate}
\end{proposition}

\begin{proof}
  Since $\coce\Phi A$ is \ff{} (\cref{cocompletion-is-ff}), \cref{rank-and-cocontinuity} provides (2) $\iff$ (3) and (3) $\implies$ (4).
  (4) $\implies$ (1) because $\Phi$-colimits are $\coce\Phi A$-absolute (\cref{cocompletion-absolute}).

  We show (1) $\implies$ (3) first for a loose-cell $p$.
  By \cref{cocompletion-loose-cell-UP}, the right lift $p(\coce\Phi A, 1) \rf \coc{\Phi}{A}(\coce\Phi A, 1)$ is the unique $\Phi$-exact $\widetilde{p(\coce\Phi A, 1)} \colon X \lto \coc{\Phi}{A}$ such that $p(\coce\Phi A, 1) \iso \widetilde{p(\coce\Phi A, 1)}(\coce\Phi A, 1)$.
  But $p$ is also a candidate for $\widetilde{p(\coce\Phi A, 1)}$, so there exists an isomorphism $p \iso p(\coce\Phi A, 1) \rf \coc{\Phi}{A}(\coce\Phi A, 1)$, exhibiting $p$ as a right lift through $\coc{\Phi}{A}(\coce\Phi A, 1)$.
  To show that (1) $\implies$ (3) for a tight-cell $f$, take $p = X(f, 1)$ to obtain $X(f, 1) \iso X(f\coce\Phi A, 1) \rf \coc{\Phi}{A}(\coce\Phi A, 1)$.
  This isomorphism exhibits $f$ as the left extension $\coce\Phi A \plx (\coce\Phi A \d f)$.
\end{proof}

\begin{corollary}
  \label{cocompletion-is-dense}
  Let $\Phi$ be a class of weights. Every $\Phi$-cocompletion $\coce\Phi A \colon A \to \coc\Phi A$ is dense.
\end{corollary}

\begin{proof}
  The identity on $\coc{\Phi}{A}$ is $\Phi$-cocontinuous, and therefore has rank $\coce\Phi A$ by \cref{Phi-preservation-is-phi-absolute-preservation}.
  The latter property is exactly density of $\coce\Phi A$ (\cref{density-via-rank}).
\end{proof}

\begin{remark}
  From \cref{cocompletion-is-dense,cocompletion-is-in-saturation}, it follows that we could reformulate the definition of $\Phi$-cocompletion (\cref{cocompletion}) by additionally asking for $\coce\Phi A$ to be dense and for $\coc\Phi A(\coce\Phi A, 1)$ to be in the colimit saturation of $\Phi$. It would then follow that any $\Phi$-cocontinuous tight-cell $\tilde f \colon \coc\Phi A \to X$ satisfying $\coce\Phi A \d \tilde f \iso f$ would automatically be a left extension, because, under these assumptions, the identity on $\coc\Phi A$ is the left extension $\coce\Phi A \plx \coce\Phi A$, which is preserved by any $\Phi$-cocontinuous tight-cell; an analogous argument follows for $\Phi$-exact loose-cells. This would permit conditions (2 \& 3) to be simplified, at the expense of the additional assumptions regarding density and saturation. We do not spell this out, as a similar but more useful reformulation will be given in \cref{intrinsic-characterisation-of-cocompletions}.
\end{remark}

We may use all of these properties to give an alternative characterisation of cocompletions.

\begin{theorem}
  \label{intrinsic-characterisation-of-cocompletions}
  Let $\Phi$ be a class of weights. A tight-cell $j \colon A \to E$ exhibits the $\Phi$-cocompletion of $A$ if and only if
  \begin{enumerate}
    \item $E$ is $\Phi$-cocomplete;
    \item $E(j, 1)$ is in the colimit saturation of $\Phi$;
    \item $E(j, 1)$ is lifting, and right lifts through $E(j, 1)$ are $\Phi$-exact;
    \item $j$ is dense and \ff{}.
  \end{enumerate}
\end{theorem}

\begin{proof}
  Assume that $j$ exhibits the $\Phi$-cocompletion of $A$. (1) is by definition; (2) is \cref{cocompletion-is-in-saturation}; (3) is by definition; (4) is \cref{cocompletion-is-dense,cocompletion-is-ff}.

  Conversely, supposing that (1 -- 4) hold, we shall verify the assumptions of \cref{cocompletion}. (1) is by assumption. For (3), we have existence and $\Phi$-exactness of the requisite right lifts by assumption, and that $(p \rf E(j, 1))(j, 1) \tto p$ is invertible follows from \ffness{} of $j$ by \cref{lift-through-ff-tight-cell}. Suppose that there is a $\Phi$-exact loose-cell $\tilde p \colon X \lto E$ such that $p \iso \tilde p(j, 1)$.
  We have
  \[\tilde p \iso \tilde p(j \plx j, 1) \iso \tilde p(j, 1) \rf E(j, 1) \iso p \rf E(j, 1)\]
  using density of $j$ and $\Phi$-exactness, using that $E(j, 1)$ is in the colimit saturation of $\Phi$, which establishes essential uniqueness.

  Finally, for (2), observe that, since $E(j, 1)$ is in the colimit saturation of $\Phi$, given a tight-cell $f \colon A \to X$ into a $\Phi$-cocomplete object, the left extension $j \plx f \colon E \to X$ exists and corepresents $X(f, 1) \rf E(j, 1)$. This suffices to establish (2), as observed in \cref{cocompletion-alternative-formulation}.
\end{proof}

A given tight-cell may exhibit cocompletions for different classes of weights, which may not be colimit equivalent. An elementary example is the functor $1 \colon \b 1 \to \Set$ picking out a singleton set, which exhibits the cocompletion of the terminal category under either (a) small coproducts; or (b) small colimits. However, for each cocompletion $j \colon A \to E$, there is a canonical class of weights for which $j$ exhibits a cocompletion: namely, the (singleton) class of weights for left extensions along $j$.

\begin{corollary}
  \label{cocompletion-for-smaller-class-of-weights}
  Let $\Phi$ be a class of weights and let $j \colon A \to E$ be a tight-cell exhibiting the {$\Phi$-cocompletion} of $A$. Then, for each class of weights $\Phi'$ such that $\{ E(j, 1) \} \colimleq \Phi' \colimleq \Phi$, the tight-cell $j$ exhibits the $\Phi'$-cocompletion of $A$. In particular, this holds when $\Phi' \defeq \{ E(j, 1) \}$.
\end{corollary}

\begin{proof}
  Assuming each of the conditions of \cref{intrinsic-characterisation-of-cocompletions} for $\Phi$, the conditions for $\Phi'$ are trivially verified.
\end{proof}

The reader will note that the conditions imposed by \cref{intrinsic-characterisation-of-cocompletions} on a tight-cell $j \colon A \to E$ are similar to those of well-behavedness (\cref{well-behaved}). Indeed, it turns out that every cocompletion is well-behaved.

\begin{proposition}
  \label{cocompletions-are-well-behaved}
  Let $\Phi$ be a class of weights. Every $\Phi$-cocompletion $\coce\Phi A \colon A \to \coc\Phi A$ is well-behaved in the sense of \cref{well-behaved}.
\end{proposition}

\begin{proof}
  $\coce\Phi A$ is \ff{} by \cref{cocompletion-is-ff}, is dense by \cref{cocompletion-is-dense}, admits left extensions of tight-cells $A \to \coc\Phi A$ along $\coce\Phi A$ by definition of the $\Phi$-cocompletion, since $\coc\Phi A$ is $\Phi$-cocomplete. These left extensions are $\coce\Phi A$-absolute by \cref{cocompletion-absolute}.
\end{proof}

\begin{remark}
  \Cref{cocompletions-are-well-behaved} almost establishes that every $\Phi$-cocompletion $\coce\Phi A \colon A \to \coc\Phi A$ exhibits a $\{ \coc\Phi A(\coce\Phi A, f) \mid f \colon X \to \coc\Phi A \}$-presheaf object: the only missing condition is monicity of $\coc\Phi A(\coce\Phi A, 1)$. In particular, assuming such a presheaf object exists, it will be equivalent to $\coc\Phi A$ (\cref{well-behaved-is-presheaf-object}). However, in practice, we wish to construct cocompletions from presheaf objects, rather than conversely (see \cref{presheaves-and-cocompletions}).
\end{remark}

\begin{remark}
  \textcite{walker2018yoneda} shows that every locally \ff{} lax-idempotent pseudomonad on a 2-category induces a Yoneda structure on the same 2-category. Conceptually, this expresses that free cocompletions (in the 2-categorical sense, \ie{} satisfying merely conditions (1 \& 2) of \cref{cocompletion}) have the abstract structure of presheaf constructions. \Cref{cocompletions-are-well-behaved} may be seen as an analogue of this result in the setting of \ve s.
\end{remark}

Conversely, being well-behaved is sufficient to be a cocompletion under some class of weights, assuming the existence of enough right lifts.

\begin{corollary}
  \label{cocompletion-via-well-behavedness}
  A tight-cell $j \colon A \to E$ exhibits the $\{ E(j, 1) \}$-cocompletion of $A$ if and only if it is well-behaved and $E(j, 1)$ is lifting.
\end{corollary}

\begin{proof}
  We invoke \cref{intrinsic-characterisation-of-cocompletions} with $\Phi \defeq \{ E(j, 1) \}$. If the cocompletion exists, then it is well-behaved by \cref{cocompletions-are-well-behaved} and satisfies the existence of right lifts by condition (3). Conversely, given that $j$ is well-behaved, $E$ admits left extensions along $j$, so that (1) holds. (2) is trivial. The existence of right lifts in (3) is by assumption; these are automatically $\{ E(j, 1) \}$-exact when they exist by \cref{rank-and-cocontinuity}. (4) is by definition of being well-behaved.
\end{proof}

In particular, using \cref{cocompletion-for-smaller-class-of-weights}, we may reformulate the characterisation of $\Phi$-cocompletions in \cref{intrinsic-characterisation-of-cocompletions} in terms of well-behavedness.

\begin{corollary}
  \label{Phi-cocompletion-via-well-behavedness}
  Let $\Phi$ be a class of weights. A tight-cell $j \colon A \to E$ exhibits the $\Phi$-cocompletion of $A$ if and only if
  \begin{enumerate}
    \item $j$ satisfies the assumptions of \cref{cocompletion-via-well-behavedness}: \ie{} it is well-behaved and $E(j, 1)$ is lifting;
    \item $E(j, 1)$ is in the colimit saturation of $\Phi$;
    \item $E$ admits $j$-absolute $\Phi$-colimits.
  \end{enumerate}
\end{corollary}

\begin{proof}
  We verify the conditions of \cref{intrinsic-characterisation-of-cocompletions}. (1, 2 \& 3) follow by assumption, as does the existence of right lifts through $E(j, 1)$. $\Phi$-exactness follows from \cref{rank-and-cocontinuity}, since the $\Phi$-colimits in $E$ are $j$-absolute. Finally, $\Phi$-colimits in a $\Phi$-cocompletion $j$ are $j$-absolute by \cref{cocompletion-absolute}.
\end{proof}

\begin{remark}
  \label{well-behavedness-is-cocompletion}
  The conditions of well-behavedness (\cref{well-behaved}) correspond to the \emph{well-behavedness} conditions of \cites[Definition~4]{altenkirch2010monads}[Definition~4.1]{altenkirch2015monads} and the \emph{eleuthericity} conditions of \cite[\S7.3 \& \S7.4]{lucyshyn2016enriched}. Strictly speaking, \cref{well-behaved-is-presheaf-object} establishes that well-behavedness characterises presheaf objects, rather than cocompletions. However, assuming the existence of enough right lifts, \cref{cocompletion-via-well-behavedness} shows that well-behavedness equivalently characterises free cocompletions for some class of weights: we state this explicitly for enriched categories in \cref{well-behaved-is-enriched-cocompletion}. This makes good on our promise in our earlier work \cite[Remark~4.31]{arkor2024formal} to establish the connection between well-behavedness and free cocompletions in formal category theory.
\end{remark}

\subsection{Characterisations of cocompleteness}

The existence of $\Phi$-cocompletions provides various convenient methods to test an object for $\Phi$-cocompleteness.

\begin{proposition}[{\cf{}~\cites[Proposition~2.2]{garner2012lex}[Proposition~4.33]{kelly1982basic}}]
  \label{characterisations-of-cocompleteness}
  Let $\Phi$ be a class of weights.
  The following are equivalent for an object $A$.
  \begin{enumerate}
      \item $A$ is $\Phi$-cocomplete.
      \item \label{reflectivity-characterisation} $A$ is reflective in a $\Phi$-cocomplete object.
  \end{enumerate}
  If $A$ admits a $\Phi$-cocompletion $\coce\Phi A \colon A \to \coc\Phi A$ then the following are also equivalent.
  \begin{enumerate}[resume]
    \item For every object $X$ admitting a $\Phi$-cocompletion $\coce\Phi X \colon X \to \coc\Phi X$ and each $g \colon X \to A$, the left extension $\coce\Phi X \plx g \colon \coc\Phi X \to A$ exists.
    \item The left extension $\coce\Phi A \plx 1_A \colon \coc\Phi A \to A$ exists.
    \item For every weight $p \in \Phi\satw$, $A$ admits the colimit $p \wc 1_A$.
    \item $\coce\Phi A \colon A \to \coc\Phi A$ admits a left adjoint.
    \item $A$ is reflective in a $\Phi$-cocompletion.
  \end{enumerate}
\end{proposition}

\begin{proof}
  The trivial implications are (1) $\implies$ (2 \& 3 \& 5) and (3) $\implies$ (4) and (6) $\implies$ (7) $\implies$ (2). Then, (2) $\implies$ (1) by \cref{reflected-colimits}, and (5) $\implies$ (4) by \cref{cocompletion-is-in-saturation}. Finally, for (4) $\implies$ (6), observe that the left extension $\coce\Phi A \plx 1_A$ is the required left adjoint.
  \begin{align*}
    A(\coce\Phi A \plx 1_A, 1)
    ~\iso~&
    A(1, 1) \rf \coc{\Phi}{A}(\coce\Phi A, 1)
    \\~\iso~&\tag{$\coce\Phi A$ is \ff{} (\cref{cocompletion-is-ff})}
    \coc{\Phi}{A}(\coce\Phi A, \coce\Phi A) \rf \coc{\Phi}{A}(\coce\Phi A, 1)
    \\~\iso~&\tag{$\coce\Phi A$ is dense (\cref{cocompletion-is-dense})}
    \coc{\Phi}{A}(1, \coce\Phi A)
  \end{align*}
\end{proof}

For each characterisation of cocompleteness, there is a corresponding characterisation of cocontinuity.

\begin{corollary}[label=characterisations-of-cocontinuity]
  Let $\Phi$ be a class of weights and let $A$ and $A'$ be $\Phi$-cocomplete objects. The following are equivalent for a tight-cell $a \colon A \to A'$.
  \begin{enumerate}
    \item $a$ is $\Phi$-cocontinuous.
    \item There exists a pseudocommutative square of tight-cells as on the left below,
    where the vertical tight-cells exhibit reflective inclusions into $\Phi$-cocomplete objects (as in \cref{reflectivity-characterisation}),
    $b$ is $\Phi$-cocontinuous,
    and the mate $(b \d \ell') \tto (\ell \d a)$ is invertible.
    \[
    \begin{tikzcd}
      B & {B'} \\
      A & {A'}
      \arrow["b", from=1-1, to=1-2]
      \arrow[""{name=0, anchor=center, inner sep=0}, "r", hook, from=2-1, to=1-1]
      \arrow["a"', from=2-1, to=2-2]
      \arrow[""{name=1, anchor=center, inner sep=0}, "{r'}"', hook, from=2-2, to=1-2]
      \arrow["\iso"{description}, draw=none, from=0, to=1]
    \end{tikzcd}
    \hspace{6em}
    \begin{tikzcd}[ampersand replacement=\&]
    	B \& {B'} \\
    	A \& {A'}
    	\arrow["b", from=1-1, to=1-2]
    	\arrow["\ell"', from=1-1, to=2-1]
    	\arrow["\iso"{description}, between={0.3}{0.7}, Rightarrow, nfold, from=1-2, to=2-1]
    	\arrow["{\ell'}", from=1-2, to=2-2]
    	\arrow["a"', from=2-1, to=2-2]
    \end{tikzcd}
    \]
  \end{enumerate}
  If $A$ admits a $\Phi$-cocompletion $\coce\Phi A \colon A \to \coc\Phi A$ then the following are also equivalent.
  \begin{enumerate}[resume]
    \item For every object $X$ admitting a $\Phi$-cocompletion $\coce{\Phi} X \colon X \to \coc\Phi X$ and each $g \colon X \to A$, the left extension $\coce{\Phi} X \plx g \colon \coc\Phi X \to A$ is preserved by $a$.
    \item The left extension $\coce\Phi A \plx 1_A$ is preserved by $a$.
    \item For every weight $p \in \Phi\satw$, the weighted colimit $p \wc 1_A$ is preserved by $a$.
  \end{enumerate}
  If furthermore $A'$ admits a $\Phi$-cocompletion $\coce\Phi {A'} \colon A' \to \coc\Phi {A'}$ then the following are also equivalent.
  \begin{enumerate}[resume]
    \item As in (2), but where $r = \coce\Phi A$, $r' = \coce\Phi {A'}$, and $b = \coc\Phi a \defeq \coce\Phi A \plx (a \d \coce\Phi {A'})$.
    \item As in (2), but where $B$ and $B'$ are $\Phi$-cocompletions and $b \iso r \plx (a \d r')$.
  \end{enumerate}
\end{corollary}

\begin{proof}
  (1) $\implies$ (2) is trivial (take $r$ and $r'$ to be identities).
  (2) $\implies$ (1) follows from \cref{reflected-colimits}.

  (1) $\implies$ (5) because $\Phi$ and $\Phi\satw$ are colimit equivalent.
  (1) $\implies$ (3) similarly, using the fact that the left extension $\coce{\Phi}{X} \plx g$ is the $\coc\Phi X(\coce{\Phi}{X}, 1)$-weighted colimit of $g$, and this weight belongs to $\Phi\satw$ by \cref{cocompletion-is-in-saturation}.
  (3) $\implies$ (4) trivially.
  To show (4) $\implies$ (2), take $r = \coce\Phi A$,
  which is a reflection with left adjoint $\ell \defeq \coce\Phi A \plx 1_A$ by \cref{characterisations-of-cocompleteness},
  and take $r' = \ell' = 1_{A'}$.
  The left extension $b \defeq \coce{\Phi}{A} \plx a$ exists, is $\Phi$-cocontinuous, and satisfies
  $(a \d r') \iso (r \d b)$,
  by the definition of $\Phi$-cocompletion.
  The mate of the latter isomorphism is the canonical 2-cell $(\coce\Phi A \plx a) \tto ((\coce{\Phi}{A} \plx 1_A) \d a)$, so invertibility of the mate is exactly (4).
  (5) $\implies$ (4) follows from \cref{cocompletion-is-in-saturation}.

  To show (3) $\implies$ (6), take $\ell = \coce\Phi A \plx 1_A$ and $\ell' = \coce\Phi {A'} \plx 1_{A'}$.
  The left extension $b = \coce\Phi{A} \plx (a \d \coce\Phi {A'})$ exists, is $\Phi$-cocontinuous, and
  satisfies $(r \d b) \iso (a \d r')$.
  Moreover, since $\ell'$ is a left adjoint,
  $(b \d \ell')$ is also $\Phi$-cocontinuous and its restriction along $\coce\Phi A$ is isomorphic to $a$,
  so we have $b \d \ell' \iso \coce\Phi A \plx a$.
  The mate $b \d \ell' \tto \ell \d a$ is therefore the canonical 2-cell $\coce\Phi A \plx a \tto (\ell \d a)$, which is invertible by (3).
  Finally, (6) $\implies$ (7) $\implies$ (2) trivially.
\end{proof}

\subsection{Adjointness theorems}

Cocompletions satisfy a particularly convenient test for adjointness.

\begin{proposition}[{\cf~\cite[Theorem~5.33]{kelly1982basic}}]
  \label{right-adjoint-from-cocomplete-object}
  Let $\Phi$ be a class of weights and let $\jAE$ be a dense tight-cell with $\Phi$-cocomplete codomain.
  Suppose that $f \colon E \to X$ is a $\Phi$-cocontinuous tight-cell for which $X(fj, 1)$ and $E(j, 1)$ are both in $\Phi\satw$.
  Then $f$ admits a right adjoint if and only if the left extension $(fj) \plx j$ exists,
  in which case we have $f \adj (fj) \plx j$.
\end{proposition}

\begin{proof}
  We have $(fj) \plx j \iso f \plx (j \plx j) \iso f \plx 1$ since $j$ is dense, using \cite[Lemma~3.16]{arkor2024formal}.
  Thus, by \cref{adjointness-via-absolute-left-extension}, it is enough to show that the left extension $f \plx 1$ is preserved by $f$.
  Both $(fj) \plx j$ and $j \plx j$ are preserved by $f$ because the relevant weights are in $\Phi\satw$.
  This implies preservation of $f \plx 1$ by the following, where we again use density of $j$ and \cite[Lemma~3.16]{arkor2024formal}.
  \begin{align*}
    f(f \plx 1)
    &\iso f((fj) \plx j)
    \\&\iso
    (fj) \plx (fj)
    \\&\iso
    f \plx (j \plx (f j))
    \\&\iso
    f \plx (f (j \plx j))
    \\&\iso
    f \plx f
    \qedhere
  \end{align*}
\end{proof}

\begin{example}
  Let $\E$ be a cocomplete category with a small dense subcategory $\A \ffto \E$ (for instance, any locally presentable category).
  Then every cocontinuous functor $\E \to \b X$ (where $\b X$ is not necessarily cocomplete) admits a right adjoint.
\end{example}

\begin{corollary}
  \label{right-adjoint-from-free-cocompletion}
  Let $g \colon A \to X$ be a tight-cell, where $X$ is $\Phi$-cocomplete, and assume that $A$ admits a $\Phi$-cocompletion $\coce{\Phi}{A} \colon A \to \coc{\Phi}{A}$.
  Also assume that $X(g, 1) \in \Phi\satw$.
  The induced left extension $\coce{\Phi}{A} \plx g \iso \tilde g \colon \coc\Phi A \to X$ admits a right adjoint if and only if the left extension $g \plx \coce{\Phi}{A}$ exists.
\end{corollary}

\begin{proof}
  Instantiate \cref{right-adjoint-from-cocomplete-object} with $j = \coce\Phi A$ and $f = \tilde g$, using \cref{cocompletion-is-in-saturation} to prove that $\coc\Phi A(j, 1) \in \Phi\satw$.
\end{proof}

\subsection{Recognition of cocompletions}

Under the assumption that a $\Phi$-cocompletion exists, we may recognise when an object is equivalent to the $\Phi$-cocompletion. (While the characterisation theorems such as \cref{intrinsic-characterisation-of-cocompletions} also serve this purpose, knowing that the cocompletion exists permits a simpler set of conditions for recognition.)

\begin{proposition}[{\cf~\cites[Theorem~5.26]{kelly1982basic}[Proposition~4.2]{kelly2005notes}}]
  \label{cocompletion-recognition}
  Let $\Phi$ be a class of weights, let $\coce{\Phi}{A} \colon A \to \coc{\Phi}{A}$ be a $\Phi$-cocompletion, and let $E$ be an object. The following are equivalent.
  \begin{enumerate}
    \item $E \equiv \coc\Phi A$.
    \item There exists a dense and \ff{} $\Phi$-atom $\jAE$ with $\Phi$-cocomplete codomain, for which $E(j, 1)$ is lifting, and $E(j, 1) \in \Phi\satw$.
  \end{enumerate}
\end{proposition}
\begin{proof}
  For (1) $\implies$ (2), take $j$ to be $\coce\Phi A$, which satisfies the assumptions by \cref{cocompletion-is-ff,cocompletion-is-dense,cocompletion-absolute,cocompletion-is-in-saturation}. For (2) $\implies$ (1), the left extension $j \plx \coce\Phi A$ exists because $E(j, 1) \in \Phi\satw$ and $\coc{\Phi}{A}$ is $\Phi$-cocomplete.
  By \cref{right-adjoint-from-free-cocompletion}, we have $\coce\Phi A \plx j \adj j \plx \coce\Phi A$. Since $j$ is a $\Phi$-atom, every right lift through $E(j, 1)$ is $\Phi\satw$-exact, and so every $\Phi\satw$-colimit is $j$-absolute by \cref{absoluteness-and-exactness}.
  In particular, this applies to the left extension $\coce{\Phi}{A} \plx j$, because $\coc{\Phi}{A}(\coce{\Phi}{A}, 1) \in \Phi\satw$ (\cref{cocompletion-is-in-saturation}). Moreover, $\coce{\Phi}{A}$ is also dense (\cref{cocompletion-is-dense}) and \ff{} (\cref{cocompletion-is-ff}) so \cref{nerve--realisation-equivalence} implies that $j \plx \coce{\Phi}{A}$ is an equivalence.
\end{proof}

\subsection{Monads relative to cocompletions}

We conclude our study of the abstract properties of cocompletions by investigating monads relative to cocompletions with respect to a class of weights, continuing our study of the formal theory of relative monads in previous work~\cite{arkor2024formal,arkor2024relative,arkor2025nerve}. Relative monads are generalisations of monads from endofunctors to arbitrary functors. Just as every adjunction induces a monad, every $j$-relative adjunction (\cref{relative-adjunction}) induces a \emph{$j$-relative monad}~\cite[Theorem~5.24]{arkor2024formal}. As we shall not need the details for what follows, and we shall not make further use of relative monads in the following sections, we refer to \cite[Definition~4.1]{arkor2024formal} for the definition of relative monads in a \ve.

Our objective in this subsection is to show that monads relative to $\Phi$-cocompletions may be characterised as $\Phi$-cocontinuous monads on $\Phi$-cocompletions (\cref{cocontinuous-monads-are-relative-monads}), thus resolving a claim in \cite[Remark~4.31]{arkor2024formal}. In doing so, we also establish a connection between the notion of cocompletion (\cref{cocompletion}) and the notion of lax-idempotent relative pseudomonad~\cite{fiore2018relative,arkor2025bicategories}.

First, we recall that every \ve{} has an underlying 2-category.

\begin{definition}[{\cite[Proposition~6.1]{cruttwell2010unified}}]
	Let $\X$ be a \vdc{} admitting loose-identities. Denote by $\tX$ the \emph{tight 2-category} associated to $\X$, having
	\begin{enumerate}
		\item objects: those of $\X$;
		\item 1-cells: tight-cells in $\X$;
		\item 2-cells $\phi \colon f \tto g$: nullary 2-cells with loose-identity codomain in $\X$ as follows.
		\[\begin{tikzcd}
			A & A \\
			B & B
			\arrow[""{name=0, anchor=center, inner sep=0}, "g", from=1-2, to=2-2]
			\arrow[""{name=1, anchor=center, inner sep=0}, "f"', from=1-1, to=2-1]
			\arrow["\shortmid"{marking}, Rightarrow, no head, nfold, from=2-2, to=2-1]
			\arrow[Rightarrow, no head, nfold, from=1-2, to=1-1]
			\arrow["\phi"{description}, draw=none, from=0, to=1]
		\end{tikzcd}\qedshift\]
	\end{enumerate}
\end{definition}

\emph{Relative pseudoadjunctions} and \emph{relative pseudomonads} are bicategorical generalisations of relative adjunctions and relative monads. Since the precise definitions are not necessary for what follows, we do not include them, but direct the reader to \cite[Definition~3.6 \& Definition~3.1]{fiore2018relative}. A relative pseudomonad is \emph{lax-idempotent} if its structure is witnessed by (nonpointwise) left extensions in a suitable sense~\cite[Proposition~5.4]{arkor2025bicategories}.

\begin{proposition}
  \label{induced-relative-pseudomonad}
  Let $\Phi$ be a class of weights. Denote by $J_\Phi \colon \tX^\Phi \ffto \tX$ the inclusion of the full sub-2-category spanned by $\Phi$-cocompletable objects (\cref{cocompletion}), and by $U_\Phi \colon \Phi\Coc \to \tX$ the locally full sub-2-category spanned by $\Phi$-cocomplete objects and $\Phi$-cocontinuous 1-cells (\cref{cocomplete}). The assignment $A \mapsto (\coce\Phi A \colon A \to \coc\Phi A)$, sending each object in $\tX^\Phi$ to its $\Phi$-cocompletion, defines the unit of a $J_\Phi$-relative pseudoadjunction, which induces a lax-idempotent $J_\Phi$-relative pseudomonad $\o\Phi$.
  \[\begin{tikzcd}
    & {\Phi\Coc} & \\
    {\tX^\Phi} && \tX
    \arrow[""{name=0, anchor=center, inner sep=0}, "{U_\Phi}", from=1-2, to=2-3]
    \arrow[""{name=1, anchor=center, inner sep=0}, "{\coc\Phi{-}}", from=2-1, to=1-2]
    \arrow["{J_\Phi}"', from=2-1, to=2-3]
    \arrow["\dashv"{anchor=center, rotate=-1}, shift right=2, draw=none, from=1, to=0]
  \end{tikzcd}\]
\end{proposition}

\begin{proof}
  To establish existence of the $J_\Phi$-relative pseudoadjunction, we must show that, for each $\Phi$-cocompletable object $A$, and each $\Phi$-cocomplete object $X$, the functor $U_\Phi\ph \c \coce\Phi A \colon \Phi\Coc(\coc\Phi A, X) \to \tX(J_\Phi A, U_\Phi X)$ is an equivalence. Essential surjectivity follows from \cref{cocompletion-tight-cell-UP}, while \ffness{} follows from the universal property of the left extension therein. Hence, by \cite[Theorem~3.8]{fiore2018relative}, the $J_\Phi$-relative pseudoadjunction induces a $J_\Phi$-relative pseudomonad $\o\Phi$. For lax-idempotence, we must show that, for each pair $A$ and $B$ of objects admitting $\Phi$-cocompletions, the assignment $(f \colon A \to \coc\Phi B) \mapsto (\tilde f \colon \coc\Phi A \to \coc\Phi B)$ is left-adjoint to precomposition by $\coce\Phi A$. This is immediate, since $\tilde f \iso \coce\Phi A \plx f$, which in particular is a nonpointwise left extension~\cite[Lemma~3.18]{arkor2024formal}.
\end{proof}

In particular, the free $\o\Phi$-pseudo algebras are precisely the $\Phi$-cocompletions. In the following, we refer to \cite[Theorem~4.1]{fiore2018relative} for the definition of the Kleisli bicategory associated to a relative pseudomonad, and to \cite[Theorem~3.9]{arkor2025bicategories} for the definition of the 2-category of pseudo algebras.

\begin{theorem}[label=formal-Eilenberg--Watts]
  Let $\Phi$ be a class of weights, inducing a relative pseudomonad $\o\Phi$ by \cref{induced-relative-pseudomonad}. The embedding of the Kleisli bicategory of $\o\Phi$ into the 2-category of $\o\Phi$-pseudo algebras witnesses a biequivalence between
  \begin{itemize}
    \item the bicategory $\Kl(\o\Phi)$, whose objects are the $\Phi$-cocompletable objects and whose hom-categories are given by $\tX(A, \coc\Phi B)$ for each $A$ and $B$;
    \item the full sub-2-category of $\Phi\Coc$ spanned by the $\Phi$-cocompletions.
  \end{itemize}
\end{theorem}

\begin{proof}
  Immediate from \cite[Corollary~6.7]{arkor2025bicategories}, which establishes that the Kleisli bicategory for a relative pseudomonad is biequivalent to the full image of the left pseudoadjoint of any relative pseudoadjunction inducing the relative pseudomonad.
\end{proof}

We will instantiate \cref{formal-Eilenberg--Watts} in the setting of enriched categories in \cref{generalised-EW-theorem}, thereby substantially generalising the classical Eilenberg--Watts theorem.

When \emph{every} object admits a $\Phi$-cocompletion, the $\o\Phi$-pseudo algebras are precisely the $\Phi$-cocomplete objects, and the $\o\Phi$-pseudo morphisms precisely the $\Phi$-cocontinuous tight-cells, thus augmenting \cref{characterisations-of-cocompleteness,characterisations-of-cocontinuity} with an additional characterisation.

\begin{proposition}[label=algebras-for-pseudomonad]
  Let $\Phi$ be a class of weights for which every object is $\Phi$-cocompletable (\cref{cocompletion}). Then the pseudoadjunction of \cref{induced-relative-pseudomonad} is pseudomonadic, exhibiting $\Phi\Coc$ as the 2-category of $\o\Phi$-pseudo algebras, pseudo morphisms, and transformations for a lax-idempotent pseudomonad $\o\Phi$ on the 2-category $\tX$.
\end{proposition}

\begin{proof}
  Since the pseudomonad is lax-idempotent, $\o\Phi$-pseudo algebra structure is essentially unique: an object $A$ admits the structure of a $\o\Phi$-pseudo algebra if and only if $\coce\Phi A \colon A \to \coc\Phi A$ admits a left adjoint~\cite[Theorem~10.7]{marmolejo1997doctrines}. By \cref{characterisations-of-cocompleteness}, this condition characterises $\Phi$-cocompleteness. Similarly, a 1-cell $f \colon A \to B$ admits the structure of a pseudo morphism if and only if the canonical 2-cell below is invertible.
  \[\begin{tikzcd}[ampersand replacement=\&]
  	{\coc\Phi A} \& {\coc\Phi B} \\
  	A \& B
  	\arrow["{\coc\Phi f}", from=1-1, to=1-2]
  	\arrow["{\coce\Phi A \plx 1_A}"', from=1-1, to=2-1]
  	\arrow[between={0.3}{0.7}, Rightarrow, nfold, from=1-2, to=2-1]
  	\arrow["{\coce\Phi B \plx 1_B}", from=1-2, to=2-2]
  	\arrow["f"', from=2-1, to=2-2]
  \end{tikzcd}\]
  By \cref{characterisations-of-cocontinuity}, this condition characterises $\Phi$-cocontinuity. Finally, every 2-cell is a transformation.
\end{proof}

\begin{remark}
  Our proof of \cref{algebras-for-pseudomonad} depends crucially on the characterisation of pseudo algebras for lax-idempotent pseudomonads in terms of adjointness, a characterisation that does not extend to pseudo algebras for lax-idempotent relative pseudomonads~\cite[\S5]{arkor2025bicategories}. Therefore, we do not know under what conditions the relative pseudoadjunction of \cref{induced-relative-pseudomonad} is relatively pseudomonadic in general (\ie{} when $\Phi\Coc$ is biequivalent to the 2-category of pseudo algebras, pseudo morphisms, and transformations for the relative pseudomonad $\o\Phi$).

  In the setting of ordinary categories and conical colimits, the pseudo algebras for relative pseudomonads exhibiting free cocompletions under certain classes of weights are exhibited in \cite[Theorem~7.14]{arkor2025bicategories} as categories admitting the appropriate class of colimits. However, as remarked in \cite[Remark~7.20]{arkor2025bicategories}, the proof technique does not readily extend to enriched categories. The issue is that the structure of a $\o\Phi$-pseudo algebra $B$ is specified in terms of providing, for each tight-cell $f \colon A \to B$ with $\Phi$-cocompletable domain, a tight-cell $f^\ddag \colon \coc\Phi A \to B$ satisfying some conditions~\cite[Definition~3.1]{arkor2025bicategories}. However, it is not clear that $f^\ddag$ is necessarily a left extension along $\coce\Phi A \colon A \to \coc\Phi A$, which would be necessary for this to imply that $B$ is $\Phi$-cocomplete (\cref{Phi-preservation-is-phi-absolute-preservation}).
\end{remark}

Finally, we remark upon a connection between the Kleisli bicategory associated to the relative pseudomonad $\o\Phi$ and the construction of a skew-multicategory $\X[j]$ given a tight-cell $j \colon A \to E$ in \cite[Definition~4.15]{arkor2024formal}.

\begin{lemma}[{\cf~\cite[Lemma~8.1]{arkor2025bicategories}}]
  \label{skew-multicategory-is-monoidal}
  Let $\Phi$ be a class of weights and let $A$ be an object admitting a $\Phi$-cocompletion $\coce\Phi A \colon A \to \coc \Phi A$. The skew-multicategory $\X[\coce\Phi A]$ is represented by the monoidal hom-category $\Kl(\o\Phi)(A, A)$.
\end{lemma}

\begin{proof}
  The proof is exactly as in \cite[Lemma~8.1]{arkor2025bicategories}, using that, since $\coce\Phi A$ is a cocompletion, it is well-behaved by \cref{cocompletions-are-well-behaved}.
\end{proof}

Since the monoids in $\X[\coce\Phi A]$ are precisely the $\coce\Phi A$-relative monads~\cite[Theorem~4.29]{arkor2024formal}, we may use this connection to characterise them as the $\Phi$-cocontinuous monads on $\coc\Phi A$.

\begin{theorem}[{\cf~\cite[Theorem~8.3]{arkor2025bicategories}}]
  \label{cocontinuous-monads-are-relative-monads}
  Let $\Phi$ be a class of weights, and let $A$ be an object admitting a $\Phi$-cocompletion $\coce\Phi A \colon A \to \coc\Phi A$. Left extension along $\coce\Phi A$ defines a coreflection
  \begin{equation}
    \label{RMnd-Mnd-adjunction}
    \RMnd(\coce\Phi A) \rightleftarrows \Mnd(\coc\Phi A)
  \end{equation}
  which restricts to an equivalence of categories
  \begin{equation}
    \label{RMnd-Mnd-equivalence}
    \RMnd(\coce\Phi A) \equiv \Mnd_\Phi(\coc\Phi A)
  \end{equation}
  between the category of $\coce\Phi A$-relative monads, and the category of $\Phi$-cocontinuous monads on $\coc\Phi A$.
\end{theorem}

\begin{proof}
  The proof is exactly as in \cite[Theorem~8.3]{arkor2025bicategories}, which establishes the result for $\X = \Cat$.
\end{proof}

It can be shown that this relationship furthermore respects algebra objects, in the sense that a $\Phi$-cocontinuous monad $T$ on $\coc\Phi A$ admits an algebra object if and only if the $\coce\Phi A$-relative monad $(\coce\Phi A \d T)$ does, in which case the algebra objects coincide (\cf~\cite[Example~5.9]{arkor2024relative}); we shall establish this as a consequence of a more general phenomenon in future work.

\section{Cocompletions via presheaf constructions}
\label{presheaves-and-cocompletions}

In the \hyperref[introduction]{introduction}, we recalled that cocompletions of enriched categories under classes of weights are exhibited by presheaf constructions. We now tie together our formal developments of presheaf objects and cocompletions by showing that this phenomenon is entirely formal: under reasonable assumptions, $\Phi$-cocompletions may be exhibited by presheaf objects with respect to a class of loose-cells $\PP$ appropriately related to the class of weights $\Phi$. We shall use this in \cref{presheaves-and-cocompletions-in-V-Cat} to construct cocompletions of categories enriched in bicategories.

\begin{theorem}[label=presheaf-object-is-cocompletion]
  Let $\Phi$ be a class of weights, let $\PP$ be a class of loose-cells, and let $A$ be a $\PP$-admissible object (\cref{Phi-embedding}).
  Assume the following.
  \begin{enumerate}
    \item For every weight $\vec p \in \Phi$ and loose-cell $q \in \PP$, the left-composite $q \odotl \vec p$ exists and is in $\PP\repl$.
    \item $\pi^\PP_A$ is in the colimit saturation of $\Phi$ (\cref{saturation}).
    \item $\pi^\PP_A \colon \psh\PP A \lto A$ is lifting (\cref{lifting}).
  \end{enumerate}
  Then $\pshe\PP A \colon A \to \psh\PP A$ exhibits the $\Phi$-cocompletion of $A$ (\cref{cocompletion}).
\end{theorem}

\begin{proof}
  We verify the four conditions of \cref{intrinsic-characterisation-of-cocompletions} with respect to $\pshe\PP A \colon A \to \psh\PP A$, using the fact that
  $\psh\PP A(\pshe\PP A, 1) \iso \pi_A$.
  \begin{enumerate}
    \item That $\psh\PP A$ is $\Phi$-cocomplete follows from \cref{presheaf-colimits-absolute}, using the assumption regarding the existence of left-composites.
    \item $\psh\PP A(\pshe\PP A, 1)$ is in the colimit saturation because the colimit saturation is replete (\cref{sat-is-replete}) and contains $\pi_A$.
    \item $\psh\PP A(\pshe\PP A, 1) \iso \pi_A$ is lifting by assumption. Right lifts through this loose-cell are $\Phi$-exact by \cref{absoluteness-and-exactness}, where we invoke \cref{presheaf-colimits-absolute} for $\pi_A$-absoluteness of $\Phi$-colimits.
    \item $\pshe\PP A$ is dense and \ff{} by \cref{Phi-embedding-is-dense-and-ff}.
    \qedhere
  \end{enumerate}
\end{proof}

\begin{remark}
  Condition (3) is equivalent to:
  \begin{enumerate}
    \item[(3$'$)] $p$ is lifting, for each $p \in \PP$ with codomain $A$.
  \end{enumerate}
  By definition, we have $\pi_A \in \PP$, so one direction is trivial. In the other, we have each $p \colon X \lto A$ in $\PP$ satisfies $p = \pi_A(1, \pshm p)$ and hence $q \rf p = q \rf \pi_A(1, \pshm p) \iso (q \rf \pi_A)(\pshm p, 1)$.
  When $q \in \PP$, this is automatic from the universal property of the $\PP$-presheaf object (\cref{Phi-presheaf-hom-is-right-lift}).
\end{remark}

In particular, every class $\PP$ of loose-cells can be viewed as a class of (unary) weights, and we may instantiate \cref{presheaf-object-is-cocompletion} in this special case. (However, in practice, it will be the more general statement that will be useful.)

\begin{corollary}[label=psh-cocompleteness, note={\cf~\cite[Theorem~10.25]{shulman2013enriched}}]
  Let $\PP$ be a class of loose-cells and let $A$ be a \mbox{$\PP$-admissible} object. If $\PP$ is closed under left-composites, then $\psh\PP A$ is $\PP$-cocomplete and every $\PP$-colimit in $\psh\PP A$ is $\pshe\PP A$-absolute. Furthermore, if $\pi_A \colon \psh\PP A \lto A$ is lifting, then $\pshe\PP A \colon A \to \psh\PP A$ exhibits the $\PP$-cocompletion of $A$.
\end{corollary}

\begin{proof}
  Immediate from \cref{presheaf-colimits-absolute,presheaf-object-is-cocompletion}.
\end{proof}

\begin{remark}
  \label{relationship-to-koudenburg}
  We briefly outline the relationship to the work of \textcite{koudenburg2024formal}, who establishes that presheaf objects (with respect to the class of all loose-cells, \ie{} the situation of \cref{presheaf-object}) in augmented \ve s exhibit cocompletions with respect to certain classes of weights and diagrams. (An \emph{augmented \ve} is a structure similar to a \ve, but which does not necessarily have identity loose-cells, instead having a collection of multiary 2-cells with nullary codomain~\cite[Definition~1.19]{koudenburg2024formal}.)

  Let $\PP$ be a class of loose-cells in a \ve{} $\X$, and take the full sub-augmented \ve{} $\X_\PP$ spanned by loose-cells in $\PP$. Suppose that $\X_\PP$ admits a presheaf object $\psh\PP A$ (with respect to all loose-cells in $\X_\PP$). Then \cite[Theorem~7.6]{koudenburg2024formal} establishes that $\psh\PP A$ exhibits the 2-categorical universal property of a free cocompletion that we consider (\ie{} conditions (1 \& 2) in \cref{cocompletion}): \citeauthor{koudenburg2024formal}'s conditions (e \& y) are analogous to our conditions (1 \& 2) in \cref{presheaf-object-is-cocompletion} (\cf~\cite[Example~5.3 \& Remark~1.26]{koudenburg2024formal}). One key difference is that we permit $\Phi$ to contain non-unary weights; we also merely require $\pi_A$ to be in $\Phi\satw$ rather than in $\Phi$. However, the most important distinction is in the universal property we establish (for which we use condition (3), which has no analogue amongst \citeauthor{koudenburg2024formal}'s assumptions).

  Crucially, as we have discussed in \cref{cocompletions}, our notion of free cocompletion has a stronger universal property than the 2-categorical universal property (namely, \cref{cocompletion-loose-cell-UP}), which does not appear to follow from the results of \cite{koudenburg2024formal}, and which is fundamental to establishing many of the important properties regarding cocompletions in \cref{cocompleteness-in-a-virtual-equipment}.
\end{remark}

\section{Copresheaf objects and completions}
\label{copresheaf-objects-and-completions}

Throughout the paper so far, we have focused on presheaf objects and cocompletions. The theory of \emph{copresheaf objects} and \emph{completions} in a \ve{} $\X$ is formally dual, corresponding to presheaf objects and cocompletions in the \ve{} $\X\co$ (\cref{dual}). For convenience, we spell out the main definitions and theorems.

\begin{manualdefinition}{\ref{dense-loose-cell}$\co$}
  A loose-cell $p \colon A \lto B$ is \emph{codense} when the identity 2-cell $1_p \colon p \tto p$ exhibits $B(1, 1) \colon B \lto B$ as the right extension $p \rx p$. A tight-cell is \emph{codense} if its companion is.
\end{manualdefinition}

\begin{manualdefinition}{\ref{Phi-presheaf-object}$\co$}
  Let $\PP$ be a class of loose-cells in $\X$ and let $A$ be an object. A \emph{$\PP$-copresheaf object} for $A$ comprises an object $\cpsh\PP A$ and a codense loose-cell $\copi^\PP_A \colon A \lto \cpsh\PP A$, such that
  \begin{enumerate}
      \item for every loose-cell $p \colon A \lto X$ in $\PP$, there is a unique tight-cell $\cpshm p \colon X \to \cpsh\PP A$ satisfying $p = \copi^\PP_A(\cpshm p, 1)$;
      \item for every tight-cell $f \colon X \to \cpsh\PP A$, the restriction $\copi^\PP_A(f, 1) \colon A \lto X$ is in $\PP$.
  \end{enumerate}
  Typically, when there is no risk of confusion, we will denote $\copi^\PP_A$ simply by $\copi_A$.
\end{manualdefinition}

\begin{manualdefinition}{\ref{Phi-embedding}$\co$}
  Let $\PP$ be a class of loose-cells. An object $A$ is \emph{$\PP$-coadmissible} if it admits a $\PP$-copresheaf object, and $A(1, 1) \in \PP\repl$.
  In this case, we denote by
  \[\cpshe\PP A \colon A \to \cpsh\PP A\]
  the \emph{$\PP$-copresheaf embedding of $A$}, which is the essentially unique tight-cell satisfying $A(1, 1) \iso \copi_A(\cpshe\PP A, 1)$.
\end{manualdefinition}

\begin{manualdefinition}{\ref{atom}$\co$}
  Let $\vec p \colon Z \lcto Y$ be a weight. A loose-cell $j \colon A \lto E$ is a \emph{$\vec p$-coatom} when every $\vec p$-weighted limit in $E$ is $j$-coabsolute. Given a class of weights $\Phi$, we say that $j$ is a \emph{$\Phi$-coatom} if it is a $\vec p$-coatom for each $\vec p \in \Phi$. A tight-cell is a \emph{$\vec p$-coatom}, respectively \emph{$\Phi$-coatom}, if its companion is.
\end{manualdefinition}

\begin{manualdefinition}{\ref{well-behaved}$\co$}
  \label{co-well-behaved}
  A tight-cell $j \colon A \to E$ is \emph{co-well-behaved} if the following conditions hold.
  \begin{enumerate}
    \item $j$ is \ff{}.
    \item $j$ is codense.
    \item Right extensions of tight-cells $A \to E$ along $j$ exist.
    \item $j$ is an $E(1, j)$-coatom, \ie the right extensions above are $j$-coabsolute.
    \qedhere
  \end{enumerate}
\end{manualdefinition}

\begin{manualdefinition}{\ref{cocomplete}$\co$}
  \label{complete}
  Let $\Phi$ be a class of weights.
  \begin{enumerate}
    \item A \emph{$\Phi$-limit} is a $\vec p$-weighted limit for some $\vec p \in \Phi$.
    \item An object $A$ is \emph{$\Phi$-complete} if, for every tight-cell $f \colon D \to A$ and weight $\vec p \colon D \lcto W$ in $\Phi$, there exists a limit $\vec p \wl f \colon W \to A$.
    \item A tight-cell $g \colon A \to B$ with $\Phi$-complete domain is \emph{$\Phi$-continuous} if it preserves every $\Phi$-limit.
    \item A loose-cell $p \colon A \lto X$ with $\Phi$-complete domain is \emph{$\Phi$-coexact} if it respects every $\Phi$-limit (in the sense of \cref{respects-limit}).
    \qedhere
  \end{enumerate}
\end{manualdefinition}

\begin{manualdefinition}{\ref{cocompletion}$\co$}
  \label{completion}
  Let $\Phi$ be a class of weights and let $A$ be an object. A \emph{$\Phi$-completion} of $A$ comprises an object $\com\Phi A$ and a tight-cell $\come\Phi A \colon A \to \com\Phi A$ such that the following conditions hold.
  \begin{enumerate}
    \item $\com\Phi A$ is $\Phi$-complete.
    \item For each tight-cell $f \colon A \to X$ with $\Phi$-complete codomain, the right extension ${\come\Phi A \prx f \colon \com\Phi A \to X}$ exists, the counit $\come\Phi A \d (\come\Phi A \prx f) \tto f$ is invertible, and $\come\Phi A \prx f$ is the essentially unique $\Phi$-continuous tight-cell $\tilde f \colon \com\Phi A \to X$ such that $f \iso (\come\Phi A \d \tilde f)$.
    \item For each loose-cell $p \colon A \lto X$, the right extension ${\com\Phi A(1, \come\Phi A) \rx p \colon \com\Phi A \lto X}$ exists, the canonical 2-cell $(\com\Phi A(1, \come\Phi A) \rx p)(1, \come\Phi A) \tto p$ is invertible, and $\com\Phi A(1, \come\Phi A) \rx p$ is the essentially unique $\Phi$-coexact loose-cell $\tilde p \colon \com\Phi A \lto X$ such that $p \iso \tilde p(1, \come\Phi A)$.
  \end{enumerate}
  An object is \emph{$\Phi$-completable} if it admits a $\Phi$-completion.
\end{manualdefinition}

\begin{manualproposition}{\ref{cocompletions-are-well-behaved}$\co$}
  Let $\Phi$ be a class of weights. Every $\Phi$-completion $\come\Phi A \colon A \to \com\Phi A$ is co-well-behaved.
\end{manualproposition}

\begin{manualtheorem}{\ref{presheaf-object-is-cocompletion}$\co$}
  Let $\Phi$ be a class of weights, let $\PP$ be a class of loose-cells, and let $A$ be a $\PP$-coadmissible object.
  Assume the following.
  \begin{enumerate}
    \item For every weight $\vec p \in \Phi$ and loose-cell $q \in \PP$, the right-composite $\vec p \odotr q$ exists and is in $\PP\repl$.
    \item $\copi^\PP_A$ is extending (\cref{extending}).
    \item $\copi^\PP_A$ is in the limit saturation of $\Phi$ (\cref{limit-saturation}).
  \end{enumerate}
  Then $\cpshe\PP A \colon A \to \cpsh\PP A$ exhibits the $\Phi$-completion of $A$.
\end{manualtheorem}

\section{Cocompletions of enriched categories}
\label{presheaves-and-cocompletions-in-V-Cat}

To demonstrate the utility of the theory we have developed, we instantiate our main results in the context of \ve s{} of enriched categories, providing constructions of completions and cocompletions of enriched categories under classes of weights. Beyond merely justifying the applicability of our results, this fills a conspicuous gap in the literature: while such constructions are well known to exist in the setting of enrichment in particularly nice monoidal categories, it does not appear that the existence of these constructions has been demonstrated even for categories enriched in non-symmetric monoidal categories, let alone for categories enriched in bicategories, despite the importance of these settings. (We give a survey of completion results in the literature in \cref{historical-context}.)

\subsection{Categories enriched in \vbs}

In fact, we shall consider a setting slightly more general than that of enrichment in bicategories. It was observed by \textcite{leinster2002generalized} that it is natural to enrich categories not merely in monoidal categories, but in \vdcs{} (enrichment in a monoidal category being recovered by considering a monoidal category to be a one-object \vdc). Here, we shall be concerned with the special case of enrichment in normal \vbs\footnotemark{}. Recall that a \vdc{} is normal when it admits loose-identities (\cref{opcartesian}).
\footnotetext{The reason for considering enrichment merely in normal \vbs{}, rather than in the most general setting, is that there are several subtleties arising in the theory of enrichment in \vdcs{} that must be addressed before developing the corresponding \fct. In the interest of space, we defer such considerations to a future paper.}%

\begin{definition}
  A (\emph{normal}) \emph{\vb} is a (normal) \vdc{} in which every tight-cell is an identity.
\end{definition}

A normal \vb{} is thus a structure akin to a bicategory -- in particular, possessing identity 1-cells -- but without the ability to compose 1-cells, instead having multiary 2-cells $f_1, \ldots, f_n \tto g$.

\begin{example}
  \label{multicategory}
  Every multicategory (and, in particular, every monoidal category) $\M$ forms a \vb{} $\Sigma\M$ with a single object, the loose-cells of $\Sigma\M$ being the objects of $\M$, and the 2-cells of $\Sigma\M$ being the multimorphisms of $\M$.
\end{example}

Since the concept of enrichment in \vbs{} is not widely known, we recall the relevant definitions. Henceforth, $\V$ shall denote a large \vb.

\begin{definition}[{\cite[Definition~2]{leinster2002generalized}}]
  \label{V-category}
  Let $\V$ be a \vb{}. A \emph{$\V$-enriched category} (or simply \emph{$\V$-category}) $\C$ comprises the following data.
  \begin{enumerate}
    \item A large set $\ob\C$ of \emph{objects}.
    \item For each $x \in \ob\C$, an \emph{extent} $\ex x \in \V$.
    \item For each $x, y \in \ob\C$, a \emph{hom-morphism}\footnotemark{} $\C(x, y) \colon \ex y \lto \ex x$ in $\V$.
    \item For each $x \in \ob\C$, an \emph{identity} 2-cell in $\V$.
      \[\begin{tikzcd}[column sep=large]
        {\ex x} & {\ex x} \\
        {\ex x} & {\ex x}
        \arrow[equals, nfold, from=1-1, to=1-2]
        \arrow[""{name=0, anchor=center, inner sep=0}, equals, nfold, from=1-1, to=2-1]
        \arrow[""{name=1, anchor=center, inner sep=0}, equals, nfold, from=1-2, to=2-2]
        \arrow["{\C(x, x)}"{inner sep=.8ex}, "\shortmid"{marking}, from=2-2, to=2-1]
        \arrow["{\I_x}"{description}, draw=none, from=0, to=1]
      \end{tikzcd}\]
    \item For each $x, y, z \in \ob\C$, a \emph{composition} 2-cell in $\V$.
      \[\begin{tikzcd}[column sep=large]
        {\ex x} & {\ex y} & {\ex z} \\
        {\ex x} && {\ex z}
        \arrow[""{name=0, anchor=center, inner sep=0}, equals, nfold, from=1-1, to=2-1]
        \arrow["{\C(x, y)}"'{inner sep=.8ex}, "\shortmid"{marking}, from=1-2, to=1-1]
        \arrow["{\C(y, z)}"'{inner sep=.8ex}, "\shortmid"{marking}, from=1-3, to=1-2]
        \arrow[""{name=1, anchor=center, inner sep=0}, equals, nfold, from=1-3, to=2-3]
        \arrow["{\C(x, z)}"{inner sep=.8ex}, "\shortmid"{marking}, from=2-3, to=2-1]
        \arrow["{\circ_{x, y, z}}"{description}, draw=none, from=0, to=1]
      \end{tikzcd}\]
  \end{enumerate}
  \footnotetext{There are two conventions for the direction of hom-morphisms for categories enriched in bicategories. We follow \textcite{street1981cauchy}, which aligns with our convention that loose-cells should be considered contravariant in their codomains (\cf~\cref{Cat}).}%
  These data are required to satisfy left-unitality, right-unitality, and associativity axioms.
\end{definition}

\begin{definition}
  Let $\V$ be a \vb. A $\V$-category is
  \begin{enumerate}
    \item \emph{small} if its class of objects is small;
    \item \emph{extentwise small}\footnotemark{} if, for each object $V \in \V$, the class of objects with extent $V$ is small;
    \item \emph{large} in general.
    \qedhere
  \end{enumerate}
  \footnotetext{Called \emph{pointwise small} in \cite[\S3]{gordon1999gabriel}.}%
\end{definition}

In particular, every small $\V$-category is extentwise small, and the converse holds if the \vb{} $\V$ is small.

We will leave most examples for \cref{examples}, but mention one familiar example here.

\begin{example}
  For $\M$ a multicategory and $\Sigma\M$ the corresponding one-object \vb{} described in \cref{multicategory}, a $\Sigma\M$-enriched category is simply a category enriched in the multicategory $\M$, in the sense of \cite[106]{lambek1969deductive}. In particular, when $\M$ is a monoidal category, this recovers the usual notion of a category enriched in a monoidal category~\cite{kelly1982basic}.
\end{example}

\begin{definition}[{\cite[Definition~3]{leinster2002generalized}}]
  \label{V-functor}
  Let $\V$ be a \vb{} and let $\C$ and $\D$ be $\V$-categories. A \emph{$\V$-functor} $f$ from $\C$ to $\D$ comprises the following data.
  \begin{enumerate}
    \item A function $\ob f \colon \ob\C \to \ob\D$ such that $f$ preserves extent, in that the following diagram commutes.
    \[\begin{tikzcd}
      {\ob\C} && {\ob\D} \\
      & \V
      \arrow["{\ob f}", from=1-1, to=1-3]
      \arrow["{\ex\ph}"', from=1-1, to=2-2]
      \arrow["{\ex\ph}", from=1-3, to=2-2]
    \end{tikzcd}\]
    \item For each $x, y \in \ob\C$, a 2-cell in $\V$.
    \[\begin{tikzcd}[column sep=6em]
      {\ex x} & {\ex y} \\
      {\ex{\ob f x}} & {\ex{\ob f y}}
      \arrow[""{name=0, anchor=center, inner sep=0}, equals, nfold, from=1-1, to=2-1]
      \arrow["{\C(x, y)}"'{inner sep=.8ex}, "\shortmid"{marking}, from=1-2, to=1-1]
      \arrow[""{name=1, anchor=center, inner sep=0}, equals, nfold, from=1-2, to=2-2]
      \arrow["{\D(\ob f x, \ob f y)}"{inner sep=.8ex}, "\shortmid"{marking}, from=2-2, to=2-1]
      \arrow["{f_{x, y}}"{description}, draw=none, from=0, to=1]
    \end{tikzcd}\]
  \end{enumerate}
  These data are required to satisfy preservation of identities and composites.
\end{definition}

\begin{definition}
  \label{V-distributor}
  Let $\V$ be a \vb{} and let $\C$ and $\C'$ be $\V$-categories. A \emph{$\V$-distributor} $p$ from $\C'$ to $\C$ comprises the following data.
  \begin{itemize}
    \item For each $x \in \ob\C$ and $y \in \ob{\C'}$, a \emph{hom-morphism} $p(x, y) \colon \ex y \lto \ex x$ in $\V$.
    \item For each $x, x' \in \ob\C$ and $y \in \ob{\C'}$, a \emph{composition} 2-cell in $\V$.
    \[\begin{tikzcd}[column sep=large]
      {\ex{x'}} & {\ex x} & {\ex y} \\
      {\ex{x'}} && {\ex y}
      \arrow[""{name=0, anchor=center, inner sep=0}, equals, nfold, from=1-1, to=2-1]
      \arrow["{\C(x', x)}"'{inner sep=.8ex}, "\shortmid"{marking}, from=1-2, to=1-1]
      \arrow["{p(x, y)}"'{inner sep=.8ex}, "\shortmid"{marking}, from=1-3, to=1-2]
      \arrow[""{name=1, anchor=center, inner sep=0}, equals, nfold, from=1-3, to=2-3]
      \arrow["{p(x', y)}"{inner sep=.8ex}, "\shortmid"{marking}, from=2-3, to=2-1]
      \arrow["{\circ_{x', x, y}}"{description}, draw=none, from=0, to=1]
    \end{tikzcd}\]
    \item For each $x \in \ob\C$ and $y, y' \in \ob{\C'}$, a \emph{composition} 2-cell in $\V$.
    \[\begin{tikzcd}[column sep=large]
      {\ex x} & {\ex y} & {\ex{y'}} \\
      {\ex x} && {\ex{y'}}
      \arrow[""{name=0, anchor=center, inner sep=0}, equals, nfold, from=1-1, to=2-1]
      \arrow["{p(x, y)}"'{inner sep=.8ex}, "\shortmid"{marking}, from=1-2, to=1-1]
      \arrow["{\C'(y, y')}"'{inner sep=.8ex}, "\shortmid"{marking}, from=1-3, to=1-2]
      \arrow[""{name=1, anchor=center, inner sep=0}, equals, nfold, from=1-3, to=2-3]
      \arrow["{p(x, y')}"{inner sep=.8ex}, "\shortmid"{marking}, from=2-3, to=2-1]
      \arrow["{\circ_{x, y, y'}}"{description}, draw=none, from=0, to=1]
    \end{tikzcd}\]
  \end{itemize}
  These data are required to satisfy compatibility axioms with the identities and composition in $\C$ and $\C'$ and an interaction axiom~\cite[\S3]{leinster2002generalized}.
\end{definition}

\begin{definition}
  \label{V-transformation}
  Let $\V$ be a \vb{}. A \emph{$\V$-natural transformation} with the following frame of $\V$-distributors and $\V$-functors,
  \[\begin{tikzcd}
    {\C_0} & \cdots & {\C_n} \\
    \D && {\D'}
    \arrow["f"', from=1-1, to=2-1]
    \arrow["{p_1}"'{inner sep=.8ex}, "\shortmid"{marking}, from=1-2, to=1-1]
    \arrow["{p_n}"'{inner sep=.8ex}, "\shortmid"{marking}, from=1-3, to=1-2]
    \arrow["{f'}", from=1-3, to=2-3]
    \arrow["q"{inner sep=.8ex}, "\shortmid"{marking}, from=2-3, to=2-1]
  \end{tikzcd}\]
  comprises, for each family $\{ x_i \in \ob{\C_i} \}_{0 \leq i \leq n}$, a 2-cell in $\V$ with the following frame,
  \[\begin{tikzcd}[column sep=large]
    {\ex{x_0}} & \cdots & {\ex{x_n}} \\
    {\ex{\ob f x_0}} && {\ex{\ob{f'} x_n}}
    \arrow[equals, nfold, from=1-1, to=2-1]
    \arrow["{p_1(x_0, x_1)}"'{inner sep=.8ex}, "\shortmid"{marking}, from=1-2, to=1-1]
    \arrow["{p_n(x_{n - 1}, x_n)}"'{inner sep=.8ex}, "\shortmid"{marking}, from=1-3, to=1-2]
    \arrow[equals, nfold, from=1-3, to=2-3]
    \arrow["{q(\ob f x_0, \ob{f'} x_n)}"{inner sep=.8ex}, "\shortmid"{marking}, from=2-3, to=2-1]
  \end{tikzcd}\]
  subject to naturality conditions~\cite[Definition~3.9]{arkor2025nerve}.
\end{definition}

\begin{example}
  The strict \ve{} $\VCat$ is defined as follows.
  \begin{enumerate}
    \item The underlying category is the category of (large) $\V$-categories (\cref{V-category}) and $\V$-functors (\cref{V-functor}).
    \item A loose-cell is a $\V$-distributor (\cref{V-distributor}).
    \item A 2-cell is a $\V$-natural transformation (\cref{V-transformation}).
    \item The loose-identity $\C(1, 1) \colon \C \lto \C$ on a $\V$-category $\C$ is given by $(c, c') \mapsto \C(c, c')$, with actions given by composition in $\C$.
    \item The restriction of a $\V$-distributor $q \colon \B' \lto \A'$ along $\V$-functors $f \colon \A \to \A'$ and $g \colon \B \to \B'$ is the $\V$-distributor $q(f, g) \defeq (a, b) \mapsto q(f(a), g(b))$.
    \qedhere
  \end{enumerate}
\end{example}

\begin{definition}
  \label{one-object-enriched-category}
  If $\V$ is a normal \vb{}, then, for every object $V \in \V$, there is a $\V$-category $\star_V$ having a single object of extent $V$, whose hom-morphism is the loose-identity on $V$.
  To give a $\V$-functor $\star_V \to \A$ is equivalently to choose a single object $a \in \ob{\A}$ of extent $V$.
  Below, we will use this correspondence implicitly, writing $a$ for both the object and the $\V$-functor.
  Similarly, $\V$-distributors $p \colon \star_V \lto \star_{V'}$ are in bijection with loose-cells $p \colon V \lto V'$.
\end{definition}

\subsection{Enriched presheaves}

In light of \cref{presheaves-and-cocompletions}, to construct cocompletions, we are led to consider enriched categories of presheaves.

\begin{definition}[label=presheaf]
  A \emph{$\V$-presheaf} $p$ on a $\V$-category $\A$ comprises an object $\ex{p}$ of $\V$ (the \emph{extent} of $p$), together with a $\V$-distributor $\star_{\ex{p}} \lto \A$: explicitly, this comprises a morphism $p(a) \colon \ex{p} \lto \ex{a}$ in $\V$ for each object $a \in \ob{\A}$; and a left-action of $\A$ on $p$, \ie{} a 2-cell $\c_{a', a} \colon \A(a', a), p(a) \tto p(a')$ in $\V$ for each pair of objects $a, a' \in \ob{\A}$, compatible with identities and composition in $\A$.
  Dually, a \emph{$\V$-copresheaf} $q$ on $\A$ comprises an object $\ex{q}$ of $\V$, together with a $\V$-distributor $\A \lto \star_{\ex{q}}$.
\end{definition}

Since we define $\V$-presheaves in terms of $\V$-distributors, we can speak of a right lift $q \rf p$ of $\V$-presheaves $p, q$ on $\A$.
Such a right lift is a $\V$-distributor $\star_{\ex{q}} \lto \star_{\ex{p}}$ or, equivalently, a loose-cell $q \rf p \colon \ex{q} \lto \ex{p}$ in $\V$.
These right lifts will provide the hom-morphisms of each $\V$-presheaf object.
When we take $\V = \Sigma\b V$ for a symmetric monoidal closed category $\b V$, a $\V$-presheaf on $\A$ is equivalently a $\b V$-functor $\A\op \to \b V$, and the right lift $q \rf p$ is the hom-object $[\A\op, \b V](p, q)$ of the $\b V$-functor category $[\A\op, \b V]$, when the latter exists.
A more detailed discussion of these points can be found in \cite[Section~8.1]{arkor2024formal}, which uses the notation $\rfP{A}{p}{q}$ for a right lift of presheaves.

Right lifts of $\V$-distributors can be reduced to right lifts of $\V$-presheaves by the following lemma.

\begin{lemma}
  Let $\vec{p} = (p_1, \dots, p_n) \colon \A_n \lcto \A_0$ be a weight in $\VCat$, with $n \ge 1$.
  \begin{enumerate}
    \item \label{right-lift-enriched-characterisation}
      Let $q \colon \B \lto \A_0$ be a $\V$-distributor.
      If it exists, the right lift $q \rf \vec{p} \colon \B \lto \A_n$ satisfies
      \[
        (q \rf \vec{p})(a, b) \iso q(1, b) \rf (p_1, \dots, p_{n - 1}, p_n(1, a))
      \]
      for all $a \in \ob{\A_n}$ and $b \in \ob{\B}$.
      If the right-hand side exists for all $a$ and $b$, then so does $q \rf \vec{p}$ (\cf{} \cite[Lemma~8.7]{arkor2024formal}).
    \item \label{composite-enriched-characterisation}
      If it exists, the left-composite $(p_1 \odotl \cdots \odotl p_n) \colon \A_n \lto \A_0$ satisfies
      \[
        (p_1 \odotl \cdots \odotl p_n)(1, a_n)
        \iso
        p_1 \odotl \cdots \odotl p_{n-1} \odotl (p_n(1, a_n))
      \]
      for all $a_n \in \ob{\A_n}$.
      If the right-hand side exists for all $a_n$, then so does $p_1 \odotl \cdots \odotl p_n$ (\cf{} \cite[Lemma~8.7]{arkor2024formal}).
  \end{enumerate}
\end{lemma}
\begin{proof}
  The first part of (1) is \cref{restriction-of-lift}.
  For the second part, we show that if the right lifts $q(1, b) \rf (p_1, \dots, p_{n - 1}, p_n(1, a))$ exist, then we can construct the right lift $q \rf \vec{p}$.
  We define the $\V$-distributor $q \rf \vec{p} \colon \B \lto \A_n$ on objects by
  \[
    (q \rf \vec{p})(a, b) = q(1, b) \rf (p_1, \dots, p_{n - 1}, p_n(1, a))
  \]
  The composition 2-cells $\A_n(a', a), (q \rf \vec{p})(a, b) \tto (q \rf \vec{p})(a', b)$ and
  $(q \rf \vec{p})(a, b), \B(b, b') \tto (q \rf \vec{p})(a, b')$ are the unique 2-cells for which the counits
  \[
    p_1, \dots, p_n(1, a), (q \rf \vec{p})(a, b) \tto q(1, b)
  \]
  form a $\V$-natural transformation $\vec{p}, q \rf \vec{p} \tto q$.
  The latter is the counit of the right lift $q \rf \vec{p}$.
  To see that it satisfies the required universal property, observe that, for $m \ge 1$, a $\V$-natural transformation
  $p_1, \dots, p_n, r_1, \dots, r_m \tto q$ is, in particular, a family of $\V$-natural transformations
  \[
    p_1, \dots, p_n(1, a), r_1(a, 1), \dots, r_m(1, b) \tto q(1, b)
  \]
  satisfying two laws.
  Under transposition, those two laws are precisely the two laws needed for the corresponding 2-cells
  \[
    r_1(a, 1), \dots, r_m(1, b) \tto (q \rf \vec{p})(a, b)
  \]
  to form a $\V$-natural transformation $r_1, \dots, r_m \tto q \rf \vec{p}$.
  For $m = 0$, we reason similarly.

  The proof of (2) is similar.
  For the first part, we have the following chain of isomorphisms.
  \begin{align*}
    (p_1 \odotl \cdots \odotl p_n)(1, a_n) &
      \iso (p_1 \odotl \cdots \odotl p_n) \odot \A_n(1, a_n) \tag{\cite[Theorem~7.16]{cruttwell2010unified}}
      \\ & \iso p_1 \odotl \cdots \odotl (p_n \odot \A_n(1, a_n)) \tag{\cref{left-composite-reassociation}}
      \\ & \iso p_1 \odotl \cdots \odotl p_n(1, a_n) \tag{\cite[Theorem~7.16]{cruttwell2010unified}}
  \end{align*}
  For the second part, we make the left-composites
  \[
    (p_1 \odotl \cdots \odotl p_n)(1, a_n)
    =
    p_1 \odotl \cdots \odotl p_{n - 1} \odotl p_n(1, a_n)
  \]
  uniquely into a $\V$-distributor $(p_1 \odotl \cdots \odotl p_n) \colon \A_n \lto \A_0$, by defining
  \begin{align*}
    p_1(a_0, 1) \odotl \cdots \odotl p_n(1, a_n), \A_n(a_n, a_n') &\tto p_1(a_0, 1) \odotl \cdots \odotl p_n(1, a'_n)
  \end{align*}
  to be the unique 2-cell for which the units $p_1, \dots, p_n(1, a_n) \tto p_1 \odotl \cdots \odotl p_n(1, a_n)$ are $\V$-natural in $a_n$.
  The universal property of the left-composite then follows.
\end{proof}

\begin{remark}[label=composite-enriched-remark]
  By duality, \cref{composite-enriched-characterisation} implies an analogous fact about right-composites. It then follows that, for loose-composites (rather than left- or right-composites), we have
  \[
    (p_1 \odot \cdots \odot p_n)(a_0, a_n)
    =
    p_1(a_0, 1) \odot \cdots \odot p_n(1, a_n)
  \]
  with the left-hand side existing if the right-hand side does.
  In particular, binary composites of $\V$-distributors can be reduced to binary composites of the form $q \odot p$ in which $q$ is a $\V$-copresheaf and $p$ is a $\V$-presheaf; composites of this form are determined by loose-cells $q \odot p \colon \ex{p} \lto \ex{q}$ in $\V$.
\end{remark}

\Cref{right-lift-enriched-characterisation} implies an analogous fact about weighted colimits, which we make explicit as the following lemma.
In particular, it follows that colimits weighted by $\V$-distributors can be reduced to colimits weighted by $\V$-presheaves (as observed by \textcite[11]{street1983enriched} for enrichment in bicategories), which agrees with the usual notion of weighted colimit in enriched category theory (\cf{}~\cite[Section~8.1]{arkor2024formal}).

\begin{corollary}[label=colimit-weighted-by-presheaves]
  Let $\vec p = (p_1, \dots, p_n) \colon \A_n \lto \A_0$ be a weight in $\VCat$, and let $f \colon \A_0 \to \B$ be a $\V$-functor.
  If the colimit $\vec p \wc f \colon \A_n \to \B$ exists, then it is given pointwise:
  \[
    (\vec p \wc f)(a) \iso (p_1, \dots, p_{n-1}, p_n, \A_n(1, a)) \wc f
  \]
  If the right-hand side exists for all $a \in \ob{\A_n}$, then so does $\vec p \wc f$.
  Thus, the following classes of weights are colimit equivalent (\cref{colimit-equivalence}).
  \[
    \{\ \vec p\ \} \colimequiv \{\ (p_1, \dots, p_{n - 1}, p_n, \A_n(1, a))\ \}_{a \in \ob{\A_n}}
  \]
\end{corollary}
\begin{proof}
  Colimits weighted by $\vec{p}$ are equivalently colimits weighted by $(p_1, \dots, p_n, \A_n(1, 1))$, so this follows from \cref{right-lift-enriched-characterisation}.
  In particular, as soon as the right-hand sides exist, the right lift $\B(f, 1) \rf \vec p$ exists, and the action of $\A_n$ thereon equips $\vec p \wc f$ with the structure of a $\V$-functor.
\end{proof}

To construct presheaf objects and cocompletions in $\VCat$, we need to know that various right lifts and composites of $\V$-distributors exist.
Under smallness assumptions, these are given as usual by limits and colimits in $\V$, as in the following lemma.
We refer to \cref{local-colimit,local-limit} for the definition of local (co)completeness for a \vdc.
When $\V$ is a bicategory (\ie{} the \vb{} $\V$ admits all loose-composites), local completeness means that each hom-category is complete (\cref{local-completeness-for-bicategories}); whereas local cocompleteness means that each hom-category is cocomplete \emph{and} that the pre- and postcomposition functors are cocontinuous (\cref{local-cocompleteness-for-bicategories}). Left-local cocompleteness drops the requirement that the postcomposition functors are cocontinuous (\cref{left-composites-and-whiskering}).

\begin{lemma}[label=small-codomain-constructions-exist]
  Let $\vec p = (p_1, \dots, p_n)$ be a weight in $\VCat$, such that the codomain of each $p_i$ is small.
  \begin{enumerate}
    \item \label{small-codomain-right-lifts-exist} If $\V$ is locally complete and admits right lifts, then $\vec p$ is lifting (\cref{lifting}).
    \item \label{small-codomain-composites-exist} If $\V$ is (left-) locally cocomplete and admits (left-) composites, then the (left-) composite ${q \odotl p_1 \odotl \cdots \odotl p_n}$ exists for all $\V$-distributors $q$.
  \end{enumerate}
\end{lemma}
\begin{proof}
  For (1), we first show that, if $p$ and $q$ are $\V$-presheaves on a small $\A$, then the right lift $q \rf p$ exists.
  We construct $q \rf p$ as follows; this construction is the same as the one given by \cite{gordon1999gabriel}, only with fewer assumptions on $\V$.
  Let ${\bb I}$ be the category consisting of two objects $\mathtt{i}_a$ and $\mathtt{j}_a$ for each $a \in \ob{\A}$,
  and a cospan $\mathtt{i}_a \to \mathtt{k}_{a, a'} \from \mathtt{j}_{a'}$ for each pair $a, a' \in \ob{\A}$.
  Define a functor ${\bb I} \to \V\lh{\ex q, \ex p}$ of categories, sending the cospan $\mathtt{i}_a \to \mathtt{k}_{a, a'} \from \mathtt{j}_{a'}$ to the following cospan of 2-cells between loose-cells $\ex q \lto \ex p$ in $\V$, the 2-cells being defined using the left-actions of $p$ and $q$.
  \[
  \begin{tikzcd}
  	{q(a) \rf p(a)} & {(q(a) \rf A(a, a')) \rf p(a')} & {q(a') \rf p(a')}
  	\arrow[Rightarrow, from=1-1, to=1-2]
  	\arrow[Rightarrow, from=1-3, to=1-2]
  \end{tikzcd}
  \]
  This diagram is small because $\A$ is.
  We take $q \rf p \colon \ex{q} \lto \ex{p}$ to be the limit of this diagram.
  We have the required universal property because a family
  $\{p(a), v_1, \dots, v_n \tto q(a)\}_{a \in \ob{\A}}$
  of 2-cells is $\V$-natural in $a \in \ob{\A}$ exactly when the corresponding family
  $\{v_1, \dots, v_n \tto q(a) \rf p(a)\}_{a \in \ob{\A}}$
  is a cone for the diagram above.

  It follows by \cref{right-lift-enriched-characterisation} that $q \rf p$ exists whenever $p$ and $q$ are both $\V$-distributors with small codomain.
  Thus, the right lift $q \rf (p_1, \dots, p_n)$ exists whenever every $p_i$ has a small codomain, since
  by \cref{iterated-lift}, we have:
  \[
    q \rf (p_1, \dots, p_n) \iso q \rf p_1 \rf \cdots \rf p_n
  \]

  The proof of (2) has a similar structure.
  We first show that $q \odotl p$ exists whenever $p$ is a $\V$-presheaf, and $q$ is a $\V$-copresheaf, on a small $\A$.
  We construct $q \odotl p$ as the colimit of the functor ${\bb I}\op \to \V\lh{\ex p, \ex q}$ sending each cospan in ${\bb I}$ to the following span of 2-cells.
  \[\begin{tikzcd}
    {q(a) \odotl p(a)} & {q(a) \odotl A(a, a') \odotl p(a')} & {q(a') \odotl p(a')}
    \arrow[Rightarrow, nfold, from=1-2, to=1-1]
    \arrow[Rightarrow, nfold, from=1-2, to=1-3]
  \end{tikzcd}\]
  Again, this diagram is small because $\A$ is.
  A family $\{\vec{v}, q(a), p(a), \vec{w} \tto u\}_{a \in \ob{\A}}$ of 2-cells is $\V$-natural in $a$ exactly when it is a cocone for the colimit $q \odotl p$, and thus $q \odotl p$ is the required composite.

  It follows by \cref{composite-enriched-characterisation} that $q \odotl p$ exists whenever $p \colon \b X \lto \A$ and $q \colon \A \lto \b Y$ are both $\V$-distributors, with $\A$ small.
  Thus $q \odotl p_1 \odotl \cdots \odotl p_n$ exists whenever the codomain of each $p_k$ is small, since,
  by \cref{left-associativity-of-left-composites}, we have:
  \[
    q \odotl p_1 \odotl \cdots \odotl p_n
    \iso
    (((q \odotl p_1) \odotl p_2) \odotl \cdots) \odotl p_n
    \qedhere
  \]
\end{proof}

\subsection{Enriched presheaf objects}

We turn now to the construction of presheaf objects (in the sense of \cref{Phi-presheaf-object}) in the \ve{} $\VCat$.
We cannot expect $\PP$-presheaf objects to exist for every class of $\V$-distributors $\PP$; indeed, an immediate consequence of the existence of a $\PP$-presheaf object on an object $A$ is that $\PP|_{{}\lto A}$ is closed under restrictions $p(1, f)$.
We thus restrict to classes of $\V$-distributors that are \emph{induced} by a class of $\V$-presheaves in the following sense.

\begin{definition}[label=presheaf-induced-class]
  A class $\PP$ of $\V$-distributors is \emph{presheaf-induced} if a $\V$-distributor $p \colon \B \lto \A$ is in $\PP$ if and only if, for each $b \in \ob\B$, the $\V$-presheaf $p(1, b)$ is in $\PP$.
\end{definition}

For every class of $\V$-presheaves, there is a corresponding maximal presheaf-induced class of $\V$-distributors.

\begin{theorem}[label=enriched-presheaf-objects]
    Let $\PP$ be a presheaf-induced class of $\V$-distributors.
    A $\V$-category $\A$ admits a $\PP$-presheaf $\V$-category (\cref{Phi-presheaf-object}) if and only if the right lift $q \rf p$ exists for all $\V$-presheaves $p, q \in \PP$ on $\A$. In this case, $\A$ is furthermore $\PP$-admissible (\cref{Phi-embedding}) if and only if the repletion of $\PP$ contains the identity $\V$-distributor on $\A$.
\end{theorem}

\begin{proof}
  First, the existence of the $\PP$-presheaf object implies the existence of the right lifts by \cref{Phi-presheaf-hom-is-right-lift,right-lift-enriched-characterisation}.

  For the converse, assume the mentioned right lifts exist; we construct the $\PP$-presheaf object as follows.
  The $\V$-category $\psh{\PP}{\A}$ has as objects the $\V$-presheaves $p \in \PP$ on $\A$, and hom-cells $\psh{\PP}{\A}(p, q) \defeq q \rf p$.
  To exhibit this $\V$-category as a $\PP$-presheaf object, we define the $\V$-distributor $\pi^\PP_\A \colon \psh{\PP}{\A} \lto \A$ by $\pi^\PP_\A(a, p) \defeq p(a)$.
  This $\V$-distributor is dense:
  \[
    \psh{\PP}{\A}(p, q) = (\pi^\PP_\A(1, q)) \rf (\pi^\PP_\A (1, p))
    = (\pi^\PP_\A \rf \pi^\PP_\A)(p, q)
  \]
  To show that this is a presheaf object, we verify conditions (1) and (2) of \cref{Phi-presheaf-object}.
  \begin{enumerate}
    \item Let $p \colon \b X \lto \A$ be a loose-cell in $\PP$, \ie{} a $\V$-distributor such that, for every $x \in \ob{\b X}$, the presheaf $p({-}, x)$ is in $\PP$.
      A $\V$-functor $\pshm p \colon \b X \to \psh{\PP}{\A}$ satisfies
      $p = \pi^{\PP}_\A(1, \pshm p)$ when the presheaves $p({-}, x)$ and $\pi^{\PP}_\A({-}, \ob{\pshm p} x) = \ob{\pshm p} x$ are equal and ${\pshm p}_{x, y} \colon \b X(x,y) \tto \psh{\PP}{\A}(\ob{\pshm p} x, \ob{\pshm p} y) = p({-}, y) \rf p({-}, x)$ is the unique factorisation of the action $\{p(a, x), \b X(x, y) \tto p(a, y)\}_{a \in \ob{\A}}$ of $\b X$ on $p$ through the counit of the right lift $p({-}, y) \rf p({-}, x)$.
      Existence and uniqueness of such a $\V$-functor therefore follow from the universal property of these right lifts.
    \item For every $\V$-functor $f \colon \b X \to \psh{\PP}{\A}$, the restriction $\pi^\PP_\A(1, f) \colon \b X \lto \A$ satisfies $\pi^\PP_\A(1, f)({-}, x) = \ob{f} x \in \ob{\psh{\PP}{\A}}$, and hence is in $\PP$.
  \end{enumerate}
  The characterisation of $\PP$-admissibility is by definition.
\end{proof}

\begin{corollary}[label=sub-V-presheaf-category]
  Let $\PP$ be a presheaf-induced class of $\V$-distributors, let $\A$ be a $\V$-category admitting a $\PP$-presheaf $\V$-category, and let $i \colon \A' \ffto \psh\PP\A$ be a full sub-$\V$-category. Denote by $\PP'$ the class of $\V$-distributors induced by the class of $\V$-presheaves $\ob{\A'}$. Then $\pi^\PP_\A(1, i) \colon \A' \lto \A$ exhibits $\A'$ as a $\PP'$-presheaf object for $\A$, and $\A' \ffto \psh\PP\A$ exhibits the $\V$-functor $\psh{\PP'}\A \ffto \psh\PP\A$ induced by $\PP' \subseteq \PP$ (\cref{subpresheaf-object}).
\end{corollary}

\begin{proof}
  Since $\PP' \subseteq \PP$, the $\V$-category $\A$ admits a $\PP'$-presheaf object by \cref{enriched-presheaf-objects}, and the stated properties follow by construction.
\end{proof}

\begin{theorem}[label=small-presheaf-object]
  Assume that $\V$ is locally complete and has right lifts.
  Every small $\V$-category $\A$ admits a $\PP$-presheaf $\V$-category, for every presheaf-induced class $\PP$ of $\V$-distributors.
\end{theorem}

\begin{proof}
  Immediate from \cref{small-codomain-right-lifts-exist,enriched-presheaf-objects}.
\end{proof}

\subsection{Enriched cocompletions}

We are already in a position to directly instantiate our main result (\cref{presheaf-object-is-cocompletion}) to construct cocompletions for certain classes of weights in $\VCat$, as illustrated by the following example.

\begin{example}[label=small-cocompletion-of-small-category]
  Let $\PP$ be the class of $\V$-distributors with small codomain or empty domain; this is induced by the class of all $\V$-presheaves on small $\V$-categories.
  Assume that $\V$ is locally complete and locally cocomplete and admits left-composites and right lifts.
  Then every small $\V$-category $\A$ is $\PP$-admissible as in \cref{small-presheaf-object}; the $\PP$-presheaf object is the $\V$-category of $\V$-presheaves on $\A$.
  The three conditions of \cref{presheaf-object-is-cocompletion} hold (using our constructions of right lifts and left-composites in \cref{small-codomain-constructions-exist}), and thus
  $\pshe{\PP}{\A} \colon \A \to \psh{\PP}{\A}$ exhibits the cocompletion of $\A$ under small colimits.
  In particular, it not only satisfies the traditional 2-categorical universal property of a cocompletion, it also satisfies our stronger virtual double categorical universal property (\cref{cocompletion}).
\end{example}

In the above example, the class of presheaves needed for the cocompletion in question is obvious.
More generally, however, to construct the $\Phi$-cocompletion of an enriched category, we can inductively construct a colimit-equivalent class $\colimclosure{\Phi}{}$ of $\V$-presheaves (viewed as $\V$-distributors, and hence as classes of unary weights).

\begin{definition}
  \label{closure}
  Let $\Phi$ be a class of weights in $\VCat$.
  If $\A$ is a $\V$-category, then, assuming the following left-composites exist, we write $\colimclosure{\Phi}{\A}$ (the \emph{colimit closure} of $\Phi$) for the smallest replete class of $\V$-presheaves on $\A$ that contains the following.
  \begin{enumerate}
    \item The $\V$-presheaf $\A(1, a)$, for every object $a \in \ob{\A}$.
    \item \label{closure-composite-condition} The $\V$-presheaf $p \odotl q_1 \odotl \cdots \odotl q_n \odotl \C(1, c)$, for every weight $(q_1, \dots, q_n) \colon \C \lcto \B$ in $\Phi$, object $c \in \ob{\C}$, and $\V$-distributor $p \colon \B \lto \A$, such that $p(1, b) \in \colimclosure{\Phi}{\A}$ for all $b \in \ob{\B}$.
  \end{enumerate}
  We will say that \emph{$\colimclosure{\Phi}{\A}$ exists} when the above left-composites exist.
  When $\colimclosure{\Phi}{\A}$ exists for every $\V$-category $\A$, we say that \emph{$\colimclosure{\Phi}{}$ exists}.
\end{definition}

\begin{remark}
  If all of the left-composites above are furthermore loose-composites, then \cref{closure-composite-condition} ensures that $\colimclosure{\Phi}{}$ forms a \emph{family of coverings} in the sense of \textcite{betti1985cocompleteness}.
  \citeauthor{betti1985cocompleteness} notes that families of coverings exhibit free cocompletions (in the 2-categorical sense), under substantially stronger assumptions (see \cref{survey-enriched-cocompletion-for-bicategory-enrichment}).
  Indeed, in this case, $\colimclosure{\Phi}{}$ is the smallest family of coverings that contains $q_1 \odot \cdots \odot q_n(1, c)$ for all $\vec{q} \in \Phi$ and $c \in \ob\C$.
\end{remark}

\begin{remark}
  Suppose we had instead defined $\colimclosure{\Phi}{}$ to be the smallest class of $\V$-distributors containing (1) all representable $\V$-distributors (2) each left-composite $p \odotl q_1 \odotl \cdots \odotl q_n(1, c)$ with $p \in \colimclosure{\Phi}{}$, $\vec q \in \Phi$, and $c \colon \b X \to \C$.
  While this may seem a natural choice (amenable to generalisation beyond \ve s of enriched categories), this class of $\V$-distributors will not, in general, be induced by a class of presheaves. We therefore would not be able to construct a cocompletion via a presheaf object with respect to $\colimclosure{\Phi}{}$.
  \Cref{closure} instead treats the $\V$-categories $\star_v$ (\cref{one-object-enriched-category}) as a collection of `generators' in $\VCat$, which does not immediately generalise to an arbitrary \ve{} (without assuming further data that appears specific to the setting of enriched categories).
\end{remark}

\begin{lemma}
  \label{closure-is-colimit-equivalent}
  Let $\Phi$ be a class of weights in $\VCat$ for which $\colimclosure{\Phi}{}$ exists.
  Then $\Phi \colimequiv \colimclosure{\Phi}{}$.
\end{lemma}

\begin{proof}
  It suffices by \cref{saturation-properties} to show that $\Phi \colimleq \colimclosure{\Phi}{} \colimleq \Phi\satw$.
  For the first inequality, consider a weight $(q_1, \dots, q_n) \colon \A \lcto \B$ in $\Phi$.
  We have $\B(1, b) \in \colimclosure{\Phi}{}$ for all $b \in \ob{\B}$, so the composite
  \[
    \B(1, 1) \odotl q_1 \odotl \cdots \odotl q_n \odotl \A(1, a) \iso q_1 \odotl \cdots \odotl q_n \odotl \A(1, a)
  \]
  exists and is in $\colimclosure{\Phi}{}$.
  By \cref{colimit-weighted-by-presheaves,right-lift-through-left-composite}, $\vec{q}$-colimits are given pointwise by $(q_1 \odotl \cdots \odotl q_n \odotl \A(1, a))$-weighted colimits, so we can conclude that $\Phi \colimleq \colimclosure{\Phi}{}$.

  For the second inequality, we show inductively that the class $\Phi\satw$ contains $\colimclosure{\Phi}{}$: it is replete by \cref{sat-is-replete}, contains the representables by \cref{sat-contains-representables}, and contains the requisite left-composites by \cref{colimit-weighted-by-presheaves,sat-contains-concatenations,sat-contains-left-composites}.
  We therefore have $\colimclosure{\Phi}{} \colimleq \Phi\satw$.
\end{proof}

The following is the main theorem of this section, showing that cocompletions of enriched categories under arbitrary classes of weights may be constructed as presheaf categories with respect to appropriate classes of distributors.

\begin{theorem}[label=closure-is-cocompletion]
  Let $\Phi$ be a class of weights in $\VCat$.
  Assume that $\colimclosure{\Phi}{}$ exists and that every weight $\vec{q} \in \Phi$ is lifting.
  Denote by $\PP$ the class of $\V$-distributors induced by $\colimclosure{\Phi}{}$.
  Every $\V$-category $\A$ is $\PP$-admissible, and the embedding $\pshe{\PP}{\A} \colon \A \to \psh{\PP}{\A}$ exhibits the $\Phi$-cocompletion of $\A$.
\end{theorem}
\begin{proof}
  We first prove that every $p \in \colimclosure{\Phi}{}$ is lifting.
  This is the case for the representables $\A(1, a)$ because $r \rf \A(1, a) \iso r(a, 1)$ by \cref{restriction-of-lift}.
  It is also the case for $p \odotl q_1 \odotl \cdots \odotl q_n\odotl C(1, c)$, because
  \begin{align*}
    r \rf (p \odotl q_1 \odotl \cdots \odotl q_n\odotl C(1, c))
    &\iso\tag{\cref{right-lift-through-left-composite}}
    r \rf (p, q_1, \dots, q_n, C(1, c))
    \\&\iso\tag{\cref{iterated-lift}}
    r \rf p \rf (q_1, \dots, q_n) \rf C(1, c)
    \\&\iso\tag{\cref{restriction-of-lift}}
    ((r \rf p) \rf \vec{q})(c, 1)
  \end{align*}

  It follows from \cref{enriched-presheaf-objects} that $\A$ is $\PP$-admissible.
  To show that $\pshe{\PP}{\A} \colon \A \to \psh{\PP}{\A}$ is the $\Phi$-cocompletion, we verify the conditions of \cref{presheaf-object-is-cocompletion}.
  \begin{enumerate}
    \item The closure under the requisite left-composites is, by \cref{composite-enriched-characterisation}, an immediate consequence of the definition of $\colimclosure{\Phi}{}$.
    \item $\pi_{\A}$ is in the colimit saturation of $\Phi$ by \cref{closure-is-colimit-equivalent}, since $\pi_{\A} \in \PP$.
    \item To show that right lifts $r \rf \pi_{\A}$ exist, it is enough by \cref{right-lift-enriched-characterisation} to show that $r \rf \pi_{\A}(1, p)$ exists for all $p \in \ob{\psh{\PP}{\A}}$; this is the case because $\pi_{\A}(1, p)$ is in $\PP$ and hence in $\colimclosure{\Phi}{}$.
    \qedhere
  \end{enumerate}
\end{proof}

\begin{corollary}[label=enriched-cocompletion]
  Let $\V$ be a locally complete and locally cocomplete bicategory with right lifts.
  If $\Phi$ is a class of weights $(p_1, \ldots, p_n)$ for which the codomain of every $p_i$ is small, then every $\V$-category $\A$ is $\Phi$-cocompletable.
\end{corollary}

\begin{proof}
  Follows from \cref{closure-is-cocompletion}.
  The required right lifts, and $\colimclosure{\Phi}{}$, exist by \cref{small-codomain-constructions-exist}.
\end{proof}

\begin{example}
  We showed in \cref{small-cocompletion-of-small-category} that, under our usual assumptions on $\V$, every small $\V$-category has a free cocompletion under small colimits.
  \Cref{enriched-cocompletion} further implies that \emph{every} $\V$-category has such a cocompletion.
\end{example}

In previous works on enriched cocompletions (\eg~\cite[Theorem~5.35]{kelly1982basic}), the $\Phi$-cocompletion of an enriched category is constructed by taking a full sub-$\V$-category of a presheaf $\V$-category (often the $\V$-category of \emph{small} presheaves, \cf{} \cref{small-presheaf-section} below). The following lemma shows that this is essentially the same as our construction involving $\colimclosure{\Phi}{\A}$. Note, however, that our definition of colimit closure does not rely on the existence of any particular presheaf object in $\VCat$.

\begin{proposition}[label=colimclosure-as-full-subcategory]
  Let $\PP$ be a replete class of $\V$-distributors and let $\A$ be a $\PP$-admissible $\V$-category. Assume that the left-composite $q \odotl \vec{p}$ exists and is in $\PP$ for all $q \in \PP$ and $\vec{p} \in \Phi$, for a class of weights $\Phi$.
  Then $\colimclosure{\Phi}{\A}$ exists. Denote by $\QQ$ the class of $\V$-distributors induced by $\colimclosure{\Phi}{\A}$. Then $\QQ|_{{}\lto\A} \subseteq \PP|_{{}\lto\A}$, $\A$ is $\QQ$-admissible, and the induced $\V$-functor $\psh{\QQ}{\A} \ffto \psh\PP\A$ exhibits the smallest full sub-$\V$-category of $\psh{\PP}{\A}$ that is closed in $\psh{\PP}{\A}$ under $\Phi$-colimits and representable $\V$-presheaves on~$\A$.
\end{proposition}

\begin{proof}
  First, the existence of the assumed left-composites implies those necessary for the existence of $\colimclosure{\Phi}{\A}$; and the class of $\V$-presheaves in $\PP$ satisfies the two conditions of \cref{closure}, so that $\QQ|_{{}\lto\A} \subseteq \PP|_{{}\lto\A}$.
  $\QQ$-admissibility of $\A$ then follows by \cref{sub-V-presheaf-category}. For the universal property, note that $\Phi$-colimits in $\psh{\PP}{\A}$ are $\pshe{\PP}{\A}$-absolute, and are given by the assumed left-composites (\cref{presheaf-colimits-absolute}).
  The result follows by using \cref{colimit-weighted-by-presheaves}.
\end{proof}

The construction of $\Phi$-cocompletions using presheaf categories in \cref{closure-is-cocompletion} is highly general: it holds not merely for classes of weights with small codomain, but more generally for classes of large weights (subject to existence of the requisite presheaf category). The following is an example involving large colimits.

\begin{example}[label=petty-presheaves]
  Consider $\V = \Sigma\Set$, so that a $\V$-category is simply a locally small category.
  A presheaf $p \colon \A\op \to \Set$ on a (possibly large) category $\A$ is weakly multirepresentable~\cite[\S1.1]{arkor2024adjoint} (\aka~\emph{proper}~\cite[\S1]{isbell1960adequate} or \emph{petty}~\cite[\S1]{freyd1968several}) when there exists an epimorphism $\sum_{i \in I} \A(1, a_i) \epito p$ in $\Set^{\A\op}$, for some small family $\{a_i \in \ob\A\}_{i \in I}$ of objects of $\A$. In particular, every small presheaf, being expressible as a coequaliser of a coproduct of representable presheaves, is weakly multirepresentable.
  \Textcite[Remark~4.39]{lack2023virtual} observe that the (locally small) category $\psh{\Phi}{\A}$ of weakly multirepresentable presheaves satisfies the 2-categorical universal property of the cocompletion of $\A$ under small colimits and (possibly large) cointersections of regular epimorphisms (\cf~\cite{kelly1981large}).
  We can apply our \cref{closure-is-cocompletion} to this case (although not \cref{enriched-cocompletion}, since the weights in question have large codomain), showing that the weakly multirepresentable presheaves satisfy our stronger definition of cocompletion. This demonstrates that our framework is an effective setting in which to study even large cocompletions (other such examples are given in \cref{weakly-representable-presheaves,topological-functors}).

  First, note that every weakly multirepresentable presheaf $p$ is lifting.
  Indeed, $\sum_{i \in I} \A(1, a_i)$ is lifting, with $q \rf \sum_{i \in I} \A(1, a_i) \iso \prod_{i \in I}(q \rf \A(1, a_i)) \iso \prod_{i \in I} q(a_i, 1)$, so $q \rf p$ is given by a sub-copresheaf of $\prod_{i \in I} q(a_i, 1)$.

  In any category, the cointersection of a $J$-indexed family of regular epimorphisms is equivalently the colimit of a diagram of shape $\W_J$, where the category $\W_J$ consists of an object $x$, and, for each $j \in J$, an object $w_j$ along with a parallel pair of morphisms $\ell_j, r_j \colon w_j \to x$.
  Thus, there is a class of weights $\Phi$ for small colimits and cointersections of regular epimorphisms, consisting of all presheaves on small categories, as well as the conical $\W_J$-indexed weights (\ie{} presheaves $u_J \colon \star \lto \W_J$, where $u_J$ sends every object of $\W_J$ to a singleton set).
  The class $\PP \defeq \colimclosure{\Phi}{}$ then consists of all weakly multirepresentable presheaves (for instance, by \cref{colimclosure-as-full-subcategory}), and it follows from \cref{closure-is-cocompletion} that the category $\A$ is $\PP$-admissible and that the embedding $\pshe{\PP}{\A} \colon \A \to \psh{\PP}{\A}$ exhibits the $\Phi$-cocompletion of $\A$.
\end{example}

\begin{example}[label=weakly-representable-presheaves]
  For a closely related example to \cref{petty-presheaves}, take $\Phi$ to be the class of (conical) weights for coequalisers and for cointersections of regular epimorphisms (but, this time, without arbitrary small weights). A presheaf $p$ on a (possibly large) category $\A$ is weakly representable when there exists an epimorphism $\A(1, a) \epito p$ in $\Set^{\A\op}$ for some object $a \in \ob\A$. In particular, representable presheaves, but not arbitrary small presheaves, are weakly representable.
  Much as in \cref{petty-presheaves}, the class $\PP \defeq \colimclosure{\Phi}{}$ consists of all weakly representable presheaves, and it follows from \cref{closure-is-cocompletion} that the category $\A$ is $\PP$-admissible and that the embedding $\pshe{\PP}{\A} \colon \A \to \psh{\PP}{\A}$ exhibits the $\Phi$-cocompletion of $\A$.
\end{example}

We may, of course, instantiate our properties of free cocompletions (\cref{cocompleteness-in-a-virtual-equipment}) in the setting of enriched categories.
In particular, we can recover \citeauthor{szlachanyi2017tensor}'s characterisation of cocompletions of small $\V$-categories as well-behaved functors, providing a detection result for cocompletions.

\begin{corollary}[label=well-behaved-is-enriched-cocompletion, note={\cf~\cite[\S8]{szlachanyi2017tensor}}]
  Suppose that $\V$ is locally complete and admits right lifts, and let $j \colon \A \to \b E$ be a $\V$-functor, with $\A$ small. Then $j$ is well-behaved (\cref{well-behaved}) if and only if there is a class $\Phi$ of weights for which $j$ exhibits $\b E$ as the free cocompletion of the $\V$-category $\A$ under $\Phi$-weighted colimits.
\end{corollary}

\begin{proof}
  Immediate from \cref{cocompletion-via-well-behavedness,small-codomain-right-lifts-exist}.
\end{proof}

\begin{remark}
  \textcite[Theorem~7.8]{lucyshyn2016enriched} establishes a similar result to \cref{well-behaved-is-enriched-cocompletion} in the setting of enrichment in a symmetric closed monoidal category $\b V$, but specifically in which the codomain of $j$ is $\b V$ itself. \citeauthor{lucyshyn2016enriched} does not impose a size constraint, but instead assumes the existence of potentially large colimits: namely, left extensions of $\b V$-functors $\b A \to \b V$ along $j$. This amounts to the assumption in \cref{cocompletion-via-well-behavedness} that $\b E(j, 1)$ is lifting.
\end{remark}

Finally, we note that a $\Phi$-cocompletion of a small $\V$-category is often small itself.

\begin{corollary}
  In addition to the assumptions of \cref{enriched-cocompletion}, assume that $\V$ has locally small hom-categories, that $\Phi$ is a small set, and that each weight in $\Phi$ has small domain and codomain.
  Then every (extentwise) small $\A$ admits an (extentwise) small $\Phi$-cocompletion.
\end{corollary}

\begin{proof}
  We prove that, if $\A$ is small, then the $\Phi$-cocompletion is equivalent to a small $\V$-category.
  The same proof works when smallness is replaced by extentwise smallness.

  Since $\V$ has locally small hom-categories, for every small $\V$-category $\B$, a small class of $\V$\nbh{}presheaves on $\A$ induces a small class of $\V$-distributors $p \colon \B \lto \A$.
  Thus an inductive argument shows that, for every small $\A$, the restriction of $\colimclosure{\Phi}{}$ to $\V$-presheaves on $\A$ is the repletion of a small set of $\V$-presheaves.
  By our construction of the cocompletion, the latter restriction is the class of objects of $\coc{\Phi}{\A}$, which thus is equivalent to a small $\V$-category.
\end{proof}

\subsection{Small presheaves}\label{small-presheaf-section}
One consequence of \cref{enriched-cocompletion} is that the small cocompletion (where $\Phi = \{\text{ $\V$-presheaves on small $\V$-categories }\}$) of a general $\V$-category $\A$ exists under mild assumptions on $\V$.
It is known that this cocompletion is given by the $\V$-category of \emph{small} $\V$-presheaves on $\A$ (\cf~\cite[Section~5.6]{kelly1982basic}), though \cref{enriched-cocompletion} does not directly give this characterisation.
More generally, for any subclass $\Phi' \subseteq \Phi$, the $\Phi'$-cocompletion is given by a sub-$\V$-category of the $\V$-category of small $\V$-presheaves.
We have so far avoided any mention of small presheaves, but for completeness we discuss them here.

\begin{definition}[label=small-presheaf]
  Let $\A$ be a $\V$-category.
  A $\V$-presheaf $p$ on $\A$ is \emph{small} when there exists a full sub-$\V$-category $i \colon \A' \ffto \A$ with small domain such that the $\V$-presheaf $p(i)$ is lifting, and such that, for every $a \in \ob{\A}$, the $\V$-natural family $\{\A(a, \ob{i}a'), p(\ob{i}a') \tto p(a)\}_{a' \in \ob{\A'}}$ formed by the left-action of $p$ is left-opcartesian in $\VCat$:
  \[
    p(a) \iso \A(a, i) \odotl p(i)
  \]
  A weight $(p_1, \dots, p_n) \colon \A_n \lcto \A_0$ in $\VCat$ is \emph{colimit-small} when the $\V$-presheaf $p_k(1, a_k)$ is small for every $1 \leq k \leq n$ and $a_k \in \ob{\A_k}$.
\end{definition}

\begin{remark}
  \Cref{composite-enriched-remark} implies that the $\V$-presheaf $p$ in \cref{small-presheaf} is the left-composite $\A(1, i) \odotl p(i)$.
\end{remark}

\begin{remark}
  We say \emph{colimit-small} instead of just \emph{small} because, while we use the same notion of weight for both colimits and limits, colimit-smallness does not coincide with \emph{limit-smallness} (which is colimit-smallness in $\VCat\co \iso \VcoCat$; see \cref{completions-of-enriched-categories}).
\end{remark}

\begin{remark}[label=small-flexibility]
  Small $\V$-presheaves are called \emph{accessible} in \cite{kelly1982basic} (for $\V$ a symmetric closed monoidal category), while our notion of colimit-small weight is similar to \citeauthor{shulman2013enriched}'s notion of \emph{$\kappa$-small profunctor} between enriched indexed categories~\cite[Definition~10.1]{shulman2013enriched}.
  Neither requires $i$ to be \ff{}, but both authors (\cite[Proposition~4.83]{kelly1982basic} and \cite[Remark~10.6]{shulman2013enriched}) show that one can always replace a non-\ff{} $i$ with a \ff{} $i$, so that this difference is irrelevant.
  \Citeauthor{shulman2013enriched}'s proof makes it clear that this fact relies on existence of certain composites; since we do not generally have these composites, we directly require that $i$ is \ff{} (for the proof of \cref{colimit-small-composite-lemma} below).
  Moreover, a direct translation of their definitions into our setting would only require the existence of \emph{some} $p'$ (not necessarily $p(i)$) such that $p \iso A(1, i) \odotl p'$.
  Again this distinction is irrelevant (\cf{} \cite[Proposition~4.83]{kelly1982basic}), since, if $i$ is \ff{}, then:
  \[
    p' \iso (\A(i, i) \odotl p')
    \iso (\A(1, i) \odotl p')(i)
    \iso p(i)
  \]
  Neither \citeauthor{kelly1982basic} nor \citeauthor{shulman2013enriched} has an explicit lifting assumption, but nor do they need one: since $p(i)$ has small codomain, it is automatically lifting when $\V$ is locally complete and has right lifts, by \cref{small-codomain-right-lifts-exist}.
\end{remark}

The main reason for considering small $\V$-presheaves is the fact that they are lifting, and hence that their presheaf-objects exist.

\begin{lemma}[label=small-right-lifts-exist]
  Every small $\V$-presheaf, and every colimit-small weight, is lifting.
  In particular, if $\PP$ is a class of $\V$-distributors induced by a class of small $\V$-presheaves, then every $\V$-category $\A$ admits a $\PP$-presheaf object.
\end{lemma}

\begin{proof}
  If $p \iso \A(1, i) \odotl p(i)$ is a small $\V$-presheaf, with $p(i)$ lifting, then we have $q(i) \rf p(i) \iso (q \rf \A(1, i)) \rf p(i) \iso q \rf (\A(1, i) \odotl p(i)) \iso q \rf p$.
  We prove that every colimit-small weight $\vec{p} = (p_1, \dots, p_n)$ is lifting by induction on $n$.
  For $n = 0$ this is trivial.
  For $n > 0$, it is enough to show that $q \rf (p_1, \dots, p_n(1, b))$ exists for all objects $b$ by \cref{right-lift-enriched-characterisation}.
  This is the case because the small $\V$-presheaf $p_n(1, b)$ is lifting, and $q \rf (p_1, \dots, p_{n - 1}) \rf p_n(1, b) \iso q \rf (p_1, \dots, p_{n-1}, p_n(1, b))$ by \cref{iterated-lift}.

  That every $\V$-category admits a $\PP$-presheaf object follows from \cref{enriched-presheaf-objects}.
\end{proof}

\begin{lemma}[label=colimit-small-composite-lemma]
  If a unary weight $p \colon \B \lto \A$ is colimit-small, then, for every small full sub-$\V$-category $j \colon \B' \ffto \B$, there is a small full sub-$\V$-category $i \colon \A' \ffto \A$ such that
  \[
    p(1, j) \iso \A(1, i) \odotl p(i, j)
  \]
  and such that $p(i, j)$ is lifting.
\end{lemma}
\begin{proof}
  For every $b \in \ob{\B'}$, since $p(1, \ob{j}b)$ is a small $\V$-presheaf, there is some small full sub-$\V$-category $i_b \colon \A'_b \ffto \A$ such that $p(1, \ob{j}b) \iso \A(1, i_b) \odotl p(i_b, \ob{j}b)$, and such that $p(i_b, \ob{j}b)$ is lifting.
  Let $i \colon \A' \to \A$ be the union of the $i_b$; this is a small union, so $\A'$ is small.
  \[\textstyle
    \ob{\A'} = \bigcup_{b \in \ob{\B'}} \ob{\A'_b}
  \]
  Since each $i_b$ factors through $i$, the latter witnesses smallness of the $\V$-presheaf $p(1, \ob{j}b)$ for every $b$:
  \begin{align*}
    p(1, \ob{j}b)
    ~\iso~& \A(1, i_b) \odotl p(i_b, \ob{j}b)
    \\~\iso~& (\A(1, i) \odot \A(i, i_b)) \odotl p(i_b, \ob{j}b)
    \\~\iso~& \A(1, i) \odotl (\A(i, i_b) \odotl p(i_b, \ob{j}b))
  \end{align*}
  It follows from \cref{composite-enriched-characterisation} that $p(1, j) \iso \A(1, i) \odotl p(i, j)$.
  To see that $p(i, j)$ is lifting, it is enough by \cref{right-lift-enriched-characterisation} to show that $p(i, \ob{j}b)$ is lifting for all $b$.
  This is the case because $\A(i, i_b)$ is a companion, and is therefore lifting by \cite[Example~3.3]{arkor2024formal}, so we can construct right lifts through
  $p(i, \ob{j}b)$ by the following.
  \[
    (q \rf \A(i, i_b)) \rf p(i_b, \ob{j}b)
    \iso
    q \rf (\A(i, i_b) \odotl p(i_b, \ob{j}b))
    \iso
    q \rf p(i, \ob{j}b)
    \qedhere
  \]
\end{proof}

\begin{remark}
  \Cref{colimit-small-composite-lemma} is an adaptation of \cite[Lemma~10.7]{shulman2013enriched} to our setting.
  Our proof is similar, but subtly different to \citeauthor{shulman2013enriched}'s: since his notion of smallness is a priori more flexible than ours (\cref{small-flexibility}), he takes a \emph{disjoint} union of inclusions, which results in a non-\ff{} functor.
  Like in \cite[Proposition~5.34]{kelly1982basic}, we take a non-disjoint union.
\end{remark}

\begin{lemma}[label=small-presheaf-closure]
  Assume that $\V$-presheaves on small $\V$-categories are lifting.
  The following $\V$-presheaves are all small.
  \begin{enumerate}
    \item Every $\V$-presheaf $p$ on a small $\V$-category.
    \item Every representable $\V$-presheaf $\A(1, a)$ on a (not necessarily small) $\V$-category $\A$.
    \item Every left-composite $q_1 \odotl \cdots \odotl q_m \odotl p$ of a colimit-small weight $\vec{q} = (q_1, \dots, q_m) \colon \A_m \lcto \A_0$ with a small $\V$-presheaf $p$ on $\A_m$.
  \end{enumerate}
  It follows that:
  \begin{enumerate}[resume]
    \item \label{colimit-small-closure-exists} If $\V$ is locally cocomplete and admits left-composites, then for every class $\Phi$ of colimit-small weights, the closure $\colimclosure{\Phi}{}$ exists, and contains only small $\V$-presheaves.
  \end{enumerate}
\end{lemma}
\begin{proof}
  (1) is trivial: we can take $i$ to be the identity in the definition of smallness.

  For (2), let $i \colon \A' \ffto\A$ be the full sub-$\V$-category of $\A$ containing only $a$.
  Then we have $\A(a', a) \iso \A(a', i) \odot \A(i, a)$ as required.

  For (3), we show that there is some $i_0 \colon \A'_0 \ffto \A_0$ with small domain, such that
  \[
    q_1 \odotl \cdots \odotl q_m \odotl p
    \iso
    \A_0(1, i_0) \odotl (q_1 \odotl \cdots \odotl q_m \odotl p)(i_0, 1)
  \]
  and such that $(q_1 \odotl \cdots \odotl q_m \odotl p)(i_0, 1)$ is lifting.
  We have the following, where the full sub-$\V$-categories $i_k$ are provided by \cref{colimit-small-composite-lemma} and smallness of $p$.
  The reassociations below use \cref{left-composite-reassociation-lifting}: the prefixes $(q_1, \ldots, q_k)$ are lifting by \cref{small-right-lifts-exist,iterated-lift}.
  \begin{align*}
    &q_1 \odotl \cdots \odotl q_{m-1} \odotl q_m \odotl p
    \\~\iso~&
    q_1 \odotl \cdots \odotl q_{m-1} \odotl q_m \odotl (\A_m(1, i_m) \odotl p(i_m))
    \\~\iso~&
    q_1 \odotl \cdots \odotl q_{m-1} \odotl q_m \odotl \A_m(1, i_m) \odotl p(i_m)
    \\~\iso~&
    q_1 \odotl \cdots \odotl q_{m-1} \odotl q_m(1, i_m) \odotl p(i_m)
    \\~\iso~&
    q_1 \odotl \cdots \odotl q_{m-1} \odotl (\A_{m-1}(1, i_{m-1}) \odotl q_m(i_{m-1}, i_m)) \odotl p(i_m)
    \\~\iso~&
    q_1 \odotl \cdots \odotl q_{m-1} \odotl \A_{m-1}(1, i_{m-1}) \odotl q_m(i_{m-1}, i_m) \odotl p(i_m)
    \\~\iso~&
    q_1 \odotl \cdots \odotl q_{m-1}(1, i_{m-1}) \odotl q_m(i_{m-1}, i_m) \odotl p(i_m)
    \\~\iso~&
    \cdots
    \\~\iso~&
    (\A_0(1, i_0) \odotl q_1(i_0, i_1)) \odotl \cdots \odotl q_{m-1}(i_{m-2}, i_{m-1}) \odotl q_m(i_{m-1}, i_m) \odotl p(i_m)
    \\~\iso~&
    \A_0(1, i_0) \odotl q_1(i_0, i_1) \odotl \cdots \odotl q_{m-1}(i_{m-2}, i_{m-1}) \odotl q_m(i_{m-1}, i_m) \odotl p(i_m)
  \end{align*}
  By similar reasoning, we can show that
  \begin{align*}
    & \A_0(1, i_0) \odotl q_1(i_0, i_1) \odotl \cdots \odotl q_{m-1}(i_{m-2}, i_{m-1}) \odotl q_m(i_{m-1}, i_m) \odotl p(i_m)
    \\~\iso~&
    \A_0(1, i_0) \odotl q_1(i_0, 1) \odotl q_2 \odotl \cdots \odotl q_m \odotl p
  \end{align*}
  and thus
  \begin{align*}
    &q_1 \odotl \cdots \odotl q_{m-1} \odotl q_m \odotl p
    \\~\iso~&
    \A_0(1, i_0) \odotl q_1(i_0, i_1) \odotl \cdots \odotl q_{m-1}(i_{m - 2}, i_{m-1}) \odotl q_m(i_{m-1}, i_m) \odotl p(i_m)
    \\~\iso~&
    \A_0(1, i_0) \odotl (q_1(i_0, 1) \odotl q_2 \odotl \cdots \odotl q_m \odotl p)
    \\~\iso~&
    \A_0(1, i_0) \odotl (q_1 \odotl q_2 \odotl \cdots \odotl q_m \odotl p)(i_0, 1)
  \end{align*}
  as required.
  The final isomorphism in this chain follows from \cref{left-composite-reassociation-lifting}, using the fact that, since $i_0$ has small domain, the conjoint of $i_0$ is lifting.
  Finally, $q_1(i_0, i_1) \odotl \cdots \odotl q_{m-1}(i_{m - 2}, i_{m-1}) \odotl q_m(i_{m-1}, i_m) \odotl p(i_m)$ is lifting by \cref{right-lift-through-left-composite,iterated-lift}, using the fact that each constituent of the left-composite is lifting.
  Hence $(q_1 \odotl q_2 \odotl \cdots \odotl q_m \odotl p)(i_0, 1)$ is also lifting.

  (4) follows from (2) and (3), noting that in the proof of (3), we reduce the left-composite to a left-composite involving small $\V$-categories, which exists by \cref{small-codomain-composites-exist}.
\end{proof}

\begin{example}
  Let $\Phi = \{\text{ $\V$-presheaves on small $\V$-categories }\}$. We showed in \cref{enriched-cocompletion} that, under our usual assumptions on $\V$, every $\V$-category $\A$ has a $\Phi$-cocompletion, given by the $\colimclosure{\Phi}{}$-embedding of $\A$.
  We can now recover the fact that this cocompletion consists of the small $\V$-presheaves on $\A$.
  Indeed, by \cref{small-presheaf-closure}, $\colimclosure{\Phi}{}$ only contains small $\V$-presheaves, while the same lemma also implies that it contains every composite $\A(1, i) \odot p(i)$ where $i$ has small domain, and hence contains every small $\V$-presheaf.
\end{example}

\begin{remark}[label=composite-as-colimit]
  In previous works on enriched cocompletions (\eg~\cite[Theorem~5.35]{kelly1982basic}), the small presheaves play a larger role than we have given them here.
  \citeauthor{kelly1982basic} first constructs the $\V$-category of small $\V$-presheaves, shows that it admits small colimits, then constructs the $\Phi$-cocompletion by taking the smallest full sub-$\V$-category closed under the representables and under $\Phi$-colimits (see \cref{colimclosure-as-full-subcategory}).
  We avoid any mention of small $\V$-presheaves in our definition by working directly with the composites appearing in $\colimclosure{\Phi}{}$ (\cref{closure}), rather than the corresponding colimits in the presheaf $\V$-category.
\end{remark}

The above results enable us to apply \cref{closure-is-cocompletion} to construct $\Phi$-cocompletions when $\Phi$ contains only colimit-small weights, assuming the required left-composites exist.
As a result, we have the following theorem, which summarises the results above in the situation of enrichment in a nice bicategory.

\begin{theorem}[label=cocompletion-in-nice-situation]
  Let $\V$ be a locally complete and locally cocomplete bicategory with right lifts, and let $\Phi$ be a class of colimit-small weights.
  For every $\V$-category $\A$ (not assumed to be small), the free cocompletion of $\A$ under $\Phi$-weighted colimits exists, and is given by the smallest sub-$\V$-category of the $\V$-category of small presheaves on $\A$ that is closed under $\Phi$-weighted colimits and representable presheaves.
\end{theorem}

\begin{proof}
  Existence of the $\Phi$-cocompletion follows from \cref{closure-is-cocompletion}, using \cref{colimit-small-closure-exists} to establish existence of the closure, and \cref{small-right-lifts-exist} for the lifting assumption.
  The description of the $\Phi$-cocompletion as a closure is immediate from \cref{colimclosure-as-full-subcategory}.
\end{proof}

\subsection{Completions of enriched categories}
\label{completions-of-enriched-categories}

Following \cref{copresheaf-objects-and-completions}, the theory of copresheaf objects and completions in $\VCat$ is the theory of presheaf objects and cocompletions in $\VCat\co$. The following observation, which generalises \cite[Proposition~8.22]{arkor2024formal} from monoidal categories to \vbs{}\footnotemark, means that it is unnecessary to study $\VCat\co$ specifically, as it is implicitly included in the preceding development.
\footnotetext{In fact, it holds more generally for categories enriched in \vdcs.}%

\begin{proposition}[label=V^co-Cat]
  For each \vb{} $\V$, there is an isomorphism of \ve s $\VCat\co \iso \VcoCat$.
\end{proposition}

\begin{proof}
  By inspection.
\end{proof}

Consequently, copresheaf objects and free completions in $\VCat$ are respectively presheaf objects and free cocompletions in $\VcoCat$. We therefore immediately obtain construction theorems therefor; for convenience, we spell out the main definitions and results.

\begin{manualdefinition}{\ref{presheaf-induced-class}$\co$}
  A class $\PP$ of $\V$-distributors is \emph{copresheaf-induced} if a $\V$-distributor $p \colon \A \lto \B$ is in $\PP$ if and only if, for each $b \in \ob\B$, the $\V$-copresheaf $p(b, 1)$ is in $\PP$.
\end{manualdefinition}

\begin{manualtheorem}{\ref{enriched-presheaf-objects}$\co$}
  Let $\PP$ be a copresheaf-induced class of $\V$-distributors.
  A $\V$-category $\A$ admits a $\PP$-copresheaf $\V$-category if and only if the right extension $p \rx q$ exists for all $\V$-copresheaves $p, q \in \PP$ on $\A$. In this case, $\A$ is furthermore $\PP$-coadmissible if and only if the repletion of $\PP$ contains the identity $\V$-distributor on $\A$.
\end{manualtheorem}

\begin{manualtheorem}{\ref{small-presheaf-object}$\co$}
  \label{enriched-copresheaf-objects}
  Assume that $\V$ is locally complete and has right extensions. Every small $\V$-category $\A$ admits a $\PP$-copresheaf $\V$-category, for every copresheaf-induced class $\PP$ of $\V$-distributors.
\end{manualtheorem}

\begin{manualdefinition}{\ref{closure}$\co$}
  Let $\Psi$ be a class of weights in $\VCat$.
  If $\A$ is a $\V$-category, then, assuming the following right-composites exist, we write $\limclosure{\Psi}{\A}$ (the \emph{limit closure} of $\Psi$) for the smallest replete class of $\V$-copresheaves on $\A$ that contains the following.
  \begin{enumerate}
    \item The $\V$-copresheaf $\A(a, 1)$, for every object $a \in \ob{\A}$.
    \item The $\V$-copresheaf $C(c, 1) \odotr q_1 \odotr \cdots \odotr q_n \odotr p$, for every weight $(q_1, \dots, q_n) \colon \B \lcto \C$ in $\Psi$, object $c \in \ob{\C}$, and $\V$-distributor $p \colon \A \lto \B$, such that $p(b, 1) \in \limclosure{\Psi}{\A}$ for all $b \in \ob{\B}$.
  \end{enumerate}
  We will say that \emph{$\limclosure{\Psi}{\A}$ exists} when the above right-composites exist.
  When $\limclosure{\Psi}{\A}$ exists for every $\V$-category $\A$, we say that \emph{$\limclosure{\Psi}{}$ exists}.
\end{manualdefinition}

\begin{manualtheorem}{\ref{closure-is-cocompletion}$\co$}
  \label{closure-is-completion}
  Let $\Psi$ be a class of weights in $\VCat$.
  Assume that $\limclosure{\Psi}{}$ exists and that every weight $\vec{q} \in \Psi$ is extending.
  Denote by $\PP$ the class of $\V$-distributors induced by $\limclosure{\Psi}{}$.
  Every $\V$-category $\A$ is $\PP$-coadmissible, and the embedding $\cpshe{\PP}{\A} \colon \A \to \cpsh{\PP}{\A}$ exhibits the $\Psi$-completion of $\A$.
\end{manualtheorem}

\begin{manualtheorem}{\ref{cocompletion-in-nice-situation}$\co$}
  \label{completion-in-nice-situation}
  Let $\V$ be a locally complete and locally cocomplete bicategory with right extensions and let $\Psi$ be a class of limit-small weights. For every $\V$-category $\A$ (not assumed to be small), the free completion of $\A$ under $\Psi$-weighted limits exists, and is given by the smallest sub-$\V$-category of the $\V$-category of small copresheaves on $\A$ that is closed under $\Psi$-weighted limits and corepresentable copresheaves.
\end{manualtheorem}

\section{Context and examples}
\label{examples}

We conclude by (1) explaining the relationship between our construction of free cocompletions of enriched categories and those appearing in the literature, (2) describing some specific classes of weights of interest in bicategory-enriched category theory, and (3) illustrating these by means of examples of categories enriched in certain bicategories.

\subsection{Historical context}
\label{historical-context}

The possibility of freely completing a category under a class $\J$ of colimits was considered early in the development of category theory. However, it took some time before a general construction was realised; typically, the class $\J$ was subject to certain compatibility conditions, or the construction was limited to small categories. Below, we provide a survey of constructions that appeared in the literature.

\begin{example}[Ordinary categories]
  First, we consider the case of the cocompletion of a locally small category $\A$ (\ie{} a $\Set$-enriched category) under a class $\J$ of small categories.
  \begin{center}
  \begin{tblr}{X[c]cX[c]}
    \bf Class $\J$ & \bf $\A$ can be large? & \bf Reference \\
    \hline
    Small categories & No & \cite[Proposition~9.1]{andre1966categories} \\
    Small categories & Yes & \cite[Remark~2.29]{ulmer1968properties} \\
    Any class \emph{admitting a calculus of regular prelimits}~\cite[Definition~8.1]{kock1967limit} & Yes & \cite[\S8]{kock1967limit} \\
    Any of (a) idempotents; (b) $\alpha$-small categories; (c) small categories & No & \cite[Beispiele~2.14]{gabriel1971lokal} \\
    Any class of small \emph{stable} graphs~\cite[Definition~2.1]{wood1977free} & No & \cite[Theorem~2.3]{wood1977free} \\
    Any class of small categories & No & \cites{street1979comprehensive}[\S9.6]{street1981conspectus} \\
    Any \emph{admissible} class of small categories~\cite[600]{tholen1983completions} & Yes & \cite[Theorems~2 \& 4]{tholen1983completions}
  \end{tblr}
  \end{center}
  As far as we are aware, a general construction of the free cocompletion of a not-necessarily-small category under colimits indexed by an arbitrary class $\J$ of small categories did not appear before it was described in the more general setting of enriched categories by \textcite{kelly1982basic} (see \cref{survey-enriched-cocompletions}).
\end{example}

\begin{example}
  \label{survey-enriched-cocompletions}
  In the case of enrichment in a complete and cocomplete locally small symmetric closed monoidal category $\b V$:
  \begin{enumerate}
    \item the description of the free cocompletion of a large $\b V$-category as a $\b V$-category of small presheaves appears as \cite[Theorem~2.11]{lindner1974morita};
    \item for an arbitrary class $\Phi$ of weights with small codomain, the $\Phi$-cocompletion of a large $\b V$-category is constructed in \cite[Theorem~5.35]{kelly1982basic}.
  \end{enumerate}
  In the non-symmetric setting, a claim that free cocompletions exist appears without proof in \cite[Theorem~6]{power2000representation}. However, we are not aware of an explicit reference. It is worth mentioning that, in the setting of bicategories enriched in (non-symmetric) monoidal bicategories, a construction of the free cocompletion under a class of weights appears as \cite[Theorem~12.1]{garner2016enriched}, from which the result for categories enriched in monoidal categories may be deduced.
\end{example}

\begin{example}
  In the special case of enrichment in a quantaloid -- \ie{} a bicategory enriched in the 2-category of complete lattices~\cite{rosenthal1991free} -- free cocompletions of potentially large enriched categories, under saturated classes of weights, are constructed in \cite[\S4]{stubbe2010hausdorff}.
\end{example}

\begin{example}[label=survey-enriched-cocompletion-for-bicategory-enrichment]
  In the case of enrichment in a locally complete and locally cocomplete closed bicategory $\V$:
  \begin{enumerate}
    \item the free cocompletion of a small\footnotemark{} $\V$-category appears in \cite[\S4]{street1983enriched};
    \footnotetext{Let us say that a $\V$-category $\A$ is \emph{loosely small} (or \emph{Morita small}~\cite[Remark~8.6]{arkor2025nerve}) when it is equivalent, in the \vb{} of $\V$-distributors, to a small $\V$-category. Strictly speaking, \textcite[\S5]{street1983enriched} considers loosely small, rather than small, $\V$-categories, but this is an inessential difference.}
    \item the free cocompletion of a small $\V$-category under a \emph{family of coverings} -- \ie{} a class $\Phi$ of weights (meaning, for each $\V$-category, a set of presheaves thereon) satisfying a saturation condition -- appears in \cite[{p.\ [4]}]{betti1985cocompleteness}.
  \end{enumerate}
  As far as we are aware, neither the existence of free cocompletions of large $\V$-categories, nor of free cocompletions of (even small) $\V$-categories under arbitrary classes of weights, have previously been established in the literature.
\end{example}

\begin{example}
  Let $\b V$ be a unital multicategory. Loose-composites in $\Sigma\b V$ correspond to tensor products, and right lifts and right extensions to (left and right) internal homs. Some basic properties of the theory of presheaves enriched in (the envelope of) a multicategory were developed by \textcite{linton1971multilinear}. However, free cocompletions have not previously been considered.

  In particular, (non-monoidal) closed categories~\cite{eilenberg1966closed}, unital promonoidal categories~\cite{day1970closed}, and normal colax monoidal categories~\cite{day2003lax} are all examples of normal multicategories, for which our theorem gives constructions of free cocompletions, each of which appears to be new.
\end{example}

\subsection{Specific classes of weights}

We mention some examples of free cocompletions under specific classes of weights that have proven of interest in the literature, which are subsumed by the general construction above. First, we extend the notion of copower for categories enriched in bicategories~\cite[Definition~2.5]{gordon1999gabriel} to categories enriched in normal \vbs, which will be relevant in several of the examples.

\begin{definition}[Copowers]
  \label{copowers}
  The \emph{copower} of an object $a \in \ob\A$ in a $\V$-enriched category by a chain of loose-cells $(p_1, \ldots, p_n) \colon V_n \lcto V_0$ in $\V$ is a $(\star_{p_1}, \ldots, \star_{p_n})$-weighted colimit of $a \colon \star_{V_0} \to \A$ (where $\star_{p_i}$ is defined in \cref{one-object-enriched-category}).
\end{definition}

\begin{example}
  \begin{itemize}
    \item Taking $\Phi$ to be any class contained in the class of representables, the $\Phi$-cocompletion is (up to equivalence) simply the identity.
    \item Taking $\Phi$ to be the class of all weights $(p_1, \ldots, p_n)$ for which the codomain of every $p_i$ is small, the $\Phi$-cocompletion is simply the free cocompletion.
    \item \label{Cauchy-completion} Taking $\Phi$ to be the class of left adjoints in $\VCat$, the $\Phi$-cocompletion is the \emph{Cauchy completion} of \cites[4]{betti1982cauchy}[\S4]{street1983enriched}.
    \item \label{torsors} Given a notion of coverage on a bicategory $\V$, \citeauthor{street1983enriched} defines a \emph{torsor} to be an absolute weight satisfying a compatibility condition with the coverage~\cite[\S5]{street1983enriched}, and a \emph{stack} to be a $\V$-category admitting torsor-weighted colimits. Taking $\Phi$ to be the class of torsors, the $\Phi$-cocompletion recovers the stack associated to a $\V$-category~\cite[Theorem~5.4]{street1983enriched}.
    \item \label{absolute-copowers} Taking $\Phi$ to be the class of copowers by left adjoints (\aka{} \emph{maps}) in $\V$, the $\Phi$-cocompletion is the \emph{map-tensor\footnotemark{} completion} of \cite[\S5]{street1983enriched}.
    \footnotetext{\citeauthor{street1983enriched} uses \emph{tensor} for what we call \emph{copower}.}%
    \item Taking $\Phi$ to be the class of copowers by representable loose-cells in the \vb{} $\V$ underlying a \ve{}, the $\Phi$-cocompletion is the \emph{substitution completion} of \cite[\S2.4]{betti1987completeness} (\aka{} the \emph{restriction completion} of \cite[\S4]{betti1985closed}).
    \item Let $\b J$ be an ordinary category and let $V$ be an object of $\V$.
      Suppose that $\V$ admits local copowers of loose-cells $V \lto V$.
      (When $\V$ is a bicategory, this means that the ordinary category of loose-cells $V \lto V$ admits copowers, and that these are preserved by composition of loose-cells on both sides.)
      There is a $\V$-category $\b J_V$ whose objects are those of $\b J$ and that all have extent $V$. The hom-morphisms are given by $\b J_V(j, j') \defeq \b J(j, j') \copow V(1, 1)$.
    A \emph{conical colimit} of $\b J$ is a colimit weighted by the $\V$-distributor $\star_V \lto \b J_V$ each of whose hom-morphisms is $V(1, 1)$. When $\V$ is a bicategory, this recovers the notion of conical colimit considered in \cite[\S4]{gordon1999gabriel}. In particular, when $\b J$ is discrete, such conical colimits are coproducts of objects with extent $V$~\cite[\S3.6]{betti1985cocompleteness}.
    \qedhere
  \end{itemize}
\end{example}

\subsection{Examples of free cocompletions}

To indicate our motivation for providing a general construction of free (co)completions for categories enriched in (virtual) bicategories, we recall some notable examples of (co)completions with regard to specific bases of enrichment.

\begin{example}
  \label{Rel}
  Let $\C$ be a small category and denote by $\b R(\C)$ the full sub-bicategory of $\b{Rel}(\Set^{\C\op})$ spanned by the representable presheaves. Coverages $J$ on $\C$ are in bijection with locally finitely continuous and \ioo{} lax endofunctors on $\b R(\C)$ whose action on each hom-category is an idempotent monad~\cite[Theorem]{betti1983notion}. Denote by $\b R(\C, J)$ the bicategory of $J$-algebras. The category of sheaves with respect to $(\C, J)$ is equivalent to the category of skeletal symmetric $\b R(\C, J)$-enriched categories admitting colimits weighted by left adjoints (\cref{Cauchy-completion})~\cite[Theorem]{walters1982sheaves}.

  Denoting by $J_0$ the trivial coverage on a small category $\C$, there is a morphism of coverages $J_0 \to J$, which induces a lax functor $\b R(\C, J_0) \to \b R(\C, J)$ and hence a functor $\b R(\C, J_0)\h\b{Cat} \to \b R(\C, J)\h\b{Cat}$. The composite
  \[\Set^{\C\op} \to \b R(\C, J_0)\h\b{Cat} \to \b R(\C, J)\h\b{Cat} \to \b R(\C, J)\h\b{Cat}\]
  given consecutively by the inclusion, change of base, and cocompletion under left-adjoint weights preserves skeletality and symmetry and hence takes values in sheaves with respect to $(\C, J)$: it is precisely the sheafification functor~\cite[\S4]{betti1981symmetry}.
\end{example}

\begin{example}
  In a similar vein to \cref{Rel}, every locally small fibration over a category $\E$ with pullbacks induces a $\Span(\E)$-enriched category. If $\E$ is equipped with a coverage $J$, then the stackification of a locally small fibration may be obtained via the free cocompletion under colimits weighted by torsors (\cref{torsors})~\cite[Example~5.6]{street1983enriched}.
\end{example}

More generally, limits in $\Span(\E)$-enriched categories have been studied in the context of locally internal categories by \textcite[\S3]{betti1987completeness}.

\begin{example}[label=topological-functors]
  Every quantaloid admits \emph{large} local limits and colimits, and consequently arbitrary right lifts and extensions~\cite[\S3]{garner2014topological}. For every (possibly large) category $\B$, there is a free quantaloid $\Q_\B$ on $\B$~\cite[\S3]{rosenthal1991free}, and \textcite[Proposition~3.5]{garner2014topological} shows that the 2-category of $\Q_\B$-categories is equivalent to the full sub-2-category of $\CAT/\B$ spanned by the faithful functors. Furthermore, a $\Q_\B$-category admits all (potentially \emph{large}) colimits if and only if the corresponding faithful functor is \emph{topological} in the sense of \textcite{herrlich1974topological}~\cite[Proposition~4.5 \& Theorem~5.2]{garner2014topological}. By \cref{closure-is-cocompletion}, every $\Q_\B$-category is $\Psi$-cocompletable for every class of weights $\Psi$. Consequently, taking $\Psi$ to be the class of all (possibly large) weights, the free $\Psi$-cocompletion exists and exhibits the free topological functor on a faithful functor.
\end{example}

\begin{example}
  For a quantaloid $\Q$ and a $\Q$-category $\A$, \textcite[\S5.2]{stubbe2010hausdorff} shows that the free cocompletion of $\A$ under conical colimits witnesses a generalised notion of Hausdorff distance between subsets of $\A$.
\end{example}

\begin{example}
  Let $\B$ be a locally small category and denote by $\V = \Sigma\Set/\B$ the local cocompletion of $\B$ viewed as a locally discrete 2-category; a 1-cell $p \colon B \lto B'$ in $\Sigma\Set/\B$ comprises a set $P$ equipped with a function $P \to \B(B', B)$. $\V$-categories are in bijection with functors into $\B$~\cite[Example~4.6]{fujii2024oplax}, and a functor into $\B$ is a fibration if and only if the corresponding $\V$-category admits powers by 1-cells with singleton domain~\cite[Proposition~5.4]{fujii2024oplax}.

  $\V$ is locally complete and cocomplete and admits right extensions (by an analogous argument to that establishing that convolution monoidal structures are complete and closed), so we can form the free $\Psi$-completion of a $\V$-category, for $\Psi$ the class of weights for the aforementioned powers, which exhibits the free fibration on a category over $\B$. The process of forming free fibrations has long been known to exhibit a colax-idempotent pseudomonad~\cite[118]{street1974fibrations} and, conceptually, colax-idempotent pseudomonads correspond to completions under classes of limits. However, the perspective of enrichment in a bicategory shows that this process really is a free completion in the conventional sense. This perspective moreover extends from fibrations of categories to fibrations of enriched categories~\cite[Theorem~5.5]{fujii2024oplax}.
\end{example}

\begin{example}
  Let $\V$ be a \vb{} and let $\A$ be a monad on an object $V$ of $\V$, viewed as a one-object $\V$-category with extent $V$. Supposing the presheaf $\V$-category on $\A$ exists, its objects are precisely the $\A$-algebras (where the extent of each algebra is its domain)~\cite[Example~3.1]{kasangian1986bcategories}.
\end{example}

\begin{example}
  Denote by $\b{LexDist}$ the bicategory whose objects are small finitely complete categories, whose 1-cells are distributors preserving finite limits in their covariant parameter, and whose 2-cells are natural transformations. A monad in $\b{LexDist}$ on a finitely complete category $\A$ (\ie{} a one-object $\b{LexDist}$-enriched category with extent $\A$) corresponds via Gabriel--Ulmer duality to a finitary monad on $\b{Lex}(\A, \Set)$~\cite[Proposition~4.2]{garner2018enriched}. The free cocompletion of a one-object $\b{LexDist}$-enriched category $\A$ under copowers by left adjoints (\cref{absolute-copowers}) exhibits the $(\A\op \to \b{Lex}(\A, \Set))$-theory associated to a finitary monad~\cite[Theorem~8.1]{garner2018enriched}.
\end{example}

\subsection{A generalised Eilenberg--Watts theorem}
\label{generalised-EW-theorem}

In 1960, \citeauthor{eilenberg1960abstract} and \citeauthor{watts1960intrinsic} independently provided a characterisation of those additive functors between categories of modules for rings that are induced by tensoring with a bimodule: this characterisation has come to be known as the \emph{Eilenberg--Watts theorem}~\cite{eilenberg1960abstract,watts1960intrinsic}. To conclude, we observe that the theory we have developed herein provides a four-fold generalisation of the classical Eilenberg--Watts theorem.
\begin{enumerate}
  \item From enrichment in the category of Abelian groups to enrichment in any well behaved bicategory.
  \item From monoids in a monoidal category (which may be seen as one-object enriched categories) to (many-object) enriched categories.
  \item From a fixed class of weights (namely, the weights for small colimits) to an arbitrary class of colimit-small weights.
  \item Establishing not merely an equivalence of categories (with respect to the bimodules between two fixed monoids), but a biequivalence of bicategories (permitting the enriched categories to vary).
\end{enumerate}

\begin{theorem}
  Let $\V$ be a locally complete and locally cocomplete bicategory with right lifts, and let $\Phi$ be a class of colimit-small weights. The process of left extension and restriction along the embeddings of $\Phi$-cocompletions exhibits a biequivalence between
  \begin{itemize}
    \item the bicategory whose objects are $\V$-categories, for which a 1-cell from $\b A$ to $\b B$ is a \mbox{$\V$-functor} $\b A \to \coc\Phi{\b B}$ from $\b A$ to the $\Phi$-cocompletion of $\B$, and whose 2-cells are $\V$-natural transformations;
    \item the full sub-2-category of $\Phi$-cocomplete $\V$-categories, $\Phi$-cocontinuous $\V$-functors, and $\V$-natural transformations spanned by the $\Phi$-cocompletions.
  \end{itemize}
\end{theorem}

\begin{proof}
  Immediate from \cref{formal-Eilenberg--Watts,cocompletion-in-nice-situation}.
\end{proof}

\appendix

\section{Local (co)completeness of \vdcs}
\label{local-(co)completeness}

In this appendix, we introduce the notions of local limits and local colimits in \vdcs, necessary to establish the existence of free cocompletions of enriched categories in \cref{presheaves-and-cocompletions-in-V-Cat}, and establish the relation to local limits and local colimits in bicategories. Since \vdcs{} are not closed under the reversal of tight-cells (since 2-cells have multiary domain but unary codomain), the notion of local limit is \emph{not} dual to the notion of local colimit. Consequently, their specialisations to bicategories are also not dual: local limits, in contrast to local colimits, are not required to be stable under composition.

\subsection{Local limits}

\begin{definition}
  \label{local-limit}
  Let $\X$ be a \vdc{}, let $A, B$ be objects of $\X$, and let $p \colon \b J \to \X\lh{B, A}$ be a functor between categories. A \emph{local limit} of $p$ comprises the following data.
  \begin{enumerate}
    \item A loose-cell $\lim p \colon B \lto A$.
    \item For each $j \in \b J$, a globular 2-cell $\pi_j \colon \lim p \tto p_j$. Furthermore, for each $\jmath \colon j \to j'$ in $\b J$, the following equation should hold.
    \[
    \begin{tikzcd}[ampersand replacement=\&]
    	A \& B \\
    	A \& B
    	\arrow["{\lim p}"'{inner sep=.8ex}, "\shortmid"{marking}, from=1-2, to=1-1]
    	\arrow[""{name=0, anchor=center, inner sep=0}, equals, nfold, from=2-1, to=1-1]
    	\arrow[""{name=1, anchor=center, inner sep=0}, equals, nfold, from=2-2, to=1-2]
    	\arrow["{p_{j'}}"{inner sep=.8ex}, "\shortmid"{marking}, from=2-2, to=2-1]
    	\arrow["{\pi_{j'}}"{description}, draw=none, from=0, to=1]
    \end{tikzcd}
    \quad = \quad
    \begin{tikzcd}[column sep=large]
      A & B \\
      A & B \\
      A & B
      \arrow["{\lim p}"'{inner sep=.8ex}, "\shortmid"{marking}, from=1-2, to=1-1]
      \arrow[""{name=0, anchor=center, inner sep=0}, equals, nfold, from=2-1, to=1-1]
      \arrow[""{name=1, anchor=center, inner sep=0}, equals, nfold, from=2-2, to=1-2]
      \arrow["{p_j}"{description}, from=2-2, to=2-1]
      \arrow[""{name=2, anchor=center, inner sep=0}, equals, nfold, from=3-1, to=2-1]
      \arrow[""{name=3, anchor=center, inner sep=0}, equals, nfold, from=3-2, to=2-2]
      \arrow["{p_{j'}}"{inner sep=.8ex}, "\shortmid"{marking}, from=3-2, to=3-1]
      \arrow["{\pi_j}"{description}, draw=none, from=0, to=1]
      \arrow["{p_\jmath}"{description}, draw=none, from=2, to=3]
    \end{tikzcd}
    \]
    \item \label{local-limit-UP} For each family of 2-cells in $\X$
    \[\begin{tikzcd}
      {X_0} & \cdots & {X_n} \\
      A && B
      \arrow[""{name=0, anchor=center, inner sep=0}, "x"', from=1-1, to=2-1]
      \arrow["{q_1}"'{inner sep=.8ex}, "\shortmid"{marking}, from=1-2, to=1-1]
      \arrow["{q_n}"'{inner sep=.8ex}, "\shortmid"{marking}, from=1-3, to=1-2]
      \arrow[""{name=1, anchor=center, inner sep=0}, "{x'}", from=1-3, to=2-3]
      \arrow["{p_j}"{inner sep=.8ex}, "\shortmid"{marking}, from=2-3, to=2-1]
      \arrow["{\chi_j}"{description}, draw=none, from=0, to=1]
    \end{tikzcd}\]
    such that, for each $\jmath \colon j \to j'$ in $\b J$, the following equation holds,
    \[
    \begin{tikzcd}
      {X_0} & \cdots & {X_n} \\
      A && B \\
      A && B
      \arrow[""{name=0, anchor=center, inner sep=0}, "x"', from=1-1, to=2-1]
      \arrow["{q_1}"'{inner sep=.8ex}, "\shortmid"{marking}, from=1-2, to=1-1]
      \arrow["{q_n}"'{inner sep=.8ex}, "\shortmid"{marking}, from=1-3, to=1-2]
      \arrow[""{name=1, anchor=center, inner sep=0}, "{x'}", from=1-3, to=2-3]
      \arrow[""{name=2, anchor=center, inner sep=0}, equals, nfold, from=2-1, to=3-1]
      \arrow["{p_j}"{description}, from=2-3, to=2-1]
      \arrow[""{name=3, anchor=center, inner sep=0}, equals, nfold, from=2-3, to=3-3]
      \arrow["{p_{j'}}"{inner sep=.8ex}, "\shortmid"{marking}, from=3-3, to=3-1]
      \arrow["{\chi_j}"{description}, draw=none, from=0, to=1]
      \arrow["{p_\jmath}"{description}, draw=none, from=2, to=3]
    \end{tikzcd}
    \quad = \quad \chi_{j'}
    \]
    a unique 2-cell
    \[\begin{tikzcd}
      {X_0} & \cdots & {X_n} \\
      A && B
      \arrow[""{name=0, anchor=center, inner sep=0}, "x"', from=1-1, to=2-1]
      \arrow["{q_1}"'{inner sep=.8ex}, "\shortmid"{marking}, from=1-2, to=1-1]
      \arrow["{q_n}"'{inner sep=.8ex}, "\shortmid"{marking}, from=1-3, to=1-2]
      \arrow[""{name=1, anchor=center, inner sep=0}, "{x'}", from=1-3, to=2-3]
      \arrow["{\lim p}"{inner sep=.8ex}, "\shortmid"{marking}, from=2-3, to=2-1]
      \arrow["{\tp{\chi_j}_{j \in \b J}}"{description}, draw=none, from=0, to=1]
    \end{tikzcd}\]
    such that, for each $j \in \b J$, the following equation holds.
    \[
    \begin{tikzcd}
      {X_0} & \cdots & {X_n} \\
      A && B \\
      A && B
      \arrow[""{name=0, anchor=center, inner sep=0}, "x"', from=1-1, to=2-1]
      \arrow["{q_1}"'{inner sep=.8ex}, "\shortmid"{marking}, from=1-2, to=1-1]
      \arrow["{q_n}"'{inner sep=.8ex}, "\shortmid"{marking}, from=1-3, to=1-2]
      \arrow[""{name=1, anchor=center, inner sep=0}, "{x'}", from=1-3, to=2-3]
      \arrow[""{name=2, anchor=center, inner sep=0}, equals, nfold, from=2-1, to=3-1]
      \arrow["{\lim p}"{description}, from=2-3, to=2-1]
      \arrow[""{name=3, anchor=center, inner sep=0}, equals, nfold, from=2-3, to=3-3]
      \arrow["{p_j}"{inner sep=.8ex}, "\shortmid"{marking}, from=3-3, to=3-1]
      \arrow["{\tp{\chi_j}_{j \in \b J}}"{description}, draw=none, from=0, to=1]
      \arrow["{\pi_j}"{description}, draw=none, from=2, to=3]
    \end{tikzcd}
    \quad = \quad
    \chi_j
    \]
  \end{enumerate}
  A functor $F \colon \X \to \Y$ of \vdcs{} \emph{preserves} the local limit if the induced loose-cell $F(\lim p)$ together with the family $\{ F(\pi_j) \}_{j \in \b J}$ exhibits a local limit of $F_{B, A} \c p \colon \b J \to \Y\lh{FB, FA}$.

  A \emph{small} local limit is a $\b J$-indexed local limit with $\b J$ small.
  A \vdc{} is \emph{locally complete} if it admits small local limits; a functor is \emph{locally continuous} if it preserves small local limits.
\end{definition}

Given a class $\J$ of categories, we will use terminology such as \emph{local $\J$-limits} to refer to `local $\b J$-limits for each category $\b J \in \J$'. Note that, since the notion of local limit is symmetric with respect to the loose-cells, a \vdc{} $\X$ is locally ($\J$-) complete if and only if $\X\co$ (\cref{dual}) is locally ($\J$-) complete.

\begin{remark}
  A pseudo double category admits local products in the sense of \cite[Example~6.10]{patterson2026products} if and only if, when viewed as a representable \vdc{}, it admits local $\J$-limits, for $\J$ the class of discrete categories.
\end{remark}

\begin{lemma}
  \label{local-limits-with-restrictions}
  Let $\X$ be a \ve. $\X$ admits the local limit of a functor $p \colon \b J \to \X\lh{B, A}$ if and only if the loose \vb{} underlying $\X$ admits a local limit of $p$.
\end{lemma}

\begin{proof}
  Clearly data (1 \& 2) do not depend on tight-cells in $\X$. For each family of 2-cells in $\X$,
  \[\begin{tikzcd}
    {X_0} & \cdots & {X_n} \\
    A && B
    \arrow[""{name=0, anchor=center, inner sep=0}, "x"', from=1-1, to=2-1]
    \arrow["{q_1}"'{inner sep=.8ex}, "\shortmid"{marking}, from=1-2, to=1-1]
    \arrow["{q_n}"'{inner sep=.8ex}, "\shortmid"{marking}, from=1-3, to=1-2]
    \arrow[""{name=1, anchor=center, inner sep=0}, "{x'}", from=1-3, to=2-3]
    \arrow["{p_j}"{inner sep=.8ex}, "\shortmid"{marking}, from=2-3, to=2-1]
    \arrow["{\chi_j}"{description}, draw=none, from=0, to=1]
  \end{tikzcd}\]
  we have a family of globular 2-cells in $\X$.
  \[\begin{tikzcd}
    A & {X_0} & \cdots & {X_n} & B \\
    A &&&& B
    \arrow[""{name=0, anchor=center, inner sep=0}, equals, nfold, from=1-1, to=2-1]
    \arrow["{A(1, x)}"'{inner sep=.8ex}, "\shortmid"{marking}, from=1-2, to=1-1]
    \arrow["{q_1}"'{inner sep=.8ex}, "\shortmid"{marking}, from=1-3, to=1-2]
    \arrow["{q_n}"'{inner sep=.8ex}, "\shortmid"{marking}, from=1-4, to=1-3]
    \arrow["{B(x', 1)}"'{inner sep=.8ex}, "\shortmid"{marking}, from=1-5, to=1-4]
    \arrow[""{name=1, anchor=center, inner sep=0}, equals, nfold, from=1-5, to=2-5]
    \arrow["{p_j}"{inner sep=.8ex}, "\shortmid"{marking}, from=2-5, to=2-1]
    \arrow["{\widecheck{\chi_j}}"{description}, draw=none, from=0, to=1]
  \end{tikzcd}\]
  Since $p$ admits a local limit in the loose \vb, we have a 2-cell as follows,
  \[\begin{tikzcd}
    A & {X_0} & \cdots & {X_n} & B \\
    A &&&& B
    \arrow[""{name=0, anchor=center, inner sep=0}, equals, nfold, from=1-1, to=2-1]
    \arrow["{A(1, x)}"'{inner sep=.8ex}, "\shortmid"{marking}, from=1-2, to=1-1]
    \arrow["{q_1}"'{inner sep=.8ex}, "\shortmid"{marking}, from=1-3, to=1-2]
    \arrow["{q_n}"'{inner sep=.8ex}, "\shortmid"{marking}, from=1-4, to=1-3]
    \arrow["{B(x', 1)}"'{inner sep=.8ex}, "\shortmid"{marking}, from=1-5, to=1-4]
    \arrow[""{name=1, anchor=center, inner sep=0}, equals, nfold, from=1-5, to=2-5]
    \arrow["{\lim p}"{inner sep=.8ex}, "\shortmid"{marking}, from=2-5, to=2-1]
    \arrow["{\tp{\widecheck{\chi_j}}_{j \in \b J}}"{description}, draw=none, from=0, to=1]
  \end{tikzcd}\]
  hence a 2-cell as follows, where $\urcorner$ and $\ulcorner$ are the evident 2-cells induced by the universal property of restriction.
  \[\begin{tikzcd}[column sep=huge]
    {X_0} & {X_0} & \cdots & {X_n} & {X_n} \\
    A & {X_0} & \cdots & {X_n} & B \\
    A &&&& B
    \arrow[""{name=0, anchor=center, inner sep=0}, "x"', from=1-1, to=2-1]
    \arrow[equals, nfold, from=1-2, to=1-1]
    \arrow[""{name=1, anchor=center, inner sep=0}, equals, nfold, from=1-2, to=2-2]
    \arrow["{q_1}"'{inner sep=.8ex}, "\shortmid"{marking}, from=1-3, to=1-2]
    \arrow[""{name=2, anchor=center, inner sep=0}, "\cdots"{description}, draw=none, from=1-3, to=2-3]
    \arrow["{q_n}"'{inner sep=.8ex}, "\shortmid"{marking}, from=1-4, to=1-3]
    \arrow[equals, nfold, from=1-4, to=1-5]
    \arrow[""{name=3, anchor=center, inner sep=0}, equals, nfold, from=1-4, to=2-4]
    \arrow[""{name=4, anchor=center, inner sep=0}, "{x'}", from=1-5, to=2-5]
    \arrow[""{name=5, anchor=center, inner sep=0}, equals, nfold, from=2-1, to=3-1]
    \arrow["{A(1, x)}"{description}, from=2-2, to=2-1]
    \arrow["{q_1}"{description}, from=2-3, to=2-2]
    \arrow["{q_n}"{description}, from=2-4, to=2-3]
    \arrow["{B(x', 1)}"{description}, from=2-5, to=2-4]
    \arrow[""{name=6, anchor=center, inner sep=0}, equals, nfold, from=2-5, to=3-5]
    \arrow["{\lim p}"{inner sep=.8ex}, "\shortmid"{marking}, from=3-5, to=3-1]
    \arrow["\urcorner"{description}, draw=none, from=0, to=1]
    \arrow["{=}"{description}, draw=none, from=1, to=2]
    \arrow["{=}"{description}, draw=none, from=2, to=3]
    \arrow["\ulcorner"{description}, draw=none, from=3, to=4]
    \arrow["{\tp{\widecheck{\chi_j}}_{j \in \b J}}"{description}, draw=none, from=5, to=6]
  \end{tikzcd}\]
  This satisfies the desired universal property by the universal property for the local limit in the loose \vb{} together with the universal property of the restriction. The converse is trivial.
\end{proof}

\begin{proposition}
  \label{local-completeness-for-bicategories}
  Each hom-category of a bicategory is $\J$-complete if and only if the bicategory, when viewed as a tightly-discrete representable \vdc{}, admits local $\J$-limits in the sense of \cref{local-limit}.
\end{proposition}

\begin{proof}
  In the presence of loose-composites, it is clear that, in the universal property of a local limit \eqref{local-limit-UP}, it suffices to consider unary 2-cells. The universal property is then precisely that of a limit in the hom-category.
\end{proof}

\subsection{Local colimits}

\begin{definition}
  \label{local-colimit}
  Let $\X$ be a \vdc{}, let $A, B$ be objects of $\X$, and let $p \colon \b J \to \X\lh{B, A}$ be a functor between categories. A \emph{local colimit} of $p$ comprises the following data.
  \begin{enumerate}
    \item A loose-cell $\colim p \colon B \lto A$.
    \item For each $j \in \b J$, a globular 2-cell $\copi_j \colon p_j \tto \colim p$. Furthermore, for each $\jmath \colon j \to j'$ in $\b J$, the following equation should hold.
    \[
    \begin{tikzcd}
      A & B \\
      A & B
      \arrow[""{name=0, anchor=center, inner sep=0}, equals, nfold, from=1-1, to=2-1]
      \arrow["{p_j}"'{inner sep=.8ex}, "\shortmid"{marking}, from=1-2, to=1-1]
      \arrow[""{name=1, anchor=center, inner sep=0}, equals, nfold, from=1-2, to=2-2]
      \arrow["{\colim p}"{inner sep=.8ex}, "\shortmid"{marking}, from=2-2, to=2-1]
      \arrow["{\copi_j}"{description}, draw=none, from=0, to=1]
    \end{tikzcd}
    \quad = \quad
    \begin{tikzcd}[column sep=large]
      A & B \\
      A & B \\
      A & B
      \arrow[""{name=0, anchor=center, inner sep=0}, equals, nfold, from=1-1, to=2-1]
      \arrow["{p_j}"'{inner sep=.8ex}, "\shortmid"{marking}, from=1-2, to=1-1]
      \arrow[""{name=1, anchor=center, inner sep=0}, equals, nfold, from=1-2, to=2-2]
      \arrow[""{name=2, anchor=center, inner sep=0}, equals, nfold, from=2-1, to=3-1]
      \arrow["{p_{j'}}"{description}, from=2-2, to=2-1]
      \arrow[""{name=3, anchor=center, inner sep=0}, equals, nfold, from=2-2, to=3-2]
      \arrow["{\colim p}"{inner sep=.8ex}, "\shortmid"{marking}, from=3-2, to=3-1]
      \arrow["{p_\jmath}"{description}, draw=none, from=0, to=1]
      \arrow["{\copi_{j'}}"{description}, draw=none, from=2, to=3]
    \end{tikzcd}
    \]
    \item \label{local-colimit-UP} For each family of 2-cells in $\X$
    \[\begin{tikzcd}
      {X_0} & \cdots & A & B & \cdots & {X_n} \\
      X &&&&& {X'}
      \arrow[""{name=0, anchor=center, inner sep=0}, "x"', from=1-1, to=2-1]
      \arrow["{q_1}"'{inner sep=.8ex}, "\shortmid"{marking}, from=1-2, to=1-1]
      \arrow["{q_k}"'{inner sep=.8ex}, "\shortmid"{marking}, from=1-3, to=1-2]
      \arrow["{p_j}"'{inner sep=.8ex}, "\shortmid"{marking}, from=1-4, to=1-3]
      \arrow["{q_{k + 1}}"'{inner sep=.8ex}, "\shortmid"{marking}, from=1-5, to=1-4]
      \arrow["{q_n}"'{inner sep=.8ex}, "\shortmid"{marking}, from=1-6, to=1-5]
      \arrow[""{name=1, anchor=center, inner sep=0}, "{x'}", from=1-6, to=2-6]
      \arrow["q"{inner sep=.8ex}, "\shortmid"{marking}, from=2-6, to=2-1]
      \arrow["{\chi_j}"{description}, draw=none, from=0, to=1]
    \end{tikzcd}\]
    such that, for each $\jmath \colon j \to j'$ in $\b J$, the following equation holds,
    \[
    \begin{tikzcd}[column sep=huge]
      {X_0} & \cdots & A & B & \cdots & {X_n} \\
      {X_0} & \cdots & A & B & \cdots & {X_n} \\
      X &&&&& {X'}
      \arrow[""{name=0, anchor=center, inner sep=0}, equals, nfold, from=1-1, to=2-1]
      \arrow["{q_1}"'{inner sep=.8ex}, "\shortmid"{marking}, from=1-2, to=1-1]
      \arrow["{q_k}"'{inner sep=.8ex}, "\shortmid"{marking}, from=1-3, to=1-2]
      \arrow[""{name=1, anchor=center, inner sep=0}, equals, nfold, from=1-3, to=2-3]
      \arrow["{p_j}"'{inner sep=.8ex}, "\shortmid"{marking}, from=1-4, to=1-3]
      \arrow[""{name=2, anchor=center, inner sep=0}, equals, nfold, from=1-4, to=2-4]
      \arrow["{q_{k +1}}"'{inner sep=.8ex}, "\shortmid"{marking}, from=1-5, to=1-4]
      \arrow["{q_n}"'{inner sep=.8ex}, "\shortmid"{marking}, from=1-6, to=1-5]
      \arrow[""{name=3, anchor=center, inner sep=0}, equals, nfold, from=1-6, to=2-6]
      \arrow[""{name=4, anchor=center, inner sep=0}, "x"', from=2-1, to=3-1]
      \arrow["{q_1}"{description}, from=2-2, to=2-1]
      \arrow["{q_k}"{description}, from=2-3, to=2-2]
      \arrow["{p_{j'}}"{description}, from=2-4, to=2-3]
      \arrow["{q_{k + 1}}"{description}, from=2-5, to=2-4]
      \arrow["{q_n}"{description}, from=2-6, to=2-5]
      \arrow[""{name=5, anchor=center, inner sep=0}, "{x'}", from=2-6, to=3-6]
      \arrow["q"{inner sep=.8ex}, "\shortmid"{marking}, from=3-6, to=3-1]
      \arrow["{=}"{description}, draw=none, from=0, to=1]
      \arrow["{p_\jmath}"{description}, draw=none, from=1, to=2]
      \arrow["{=}"{description}, draw=none, from=2, to=3]
      \arrow["{\chi_{j'}}"{description}, draw=none, from=4, to=5]
    \end{tikzcd}
    \quad = \quad \chi_j
    \]
    a unique 2-cell
    \[\begin{tikzcd}
      {X_0} & \cdots & A & B & \cdots & {X_n} \\
      X &&&&& {X'}
      \arrow[""{name=0, anchor=center, inner sep=0}, "x"', from=1-1, to=2-1]
      \arrow["{q_1}"'{inner sep=.8ex}, "\shortmid"{marking}, from=1-2, to=1-1]
      \arrow["{q_k}"'{inner sep=.8ex}, "\shortmid"{marking}, from=1-3, to=1-2]
      \arrow["{\colim p}"'{inner sep=.8ex}, "\shortmid"{marking}, from=1-4, to=1-3]
      \arrow["{q_{k + 1}}"'{inner sep=.8ex}, "\shortmid"{marking}, from=1-5, to=1-4]
      \arrow["{q_n}"'{inner sep=.8ex}, "\shortmid"{marking}, from=1-6, to=1-5]
      \arrow[""{name=1, anchor=center, inner sep=0}, "{x'}", from=1-6, to=2-6]
      \arrow["q"{inner sep=.8ex}, "\shortmid"{marking}, from=2-6, to=2-1]
      \arrow["{[\chi_j]_{j \in \b J}}"{description}, draw=none, from=0, to=1]
    \end{tikzcd}\]
    such that, for each $j \in \b J$, the following equation holds.
    \[
    \begin{tikzcd}[column sep=huge]
      {X_0} & \cdots & A & B & \cdots & {X_n} \\
      {X_0} & \cdots & A & B & \cdots & {X_n} \\
      X &&&&& {X'}
      \arrow[""{name=0, anchor=center, inner sep=0}, equals, nfold, from=1-1, to=2-1]
      \arrow["{q_1}"'{inner sep=.8ex}, "\shortmid"{marking}, from=1-2, to=1-1]
      \arrow["{q_k}"'{inner sep=.8ex}, "\shortmid"{marking}, from=1-3, to=1-2]
      \arrow[""{name=1, anchor=center, inner sep=0}, equals, nfold, from=1-3, to=2-3]
      \arrow["{p_j}"'{inner sep=.8ex}, "\shortmid"{marking}, from=1-4, to=1-3]
      \arrow[""{name=2, anchor=center, inner sep=0}, equals, nfold, from=1-4, to=2-4]
      \arrow["{q_{k + 1}}"'{inner sep=.8ex}, "\shortmid"{marking}, from=1-5, to=1-4]
      \arrow["{q_n}"'{inner sep=.8ex}, "\shortmid"{marking}, from=1-6, to=1-5]
      \arrow[""{name=3, anchor=center, inner sep=0}, equals, nfold, from=1-6, to=2-6]
      \arrow[""{name=4, anchor=center, inner sep=0}, "x"', from=2-1, to=3-1]
      \arrow["{q_1}"{description}, from=2-2, to=2-1]
      \arrow["{q_k}"{description}, from=2-3, to=2-2]
      \arrow["{\colim p}"{description}, from=2-4, to=2-3]
      \arrow["{q_{k + 1}}"{description}, from=2-5, to=2-4]
      \arrow["{q_n}"{description}, from=2-6, to=2-5]
      \arrow[""{name=5, anchor=center, inner sep=0}, "{x'}", from=2-6, to=3-6]
      \arrow["q"{inner sep=.8ex}, "\shortmid"{marking}, from=3-6, to=3-1]
      \arrow["{=}"{description}, draw=none, from=0, to=1]
      \arrow["{\copi_j}"{description}, draw=none, from=1, to=2]
      \arrow["{=}"{description}, draw=none, from=2, to=3]
      \arrow["{[\chi_j]_{j \in \b J}}"{description}, draw=none, from=4, to=5]
    \end{tikzcd}
    \quad = \quad
    \chi_j
    \]
  \end{enumerate}
  A functor $F \colon \X \to \Y$ of \vdcs{} \emph{preserves} the local colimit if the induced loose-cell $F(\colim p)$ together with the family $\{ F(\copi_j) \}_{j \in \b J}$ exhibits a local colimit of $F_{B, A} \c p \colon \b J \to \Y\lh{FB, FA}$.

  A \emph{small} local colimit is a $\b J$-indexed local colimit with $\b J$ small.
  A \vdc{} is \emph{locally cocomplete} if it admits small local colimits; a functor is \emph{locally cocontinuous} if it preserves small local colimits.
\end{definition}

Given a class $\J$ of categories, we will use terminology such as \emph{local $\J$-colimits} to refer to `local $\b J$-colimits for each category $\b J \in \J$'.

\begin{definition}
  A \emph{weak local colimit} satisfies the universal property of \cref{local-colimit-UP} only when $n = k = 0$. A \emph{left-local colimit} satisfies the universal property only when $k = 0$, and a \emph{right-local colimit} only when $n = k$. A \vdc{} is \emph{weakly locally cocomplete}, \emph{left-locally cocomplete}, or \emph{right-locally cocomplete} if it admits all weak local colimits, left-local colimits, or right-local colimits respectively.
\end{definition}

(A \emph{weak local colimit} should be distinguished from a \emph{local weak colimit}, which satisfies only the existence aspect of the universal property, but nevertheless satisfies this existence property with respect to multiary 2-cells.) Note that a \vdc{} $\X$ admits (left) local colimits if and only if $\X\co$ (\cref{dual}) admits (right) local colimits.

In the presence of restrictions and right lifts, left-local colimits are furthermore local colimits. (This justifies our asking for local cocompleteness throughout \cref{presheaves-and-cocompletions-in-V-Cat}, rather than mere left-local cocompleteness.)

\begin{lemma}[label=right-lifts-imply-left-local-colimits-are-local-colimits]
  Suppose that all restrictions exist and that every loose-cell is lifting. Then every left-local colimit is moreover a local colimit.
\end{lemma}

\begin{proof}
  The proof is completely analogous to that of \cref{right-lifts-imply-left-composites-are-composites}.
\end{proof}

\begin{remark}
  The similarity between \cref{right-lifts-imply-left-composites-are-composites} and \cref{right-lifts-imply-left-local-colimits-are-local-colimits} is explained by observing that one may consider a more general notion of local colimit than \cref{local-colimit}, which is parameterised by a functor from $\J$ into the category of weights $B \lcto A$, in which morphisms are chains of (not necessarily unary) 2-cells. In particular, a local colimit in this sense that is indexed by $\b1$ and picks out a chain $(p_1, \ldots, p_n)$ in $\X$ is precisely a loose composite in the sense of \cref{opcartesian}. For our purposes, however, the more limited notion of local colimit suffices.
\end{remark}

The following extends an observation in \cite[Remark~2.1.2]{pare2013composition} from \dcs{} to \vdcs.

\begin{lemma}
  \label{local-colimits-with-restrictions}
  Let $\X$ be a \vdc{} with restrictions. $\X$ admits the local colimit of a functor $p \colon \b J \to \X\lh{B, A}$ if and only if the loose \vb{} underlying $\X$ admits a local colimit of $p$.
\end{lemma}

\begin{proof}
  Clearly data (1 \& 2) do not depend on tight-cells in $\X$. For each family of 2-cells in $\X$,
  \[\begin{tikzcd}
    {X_0} & \cdots & A & B & \cdots & {X_n} \\
    X &&&&& {X'}
    \arrow[""{name=0, anchor=center, inner sep=0}, "x"', from=1-1, to=2-1]
    \arrow["{q_1}"'{inner sep=.8ex}, "\shortmid"{marking}, from=1-2, to=1-1]
    \arrow["{q_k}"'{inner sep=.8ex}, "\shortmid"{marking}, from=1-3, to=1-2]
    \arrow["{p_j}"'{inner sep=.8ex}, "\shortmid"{marking}, from=1-4, to=1-3]
    \arrow["{q_{k + 1}}"'{inner sep=.8ex}, "\shortmid"{marking}, from=1-5, to=1-4]
    \arrow["{q_n}"'{inner sep=.8ex}, "\shortmid"{marking}, from=1-6, to=1-5]
    \arrow[""{name=1, anchor=center, inner sep=0}, "{x'}", from=1-6, to=2-6]
    \arrow["q"{inner sep=.8ex}, "\shortmid"{marking}, from=2-6, to=2-1]
    \arrow["{\chi_j}"{description}, draw=none, from=0, to=1]
  \end{tikzcd}\]
  we have a family of globular 2-cells in $\X$.
  \[\begin{tikzcd}
    {X_0} & \cdots & A & B & \cdots & {X_n} \\
    {X_0} &&&&& {X_n}
    \arrow[""{name=0, anchor=center, inner sep=0}, equals, nfold, from=1-1, to=2-1]
    \arrow["{q_1}"'{inner sep=.8ex}, "\shortmid"{marking}, from=1-2, to=1-1]
    \arrow["{q_k}"'{inner sep=.8ex}, "\shortmid"{marking}, from=1-3, to=1-2]
    \arrow["{p_j}"'{inner sep=.8ex}, "\shortmid"{marking}, from=1-4, to=1-3]
    \arrow["{q_{k + 1}}"'{inner sep=.8ex}, "\shortmid"{marking}, from=1-5, to=1-4]
    \arrow["{q_n}"'{inner sep=.8ex}, "\shortmid"{marking}, from=1-6, to=1-5]
    \arrow[""{name=1, anchor=center, inner sep=0}, equals, nfold, from=1-6, to=2-6]
    \arrow["{q(x, x')}"{inner sep=.8ex}, "\shortmid"{marking}, from=2-6, to=2-1]
    \arrow["{\widecheck{\chi_j}}"{description}, draw=none, from=0, to=1]
  \end{tikzcd}\]
  Since $p$ admits a local colimit in the loose \vb, we have a 2-cell as follows,
  \[\begin{tikzcd}
    {X_0} & \cdots & A & B & \cdots & {X_n} \\
    {X_0} &&&&& {X_n}
    \arrow[""{name=0, anchor=center, inner sep=0}, equals, nfold, from=1-1, to=2-1]
    \arrow["{q_1}"'{inner sep=.8ex}, "\shortmid"{marking}, from=1-2, to=1-1]
    \arrow["{q_k}"'{inner sep=.8ex}, "\shortmid"{marking}, from=1-3, to=1-2]
    \arrow["{\colim p}"'{inner sep=.8ex}, "\shortmid"{marking}, from=1-4, to=1-3]
    \arrow["{q_{k + 1}}"'{inner sep=.8ex}, "\shortmid"{marking}, from=1-5, to=1-4]
    \arrow["{q_n}"'{inner sep=.8ex}, "\shortmid"{marking}, from=1-6, to=1-5]
    \arrow[""{name=1, anchor=center, inner sep=0}, equals, nfold, from=1-6, to=2-6]
    \arrow["{q(x, x')}"{inner sep=.8ex}, "\shortmid"{marking}, from=2-6, to=2-1]
    \arrow["{[\widecheck{\chi_j}]_{j \in \b J}}"{description}, draw=none, from=0, to=1]
  \end{tikzcd}\]
  hence a 2-cell as follows.
  \[\begin{tikzcd}
    {X_0} & \cdots & A & B & \cdots & {X_n} \\
    {X_0} &&&&& {X_n} \\
    X &&&&& {X'}
    \arrow[""{name=0, anchor=center, inner sep=0}, equals, nfold, from=1-1, to=2-1]
    \arrow["{q_1}"'{inner sep=.8ex}, "\shortmid"{marking}, from=1-2, to=1-1]
    \arrow["{q_k}"'{inner sep=.8ex}, "\shortmid"{marking}, from=1-3, to=1-2]
    \arrow["{\colim p}"'{inner sep=.8ex}, "\shortmid"{marking}, from=1-4, to=1-3]
    \arrow["{q_{k + 1}}"'{inner sep=.8ex}, "\shortmid"{marking}, from=1-5, to=1-4]
    \arrow["{q_n}"'{inner sep=.8ex}, "\shortmid"{marking}, from=1-6, to=1-5]
    \arrow[""{name=1, anchor=center, inner sep=0}, equals, nfold, from=1-6, to=2-6]
    \arrow[""{name=2, anchor=center, inner sep=0}, "x"', from=2-1, to=3-1]
    \arrow["{q(x, x')}"{description}, from=2-6, to=2-1]
    \arrow[""{name=3, anchor=center, inner sep=0}, "{x'}", from=2-6, to=3-6]
    \arrow["q"{inner sep=.8ex}, "\shortmid"{marking}, from=3-6, to=3-1]
    \arrow["{[\widecheck{\chi_j}]_{j \in \b J}}"{description}, draw=none, from=0, to=1]
    \arrow["\cart"{description}, draw=none, from=2, to=3]
  \end{tikzcd}\]
  This satisfies the desired universal property by the universal property for the local colimit in the loose bicategory together with the universal property of the restriction. The converse is trivial.
\end{proof}

\begin{proposition}
  If $\X$ admits loose-composites, then the following are equivalent.
  \begin{enumerate}
    \item $\X$ admits a local colimit of $p$.
    \item \label{representable-local-colimits} $\X$ admits a weak local colimit of $p$, and the weak local $\colim p$ is preserved by pre- and post-whiskering of loose-cells, in the sense that the canonical 2-cell $r \odot p_j \odot s \tto r \odot \colim p \odot s$ induced by the following 2-cell exhibits the coprojection at $j \in \b J$ of a weak local colimit of $\b J \to \X\lh{B, A} \to \X\lh{S, R}$.
    \[\begin{tikzcd}[column sep=huge]
      R & A & B & S \\
      R & A & B & S \\
      R &&& S
      \arrow[""{name=0, anchor=center, inner sep=0}, equals, nfold, from=1-1, to=2-1]
      \arrow["r"'{inner sep=.8ex}, "\shortmid"{marking}, from=1-2, to=1-1]
      \arrow[""{name=1, anchor=center, inner sep=0}, equals, nfold, from=1-2, to=2-2]
      \arrow["{p_j}"'{inner sep=.8ex}, "\shortmid"{marking}, from=1-3, to=1-2]
      \arrow[""{name=2, anchor=center, inner sep=0}, equals, nfold, from=1-3, to=2-3]
      \arrow["s"'{inner sep=.8ex}, "\shortmid"{marking}, from=1-4, to=1-3]
      \arrow[""{name=3, anchor=center, inner sep=0}, equals, nfold, from=1-4, to=2-4]
      \arrow[""{name=4, anchor=center, inner sep=0}, equals, nfold, from=2-1, to=3-1]
      \arrow["r"{description}, from=2-2, to=2-1]
      \arrow["{\colim p}"{description}, from=2-3, to=2-2]
      \arrow["s"{description}, from=2-4, to=2-3]
      \arrow[""{name=5, anchor=center, inner sep=0}, equals, nfold, from=2-4, to=3-4]
      \arrow["{r \odot \colim p \odot s}", from=3-4, to=3-1]
      \arrow["{=}"{description}, draw=none, from=0, to=1]
      \arrow["{\copi_j}"{description}, draw=none, from=1, to=2]
      \arrow["{=}"{description}, draw=none, from=2, to=3]
      \arrow["\opcart"{description}, draw=none, from=4, to=5]
    \end{tikzcd}\]
  \end{enumerate}
\end{proposition}

\begin{proof}
  (1) $\implies$ (2). Trivially $\X$ satisfies the universal property with respect to unary 2-cells, so it remains to show that the colimit is preserved by composition.
  Given a 2-cell as in \cref{local-colimit-UP}, we define a 2-cell as follows.
  \[\begin{tikzcd}
    {X_0} & \cdots & R & A & B & S & \cdots & {X_n} \\
    {X_0} & \cdots & R &&& S & \cdots & {X_n} \\
    X &&&&&&& {X'}
    \arrow[""{name=0, anchor=center, inner sep=0}, equals, nfold, from=1-1, to=2-1]
    \arrow["{q_1}"'{inner sep=.8ex}, "\shortmid"{marking}, from=1-2, to=1-1]
    \arrow["{q_k}"'{inner sep=.8ex}, "\shortmid"{marking}, from=1-3, to=1-2]
    \arrow[""{name=1, anchor=center, inner sep=0}, equals, nfold, from=1-3, to=2-3]
    \arrow["r"'{inner sep=.8ex}, "\shortmid"{marking}, from=1-4, to=1-3]
    \arrow["{p_j}"'{inner sep=.8ex}, "\shortmid"{marking}, from=1-5, to=1-4]
    \arrow["s"'{inner sep=.8ex}, "\shortmid"{marking}, from=1-6, to=1-5]
    \arrow[""{name=2, anchor=center, inner sep=0}, equals, nfold, from=1-6, to=2-6]
    \arrow["{q_{k + 1}}"'{inner sep=.8ex}, "\shortmid"{marking}, from=1-7, to=1-6]
    \arrow["{q_n}"'{inner sep=.8ex}, "\shortmid"{marking}, from=1-8, to=1-7]
    \arrow[""{name=3, anchor=center, inner sep=0}, equals, nfold, from=1-8, to=2-8]
    \arrow[""{name=4, anchor=center, inner sep=0}, "x"', from=2-1, to=3-1]
    \arrow["{q_1}"{description}, from=2-2, to=2-1]
    \arrow["{q_k}"{description}, from=2-3, to=2-2]
    \arrow["{r \odot p_j \odot s}"{description}, from=2-6, to=2-3]
    \arrow["{q_{k + 1}}"{description}, from=2-7, to=2-6]
    \arrow["{q_n}"{description}, from=2-8, to=2-7]
    \arrow[""{name=5, anchor=center, inner sep=0}, "{x'}", from=2-8, to=3-8]
    \arrow["q"{inner sep=.8ex}, "\shortmid"{marking}, from=3-8, to=3-1]
    \arrow["{=}"{description}, draw=none, from=0, to=1]
    \arrow["\opcart"{description}, draw=none, from=1, to=2]
    \arrow["{=}"{description}, draw=none, from=2, to=3]
    \arrow["{\chi_j}"{description}, draw=none, from=4, to=5]
  \end{tikzcd}\]
  Consequently, this factors uniquely as the following 2-cell.
  \[\begin{tikzcd}[column sep=large]
    {X_0} & \cdots & R & A & B & S & \cdots & {X_n} \\
    {X_0} & \cdots & R & A & B & S & \cdots & {X_n} \\
    X &&&&&&& {X'}
    \arrow[""{name=0, anchor=center, inner sep=0}, equals, nfold, from=1-1, to=2-1]
    \arrow["{q_1}"'{inner sep=.8ex}, "\shortmid"{marking}, from=1-2, to=1-1]
    \arrow["{q_k}"'{inner sep=.8ex}, "\shortmid"{marking}, from=1-3, to=1-2]
    \arrow["r"'{inner sep=.8ex}, "\shortmid"{marking}, from=1-4, to=1-3]
    \arrow[""{name=1, anchor=center, inner sep=0}, equals, nfold, from=1-4, to=2-4]
    \arrow["{p_j}"'{inner sep=.8ex}, "\shortmid"{marking}, from=1-5, to=1-4]
    \arrow[""{name=2, anchor=center, inner sep=0}, equals, nfold, from=1-5, to=2-5]
    \arrow["s"'{inner sep=.8ex}, "\shortmid"{marking}, from=1-6, to=1-5]
    \arrow["{q_{k + 1}}"'{inner sep=.8ex}, "\shortmid"{marking}, from=1-7, to=1-6]
    \arrow["{q_n}"'{inner sep=.8ex}, "\shortmid"{marking}, from=1-8, to=1-7]
    \arrow[""{name=3, anchor=center, inner sep=0}, equals, nfold, from=1-8, to=2-8]
    \arrow[""{name=4, anchor=center, inner sep=0}, "x"', from=2-1, to=3-1]
    \arrow["{q_1}"{description}, from=2-2, to=2-1]
    \arrow["{q_k}"{description}, from=2-3, to=2-2]
    \arrow["r"{description}, from=2-4, to=2-3]
    \arrow["{\colim p}"{description}, from=2-5, to=2-4]
    \arrow["s"{description}, from=2-6, to=2-5]
    \arrow["{q_{k + 1}}"{description}, from=2-7, to=2-6]
    \arrow["{q_n}"{description}, from=2-8, to=2-7]
    \arrow[""{name=5, anchor=center, inner sep=0}, "{x'}", from=2-8, to=3-8]
    \arrow["q"{inner sep=.8ex}, "\shortmid"{marking}, from=3-8, to=3-1]
    \arrow["{=}"{description}, draw=none, from=0, to=1]
    \arrow["{\copi_j}"{description}, draw=none, from=1, to=2]
    \arrow["{=}"{description}, draw=none, from=2, to=3]
    \arrow["{[(q_1, \ldots, q_k, \opcart, q_{k + 1}, \ldots, q_n) \d \chi_j]_{j \in \b J}}"{description}, draw=none, from=4, to=5]
  \end{tikzcd}\]
  By opcartesianness, this is equal to the following 2-cell, from which the required property again follows from opcartesianness.
  \[\begin{tikzcd}[column sep=large]
    {X_0} & \cdots & R & A & B & S & \cdots & {X_n} \\
    && R & A & B & S \\
    {X_0} & \cdots & R &&& S & \cdots & {X_n} \\
    X &&&&&&& {X'}
    \arrow[""{name=0, anchor=center, inner sep=0}, equals, nfold, from=1-1, to=3-1]
    \arrow["{q_1}"'{inner sep=.8ex}, "\shortmid"{marking}, from=1-2, to=1-1]
    \arrow["{q_k}"'{inner sep=.8ex}, "\shortmid"{marking}, from=1-3, to=1-2]
    \arrow[""{name=1, anchor=center, inner sep=0}, equals, nfold, from=1-3, to=2-3]
    \arrow["r"'{inner sep=.8ex}, "\shortmid"{marking}, from=1-4, to=1-3]
    \arrow[""{name=2, anchor=center, inner sep=0}, equals, nfold, from=1-4, to=2-4]
    \arrow["{p_j}"'{inner sep=.8ex}, "\shortmid"{marking}, from=1-5, to=1-4]
    \arrow[""{name=3, anchor=center, inner sep=0}, equals, nfold, from=1-5, to=2-5]
    \arrow["s"'{inner sep=.8ex}, "\shortmid"{marking}, from=1-6, to=1-5]
    \arrow[""{name=4, anchor=center, inner sep=0}, equals, nfold, from=1-6, to=2-6]
    \arrow["{q_{k + 1}}"'{inner sep=.8ex}, "\shortmid"{marking}, from=1-7, to=1-6]
    \arrow["{q_n}"'{inner sep=.8ex}, "\shortmid"{marking}, from=1-8, to=1-7]
    \arrow[""{name=5, anchor=center, inner sep=0}, equals, nfold, from=1-8, to=3-8]
    \arrow[""{name=6, anchor=center, inner sep=0}, equals, nfold, from=2-3, to=3-3]
    \arrow["r"{description}, from=2-4, to=2-3]
    \arrow["{\colim p}"{description}, from=2-5, to=2-4]
    \arrow["s"{description}, from=2-6, to=2-5]
    \arrow[""{name=7, anchor=center, inner sep=0}, equals, nfold, from=2-6, to=3-6]
    \arrow[""{name=8, anchor=center, inner sep=0}, "x"', from=3-1, to=4-1]
    \arrow["{q_1}"{description}, from=3-2, to=3-1]
    \arrow["{q_k}"{description}, from=3-3, to=3-2]
    \arrow["{r \odot \colim p \odot s}"{description}, from=3-6, to=3-3]
    \arrow["{q_{k + 1}}"{description}, from=3-7, to=3-6]
    \arrow["{q_n}"{description}, from=3-8, to=3-7]
    \arrow[""{name=9, anchor=center, inner sep=0}, "{x'}", from=3-8, to=4-8]
    \arrow["q"{inner sep=.8ex}, "\shortmid"{marking}, from=4-8, to=4-1]
    \arrow["{=}"{description, pos=0.6}, draw=none, from=0, to=2-3]
    \arrow["{=}"{description}, draw=none, from=1, to=2]
    \arrow["{\copi_j}"{description}, draw=none, from=2, to=3]
    \arrow["{=}"{description}, draw=none, from=3, to=4]
    \arrow["\opcart"{description}, draw=none, from=6, to=7]
    \arrow["{=}"{description, pos=0.4}, draw=none, from=2-6, to=5]
    \arrow["{[(q_1, \ldots, q_k, \opcart, q_{k + 1}, \ldots, q_n) \d \chi_j]_{j \in \b J}}"{description}, draw=none, from=8, to=9]
  \end{tikzcd}\]
  (2) $\implies$ (1). Given a family $\{ \chi_j \}_{j \in \b J}$ as in \cref{local-colimit-UP}, we have a 2-cell as follows.
  \[\begin{tikzcd}
    {X_0} &&&&& {X_n} \\
    X &&&&& {X'}
    \arrow[""{name=0, anchor=center, inner sep=0}, "x"', from=1-1, to=2-1]
    \arrow["{q_1 \odot \cdots \odot q_k \odot p_j \odot q_{k + 1} \odot \cdots \odot q_n}"'{inner sep=.8ex}, "\shortmid"{marking}, from=1-6, to=1-1]
    \arrow[""{name=1, anchor=center, inner sep=0}, "{x'}", from=1-6, to=2-6]
    \arrow["q"{inner sep=.8ex}, "\shortmid"{marking}, from=2-6, to=2-1]
    \arrow["{\pshm{\chi_j}}"{description}, draw=none, from=0, to=1]
  \end{tikzcd}\]
  This consequently factors uniquely as the following, from which the required universal property follows from opcartesianness.
  \[\begin{tikzcd}[column sep=huge]
    {X_0} &&&&& {X_n} \\
    {X_0} &&&&& {X_n} \\
    X &&&&& {X'}
    \arrow[""{name=0, anchor=center, inner sep=0}, equals, nfold, from=1-1, to=2-1]
    \arrow["{q_1 \odot \cdots \odot q_k \odot p_j \odot q_{k + 1} \odot \cdots \odot q_n}"'{inner sep=.8ex}, "\shortmid"{marking}, from=1-6, to=1-1]
    \arrow[""{name=1, anchor=center, inner sep=0}, equals, nfold, from=1-6, to=2-6]
    \arrow[""{name=2, anchor=center, inner sep=0}, "x"', from=2-1, to=3-1]
    \arrow["{q_1 \odot \cdots \odot q_k \odot \colim p \odot q_{k + 1} \odot \cdots \odot q_n}"{description}, from=2-6, to=2-1]
    \arrow[""{name=3, anchor=center, inner sep=0}, "{x'}", from=2-6, to=3-6]
    \arrow["q"{inner sep=.8ex}, "\shortmid"{marking}, from=3-6, to=3-1]
    \arrow["{\copi_{q_1 \odot \cdots \odot q_k \odot p_j \odot q_{k + 1} \odot \cdots \odot q_n}}"{description}, draw=none, from=0, to=1]
    \arrow["{[\pshm{\chi_j}]_{j \in \b J}}"{description}, draw=none, from=2, to=3]
  \end{tikzcd}\qedshift\]
\end{proof}

\begin{remark}[label=left-composites-and-whiskering]
  More generally, if $\X$ admits left-composites, then $\X$ admits a left-local colimit of $p$ if and only if $\X$ admits a weak local colimit of $p$, and $\colim p$ is preserved by pre-whiskering of loose-cells, in the sense that the canonical 2-cell $p_j \odotl s \tto \colim p \odotl s$ exhibits the coprojection at $j \in \b J$.
\end{remark}

\begin{corollary}
  A pseudo double category admits local $\J$-colimits in the sense of \cite[Definition~2.1.1]{pare2013composition} if and only if, when viewed as a representable \vdc{}, it admits local $\J$-colimits in the sense of \cref{local-colimit}.
\end{corollary}

\begin{proof}
  The definition of \cite{pare2013composition} is essentially the same as that of \cref{representable-local-colimits}, except that it also asks for each colimit to exhibit a colimit in the hom-category of the loose bicategory: this condition is redundant, as it is the special case of \cref{local-colimit-UP} for globular 2-cells.
\end{proof}

\begin{corollary}
  \label{local-cocompleteness-for-bicategories}
  A bicategory admits an enrichment in the monoidal 2-category of $\J$-cocomplete categories of \textcite[\S15.8]{garner2016enriched} if and only if, when viewed as the loose bicategory of a tightly-discrete representable \vdc{}, it admits local $\J$-colimits in the sense of \cref{local-colimit}.
\end{corollary}

\begin{proof}
  By \cref{representable-local-colimits}, a bicategory is locally $\J$-cocomplete in the sense of \cite[\S15.1]{garner2016enriched} if and only if it admits local $\J$-colimits in the sense of \cref{local-colimit}, hence also if and only if it admits such an enrichment.
\end{proof}

\printbibliography

\end{document}